\documentclass[reqno,11pt]{amsart}

\usepackage{amsfonts,amsthm,amsmath,amssymb,amscd,mathrsfs}
\usepackage{graphics}
\usepackage{indentfirst}
\usepackage{cite}
\usepackage{latexsym}
\usepackage{color}
\usepackage{bookmark}
\usepackage{soul}
\usepackage{nccmath}

\usepackage{tikz}
\usetikzlibrary{arrows.meta,positioning,calc,decorations.pathreplacing}

\newtheorem{theorem}{Theorem}[section]
\newtheorem{corollary}[theorem]{Corollary}
\newtheorem{remark}{Remark}[section]
\newtheorem{definition}{Definition}[section]
\newtheorem{lemma}[theorem]{Lemma}
\newtheorem{proposition}[theorem]{Proposition}

\newcommand{\supp}{\operatorname{supp}}
\newcommand{\R}{\mathbb{R}}
\renewcommand{\div}{\mathrm{div}}

\begin{document}

\title{Sharp Ill-Posedness of the Euler Equations in Lorentz Spaces}

\begin{abstract}
We study vortex stretching for the three-dimensional axisymmetric Euler equations without swirl in vorticity formulation.
     Danchin \cite{Danchin07} established global existence and uniqueness for bounded vorticity $\omega_0$ provided $\omega_0/r$ lies in the endpoint Lorentz space $L^{3,1}(\mathbb{R}^3)$ (together with a decay assumption on $\omega_0$).
    We prove that this $L^{3,1}$ endpoint is sharp: for every Lorentz  exponent $q>1$, we construct multi-ring data $\omega_0 \in L^\infty (\mathbb{R}^3)$ with $\omega_0/r\in L^{3,q}(\mathbb{R}^3)$ that produce $L^\infty$-norm inflation of the vorticity; moreover, within the same class, we obtain instantaneous blow-up from data with infinitely many rings. 

    Our initial data are inspired by the Kim--Jeong dyadic ring superposition \cite{KimJeong22}, but we crucially generalize it by allowing flexible conical support geometry for the ring profile. In the regime where outer rings are dominant---a multiscale viewpoint appearing in recent works including \cite{KimJeong22,CordobaMartinezZheng25}---we obtain a forward-in-time ODE cascade for ring amplitudes and aspect ratios in which vortex stretching weakens its own future forcing: as a ring amplifies, incompressibility flattens it, the aspect ratio collapses, and the induced stretching coefficient is geometrically depleted. A key new ingredient is a profile-localization argument that freezes the relevant Biot–Savart kernel and makes this depletion explicit, enabling us to exploit a monotone ``productive window" (controlled by the cone slope) together with an exact cascade identity. This propagates stretching across scales and gives a robust lower bound on cumulative stretching, yielding ill-posedness in the full range $q>1$.    
\end{abstract}

\author{Jeaheang Bang}
\address{Institute for Theoretical Sciences, Westlake Institute for Advanced Study, Westlake University, 600 Dunyu Road, Hangzhou, Zhejiang 310024, People's Republic of China}
\email{jhbang@westlake.edu.cn}

\author{Alexey Cheskidov}
\address{Institute for Theoretical Sciences, Westlake University, 600 Dunyu Road, Hangzhou, Zhejiang 310024, People's Republic of China}
\email{cheskidov@westlake.edu.cn}

\maketitle

\setcounter{tocdepth}{2}
\tableofcontents

\section{Introduction}
Consider the initial value problem of the incompressible Euler equations in $\mathbb{R}^3$:
    \begin{align}
    \label{eq:Euler}
    \left\{
        \begin{aligned}
           \partial_t u + u\cdot \nabla u + \nabla p
        &=
            0, 
           \qquad \mathrm{div}\, u=0
           \\
           u(x,0)
        &=
            u_0(x).
        \end{aligned}
    \right.
    \end{align}
Equivalently, one can also consider its vorticity formulation:
    \begin{align}
    \label{main_eq}
    \left\{
    \begin{aligned}
        \partial_t \omega + u\cdot \nabla \omega = \omega \cdot \nabla u,
    \\
        \omega(x,0)=\omega_0(x)
    \end{aligned}
    \right.
    \end{align}
Here,  $u$ is determined by the Biot--Savart law: $u=(-\Delta)^{-1} \mathrm{curl} \, \omega$, that is,
    \begin{align}
    \label{eq:gen_Biot-Savart}
        u(x)= \frac{1}{4\pi}
        \int_{\mathbb{R}^3} 
        \frac{\omega (y)\times (x-y) }{|x-y|^3} \, dy.
    \end{align}
See our sign convention related to the Biot--Savart law in Appendix~\ref{sec:notation}.

Well-posedness and ill-posedness of the Euler equations have long been central topics in mathematical fluid mechanics. 
Advection tends to mix the flow, while the vortex stretching mechanism can amplify vorticity. 
The interplay between these two effects plays a crucial role in the analysis of well-posedness and ill-posedness. 

In two dimensions, vortex stretching is absent, and the vorticity is purely transported along  particle trajectories. As a consequence, the $L^\infty$ norm of the vorticity is conserved in time. Although the condition $\omega\in L^\infty$ only yields log-Lipschitz continuity of the velocity field $u$, this regularity is sufficient to define the Lagrangian flow map $Y(x,t)$ through
    \begin{align*}
        \frac{\partial }{\partial t} Y(x,t) = u(Y(x,t),t ), \quad Y(x,0)=x.
    \end{align*}
This observation underlies Yudovich theory \cite{Yudovich63}, which establishes global existence and uniqueness of solutions  for initial vorticity $\omega_0 \in L^\infty (\mathbb{R}^2)$ under a mild decay condition at infinity. 

This global result shows that, although the formation of small scales is not excluded---since no assumption is imposed on the vorticity gradient--- the sole boundedness of the initial vorticity is sufficient to ensure global well-defined Lagrangian dynamics.

On the other hand, in three dimensions, the presence of vortex stretching fundamentally alters the dynamics and renders the extension of Yudovich theory highly nontrivial. In general, the amplification of vorticity induced by stretching may lead to loss of regularity, even for initially smooth flows. However, in the special class of axisymmetric flows without swirl, the structure of the vorticity equation simplifies significantly:
    \begin{align}
    \label{vor_eq}
        \partial_t \omega
        +u^r \partial_r \omega
        + u^z \partial _z \omega
        &= \omega \, \frac{u^r}{r}.
    \end{align}
By abuse of notation, here $\omega$ still denotes a scalar quantity defined by
    \begin{align*}
        \operatorname{curl} u= \omega e_\theta, 
    \end{align*}
where $e_\theta= e_z\times e_r$, and
$\theta$ denotes the angular variable in cylindrical coordinates. 

In addition, we will identify 
meridional $(r,z)-$functions, such as $\omega(r,z)$, with their axisymmetric lifts to $\mathbb{R}^3$. 
The axisymmetric Biot--Savart law reads as follows.
    \begin{align}
    \begin{aligned}
    \label{eq:axis_BS_law}
        u^r (x)
    &=
        \frac{1}{4\pi} \int_{\mathbb{R}^3}
            \frac{ (z_x-z_y)\, \cos (\theta_x-\theta_y ) \,  }{|x-y|^3}
            \omega(y) \, dy,
     \\
     u^z(x)
     &=-\frac{1}{4\pi}
     \int_{\mathbb{R}^3} 
     \frac{r_x \cos (\theta_x-\theta_y) -r_y}{|x-y|^3} \, \omega (y) \, dy,
     \end{aligned}
    \end{align}
where $(r_x,\theta_x,z_x)$ denotes cylindrical coordinates of $x$.
We  call $u^r/r$, that appears on the right-hand side of the equation, \emph{vortex stretching rate}. See our sign convention related to the Biot--Savart law in Appendix~\ref{sec:notation}.

\subsection{A Global Estimate in Lorentz Spaces}
For the axisymmetric flows without swirl, the vorticity equation \eqref{vor_eq} can be re-written as
    \begin{align}
    \label{eq:rel_vor_eq}
        \partial_t \left( \frac{\omega}{r} \right) 
    +
        u^r \partial_r \left( \frac{\omega}{r} \right)
    +
        u^z \partial_z \left( \frac{\omega}{r} \right)=0.
    \end{align}
Hence, the so-called relative vorticity $\omega/r$ is transported along particle trajectories. This partial suppression of vortex stretching makes the axisymmetric, no-swirl case the closest three-dimensional analogue of the two-dimensional case.

Motivated by Yudovich theory, one may attempt to establish a \emph{global} well-posedness result in this setting under a minimal additional assumption controlling vortex stretching. Since $\omega/r$ is conserved along the flow, it is natural to impose an assumption on the initial quantity $\omega_0/r$. From a scaling perspective, the condition $\omega_0/r \in L^{3,q}(\mathbb{R}^3)$ for $q \ge 1$ has the same homogeneity as the main assumption $\omega_0 \in L^\infty(\mathbb{R}^3)$. 

Moreover, like the $L^\infty$ condition in Yudovich theory, such a Lorentz-space assumption does not impose any explicit Sobolev or H\"older regularity on the vorticity. While the condition $\omega_0/r\in L^{3,q}$ in fact does enforce a mild constraint near the symmetry axis, it remains a borderline assumption that does not directly control vorticity gradients. (A detailed discussion about possible singular behaviors in this class can be found in the next section \ref{General_L_space}.)
Thus, the family of Lorentz spaces $L^{3,q}$ provides a natural class in which to seek a Yudovich-type global result. Now, the central question is whether such a condition is sufficient to control the vortex stretching rate $u^r/r$.

To answer this question, one observes that $\omega/r$ and $u^r/r$ behave as if they are related by the singular kernel $|x|^{-2}$. Since this kernel belongs to the Lorentz space $L^{\frac{3}{2},\infty}(\mathbb{R}^3)$, one obtains the estimate
\begin{align}
\label{Danchin_est}
\left\| \frac{u^r}{r} \right\|_{L^\infty(\mathbb{R}^3)}
\le C
\left\| \frac{\omega}{r} \right\|_{L^{3,1}(\mathbb{R}^3)}
\left\| \frac{1}{|x|^2} \right\|_{L^{\frac{3}{2},\infty}(\mathbb{R}^3)}
\le C
\left\| \frac{\omega}{r} \right\|_{L^{3,1}(\mathbb{R}^3)}.
\end{align}
Here, among the Lorentz spaces $L^{3,q}, q\geq 1$, the space $L^{3,1}$ arises naturally from duality, since $(L^{3,1})^* = L^{\frac{3}{2},\infty}$.

If $\omega_0/r \in L^{3,1}(\mathbb{R}^3)$, then the Lorentz norm of $\omega/r$ is conserved in time, and the estimate \eqref{Danchin_est} yields a uniform-in-time bound on the vortex stretching rate $u^r/r$. This observation explains the main heuristic behind Danchin’s global existence result. More precisely, Danchin proved global existence and uniqueness of solutions under the assumptions
\begin{align}
\label{Danchin_cond}
\omega_0 \in L^\infty(\mathbb{R}^3) \cap L^{3,1}(\mathbb{R}^3),
\qquad
\frac{\omega_0}{r} \in L^{3,1}(\mathbb{R}^3).
\end{align}
The additional condition $\omega_0 \in L^{3,1}(\mathbb{R}^3)$ ensures sufficient decay at infinity; see
\cite{Serfati94, Saint-Raymond94, ShirotaYanagisawa94, AbidiHmidiKeraani10} for related results in various settings.

In contrast to the two-dimensional case, the additional Lorentz condition $\omega_0/r\in L^{3,1}(\mathbb{R}^3)$ in \eqref{Danchin_cond} is needed to control the stretching mechanism.
However, the assumption $\omega_0/r \in L^{3,1}(\mathbb{R}^3)$ is weak in the sense that it has the same scaling as the primary boundedness condition $\omega_0 \in L^\infty(\mathbb{R}^3)$. Moreover, the main estimate \eqref{Danchin_est} fails if $L^{3,1}$ is replaced by $L^{3,q}$ for any $q>1$. For this reason, Danchin’s result is  regarded as a sharp extension of Yudovich theory to the three-dimensional axisymmetric no-swirl setting.

\subsection{Worst-Case Singular Geometry in \texorpdfstring{$L^{3,q}$}{L3q}}\label{General_L_space}

More precisely, Danchin's result is expected to be sharp in the sense that if one
instead considers Lorentz spaces $L^{3,q}(\mathbb{R}^3)$ for any $q>1$ in
\eqref{Danchin_cond}, then ill-posedness might occur; see
\cite[p.~32, footnote~3]{LimJeong25}.

To investigate this ill-posedness, let us consider a worst-case scenario in the
class
\[
\omega_0 \in L^\infty(\mathbb{R}^3),
\qquad
\frac{\omega_0}{r} \in L^{3,q}(\mathbb{R}^3),
\]
for a fixed $q>1$ (together with a decay condition at infinity).

First, any potentially singular behavior is strongest when it is concentrated on the $z$-axis. 
Otherwise,
it does not contribute to amplifying the vortex stretching in view of Danchin's
result. In addition, if the singularity is uniformly distributed along the $z$-axis, then the Lorentz
condition in \eqref{Danchin_cond} forces $\omega$ to be H\"older continuous along
the axis; in that case one has a classical local well-posedness theory (explained
below). Hence, it is natural  to consider a \emph{singularity concentrated at a point}
on the axis, say the origin $(r,z)=(0,0)$.

In addition, since $\omega/r$ is conserved along particle trajectories, outward radial motion of a fluid particle amplifies the vorticity. This suggests
considering flows that are \emph{odd in $z$} such that 
\[
\omega(r,z)\le 0 \qquad \text{when} \quad  z>0.
\]

In this case, $\omega_0/r$ should behave like a typical example of a singular
function that barely belongs to $L^{3,q}_{\mathrm{loc}}(\mathbb{R}^3)$:
\begin{align*}
        \frac{\omega_0}{r}
    \sim 
        -\frac{1}{|x|} 
        \frac{1}{\log ^\alpha (1/|x|)}\mathrm{sgn}(z),
    \qquad \text{that is,}\qquad  
        \omega_0 \sim -\frac{r}{|x|} 
        \frac{1}{\log^\alpha (1/|x|)}\mathrm{sgn}(z),
\end{align*}
where $\alpha q>1$.

Near the origin, one expects $u^r_0(r,z)\approx r\,\partial_r u^r_0(0,0)$ (using
that $u^r_0(0,0)=0$ and $\partial_z u^r_0 (0,0)=0$ due to the axisymmetry). Therefore, one can approximate
the initial vortex stretching rate $u^r_0/r$ via the Biot--Savart law as follows:
\begin{align}
\label{ini_vs}
\begin{aligned}
        \frac{u^r_0(r,z)}{r}
    \approx
        \partial_r u^r_0(0,0)
    &=
        C\iint \frac{r^2\, z}{(r^2+z^2)^{5/2}} \,\big(- \omega_0 (r,z) \big) \, dr\,dz
        \\
    &\sim
        C\iint_{r\sim |z|} \frac{r^2\, z}{(r^2+z^2)^{5/2}} \,\big(- \omega_0 (r,z) \big) \, dr\,dz \ge 0.
\end{aligned}
\end{align}
Here, the kernel carries factors $r^2$ and $z$, so the dominant contribution near
the origin comes from the region where $r$ and $|z|$ are comparable. Restricting
to $r\sim |z|$ therefore still captures the worst singular behavior.

Consequently, we focus on the region $r\sim |z|$, in which
\begin{align}
\label{worst_m}
        \omega_0 \sim -\frac{1}{\log^\alpha (1/|x|)} \mathrm{sgn}(z).
\end{align}
This provides a simple worst-case scenario in the class
$\omega_0/r\in L^{3,q}(\mathbb{R}^3)$.
Note that if $q>1$ is close to $1$, then one may take $\alpha<1$ arbitrarily close
to $1$ while still having $\alpha q>1$.

\subsection{Ill-Posedness Beyond Danchin's Regime}
\label{sharp_Danchin}
To study ill-posedness beyond Danchin's regime, we will focus on norm inflation or instantaneous blow-up and use the worst-case scenario
\eqref{worst_m} as a model for that.

One of the main motivations behind ill-posedness is that for any $q>1$ (equivalently,
for any $\alpha<1$ in \eqref{worst_m}), the initial vortex stretching rate $u^r_0/r$ near the singular point can become very large.

A first (too naive) heuristic is to insert \eqref{worst_m} into \eqref{ini_vs} and obtain that for small $x$,
\begin{align*}
    \frac{u_0^r(x)}{r}
    \sim C \int_{\substack{|y|\le C,\\r_y \sim |z_y|,\\z_y\geq 0}}\frac{-\omega_0(y)}{|y|^3}\,dy
    \sim C\int_{\substack{|y|\le C,\\ r_y\sim |z_y|}}\frac{1}{|y|^3\log^\alpha(1/|y|)}\,dy=\infty
\end{align*}
 for any $\alpha<1$ (indeed, this integral diverges even for $\alpha=1$). Hence the stretching rate becomes very large near the singular point.

However, this heuristic does not take into account the competition between the blow-up of vortex stretching rate $u^r/r$ and decay of the initial data $\omega_0$, which occurs as one approaches  the singular point $(r,z)=(0,0)$.
More precisely, the vorticity equation \eqref{vor_eq} can be re-written via the Lagrangian flow map $Y$ as
    \begin{align}
    \label{Cauchy_for}
          \omega(Y(r,z;t),t)
    &=
        \omega_0 (r,z) 
        \exp\left(\int_0^t \frac{u^r(Y(r,z;s),s)}{Y^r(r,z;s)} \, ds \right).
    \end{align}
As $|x|\to 0$, the initial data $\omega_0$ on the right-hand side decays to $0$ while the vortex stretching factor blows up.

To account for this competition, 
    \begin{align}
    \label{str_ini_vs2}
        \left| \frac{u^r_0(x)}{r} \right|
        \leq
        C\int_{\substack{r_y\sim |z_y|, \\ |y|\leq 1/2}} \frac{1}{|x-y|^2 \log ^{\alpha}(1/|y|)} dy\sim \log ^{1-\alpha} (1/|x|) 
    \end{align}
where $x$ is also in a region where $r\sim |z|$.
Indeed, we decompose the domain of integration into the outer region: $|y|\geq 2 |x|$ and the remaining region $|y|\leq 2|x|$ and  denote each integral by $I_{\mathrm{out}}, I_{\mathrm{rem}}$. Then
    \begin{align*}
        I_{\mathrm{out}}\sim 
        \int_{2|x|\leq |y|\leq1/2} \frac{1}{|y|^2} \frac{1}{\log^\alpha (1/|y|)}dy\sim \log ^{1-\alpha} \left( \frac{1}{|x|} \right), 
    \qquad
        I_{\mathrm{rem}} 
        \sim \frac{1}{\log^{\alpha} (1/|x|)}.
    \end{align*}
Hence the claim \eqref{str_ini_vs2} follows, and note that the main contribution arises from the outer region $|y|\geq 2 |x|.$

Moreover, if $\omega_0$ is odd in $z$ and $\omega_0\leq0$ for $z>0$, then 
the dominant contribution to the Biot--Savart integral has a definite sign. In particular, the outer-region contribution yields  a genuine lower bound for $u_0^r (x)/r$ of order $\log^{1-\alpha}(1/|x|)$.

By contrast, in the borderline case $\alpha=1$ one instead has $\frac{u^r_0(x)}{r}\sim \log\log\!\Big(\frac1{|x|}\Big).$

With this in mind, for very short times, the trajectory has not moved far, and we may heuristically
approximate the stretching factor by its initial value, obtaining from \eqref{Cauchy_for} that for $z>0$,
\begin{align*}
    \omega(Y(r,z;t),t)
    &\sim
    \omega_0(r,z)\exp\!\left(Ct\,\frac{u^r_0(x)}{r}\right)
    \sim
    -\frac{1}{\log^\alpha\!\frac1{|x|}}
    \exp\!\left(Ct\,\log^{1-\alpha}\!\frac1{|x|}\right)
    \to -\infty
    \quad \text{as }|x|\to 0,
\end{align*}
for any fixed $t>0$ provided $\alpha<1$.

Thus the initial vortex stretching rate is strong enough to overwhelm the
initial decay of $\omega_0$ near the singular point for every $\alpha<1$ (i.e.\ for every
$q>1$). In contrast, when $\alpha=1$ the growth of $\frac{u^r_0(x)}{r}$ is only of size
$\log\log \frac1{|x|}$, and the above short-time heuristic does not force instantaneous
blow-up  of $\omega$.

\subsection{Geometric Self-Slowdown Mechanism}
\label{geo_ss_mc}

Although the approximation by the initial stretching rate $u_0^r/r$ in
Subsection~\ref{sharp_Danchin} suggests ill-posedness, there is another
important mechanism that competes with strong vortex stretching and can
significantly weaken it over time. We refer to this effect as the
\emph{geometric self-slowdown mechanism}. 
This is also referred to as squeezing of vorticity; see \cite[p.~8]{KimJeong22}.

We begin with a qualitative explanation that does not rely on any detailed
kernel computation. Consider a portion of vorticity concentrated in a region
where
$
r \sim R, \,\, z \sim \pm H,$ more precisely $|r-R|\leq cR, |z\pm H|\leq cH $. (In a general distribution of vorticity, we decompose the vorticity into small regions like this.)
As the flow evolves,
vortex stretching tends to increase the radial scale $R$ of this region.
Since the flow is incompressible, this radial expansion must be accompanied
by a collapse of the vorticity support toward the mid-plane $z=0$, so that
the vertical scale $H$ decreases in time. Consequently, the aspect ratio
$H/R$ becomes smaller as vortex stretching proceeds.

The key observation is that the Biot--Savart law governing the vortex
stretching rate $u^r/r$ is essentially homogeneous and introduces no preferred
length scale. As a result, the contribution of a vortex packet to the
stretching rate $u^r/r$ depends primarily on the \emph{shape} of the vorticity
distribution rather than its absolute size. In particular, for a vortex
packet that becomes increasingly flattened in the $z$-direction, the
contribution to the stretching rate is suppressed as $H/R$ decreases.

This suppression has a simple geometric origin. Since the relevant kernel is
odd in $z$, the leading contribution is controlled by a vertical moment, $\int z\, \omega \, dz$, of
the vorticity distribution. As the support collapses toward $z=0$, this
moment decreases, and so does the induced vortex stretching. Thus, vortex
stretching deforms the vorticity support in a way that weakens further
stretching.

To illustrate this slowdown mechanism more concretely, we now present a
representative heuristic computation based on the structure of the
axisymmetric Biot--Savart law. Suppose that $\omega$ is essentially supported
in a region where $r\sim R$ and $z\sim \pm H$, and consider a point $x$ located
much closer to the origin than to the support of $\omega$. Using \eqref{ini_vs}, we estimate
\begin{align*}
    \frac{u^r(x)}{r}
    &\sim
    \|\omega\|_{L^\infty}
    \iint_{ \substack{r\sim R,\\ z\sim H}}
    \frac{r^2 z}{(r^2+z^2)^{5/2}} \, dr\,dz.
\end{align*}
Since $r\sim R$ on the support of $\omega$, the kernel contributes a factor of
order $R^{-2}$ as $H/R $ becomes small. 
Then we obtain
\begin{align*}
    \frac{u^r(x)}{r}
    &\sim
    \frac{\|\omega\|_{L^\infty}}{R^2}
    \int_{z\sim H} z \, dz
    \sim
    \|\omega\|_{L^\infty}\,\frac{H^2}{R^2}.
\end{align*}
This computation shows explicitly that the stretching rate decreases as the
aspect ratio $H/R$ becomes small. 

In summary, strong initial vortex stretching does not by itself guarantee
ill-posedness. One must instead analyze the competition between the initial
amplification of the stretching rate $u^r/r$ and its subsequent suppression due to
the self-slowdown mechanism.

\subsection{Kim--Jeong's Ill-Posedness for Large \texorpdfstring{$q$}{q}}

The heuristic analysis in Subsection~\ref{sharp_Danchin} shows that, for the
worst-case scenario \eqref{worst_m}, the dominant contribution to the initial vortex
stretching rate arises from an \emph{outer region} of the Biot--Savart
integral. Kim and Jeong \cite{KimJeong22} exploit this same mechanism in a
multiscale setting and show that, for sufficiently large $q\gg1$, the
outer-region stretching can be made effective on suitable time scales before
the slowdown mechanism plays a role, leading to norm inflation.

Motivated by \eqref{worst_m}, Kim--Jeong replace the continuous logarithmic
profile by a dyadic superposition of localized vortex rings. More precisely,
they consider initial data of the form
\begin{align}
\label{KJdata}
    \omega_0^{(m)}(r,z)
    =
    \sum_{k=n_0}^{m} \omega_{0,k}(r,z),
    \qquad
    \omega_{0,k}(r,z)
    =
    \frac{1}{k^\alpha}
    \phi\!\left( \frac{r}{(1/8)^{k-1}}, \frac{z}{(1/8)^k} \right),
    \quad \frac{1}{q}<\alpha<1.
\end{align}
where $\phi$ is a fixed profile that is odd in $z$, nonpositive for $z>0$, and
supported near $(r,z)=(1,\pm1)$. Each $\omega_{0,k}$ represents a vortex ring
localized at spatial scale $1/8^k$, and the amplitude $k^{-\alpha}$ reflects the
logarithmic decay in \eqref{worst_m}. In particular, $\omega_0^{(m)}/r \in
L^{3,q}(\mathbb{R}^3)$ since $1/q<\alpha<1$.  See our sign convention related to the Biot--Savart law in Appendix~\ref{sec:notation}.

Kim--Jeong's key lemma turns the outer-region mechanism in the Biot--Savart law into an estimate usable for short times after $t=0$: it writes the stretching rate $u^r(x,t)/r$ as a sign-definite outer-region
integral up to controllable errors. When $x$ is associated with
the $k$-th ring and the rings remain well-separated, the main term depends essentially
only on vorticity at larger scales $n<k$.

A further issue in using the strong initial stretching is that the outer-region contribution could, in principle, weaken
due to geometric flattening (self-slowdown). Kim--Jeong  overcame this difficulty by combining
the key lemma with a scale-dependent \emph{stability} statement for the outer rings:
on a scale-dependent time interval, each outer ring remains localized in a
comparable annulus, so that the contributions by the outer rings stay
comparable to its initial value. Summing these stable outer-ring contributions yields a
robust lower bound on the time-integrated stretching rate felt by the $k$-th ring.

The
assumption that $q$ is large (equivalently, that $\alpha$ is small) is crucial here,
since it ensures that the amplitudes of the outer rings decay slowly with scale, allowing
the outer-ring contributions to accumulate before the self-slowdown mechanism becomes
effective.

Although the construction replaces the continuous worst-case profile \eqref{worst_m} by a
dyadic superposition of localized rings, this still captures the strong initial vortex stretching rate $u^r_0/r$.
 In addition, the dyadic superposition provides a \emph{scale-by-scale
decomposition} of the outer-region contribution, and the contribution of each outer ring can be tracked (for a short, scale-dependent
time) before geometric self-slowdown degrades its effect.

Kim–Jeong show that outer-region forcing can produce norm inflation when 
$q$ is sufficiently large because the dyadic amplitudes decay slowly enough for cumulative stretching to build up before geometric flattening suppresses the Biot--Savart coefficient. The main difficulty in the present paper is to recover this mechanism for every 
$q>1$, where the amplitudes decay much faster. Our key new idea is to allow a flexible conical geometry for the rings and to localize the relevant Biot--Savart kernel, which isolates a monotone productive window and yields an effective ODE cascade even near the endpoint 
$q=1$.

\subsection{Initial Data and Main Results}

Motivated by \eqref{KJdata}, we use generalized data in the following form:
    \begin{align}
\label{KJdata2}
    \omega_0^{(m)}(r,z)
    =
    \sum_{k=1}^{m} \omega_{0,k}(r,z),
    \qquad
    \omega_{0,k}(r,z)
    =
    \frac{\varepsilon}{k^\alpha}
    \phi\!\left( \frac{r}{d^k}, \frac{z}{d^k} \right),
    \quad d:=10^{-2}
\end{align}
where $\phi$ is a fixed smooth profile satisfying
\begin{align}
\label{cond_phi}
    \begin{cases}
    \text{(i)}\ \phi \text{ is odd in } z,\\
    \text{(ii)}\ \phi \le 0 \text{ for } z>0,\\
    \text{(iii)}\ \|\phi\|_{L^\infty(\mathbb{R}^3)}\geq 1,\\
    \text{(iv)}\ 
        \mathrm{supp} \phi\cap\{z>0\}\ \subset\
\big[(1-\eta)r_0,(1+\eta)r_0\big]\times \big[(1-\eta)z_0,(1+\eta )z_0\big],
    \end{cases}
\end{align}
where
    \begin{align*}
        r_0:=1, \quad \eta:=1/4, \quad L:=z_0/r_0>0.
    \end{align*}
 We call $L$ the \emph{cone-slope} parameter, $\alpha$ the \emph{logarithmic decay exponent}, and $\eta$ the \emph{localization} parameter. The specific choice of $d, r_0, \eta$ is not necessary. 

A schematic evolution of two consecutive vortex-ring supports is drawn in Figure~\ref{fig:sch_evol}.

\begin{figure}[t]
\centering
\begin{tikzpicture}[x=1cm,y=1cm,>=Latex, line cap=round, line join=round]

% ---------------------------
% Parameters
% ---------------------------
\pgfmathsetmacro{\xmax}{12.6}
\pgfmathsetmacro{\ymax}{6.6}
\pgfmathsetmacro{\Lcone}{2.30}

% Initial supports (same cone z = L r)
\pgfmathsetmacro{\rk}{2.20}
\pgfmathsetmacro{\zk}{\Lcone*\rk}
\pgfmathsetmacro{\radk}{0.92}

\pgfmathsetmacro{\rkp}{0.85}
\pgfmathsetmacro{\zkp}{\Lcone*\rkp}
\pgfmathsetmacro{\radkp}{0.24}

% Later-time supports (common time t>0)
\pgfmathsetmacro{\Rk}{9.05}
\pgfmathsetmacro{\Hk}{2.60}
\pgfmathsetmacro{\ak}{2.70}
\pgfmathsetmacro{\bk}{0.28}

\pgfmathsetmacro{\Rkp}{4.80}
\pgfmathsetmacro{\Hkp}{1.05}
\pgfmathsetmacro{\akp}{1.00}
\pgfmathsetmacro{\bkp}{0.24}

% Arrow attachment angles
\pgfmathsetmacro{\thetaTopStart}{-28}
\pgfmathsetmacro{\thetaTopEnd}{168}
\pgfmathsetmacro{\thetaBotStart}{-25}
\pgfmathsetmacro{\thetaBotEnd}{165}

% ---------------------------
% Styles
% ---------------------------
\tikzset{
  axis/.style={very thick,-{Latex[length=3mm,width=2.2mm]}},
  ring/.style={thick},
  transport/.style={very thick,-{Latex[length=3mm,width=2.2mm]}},
  guideDotted/.style={densely dotted, line width=0.8pt},
  group/.style={densely dashed, rounded corners=4pt, line width=0.7pt},
  annot/.style={font=\small}
}

% ---------------------------
% Axes
% ---------------------------
\draw[axis] (0,0) -- (\xmax,0) node[annot,below right] {$r$};
\draw[axis] (0,0) -- (0,\ymax) node[annot,above left] {$z$};

% ---------------------------
% Cone guide for initial centers
% ---------------------------
\draw[guideDotted] (0,0) -- ({2.55},{\Lcone*2.55});

% ---------------------------
% Initial supports
% ---------------------------
\draw[ring] (\rk,\zk) circle (\radk);
\fill (\rk,\zk) circle (1.15pt);

\draw[ring] (\rkp,\zkp) circle (\radkp);
\fill (\rkp,\zkp) circle (1.05pt);

\node[annot,anchor=south west] at (0.88,6)
  {$\operatorname{supp}\,\omega_k(\cdot,0)$};
\node[annot,anchor=west] at (1.22,2.28)
  {$\operatorname{supp}\,\omega_{k+1}(\cdot,0)$};

\node[annot,anchor=west] at (3.1,5.12)
  {center$= (r_0 d^k,\,z_0 d^k)$};

% ---------------------------
% Later-time supports
% ---------------------------
\draw[ring] (\Rk,\Hk) ellipse ({\ak} and {\bk});
\fill (\Rk,\Hk) circle (1.15pt);
\node[annot,anchor=west] at (\Rk+\ak+0.35,\Hk-0.02)
  {$\operatorname{supp}\,\omega_k(\cdot,t)$};

\draw[ring] (\Rkp,\Hkp) ellipse ({\akp} and {\bkp});
\fill (\Rkp,\Hkp) circle (1.05pt);
\node[annot,anchor=west] at (\Rkp+\akp+0.35,\Hkp-0.02)
  {$\operatorname{supp}\,\omega_{k+1}(\cdot,t)$};

% Local size annotations for the k-th later support
\draw[<->,line width=0.9pt] (\Rk-\ak,\Hk-0.72)
  -- node[annot,below=4pt] {$\sim R_k$} (\Rk+\ak,\Hk-0.72);

\draw[<->,line width=0.9pt] (\Rk-\ak-0.55,\Hk-\bk)
  -- node[annot,left=5pt] {$\sim H_k$} (\Rk-\ak-0.55,\Hk+\bk);

\node[annot,anchor=east] at (\Rk+2,\Hk+0.52)
  {center $\sim (r_0\, R_k,z_0 \, H_k)$};

% ---------------------------
% Transport arrows
% ---------------------------
% Top arrow
\pgfmathsetmacro{\xtS}{\rk + \radk*cos(\thetaTopStart)}
\pgfmathsetmacro{\ytS}{\zk + \radk*sin(\thetaTopStart)}
\pgfmathsetmacro{\xtE}{\Rk + \ak*cos(\thetaTopEnd)}
\pgfmathsetmacro{\ytE}{\Hk + \bk*sin(\thetaTopEnd)}
\pgfmathsetmacro{\aTop}{(\ytS*\xtS - \ytE*\xtE)/(\ytS-\ytE)}
\pgfmathsetmacro{\CTop}{\ytS*(\xtS-\aTop)}

\draw[transport]
  plot[smooth,domain=\xtS:\xtE,samples=120]
    (\x,{\CTop/(\x-\aTop)});

% Bottom arrow
\pgfmathsetmacro{\xbS}{\rkp + \radkp*cos(\thetaBotStart)}
\pgfmathsetmacro{\ybS}{\zkp + \radkp*sin(\thetaBotStart)}
\pgfmathsetmacro{\xbE}{\Rkp + \akp*cos(\thetaBotEnd)}
\pgfmathsetmacro{\ybE}{\Hkp + \bkp*sin(\thetaBotEnd)}
\pgfmathsetmacro{\aBot}{(\ybS*\xbS - \ybE*\xbE)/(\ybS-\ybE)}
\pgfmathsetmacro{\CBot}{\ybS*(\xbS-\aBot)}

\draw[transport]
  plot[smooth,domain=\xbS:\xbE,samples=120]
    (\x,{\CBot/(\x-\aBot)});

\end{tikzpicture}
\caption{Schematic evolution of two consecutive vortex-ring supports in the upper half-plane. The left pair represents the initial supports $\operatorname{supp}\,\omega_k(\cdot,0)$ and $\operatorname{supp}\,\omega_{k+1}(\cdot,0)$, whose centers lie on the same cone $z=(z_0/r_0)r$. The right pair represents the corresponding supports at a common later time $t>0$: the $k$-th ring has moved outward and flattened into an ellipse with center $\sim (r_0 \, R_k,z_0 \, H_k)$, while the $(k+1)$-th ring remains in a smaller inner annulus. The picture is schematic and not to scale; its purpose is only to indicate the common initial cone, the qualitative outward transport and flattening, and the outer/inner annular ordering at later time. In particular, the arrows represent qualitative transport and deformation rather than exact trajectories, and the relative placement of supports belonging to different indices and different times is not intended to encode any further quantitative relation. Only the upper half-plane is shown; the reflected lower half-plane is omitted.}
\label{fig:sch_evol}
\end{figure}
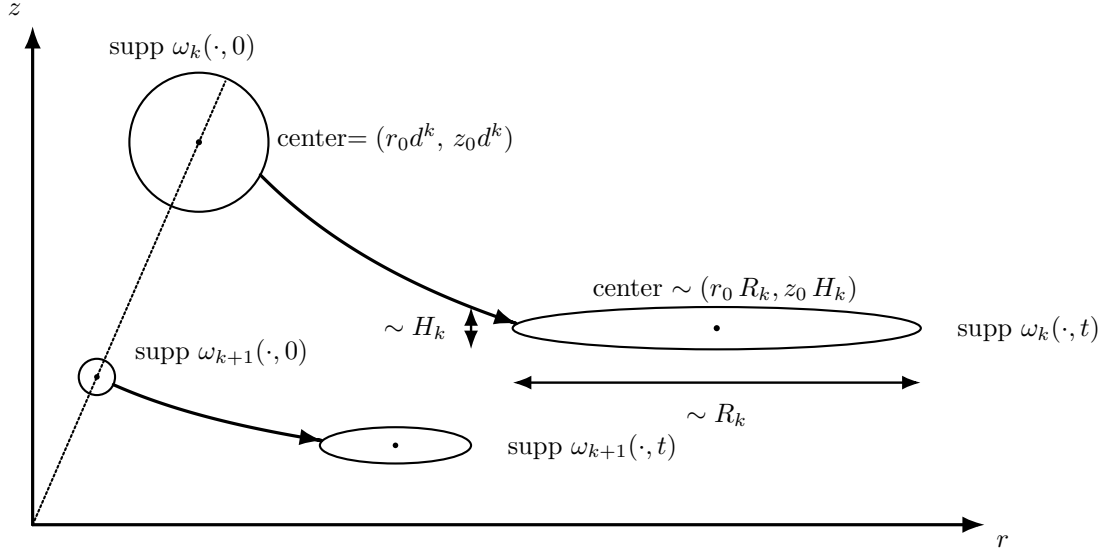

For more details on the choice of parameters $r_0, \eta, d, \ldots,$ see the list of notation at the end of the paper.

First of all, similar to Kim--Jeong's data \eqref{KJdata}, one can show our data \eqref{KJdata2} with \eqref{cond_phi}  also satisfies $\omega_0^{(m)}/r\in L^{3,q}(\mathbb{R}^3), q>1,$ for $\alpha\in (1/q,1)$.

Compared with \eqref{KJdata}, this class of data \eqref{KJdata2} is  more general because the geometry of $\mathrm{supp} \, \phi$ is more flexible. Especially, the cone slope parameter $L$ is not fixed for \eqref{KJdata2}. This difference plays a critical role in our result. The other differences do not seem to be essential. In particular, the condition on the $L^\infty$ norm of $\phi$ and the specific choice of $r_0, \eta,d $ are made just for convenience. See \S~\ref{sec:out_paper_not} for more details about this choice.

Using this data \eqref{KJdata2} and tuning the cone-slope parameter $L$, we prove norm inflation as follows:

\begin{theorem}[Norm inflation]
\label{main_thm1}
    Fix any $q>1$. For any $\tilde \varepsilon, \delta, A>0$, there exists an axisymmetric initial datum $\omega_0 \in C_c^\infty (\mathbb{R}^3)$ such that 
        \begin{align*}
           \| \omega_0 \|_{L^\infty \cap L^{1}(\mathbb{R}^3)} 
        +
            \left\|
                \frac{\omega_0}{r}
            \right\|_{L^{3,q}(\mathbb{R}^3) } 
            \leq  \tilde \varepsilon
        \end{align*}
    and the unique global-in-time solution $\omega (x,t)$ to the vorticity equation \eqref{vor_eq} with the initial condition $\omega|_{t=0}=\omega_0$ satisfies
        \begin{align*}
            \sup _{0\leq t \leq \delta} 
            \|\omega (\cdot,t) \|_{L^\infty (\mathbb{R}^3)} \geq A.
        \end{align*}
\end{theorem}

\begin{remark}
To elaborate on how the cone-slope parameter $L$ plays a role in our proof for norm inflation (Theorem~\ref{main_thm1}) and thus for instantaneous blow-up (Theorem~\ref{thm:instant_blow-up} below), 
given any $q>1$, we choose a large cone-slope parameter $L=L_q>0$ such that
    \begin{align}
    \label{eq:cond_L5}
        \frac{1}{q} 
        <
        \frac{\log (L_q/\zeta_\eta)}
        {3\Theta_\mu + \log (L_q/\zeta_\eta)} (<1),
    \quad
        \Theta_\mu:= \left( \frac{1+\mu}{1-\mu} \right)^8,
    \quad \mu:= \frac{1}{20},
    \quad
        \zeta_\eta  := \frac{1+\eta}{\sqrt{2}(1-\eta)}
    \end{align}
Then we can establish norm inflation for all $\alpha$ such that 
    \begin{align}
    \label{eq:norm_inflation_range_rem}
        \frac{1}{q} 
        <\alpha<
        \frac{\log (L_q/\zeta_\eta)}
        {3\Theta_\mu + \log (L_q/\zeta_\eta )} (<1).
    \end{align}
See Proposition~\ref{prop:norm_inflation_varcoef} and Remark~\ref{rem:role_L} for more details.

In fact, for a fixed cone-slope parameter $L$, if one can extend our norm inflation range \eqref{eq:norm_inflation_range_rem} to the full range $1/q<\alpha<1$, then one does not need to adjust $L$ to prove norm inflation for any $q>1$. However, such extension of our norm inflation range is impossible. Our proof also shows that for a fixed $L$, one can rule out norm inflation when $\alpha$ is close to $1$. See Appendix~\ref{frozen_pro_ODE_appdx} for a proof. Therefore, tuning the geometry of $\supp \phi$, or more precisely the cone-slope parameter $L$, is necessary in order to prove norm inflation for any $q>1$.

Regarding the cone slope, it is worth comparing the initial cone slope with the slope of the $m$-th ring center at the norm inflation time. From Figure~\ref{fig:sch_evol}, one can find the cone slope of the $m$-th ring center:
    \begin{align*}
        L_q\frac{H_m(t)}{R_m(t)}=L_q \frac{1}{\tilde x_m(t)^3 }.
    \end{align*}
At the norm-inflation time, $\tilde x_m(t)\sim Am^\alpha /\varepsilon,$ and hence the cone slope at the norm-inflation time is of the following order
    \begin{align*}
        L_q \left( \frac{\varepsilon}{Am^\alpha} \right)^3 \ll 1
    \end{align*}
for large $m$.
Thus, although the initial profile is placed on a large cone \(z_0/r_0\sim L_q\),
the \(m\)-th ring is strongly flattened by the time it produces norm inflation.
\end{remark}

Now to state our instantaneous blow-up result, let us first define a distributional solution and a Yudovich-type weak solution of the initial value problem:
    \begin{align}
    \label{eq:IVP_rel_vor}
    \left\{
    \begin{aligned}
        \partial_t \left(\frac{\omega}{r}\right)
        + u^r \, \partial_r  \left(\frac{\omega}{r}\right)
        +u^z \, \partial_z  \left(\frac{\omega}{r}\right)&=0, 
        \\
        \omega(x,0)&=\omega_0(x).
    \end{aligned}
    \right.
    \end{align}
\begin{definition}
\label{def:dist_sol_rel_vor}
Let \(T\in (0,\infty]\) and let \(\omega_0=\omega_0(r,z)\) be axisymmetric such that $\omega_0/r\in L^1_{\mathrm{loc}} (\mathbb{R}^3)$. 
\begin{itemize}
    \item 
We say that a pair
\((u,\omega)\) is a \emph{distributional solution} of the initial value problem 
\eqref{eq:IVP_rel_vor} on \((0,T)\)  if
\[
u\in L^1_{\mathrm{loc}}\bigl((0,T)\times\mathbb{R}^3\bigr),
\qquad
\frac{\omega}{r}\in L^1_{\mathrm{loc}}\bigl([0,T)\times\mathbb{R}^3\bigr),
\qquad 
u\,\frac{\omega}{r}\in L^1_{\mathrm{loc}}\bigl([0,T)\times\mathbb{R}^3\bigr),
\]
\(u=u^r(r,z,t)e_r+u^z(r,z,t)e_z\) is divergence-free and axisymmetric without swirl,
\[
\operatorname{curl}u=\omega e_\theta
\qquad\text{in }\mathcal{D}'\bigl((0,T)\times\mathbb{R}^3\bigr),
\]
and
\begin{equation}
\int_0^T\int_{\mathbb{R}^3}\frac{\omega(x,t)}{r}
\bigl(\partial_t\psi(x,t)+u(x,t)\cdot\nabla_x\psi(x,t)\bigr)\,dx\,dt
+\int_{\mathbb{R}^3}\frac{\omega_0(x)}{r}\psi(x,0)\,dx
=0
\label{eq:transport_distributional}
\end{equation}
for every \(\psi\in C_c^\infty(\mathbb{R}^3\times[0,T))\).
\item And we say that a distributional solution $(u,\omega)$ is a \emph{Yudovich-type weak solution} if $\omega$ additionally satisfies $\omega\in L^\infty (0,T;L^\infty \cap L^1(\mathbb{R}^3))$, and the velocity field $u$ is associated with $\omega$ via the Biot--Savart law~\eqref{eq:axis_BS_law}.
\end{itemize}

\end{definition}

For the definition of a distributional solution, we only impose a minimal regularity requirement to make sense of the weak formulation~\eqref{eq:transport_distributional}. 
On the other hand, we include boundedness and decay in the definition of a Yudovich-type weak solution, which aligns with a convention as in \cite[Definition 8.1]{MajdaBertozzi01}. 

In what follows, we refer to these as distributional and Yudovich-type weak solutions of \eqref{eq:IVP_rel_vor}. We do not explicitly include axisymmetry or the no-swirl condition in the terminology, since these properties are already encoded in the formulation of \eqref{eq:IVP_rel_vor}. In particular, any solution in the above sense is necessarily axisymmetric without swirl.

Using our data~\eqref{KJdata2} with $m=\infty$ (infinitely many rings), we prove instantaneous blow-up, which consists of two types of results: non-existence of a Yudovich-type weak solution, and existence of a distributional solution that blows up instantaneously.

\begin{theorem}[Instantaneous blow-up]
\label{thm:instant_blow-up}
For any \(\widetilde{\varepsilon}>0\) and \(q>1\), there exists an axisymmetric initial datum
\(\omega_0\in C^\infty(\mathbb{R}^3\setminus \{0\})\) with compact support such that
\[
\|\omega_0\|_{L^\infty\cap L^1(\mathbb{R}^3)}
+\left\|\frac{\omega_0}{r}\right\|_{L^{3,q}(\mathbb{R}^3)}
\le \widetilde{\varepsilon},
\]
and the following  properties hold: 
\begin{enumerate}
\item 
For any $T>0$, there does not exist a Yudovich-type weak solution $(u,\omega)$ on $(0,T)$ of \eqref{eq:IVP_rel_vor} with the initial data $\omega_0$
in the sense of Definition~\ref{def:dist_sol_rel_vor};
\item 
There exists a distributional solution $(u,\omega)$ in $(0,\infty)$ of \eqref{eq:IVP_rel_vor} with the same initial data $\omega_0$ in the sense of Definition~\ref{def:dist_sol_rel_vor} such that
    \begin{align}
    \label{eq:instant_blow_up8}
        \operatorname*{ess \,sup}_{0<t<T}
        \|\omega(t)\|_{L^{\infty}(\mathbb{R}^3)}
        =+\infty
        \qquad 
        \text{for any }T\in (0,\infty).
    \end{align}
\end{enumerate}
\end{theorem}

We call \eqref{eq:instant_blow_up8} \emph{instantaneous blow-up}.

\begin{remark}
Although Theorem~\ref{thm:instant_blow-up}-(2) is stated only in terms of the
existence of a distributional solution and instantaneous blow-up of
\(\|\omega(t)\|_{L^\infty(\mathbb{R}^3)}\), the solution constructed in the proof
has additional local energy-type bounds. We record these bounds
only as properties obtained by our construction; no optimality of the regularity
is asserted.

More precisely, the proof constructs a single global-in-time axisymmetric
no-swirl distributional solution \((u,\omega)\) as a diagonal subsequential limit
of the smooth finite-ring solutions; see
Proposition~\ref{prop:finite_ring_limit_distributional} and
Corollary~\ref{cor:global_distributional_limit}. Writing
\[
\xi:=\frac{\omega}{r},
\]
the constructed solution satisfies
\[
u\in L^\infty_{\mathrm{loc}}\bigl([0,\infty);L^2(\mathbb{R}^3)\bigr)
\cap L^\infty_{\mathrm{loc}}\bigl([0,\infty);H^1_{\mathrm{loc}}(\mathbb{R}^3)\bigr),
\qquad
\xi\in L^\infty_{\mathrm{loc}}\bigl([0,\infty);L^2_{\mathrm{loc}}(\mathbb{R}^3)\bigr).
\]
Moreover, for every \(T>0\) and \(R>0\), along a suitable subsequence of the finite-ring
approximations \((u^{(m)},\omega^{(m)})\), one has
\[
u^{(m)} \overset{*}{\rightharpoonup} u
\quad\text{in }L^\infty\bigl(0,T;L^2(\mathbb{R}^3)\bigr),
\qquad
u^{(m)} \to u
\quad\text{strongly in }L^2\bigl((0,T)\times B_R\bigr),
\]
and
\[
\frac{\omega^{(m)}}{r}
\overset{*}{\rightharpoonup}
\frac{\omega}{r}
\quad\text{in }L^\infty\bigl(0,T;L^2(B_R)\bigr).
\]
Thus, the singular behavior in Theorem~\ref{thm:instant_blow-up}-(2) is specific to the
\(L^\infty_x\)-norm of the vorticity: despite instantaneous blow-up of
\(\|\omega(t)\|_{L^\infty}\), the constructed solution still retains the above local
energy-type regularity and local \(L^2\)-control of the relative vorticity. 
\end{remark}

\begin{remark}
For each finite-ring approximation, the vorticity is compactly supported at
every time. Moreover, on the bootstrap time intervals used in the norm-inflation
argument, the individual rings remain quantitatively localized and scale-separated.
It is therefore natural to expect that the limiting distributional solution also
has compactly supported vorticity, at least for a suitable representative and for
a.e. time. However, this compact-support property does not follow from the
present compactness argument.

The obstruction is that the uniform estimates available for the finite-ring approximations are the
transported bound
\[
    \left\|\xi^{(m)}\right\|_{L^\infty(0,\infty; L^{3,q}(\mathbb{R}^3))}\le C,
    \qquad \xi^{(m)}=\omega^{(m)}/r,
\]
and the conserved energy
\[
    \left\|u^{(m)}\right\|_{L^\infty (0,\infty;L^2(\mathbb{R}^3))} \leq C,
\]
both of which are independent of \(m\). Together with local elliptic estimates, these bounds give the local compactness
used to construct the distributional limit, but they do not give a uniform-in-\(m\)
finite-propagation estimate for the supports. In particular, for \(q>1\), they do
not yield a uniform \(L^\infty_x\)-bound on the velocities \(u^{(m)}\), nor on the
limiting velocity. Consequently, we do not claim compact support of the limiting
distributional solution.
\end{remark}

\begin{remark}
    We prove Theorem~\ref{thm:instant_blow-up}-(2) by using the finite-ring approximations from Theorem~\ref{main_thm1} and taking its subsequential limit. This part in fact shows this limiting solution blows up instantaneously. In contrast, Theorem~\ref{thm:instant_blow-up}-(1) shows there is no Yudovich-type weak solution, that is, there is no distributional solution $\omega$ in $L^\infty(0,T; L^\infty \cap L^1(\mathbb{R}^3))$ with the Biot--Savart law~\eqref{eq:axis_BS_law}. 
    Admittedly, Theorem~\ref{thm:instant_blow-up}-(1)
    does not rule out the existence of other distributional solutions $\omega$ that remain bounded in $L_x^\infty$ but fail the $L^1_x$ condition, that is, the non-existence of a distributional solution $\omega\in L^\infty(0,T; L^\infty(\mathbb{R}^3))$ but $\omega\notin L^\infty(0,T;L^1(\mathbb{R}^3))$.
    We leave it open. See Remark~\ref{rem:diff_bd_dis_sol} for detailed discussion on difficulties of this problem.    
\end{remark}

\begin{remark}
\label{rem_DeG}
In a forthcoming paper, we prove an analogous result (norm inflation) for the De Gregorio model, a one-dimensional model of the three-dimensional  Euler equation. More precisely, the corresponding Lorentz quantity we consider is  $\|\omega_0/|x|\|_{L^{1,q}(\mathbb{R})}$. The one-dimensional geometry eliminates the geometric self-slowdown mechanism, which leads to what we call \emph{strong ODEs}. See \S~\ref{der_weak_ODE} for more discussions. The absence of the slowdown mechanism makes the analysis significantly simpler. 
\end{remark}

\subsection{Framework: Dominant Outer-Rings and ODE Reduction}
\label{sec:framework}

Inspired by \cite{KimJeong22} and \cite{CordobaMartinezZheng25}, we will perform our analysis in a regime where outer-ring contributions to vortex stretching dominate over inner-ring and self-induced contribution, which motivates an ansatz and ODE reduction.
The authors of \cite{CordobaMartinezZheng25} introduced this ansatz to construct finite-time blow-up solutions starting from H\"older continuous initial data. See \S~\ref{subsec:CMZ_comparison} for further discussion of their results.

Using the  data \eqref{KJdata2}, we consider the solution prior to collision, and decompose the vorticity into individual vortex rings:
\begin{align*}
    \omega^{(m)}(r,z,t) = \sum_{k=1}^m \omega_k(r,z,t).
\end{align*}
For a fixed index $k$, we define the \emph{outer} and \emph{inner} rings by
\[
    \omega_- := \sum_{j=1}^{k-1} \omega_j,
    \qquad
    \omega_+ := \sum_{j=k+1}^{m} \omega_j,
\]
so that $\omega^{(m)} = \omega_- + \omega_k + \omega_+$. Correspondingly, we decompose the velocity field as
\[
    u^{(m)} = u_- + u_k + u_+,
\]
where each component is generated by the corresponding vorticity through the Biot--Savart law.

The evolution of a fixed vortex ring $\omega_k$ is governed by
\begin{align}
\label{vor_k}
    \partial_t \omega_k
    + u^{(m),r} \partial_r \omega_k
    + u^{(m),z} \partial_z \omega_k
    =
    \omega_k \frac{u^{(m),r}}{r}.
\end{align}

The contribution of the inner rings $u_+$ does not enhance vortex stretching. For very short times $t>0$, the odd symmetry in $z$ ensures that this contribution remains negligible. Moreover, the self-induced velocity $u_k$ turns out to be subleading. Therefore, outer rings dominate, and thus
one can approximate the contribution of the outer rings $u_-$ near the origin $(r,z)=(0,0)$:
\begin{align*}
    u_-^r(r,z,t) \approx r\,\partial_r u_-^r(0,0,t)=r\, a(t), 
    \qquad \text{where }a(t):= \partial_r u_-^r(0,0,t). 
\end{align*}

In fact, we justified the outer-region dominance in a more general setting without relying on the multi-ring structure of a solution. See Section~\ref{sec:cone_free_key_lem} and Proposition~\ref{prop:cone_free_key_lem} for its statement, history, and our motivation behind it.

By axisymmetry and odd symmetry in $z$, we have
\[
    u_-^r(0,0,t)=0,
    \qquad
    \partial_z u_-^r(0,0,t)=0.
\]
Furthermore, using the divergence-free condition,
\begin{align*}
    u_-^z(r,z,t)
    \approx
    z\,\partial_z u_-^z(0,0,t)
    =
    -2z\,a(t).
\end{align*}

Under these approximations, \eqref{vor_k} reduces to
\begin{align*}
    \partial_t \omega_k
    + r\,a(t)\,\partial_r \omega_k
    - 2z\,a(t)\,\partial_z \omega_k
    =
    \omega_k\,a(t).
\end{align*}
Solving this equation by the method of characteristics yields
\begin{align*}
    \omega_k(r,z,t)
    &=
    \exp\left( \int_0^t a(s) \,ds \right) 
    \,
    \omega_{0,k}
    \!\left(
        r \, \exp\left( -\int_0^t a(s) \,ds \right) ,
        z \, \exp\left( 2\int_0^t a(s) \,ds \right) 
    \right)
    \\
    &=
    x_k^0
   \exp\left( \int_0^t a(s) \,ds \right) 
   \,
    \phi\!\left(
        \frac{r}{d^k \exp\left( \int_0^t a(s) \,ds \right)  },
        \frac{z \exp\left( 2\int_0^t a(s) \,ds \right) }{d^k}
    \right).
\end{align*}
Here, 
    \begin{align*}
        x_k^0=\frac{\varepsilon}{k^\alpha}, \qquad \frac{1}{q}<\alpha<1.
    \end{align*}
Defining
\begin{align}
\label{tilde_xk_rk_hk}
    x_k(t) := x_k^0 \exp \left( {\int_0^t a(s)\,ds} \right) , 
    \quad
    \tilde x_k(t) := \frac{x_k(t)}{x_k^0},
    \quad 
    R_k(t) := d^k \tilde x_k(t),
    \quad
    H_k(t) := \frac{d^k}{(\tilde x_k(t))^2},
\end{align}
we obtain the representation
\[
    \omega_k(r,z,t)
    =
    x_k(t)\,
    \phi\!\left(
        \frac{r}{R_k(t)},
        \frac{z}{H_k(t)}
    \right).
\]

Motivated by this formal computation, we adopt the ansatz
\begin{align}
\label{eq:ansatz}
    \omega_k(r,z,t)
    =
    x_k(t)\,
    W_k\!\left(
        \frac{r}{R_k(t)},
        \frac{z}{H_k(t)},
        t
    \right),
\end{align}
where the weak dependence of $W_k$ on $k$ and $t$ is introduced to account for errors between the approximate dynamics and the true solution, although $W_k$ is expected to remain close to the reference profile $\phi$.

In addition, regarding the data \eqref{KJdata2}, a direct computation then shows that
\begin{align*}
    \left\| \frac{\omega_0^{(m)}}{r} \right\|_{L^{3,q}(\mathbb{R}^3)}^q
    \sim
    \sum_{k=1}^{m} \frac{\varepsilon^q}{k^{\alpha q}}
    \leq
    \sum_{k=1}^{\infty} \frac{\varepsilon^q}{k^{\alpha q}}
    < \infty
\end{align*}
as expected.
In particular, by choosing $\varepsilon$ sufficiently small, this norm can be made arbitrarily small uniformly in $m$.

\subsection{Main Idea: Evolution of the Stretching Rate \texorpdfstring{$S_k$}{Sk} and Front-Migration}
\label{ODE_intro}
A central point of this paper is that the stretching--slowdown competition can be
captured, to leading order, by a coupled ODE cascade for the amplitudes and
aspect ratios of the vortex rings.  The advantage of this viewpoint is that it lets us
track the evolution of the stretching rate itself and quantify how stretching combined with geometric 
flattening can either deplete or enhance the influence of outer rings. 

We use superscript/subscript, such as orig, loc, froz, to distinguish various types of simplifications. In addition, we will denote $x_k^{\mathrm{orig}}(t):=x_k(t)$ throughout this Section~\ref{ODE_intro}, but after this section, we will come back to the original notation $x_k(t)$.

\subsubsection{The Weak ODEs, and Overview.}
\label{der_weak_ODE}
Under the ansatz and the outer-ring linearization described above, the
$L^\infty$-amplitude $x_k^{\mathrm{orig}}(t)$ of the $k$-th ring satisfies
\[
    \frac{d}{dt}\log x_k^{\mathrm{orig}}(t) =\partial_r u_-^r(0,0,t),
\]
where $u_-=\sum_{j<k}u_j$ is generated by the outer rings
$\omega_-=\sum_{j<k}\omega_j$.

From the axisymmetric Biot--Savart law, one can obtain
    \begin{align*}
        \partial_r u^{r}_- (0,0,t)
    =
        C\int_{-\infty}^\infty \int _0^\infty \frac{r^2z}{(r^2+z^2)^{5/2}} \big(-\omega_-(r,z,t)\big) \, dr dz.
    \end{align*}

Rescaling the Biot--Savart contribution of the $j$-th ring by its length scales
$(R_j(t),H_j(t))$ yields a factor $\Gamma_j(t)^2$ together with a
profile-dependent correction:
\begin{align}
\label{original_ODEs}
\begin{aligned}
    \frac{d}{dt}\log x_k^{\mathrm{orig}}(t)
    &=
    \sum_{j=1}^{k-1} x_j^{\mathrm{orig}}(t)\,\Gamma_j^{\mathrm{orig}}(t)^2
    \iint \frac{r^2 z}{(r^2+\Gamma_j^{\mathrm{orig}}(t)^2 z^2)^{5/2}} \big( -W_j(r,z,t) \big)\, drdz, 
    \\
    &=\sum_{j=1}^{k-1}
    x_j^{\mathrm{orig}}(t) \Gamma_j^{\mathrm{orig}}(t)^2 \Lambda_j (t,\Gamma_j^{\mathrm{orig}}(t)),
    \qquad
    \Gamma_j^{\mathrm{orig}}(t):=
    \frac{H_j(t)}{R_j(t)}
    =
    \left(
    \frac{x_j^{\mathrm{orig}}(0)}{x_j^{\mathrm{orig}}(t)}
    \right)^3
\end{aligned}
\end{align}
where
\begin{equation*}
\Lambda_j(t,\Gamma)
:=\int_{-\infty}^\infty \int_{0}^\infty
K_\Gamma(r,z)\,\big( - W_j(r,z,t) \big)\,dr\,dz,
\qquad
K_\Gamma(r,z):=\frac{r^2 z}{(r^2+\Gamma^2 z^2)^{5/2}}.
\end{equation*}
As the profiles $W_j(r,z,t)\approx \phi(r,z)$ are supposed to vary little in $j,t$, let us consider the following model with a ``frozen" profile $\phi$:
\begin{equation}
\label{ode_full}
    \frac{d}{dt}\log x_k^{\mathrm{froz}}(t)
    =
    \sum_{j<k} x_j^{\mathrm{froz}}(t)\,\Gamma_j^{\mathrm{froz}}(t)^2\,\Lambda_{\mathrm{froz}}(\Gamma_j^{\mathrm{froz}}(t)),
    \qquad \Gamma_j^{\mathrm{froz}}(t)
    :=\left(
        \frac{x_j^{\mathrm{froz}}(0)}{x_j^{\mathrm{froz}}(t)}
    \right)^3,
\end{equation}
where
\begin{equation*}
    \Lambda_{\mathrm{froz}}(\Gamma)
    :=
    \int_{-\infty}^\infty \int _{0}^\infty
    \frac{r^2 z}{(r^2+\Gamma^2 z^2)^{5/2}}
    \big( -\phi(r,z) \big)\,drdz.
\end{equation*}
Let us call $\Lambda_j(t,\Gamma),\Lambda_{\mathrm{froz}}(\Gamma)$ an \emph{unfrozen and frozen Biot--Savart coefficient} respectively.
(The dependence on the reference profile $\phi$ is suppressed in the notation.)

The geometric self-slowdown mechanism is already encoded in the ODE structure:
incompressibility forces $\Gamma_j^{\mathrm{froz}}(t)$ to decrease as $x_j^{\mathrm{froz}}(t)$ grows, which can
weaken the subsequent contribution of the $j$-th ring to vortex stretching.

\medskip
\noindent \emph{Strong ODEs without Geometric Self-Slowdown.}
If one ignores the geometric self-slowdown mechanism by treating the factor $\Gamma ^2 \Lambda (\Gamma)\sim 1$, then vortex stretching enhances itself without any depletion because
the ODEs become
    \begin{align*}
        \frac{d}{dt} \log x_k^{\mathrm{str}}(t)
        = \sum_{j<k} x_j^{\mathrm{str}}(t).
    \end{align*}
Let us call this equation \emph{strong} ODEs whereas we call the original ODE \eqref{original_ODEs} \emph{weak} ODEs.
    
For this model, as it is expected from our discussion in \S~\ref{sharp_Danchin}, one can simply use the strong initial vortex stretching by estimating $\sum_{j<k}x_j^{\mathrm{str}}(t) \geq \sum_{j<k}x_j^0 \sim k^{1-\alpha}$ on the right-hand side in order to establish norm inflation for any $\alpha<1.$ This is how vortex stretching is governed for the De Gregorio model, which is a one-dimensional model of the three-dimensional  Euler equation \eqref{eq:Euler}. See Remark~\ref{rem_DeG}. And regarding the Euler equation \eqref{eq:Euler}, if one considers a backward problem instead of our forward problem, then it yields strong ODEs. However, we cannot use a backward problem in our case. See \S~\ref{subsub:comparison} for discussion on this issue.

As discussed in \S~\ref{geo_ss_mc}, in our case, the geometric self-slowdown depletes the vortex stretching rate $u^{r}_-/r$, which corresponds to $\frac{d}{dt}\log x_k^{\mathrm{orig}}(t)$ at the ODE level, as vorticity stretches. This entire section \ref{ODE_intro} is devoted to investigate how to quantify the stretching-slowdown competition and overcome the difficulty caused by the slowdown.

Proving norm inflation for \eqref{vor_eq} reduces to proving norm inflation at the ODE
level, i.e.\ that for a fixed short time $t>0$ one has $x_k^{\mathrm{orig}}(t)\to\infty$ as
$k\to\infty$. Let us call this type of result a \emph{norm inflation result} (at the ODE level). In contrast, when one proves the other type of result, that is, for a fixed short time $t>0$, $x_k^{\mathrm{orig}}(t) \to 0$ as $k\to\infty$, let us call it a \emph{decay result} (at the ODE level).

\medskip
\noindent \emph{Overview of Section~\ref{ODE_intro}.} 
We now outline the structure of this section. Our primary goal is to motivate the proof of norm inflation at the ODE level for the original, unsimplified \emph{weak} ODEs \eqref{original_ODEs} for any $\alpha<1$ after tuning $L$ appropriately. To achieve this, we introduce several simplifications of the ODE system and analyze them using various techniques, demonstrating how the insights gained from these simplified models enable us to prove norm inflation for the full system. A diagram that shows the hierarchy of these ODE simplifications is drawn in Figure~\ref{fig:hier_ODE}.
    \begin{enumerate}
        \item We first freeze the profile $W_k(r,z,t)\approx \phi(r,z)$ and obtain the \emph{frozen-profile} model \eqref{ode_full} (or \emph{a frozen model} in short):
        \begin{align*}
            \frac{d}{dt}(\log x_k^{\mathrm{froz}}(t))
    =
    \sum_{j=1}^{k-1} x_j^{\mathrm{froz}}(t)\,\Gamma_j^{\mathrm{froz}}(t)^2\,\Lambda_{\mathrm{froz}}(\Gamma_j^{\mathrm{froz}}(t)),
        \end{align*}
    \item (Section~\ref{flattened_kernel_ODE_intro})
        As vortex stretching makes $\Gamma_j(t)$ small in general, we simplify it further by approximating the Biot--Savart coefficient $\Lambda_{\mathrm{froz}} (\Gamma)\approx \Lambda_{\mathrm{froz}}(0)$ and obtain \emph{a model with a flattened Biot--Savart coefficient} (or \emph{a flattened model} in short):
            \begin{align*}
                \frac{d}{dt}(\log x_k^{\mathrm{flat}}(t))
                =
            \sum_{j=1}^{k-1}
            x_j^{\mathrm{flat}}(t) \Gamma_j^{\mathrm{flat}}(t)^2.
            \end{align*}
        We prove sharp norm inflation for $\alpha<2/7$ for this model by investigating evolution of vortex stretching rate $S_k^{\mathrm{flat}}(t)=\sum_{j\leq k} x_j^{\mathrm{flat}}(t) \Gamma_j^{\mathrm{flat}}(t)^2$ and obtaining a Riccati-type inequality. (A rigorous proof for this model can be found in Appendix \ref{HAodes}.)
    \item (Section~\ref{variable_coeff_intro})
        Going back to the \emph{frozen-profile} model \eqref{ode_full},
        we apply similar ideas and obtain a similar Riccati-type inequality but with a profile-dependent correction.
        We can prove the same $2/7$-norm inflation, which is now sharp only for the profile-independent approach. It suggests to utilize both the geometry of $\supp \phi$ and the profile-dependent correction to prove norm inflation beyond $2/7$. (A rigorous proof for this part can be found in Appendix \ref{frozen_pro_ODE_appdx}.)
    \item (Section~\ref{intro:front_mig})
        Therefore, we localize the profile $\phi$, which in turn ``localizes" the frozen Biot--Savart coefficient $\Lambda _{\mathrm{froz}}(\Gamma)\approx \Lambda_{\mathrm{loc}} (\Gamma)$. Here, $\Lambda_{\mathrm{loc}} (\Gamma)$ has an explicit expression. Through change of variables, this localization leads to a \emph{model with a localized Biot--Savart coefficient} (or a \emph{localized model} in short):
        \begin{align}
        \label{eq:loc_coe_model}
        \begin{aligned}
            \frac{d}{dt} (\log x_k^{\mathrm{loc}}(t))
            &\approx 
            \sum_{j=1}^{k-1}
             x_j^{\mathrm{loc}}(t) \Gamma_j^{\mathrm{loc}}(t)^2 \Lambda_{\mathrm{loc}}(\Gamma_j^{\mathrm{loc}}(t))
            \\
            &=
            \sum_{j=1}^{k-1}
            x_j^{\mathrm{loc}}(0) L^{-2/3} \Psi (\zeta_j^{\mathrm{loc}}(t)),
            \quad \Psi(\zeta)=\zeta^{5/3}(1+\zeta^2)^{-5/2}, 
            \quad \zeta_j^{\mathrm{loc}}(t) =L  \Gamma_j^{\mathrm{loc}}(t)
        \end{aligned}
        \end{align}
        where $L$ is a cone slope parameter related to the geometry of $\supp \phi$. 
        Through the monotonicity of $\Psi$, which is related to what we call a front-migration mechanism, and an exact cascade identity for $\Gamma_j$ , we can prove norm inflation in a certain range of $\alpha$ that depends on the geometry of $\supp \phi$, and this range reaches the full range $\alpha<1$ as $L\to\infty$.
    \item (Section~\ref{local_phi_final}) We generalize this idea to the original weak ODEs (un-frozen profile) as the monotonicity of $\Psi$ and the exact cascade identity are stable under multiplicative comparability bounds.
    \end{enumerate}
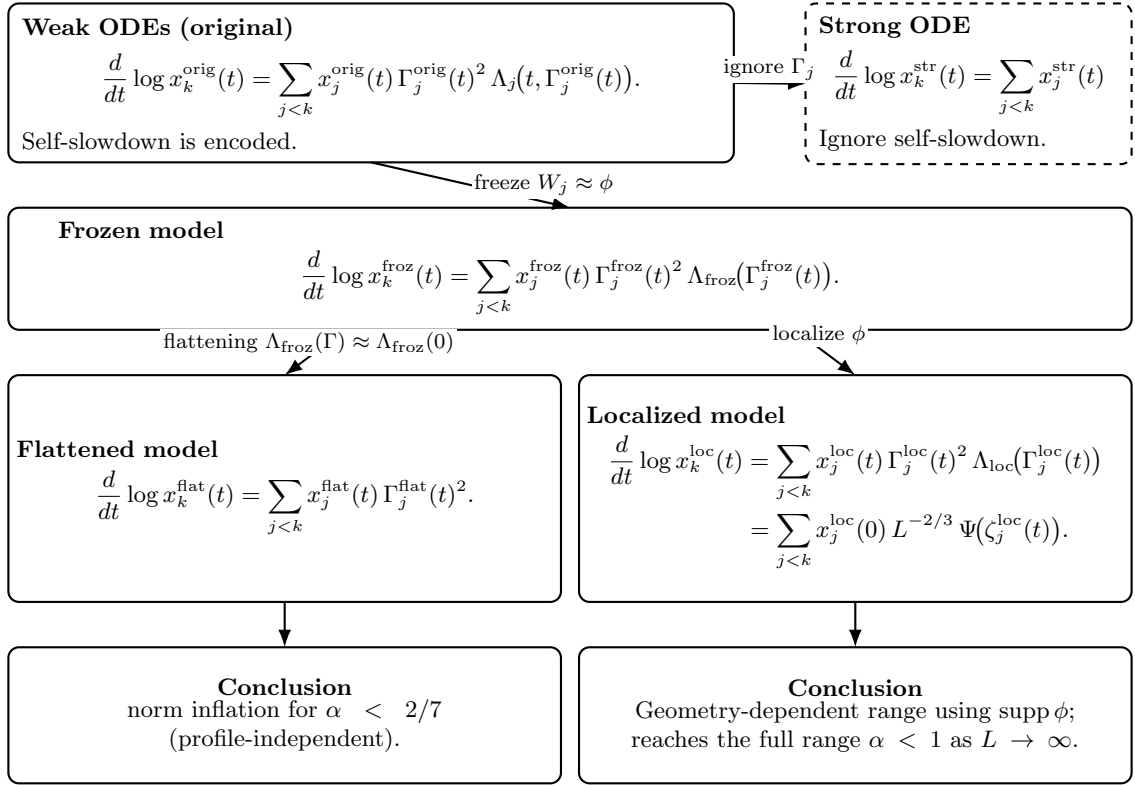
\begin{figure}[t]
\centering
\begin{tikzpicture}[
  >=Latex,
  font=\footnotesize,
  arr/.style={->, thick},
  box/.style={
    draw, rounded corners, thick, align=left,
    outer sep=0pt
  },
  weakbox/.style={
    box, inner sep=5pt,
    text width=.56\linewidth
  },
  sidebox/.style={
    box, dashed, inner sep=5pt,
    text width=.24\linewidth
  },
  widebox/.style={
    box, inner sep=5pt,
    text width=.82\linewidth,
    minimum width=.90\linewidth
  },
  halfbox/.style={
    box,
    inner xsep=3pt,
    inner ysep=5pt,
    text width=.43\linewidth,
    minimum height=30mm
  },
  conclbox/.style={
    box,
    inner xsep=3pt,
    inner ysep=4pt,
    align=center,
    text width=.43\linewidth,
    minimum height=18mm
  }
]

% common row width = 0.90\linewidth
\coordinate (L) at (-0.45\linewidth,0);
\coordinate (R) at ( 0.45\linewidth,0);

%--- first row ---------------------------------------------------------------
\node[weakbox, anchor=north west] (weak) at (L) {%
\textbf{Weak ODEs (original)}\\[-1mm]
\[
\frac{d}{dt} \log x_k^{\mathrm{orig}}(t)
=
\sum_{j<k} x_j^{\mathrm{orig}}(t)\,\Gamma_j^{\mathrm{orig}}(t)^2\,
\Lambda_j\!\bigl(t,\Gamma_j^{\mathrm{orig}}(t)\bigr).
\]
Self-slowdown is encoded.
};

\node[sidebox, anchor=north east] (strong) at (R) {%
\textbf{Strong ODE}\\[-1mm]
\[
\frac{d}{dt}\log x_k^{\mathrm{str}}(t)
=
\sum_{j<k} x_j^{\mathrm{str}}(t)
\]
Ignore self-slowdown.
};

\draw[arr] (weak.east) -- (strong.west)
node[pos=.5, above, fill=white, inner sep=1pt]
{\scriptsize ignore \(\Gamma_j\)};

%--- second row --------------------------------------------------------------
\node[widebox, anchor=north west] (frozen)
at ([yshift=-6mm]L |- weak.south) {%
\textbf{Frozen model}\\[-1mm]
\[
\frac{d}{dt}\log x_k^{\mathrm{froz}}(t)
=
\sum_{j<k} x_j^{\mathrm{froz}}(t)\,\Gamma_j^{\mathrm{froz}}(t)^2\,
\Lambda_{\mathrm{froz}}\!\bigl(\Gamma_j^{\mathrm{froz}}(t)\bigr).
\]
};

\draw[arr] (weak.south) -- (frozen.north)
node[midway, right, fill=white, inner sep=1pt]
{\scriptsize freeze \(W_j\approx \phi\)};

%--- third row ---------------------------------------------------------------
\node[halfbox, anchor=north west] (flat)
at ([yshift=-6mm]L |- frozen.south) {%
\textbf{Flattened model}\\[-1mm]
\[
\frac{d}{dt} \log x_k^{\mathrm{flat}}(t)
=
\sum_{j<k} x_j^{\mathrm{flat}}(t)\,\Gamma_j^{\mathrm{flat}}(t)^2.
\]
};

\node[halfbox, anchor=north east] (local)
at ([yshift=-6mm]R |- frozen.south) {%
\textbf{Localized model}\\[-1mm]
\[
\begin{aligned}
\frac{d}{dt} \log x_k^{\mathrm{loc}}(t)
&=
\sum_{j<k} x_j^{\mathrm{loc}}(t)\,\Gamma_j^{\mathrm{loc}}(t)^2\,
\Lambda_{\mathrm{loc}}\!\bigl(\Gamma_j^{\mathrm{loc}}(t)\bigr)
\\
&=
\sum_{j<k} x_j^{\mathrm{loc}}(0)\,L^{-2/3}\,
\Psi\!\bigl(\zeta_j^{\mathrm{loc}}(t)\bigr).
\end{aligned}
\]
};

\coordinate (fleft)  at ($(frozen.south west)!0.30!(frozen.south east)$);
\coordinate (fright) at ($(frozen.south west)!0.70!(frozen.south east)$);

\draw[arr] (fleft) -- (flat.north)
node[pos=.60, above, fill=white, inner sep=1pt]
{\scriptsize flattening \(\Lambda_{\mathrm{froz}}(\Gamma)\approx\Lambda_{\mathrm{froz}}(0)\)};

\draw[arr] (fright) -- (local.north)
node[pos=.40, above, fill=white, inner sep=1pt]
{\scriptsize localize \(\phi\)};

%--- fourth row --------------------------------------------------------------
\node[conclbox, anchor=north west] (flatc)
at ([yshift=-6mm]L |- flat.south) {%
\textbf{Conclusion}\\[-1mm]
norm inflation for \(\alpha<2/7\)\\
(profile-independent).
};

\node[conclbox, anchor=north east] (localc)
at ([yshift=-6mm]R |- local.south) {%
\textbf{Conclusion}\\[-1mm]
Geometry-dependent range using \(\operatorname{supp}\phi\);\\
reaches the full range \(\alpha<1\) as \(L\to\infty\).
};

\draw[arr] (flat.south) -- (flatc.north);
\draw[arr] (local.south) -- (localc.north);

\end{tikzpicture}
\caption{Hierarchy of the ODE simplifications used in Section~\ref{ODE_intro}}
\label{fig:hier_ODE}
\end{figure}

\subsubsection{A Model with a Flattened Biot--Savart Coefficient and Evolution of the Stretching Rate $S_k$}
\label{flattened_kernel_ODE_intro}
To isolate the main ODE mechanism, it is convenient to simplify further the frozen-profile model \eqref{ode_full} 
by ``flattening'' the frozen Biot--Savart coefficients, that is, setting $\Lambda_{\mathrm{froz}}(\Gamma)\approx \Lambda_{\mathrm{froz}}(0)$ in the flattened regime
$\Gamma\ll 1$. This simplification is based on the fact that norm inflation makes $\Gamma$ small in general. Hence in this simplification, we suppress the $\Gamma$-dependence of $\Lambda_{\mathrm{froz}}$ and keep only the $\Gamma^2$ slowdown, isolating the core depletion mechanism.

Absorbing harmless constants by rescaling the time variable, the system \eqref{ode_full} becomes
\begin{equation}\label{ode_simplified}
    \frac{d}{dt} (\log x_k^{\mathrm{flat}}(t))
    =
    \sum_{j<k} x_j^{\mathrm{flat}}(t)\,\Gamma_j^{\mathrm{flat}}(t)^2, 
    \quad
    \Gamma_j^{\mathrm{flat}}(t)
    := 
    \left(
    \frac{x_j^{\mathrm{flat}}(0)}{x_j^{\mathrm{flat}}(t)}
    \right)^3.
\end{equation}
Define
\[
    b_j^{\mathrm{flat}}(t):=x_j^{\mathrm{flat}}(t)\,\Gamma_j^{\mathrm{flat}}(t)^2,
    \qquad
    S_k^{\mathrm{flat}}(t):=\sum_{j\le k} b_j^{\mathrm{flat}}(t).
\]
Then \eqref{ode_simplified} reads $\frac{d}{dt}\log x_k^{\mathrm{flat}}(t)=S_{k-1}^{\mathrm{flat}}(t)$.
Differentiating $b_j^{\mathrm{flat}}$ using 
\begin{align*}
\frac{d}{dt}\Gamma_j^{\mathrm{flat}}(t)=-3\Gamma_j^{\mathrm{flat}}(t)S_{j-1}^{\mathrm{flat}}(t)
\end{align*}
yields the 
depletion law
\begin{equation}\label{ode_bj}
    \frac{d}{dt}b_j^{\mathrm{flat}}(t)\;=\;-5\,b_j^{\mathrm{flat}}(t)\,S_{j-1}^{\mathrm{flat}}(t),
\quad \text{and hence}\quad
    \frac{d}{dt}S_k^{\mathrm{flat}}(t)
    \;=\;
    -5\sum_{j\le k} b_j^{\mathrm{flat}}(t)\,S_{j-1}^{\mathrm{flat}}(t).
\end{equation}
Equivalently, using $(S_j^{\mathrm{flat}})^2-(S_{j-1}^{\mathrm{flat}})^2 = 2b_j^{\mathrm{flat}}\, S_{j-1}^{\mathrm{flat}}+(b_j^{\mathrm{flat}})^2$, one obtains
\begin{equation}\label{ode_Sk_quad}
    \frac{d}{dt}S_k^{\mathrm{flat}}(t)
    \;=\;
    -\frac{5}{2}\Big(S_k^{\mathrm{flat}}(t)^2-\sum_{j\le k} b_j^{\mathrm{flat}}(t)^2\Big).
\end{equation}
As $\frac{d}{dt}b_j^{\mathrm{flat}}(t)\leq0$, it holds that $b_j^{\mathrm{flat}}(t)\leq b_j^{\mathrm{flat}}(0)=\frac{\varepsilon}{j^\alpha}\leq \varepsilon$. Hence  $\sum_{j\le k} b_j^{\mathrm{flat}}(t)^2$ is lower order. Ignoring this lower order term, $\sum b_j^{\mathrm{flat}}(t)^2\geq 0$, one is led to a Riccati-type differential inequality  of $S_k^{\mathrm{flat}}(t)$ with a $1/t$ bound of $S_k^{\mathrm{flat}}(t)$ with the Riccati constant $5/2$:
    \begin{align}
    \label{low_bd_skt}
        S_k^{\mathrm{flat}}(t) \geq \frac{1}{\frac{5}{2}t+ \frac{1}{S_k^{\mathrm{flat}}(0)}}, 
    \quad \text{and thus},
    \quad 
        \int_0^t S_k^{\mathrm{flat}}(s) ds \geq
        \log (k^{\frac{2}{5}(1-\alpha)})+C(\varepsilon,t,\alpha).
    \end{align}
Therefore, the ODEs $\frac{d}{dt} \log x_k^{\mathrm{flat}}(t)= S_{k-1}^{\mathrm{flat}}(t)$ yields
    \begin{align*}
        x_k^{\mathrm{flat}}(t) \geq
        \frac{C(\varepsilon,t,\alpha)}{k^\alpha} k^{\frac{2}{5}(1-\alpha)}.
    \end{align*}
This in turn yields the norm inflation range for the ODEs:
\begin{equation}
\label{27_ni_range}
    \frac{2}{5}(1-\alpha)>\alpha, 
    \quad \text{that is}, 
    \quad
    \alpha< \frac{2/5}{1+2/5}=\frac{2}{7}.
\end{equation}

This threshold $2/7$ is in fact sharp for the flattened model \eqref{ode_simplified}. More precisely, for $\alpha>2/7$, one can show a decay result, that is, for any fixed small time $t>0$, $x_k^{\mathrm{flat}}(t) \to 0$ as $k\to\infty$. See Appendix~\ref{HAodes} for more details.
This computation highlights the basic obstruction coming from self-slowdown when the
Biot--Savart coefficient $\Lambda$ is treated as essentially constant.

\subsubsection{A Model with a Non-Flattened Coefficient and Geometry of the Profile $\phi$}
\label{variable_coeff_intro}
For the original frozen-profile ODEs \eqref{ode_full}, one must retain the factor $\Lambda_{\mathrm{froz}}\big(\Gamma_j(t)\big)$ in
\eqref{ode_full}. Define
\begin{align}
\label{eq:orig_def_bj_froz}
    b_j^{\mathrm{froz}}(t):=x_j^{\mathrm{froz}}(t)\,\Gamma_j^{\mathrm{froz}}(t)^2\,\Lambda_{\mathrm{froz}}\big(\Gamma_j^{\mathrm{froz}}(t)\big),
    \qquad
    S_k^{\mathrm{froz}}(t):=\sum_{j\le k} b_j^{\mathrm{froz}}(t),
\end{align}
so that \eqref{ode_full} again reads $\frac{d}{dt}\log x_k^{\mathrm{froz}}(t)=S_{k-1}^{\mathrm{froz}}(t)$.
The difference is that the evolution of $b_j^{\mathrm{froz}}$ now includes a $\Gamma$-dependent
correction. Using $\frac{d}{dt}\Gamma_j^{\mathrm{froz}}(t)= -3\Gamma_j^{\mathrm{froz}}(t)S_{j-1}^{\mathrm{froz}}(t)$, one finds
\begin{align}
\label{eq:log_der_cor}
    \frac{d}{dt}b_j^{\mathrm{froz}}(t)
    =
    b_j^{\mathrm{froz}}(t)\,S_{j-1}^{\mathrm{froz}}(t)
    \Big(
        -5
        + Q_{\mathrm{froz}}\big(\Gamma_j^{\mathrm{froz}}(t))
    \Big),
    \qquad
    Q_{\mathrm{froz}}(\Gamma)= 3\Gamma\frac{|\Lambda'_{\mathrm{froz}}(\Gamma)|}{\Lambda_{\mathrm{froz}}(\Gamma)},
\end{align}
where we used $\Lambda_{\mathrm{froz}}'(\Gamma)\le 0$.

Hence the coefficient $-5$ responsible for the $2/7$-threshold in the flattened model heuristic (\S~\ref{flattened_kernel_ODE_intro}) is shifted upward by the  correction $Q_{\mathrm{froz}}(\Gamma)$. 

First of all, simply treating
$Q_{\mathrm{froz}}(\Gamma)\geq 0$ as an ``error'' recovers the same $2/7$-norm inflation result as before. 
On the other hand, using the simple upper bound $\mathrm{max}_{0\leq \Gamma \leq 1} Q_{\mathrm{froz}}(\Gamma)$ of $Q_{\mathrm{froz}}(\Gamma)$, one can establish a decay result for a certain range of $\alpha$. In addition, one can make this upper bound $\mathrm{max}_{0\leq \Gamma \leq 1} Q_{\mathrm{froz}}(\Gamma)$ arbitrary small by modifying the geometry of $\supp \phi$, which leads to the full decay range $\alpha>2/7$. See Appendix~\ref{frozen_pro_ODE_appdx} for a proof. 
Therefore, the $2/7$-norm inflation threshold is in fact sharp for the profile-independent approach.

This observation has several consequences. First, given the fact that the $2/7$-norm inflation threshold is in fact independent of a choice of a profile $\phi$ whereas our decay range depends on it, it is reasonable to speculate that the upward shifting of the coefficient $-5$ by $Q_{\mathrm{froz}}(\Gamma)$ might in fact lead to norm inflation beyond $2/7$, and the real norm inflation range might depend on the profile $\phi$. In addition, as one can make the upper bound $\mathrm{max}_{0\leq \Gamma \leq 1} Q_{\mathrm{froz}}(\Gamma)$ small by adjusting the profile $\phi$, one can also make it large, which will make our decay range a small interval near $\alpha=1$. Hence, one might be able to prove norm inflation in the full range, any $\alpha<1$, by changing the profile $\phi$.

To this end, heuristically, if one can somehow increase the coefficient $-5$ in \eqref{ode_bj} to an arbitrary small negative number $-\delta'$, then it will lead to the full inflation range:
    \begin{align}
    \label{zeta_full_range}
        \alpha< \frac{2/\delta'}{1+2/\delta'} \sim 1,
    \end{align}
It can be seen by replacing the coefficient $-5$ by $-\delta'$ in \eqref{ode_bj}--\eqref{27_ni_range}. 

However, it is delicate to extract a net gain from $Q_{\mathrm{froz}}(\Gamma)$ because early-time, large-scale enhancement of stretching can feed back into a
stronger later-time, small-scale slowdown. More precisely, when time $t>0$ is small and the scale index $k$ is small, the aspect ratio $\Gamma_k(t)$ does not become small, which leads to strong stretching. However, at the same time, it leads to stronger slowdown for larger time $t>0$ and large scale index $k$.

Therefore, it is natural to try to effectively utilize  the geometry of $\supp \phi$ in order to prove a norm inflation result beyond $\alpha=2/7$. To this end, we localize $\phi$ around the center $(r_0,\pm z_0)$ of its support and make explicit the dependence of the Biot--Savart coefficient $\Lambda_{\mathrm{froz}}(\Gamma)$ on the geometry of $\supp \phi$.

The front-migration mechanism that we present in the next section~\ref{intro:front_mig}, produces an effective small Riccati constant (depending on geometry of $\supp \phi$).

\subsubsection{A Model with a Localized Biot--Savart Coefficient, and Front-Migration}
\label{intro:front_mig}
The idea of localization suggests to introduce 
ODEs with a ``localized" Biot--Savart coefficient in which the Biot--Savart kernel $K_\Gamma$ is frozen at a representative point
$(r_0,z_0)$  of the support:
\begin{equation}\label{eq:Lambda_fr_intro}
    \Lambda_{\mathrm{froz}}(\Gamma)\ \leadsto\ \Lambda_{\mathrm{loc}}(\Gamma)
    :=
    \frac{r_0^2 z_0}{(r_0^2+\Gamma^2 z_0^2)^{5/2}}
    =
    \frac{z_0}{r_0^3}\Big(1+L^2\Gamma^2\Big)^{-5/2},
    \qquad
    L:=\frac{z_0}{r_0}.
\end{equation}
In this subsection \S~\ref{intro:front_mig}, we do not fix $L,\alpha$ by \eqref{eq:cond_L5}, \eqref{eq:norm_inflation_range_rem}. At the moment, let us regard them as unfixed parameters $L>0, \alpha\in (0,1)$. We will identify conditions of $L,\alpha$ later on.
(As usual, harmless positive constants depending on the localized profile are suppressed in
the discussion below.) Then the original frozen model \eqref{ode_full} becomes the localized one \eqref{eq:loc_coe_model}. 

Recall the frozen model
\[
    b_j^{\mathrm{froz}}(t)=x_j^{\mathrm{froz}}(t)\Gamma_j^{\mathrm{froz}}(t)^2\,\Lambda_{\mathrm{froz}}(\Gamma_j^{\mathrm{froz}}(t)),
    \quad
    S_k^{\mathrm{froz}}(t)=\sum_{i\le k} b_i^{\mathrm{froz}}(t),
    \quad
    \Gamma_j^{\mathrm{froz}}(t)=\left(\frac{x_j^{\mathrm{froz}}(0)}{x_j^{\mathrm{froz}}(t)}\right)^3.
\]
In the localized model, using $x_j^{\mathrm{loc}}(t)=x_j^{\mathrm{loc}}(0)\Gamma_j^{\mathrm{loc}}(t)^{-1/3}$, we can rewrite the outer-ring
contribution purely in terms of the single parameter $\Gamma_j^{\mathrm{loc}}(t)$:
\[
    b_j^{\mathrm{loc}}(t)
    =
    x_j^{\mathrm{loc}}(t)\Gamma_j^{\mathrm{loc}}(t)^2\,\Lambda_{\mathrm{loc}}(\Gamma_j^{\mathrm{loc}}(t))
    =
    x_j^{\mathrm{loc}}(0)\Gamma_j^{\mathrm{loc}}(t)^{5/3}\,\Lambda_{\mathrm{loc}}(\Gamma_j^{\mathrm{loc}}(t)), 
    \quad
    \Gamma_j^{\mathrm{loc}} (t):=
    \left(\frac{x_j^{\mathrm{loc}}(0)}{x_j^{\mathrm{loc}}(t)}
    \right)^3.
\]
Introducing the cone variable
\[
    \zeta_j^{\mathrm{loc}}(t):=L\Gamma_j^{\mathrm{loc}}(t),
\]
and substituting \eqref{eq:Lambda_fr_intro}, we obtain the structural form
\begin{equation}\label{eq:bj_Psi_intro}
    b_j^{\mathrm{loc}}(t) = \ x_j^{\mathrm{loc}}(0)\,L^{-2/3}\,\Psi(\zeta_j^{\mathrm{loc}}(t)),
    \qquad
    \Psi(\zeta):=\zeta^{5/3}(1+\zeta^2)^{-5/2}.
\end{equation}
We call $\Psi$ the \emph{Biot--Savart profile}.
The key point is that \(\Psi\) is monotone on a fixed ``productive'' range:
indeed, $\Psi$ is strictly decreasing on $[\zeta_\ast,\infty)$, where $\zeta_\ast:=1/\sqrt2$.
Since $x_j^{\mathrm{loc}}(t)$ is increasing, $\Gamma_j^{\mathrm{loc}}(t)$ (and hence $\zeta_j^{\mathrm{loc}}(t)$) is decreasing in time. Hence the composition $t\mapsto\Psi(\zeta_j^{\mathrm{loc}}(t))$ is  increasing in time as long as $\zeta_j^{\mathrm{loc}}(t) \geq \zeta_*$, which captures the exact window where stretching has not yet been overtaken by geometric self-slowdown. See Figure~\ref{fig:mon_psi} for an illustration of this monotonicity.

\begin{figure}[t]
\centering
\begin{tikzpicture}[
    scale=1.0,
    x=9cm,
    y=3cm,
    line width=1pt,
    line cap=round,
    line join=round,
    >=stealth
]

% parameters
\def\zstar{0.45}
\def\zinit{0.95}

% axes
\draw[->] (0,0) -- (1.05,0) node[right] {$\zeta$};
\draw[->] (0,0) -- (0,1.15) node[above] {$\Psi$};

% Psi curve (smooth schematic)
\draw[thick]
(0,0)
.. controls (0.10,0.02) and (0.15,0.35) ..
(0.25,0.70)
.. controls (0.32,0.95) and (0.40,1.00) ..
(\zstar,1.00)
.. controls (0.60,0.95) and (0.70,0.60) ..
(0.80,0.35)
.. controls (0.90,0.15) and (1.00,0.10) ..
(1.05,0.10);

% vertical dashed line at zeta*
\draw[dashed] (\zstar,0) -- (\zstar,1.00);
\node[below] at (\zstar,0) {$\zeta_*=1/\sqrt{2}$};

% initial position zeta_j(0)
\draw (\zinit,0.04) -- (\zinit,-0.04);
\node[below] at (\zinit,0) {$\zeta_j(0)=L$};

% arrow showing motion of zeta_j(t)
\draw[->] (\zinit-0.02,-0.35) -- (0.55,-0.35);
\node at (0.75,-0.50) {$\zeta_j(t)\ \text{decreases in time}$};

% arrow and label showing increase of Psi(zeta_j(t))
\draw[->]
(0.98,0.32+0.02)
.. controls (0.86,0.34+0.02) and (0.78,0.55+0.02) ..
(0.70,0.78+0.02)
.. controls (0.63,0.98+0.02) and (0.56,1.05+0.02) ..
(0.50,1.08+0.02);

\node at (0.86,1.0) {$\Psi(\zeta_j(t))$};
\node at (0.92,0.83) {$\text{increases}$};
\node at (0.92,0.66) {$\text{in time}$};
\end{tikzpicture}
\caption{
An illustration of the monotonicity of the Biot--Savart profile $\Psi$.
}
\label{fig:mon_psi}
\end{figure}
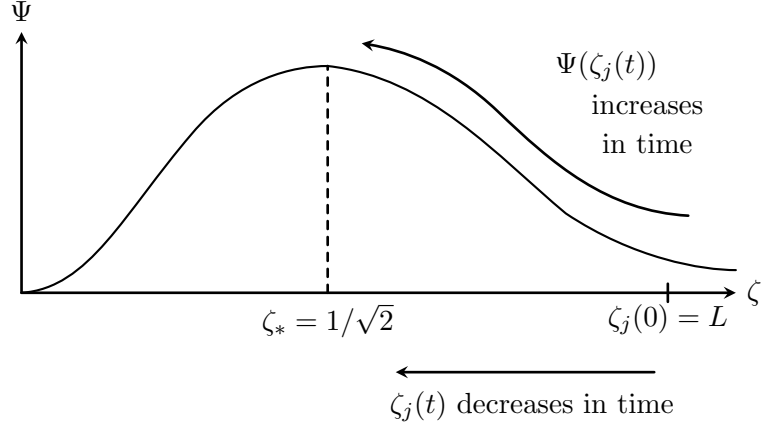

Since $\frac{d}{dt}x_1^{\mathrm{loc}}(t)=0$, one has $\Gamma_1^{\mathrm{loc}}\equiv1$ and therefore $\zeta_1^{\mathrm{loc}}\equiv L$, which is chosen to be large. In addition, $\Gamma_j^{\mathrm{loc}}(0)=1$ for all $j$, hence $\zeta_j^{\mathrm{loc}}(0)=L$ for all $j.$ Hence, starting in the productive regime, $\zeta_j^{\mathrm{loc}}(t)'s$  decreases from $L$ and then eventually passes the threshold $\zeta_*$.

\medskip

\noindent\emph{Front index and the \(t^{-1}\) lower bound for \(S_{m-1}^{\mathrm{loc}}(t)\).}
Define the time-integrated contributions
\[
    B_j^{\mathrm{loc}}(t):=\int_0^t b_j^{\mathrm{loc}}(s)\,ds,
\]
and recall the exact identity
\begin{equation*}
    \frac{d}{dt}\Gamma_{j+1}^{\mathrm{loc}}(t)=-3\Gamma_{j+1}^{\mathrm{loc}}(t)S_j^{\mathrm{loc}}(t),
    \qquad
    \frac{d}{dt}\Gamma_j^{\mathrm{loc}}(t)=-3\Gamma_j^{\mathrm{loc}}(t)S_{j-1}^{\mathrm{loc}}(t).
\end{equation*}
Subtracting the logarithmic derivatives gives
\[
    \frac{d}{dt}\Big(\log \Gamma_{j+1}^{\mathrm{loc}}(t)-\log \Gamma_j^{\mathrm{loc}}(t)\Big)
    =
    -3\big(S_j^{\mathrm{loc}}(t)-S_{j-1}^{\mathrm{loc}}(t)\big)
    =
    -3b_j^{\mathrm{loc}}(t),
\]
and integrating yields the \emph{exact $\Gamma$--cascade identity}
\begin{equation}
    \Gamma_{j+1}^{\mathrm{loc}}(t)=\Gamma_j^{\mathrm{loc}}(t)\,e^{-3B_j^{\mathrm{loc}}(t)},
    \qquad\text{and thus, }
    \qquad
    \label{eq:xi_cascade_intro}
    \zeta_{j+1}^{\mathrm{loc}}(t)=\zeta_j^{\mathrm{loc}}(t)\,e^{-3B_j^{\mathrm{loc}}(t)}.
\end{equation}
As $\zeta_j^{\mathrm{loc}}(t)$ decreases in $t$, the front index $J_{\zeta_0}^{\mathrm{loc}}(t)$ is non-increasing in time, too. We call this behavior of the  front index \emph{front-migration}.

Fix a threshold $\zeta_0\in[\zeta_\ast,L]$ and define the \emph{front index}
\[
    J_{\zeta_0}^{\mathrm{loc}}(t):=\max\Big\{1\le j\le m-1:\ \zeta_j^{\mathrm{loc}}(t)\ge \zeta_0\Big\}.
\]
A schematic profile of the cone variable $\zeta_j^{\mathrm{loc}}(t)$, which shows the behavior of the front index $J_{\zeta_0}^{\mathrm{loc}}(t)$ is drawn in Figure~\ref{fig:front_index}.

\begin{figure}[t]
\centering
\begin{tikzpicture}[
    x=1.55cm,
    y=1.10cm,
    >=Latex,
    thick,
    font=\small,
    declare function={zeta(\x)=0.58 + 3.15*exp(-0.43*(\x-0.55));}
  ]

  % parameters
  \def\yth{1.46}   % threshold height
  \def\J{3}
  \def\Jp{4}

  % axes
  \draw[->] (0,0) -- (6.75,0) node[right] {$j$};
  \draw[->] (0,-1.10) -- (0,4.05) node[above] {$\zeta_j^{\mathrm{loc}}(t)$};

  % threshold line
  \draw[gray!70] (0,\yth) -- (6.25,\yth);
  \node[right] at (6.32,\yth) {$\zeta_0$};

  % smooth monotone guide curve
  \draw[very thick,domain=0.65:6.20,samples=200]
    plot (\x,{zeta(\x)});

  % discrete points at integer indices
  \foreach \x in {1,2,3,4,5,6}
    \fill (\x,{zeta(\x)}) circle (1.15pt);

  % dashed lines ending exactly on the curve
  \draw[densely dashed] (\J,0) -- (\J,{zeta(\J)});
  \draw[densely dashed] (\Jp,0) -- (\Jp,{zeta(\Jp)});

  % ticks
  \foreach \x in {1,2,5,6}
    \draw (\x,0.07) -- (\x,-0.07);
  \draw (\J,0.07) -- (\J,-0.07);
  \draw (\Jp,0.07) -- (\Jp,-0.07);

  % ordinary index labels (smaller, slightly spread)
  \node[font=\footnotesize] at (1,-0.23) {$1$};
  \node[font=\footnotesize] at (2,-0.23) {$2$};
  \node[font=\footnotesize] at (5.10,-0.23) {$m-1$};
  \node[font=\footnotesize] at (6.03,-0.23) {$m$};

  % J and J+1 labels (separated so they do not clash)
  \node[font=\footnotesize, anchor=north east] at (\J+0.30,-0.08) {$J(t)$};
  \node[font=\footnotesize, anchor=north west] at (\Jp-0.40,-0.08) {$J(t)+1$};

  % annotation: each zeta_j decreases in time
  \draw[->] (1.05,2.80) -- (1.05,2.05);
  \draw[->] (1.23,2.68) -- (1.23,1.93);
  \node at (3.15,2.72) {$\zeta_j^{\mathrm{loc}}(t)$ decrease in time};

  % annotation: front index moves left as time increases
  \draw[->] (3.60,-0.62) -- (2.05,-0.62);
  \node[font=\footnotesize] at (3.33,-0.93) {front moves left as time increases};

\end{tikzpicture}
\caption{Schematic profile of the cone variable at a fixed time. Here
\(J(t):=J_{\zeta_0}^{\mathrm{loc}}(t)\) is the largest index such that
\(\zeta_j^{\mathrm{loc}}(t)\ge \zeta_0\), so that
\(\zeta_{J(t)}^{\mathrm{loc}}(t)\ge \zeta_0>\zeta_{J(t)+1}^{\mathrm{loc}}(t)\).
Since each \(t\mapsto \zeta_j^{\mathrm{loc}}(t)\) is nonincreasing, the values
\(\zeta_j^{\mathrm{loc}}(t)\) move downward in time and the front index moves to
smaller values of \(j\). Strictly speaking, the profile consists of the discrete values
\(\{\zeta_j^{\mathrm{loc}}(t)\}_{j=1}^m\) at integer indices \(j\); the continuous curve is
drawn only as a visual guide connecting these discrete points.}
\label{fig:front_index}
\end{figure}
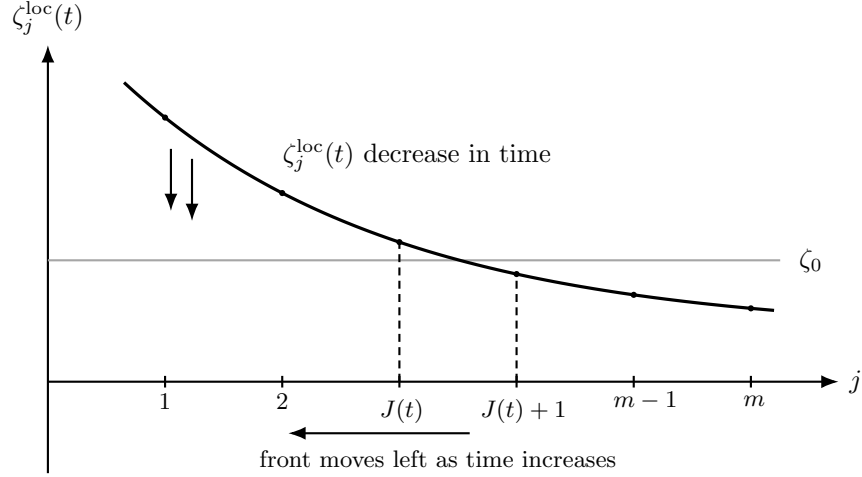

 Assume that the front has passed the smallest scale by time \(t\), i.e.
\(\zeta_m^{\mathrm{loc}}(t)<\zeta_0\).
Then, defining \(J:=J_{\zeta_0}^{\mathrm{loc}}(t)\) temporarily, we have \(\zeta_{J+1}^{\mathrm{loc}}(t)<\zeta_0\).
Iterating \eqref{eq:xi_cascade_intro} from \(1\) to \(J\) and using \(\zeta_1^{\mathrm{loc}}=L\) gives
\[
    \zeta_{J+1}^{\mathrm{loc}}(t)
    =
    L\exp\!\left(-3\sum_{k=1}^{J} B_k^{\mathrm{loc}}(t)\right),
\]
hence
\begin{equation}\label{eq:sumBk_lower_intro}
    \sum_{k=1}^{J} B_k^{\mathrm{loc}}(t)
    =
    \frac13\log\!\left(\frac{L}{\zeta_{J+1}^{\mathrm{loc}}(t)}\right)
    \ge
    \frac13\log\!\left(\frac{L}{\zeta_0}\right).
\end{equation}
This is the first input: \emph{
to drive the cone variables from their initial value $L$ across the threshold $\zeta_0$,
the system must
accumulate at least \(\frac13\log(L/\zeta_0)\) units of integrated outer influence up to the front.}

The second input is monotonicity on the productive regime.
For every \(k\le J\) and every \(s\in[0,t]\), we have
\(
\zeta_k^{\mathrm{loc}}(s)\ge \zeta_k^{\mathrm{loc}}(t)\ge \zeta_0\ge \zeta_\ast
\)
because each \(\zeta_k^{\mathrm{loc}}\) is nonincreasing in time.
Since \(\Psi\) is decreasing on \([\zeta_\ast,\infty)\), the representation \eqref{eq:bj_Psi_intro}
implies that \(s\mapsto b_k(s)\) is nondecreasing on \([0,t]\).
Consequently,
\[
    B_k^{\mathrm{loc}}(t)=\int_0^t b_k^{\mathrm{loc}}(s)\,ds \le t\,b_k^{\mathrm{loc}}(t)
    \qquad (k\le J),
\]
and summing over \(k\le J\) gives
\begin{equation}\label{eq:sumBk_upper_intro}
    \sum_{k=1}^{J} B_k^{\mathrm{loc}}(t)
    \le
    t\sum_{k=1}^{J} b_k^{\mathrm{loc}}(t)
    \le
    t\,S_{m-1}^{\mathrm{loc}}(t).
\end{equation}
Combining \eqref{eq:sumBk_lower_intro} and \eqref{eq:sumBk_upper_intro} yields the desired
\emph{front-induced \(t^{-1}\) lower bound}:
\begin{equation*}
    S_{m-1}^{\mathrm{loc}}(t)\ \ge\ \frac{1}{3t}\,\log\!\left(\frac{L}{\zeta_0}\right)
    \qquad\text{whenever}\qquad \zeta_m^{\mathrm{loc}}(t)<\zeta_0.
\end{equation*}
This estimate is the crux of the front-migration mechanism: given our discussion around \eqref{zeta_full_range}, heuristically, the large coefficient $\log (L/\zeta_0)$ here effectively replaces the Riccati constant $5/2$ in \eqref{low_bd_skt} by an arbitrary small number and yields the full norm inflation range for large $L$:
    \begin{align*}
        \alpha < \frac{\log (L/\zeta_0)}{3+\log(L/\zeta_0)}\to 1 \quad \text{as }L\to\infty.
    \end{align*}

\subsubsection{The Original Weak ODEs and Localization of $\phi$}
\label{local_phi_final}

The localized model is only an intermediate simplification. The actual weak ODEs involve the unfrozen coefficients $\Lambda_j(t,\Gamma)$ coming from the transported profiles $W_j(\cdot,\cdot,t)$. Accordingly, define
\[
    b_j^{\mathrm{orig}}(t)
    :=
    x_j^{\mathrm{orig}}(t)\bigl(\Gamma_j^{\mathrm{orig}}(t)\bigr)^2
    \Lambda_j\!\bigl(t,\Gamma_j^{\mathrm{orig}}(t)\bigr),
    \qquad
    S_k^{\mathrm{orig}}(t):=\sum_{j=1}^k b_j^{\mathrm{orig}}(t),
\]
and
\[
    B_j^{\mathrm{orig}}(t):=\int_0^t b_j^{\mathrm{orig}}(s)\,ds.
\]
Then the original weak system still has the exact cascade form
\[
    \frac{d}{dt}\log x_k^{\mathrm{orig}}(t)=S_{k-1}^{\mathrm{orig}}(t),
    \qquad
    \Gamma_{j+1}^{\mathrm{orig}}(t)
    =
    \Gamma_j^{\mathrm{orig}}(t)e^{-3B_j^{\mathrm{orig}}(t)}.
\]

To connect this system to the localized model, we use two comparisons. First, on the bootstrap interval, the flow-map bounds imply that the transported profiles stay close to the fixed profile $\phi$ after rescaling, so
\[
    \Lambda_j(t,\Gamma)\approx \Lambda_{\mathrm{froz}}(\Gamma).
\]
Equivalently, if
\[
    b_{\mathrm{froz},j}^{\mathrm{orig}}(t)
    :=
    x_j^{\mathrm{orig}}(t)\bigl(\Gamma_j^{\mathrm{orig}}(t)\bigr)^2
    \Lambda_{\mathrm{froz}}\!\bigl(\Gamma_j^{\mathrm{orig}}(t)\bigr),
    \qquad
    S_{\mathrm{froz},k}^{\mathrm{orig}}(t):=\sum_{j=1}^k b_{\mathrm{froz},j}^{\mathrm{orig}}(t),
\]
then
\[
    b_j^{\mathrm{orig}}(t)\approx b_{\mathrm{froz},j}^{\mathrm{orig}}(t),
    \qquad
    S_k^{\mathrm{orig}}(t)\approx S_{\mathrm{froz},k}^{\mathrm{orig}}(t).
\]

Second, localization acts on the frozen coefficient, not directly on $\Lambda_j(t,\Gamma)$. Since $\phi$ is supported near $(r_0,\pm z_{0})$, one has 
\[
    \Lambda_{\mathrm{froz}}(\Gamma)\approx \Lambda_{\mathrm{loc}}(\Gamma).
\]
Hence the correct chain is
\[
    \Lambda_j(t,\Gamma)\approx \Lambda_{\mathrm{froz}}(\Gamma)\approx \Lambda_{\mathrm{loc}}(\Gamma).
\]

Therefore the monotonicity and front-migration estimates proved through the localized Biot--Savart profile $\Psi$ apply first to the frozen quantities and then transfer, up to harmless constants, to the original weak system. Since the exact $\Gamma$-cascade already holds for $\Gamma_j^{\mathrm{orig}}$, the same ODE norm-inflation mechanism remains valid for the original weak ODEs.

\subsection{Scale Separation and Stability of the Profiles}
\label{subsec:scale_sep}

In our construction, it is essential that the vortex rings remain well separated
and the profiles $W_k(r,z,t)$ do not depend much on the scale index $k$ and time $t$
throughout the entire time interval relevant for the analysis, even as norm
inflation actually deforms the configuration. 

More precisely, for scale separation, we must verify that the
separation of spatial scales is preserved for all time across scales, namely that
\begin{align}
\label{scale_sep}
    R_k(t) \gg R_{k+1}(t),
    \qquad\text{equivalently}\qquad
    d\,\frac{\tilde x_{k+1}(t)}{\tilde x_k(t)} \ll 1
\end{align}
uniformly for all $k$ and for all $t$ in the time interval of interest.

In addition, for the stability of the profile $W_k(r,z,t)$, we justify it through a bootstrap argument: Using the fact that $W_k(r,z,t)$ does not change much in $k,t$ for a short time interval, that is, $W_k(r,z,t)$ satisfy ``bootstrap bounds"  for a short time interval, we improve the bootstrap bounds with a margin by using the norm inflation at the ODE level and choosing large $m$. 
It will justify the bootstrap bounds until the ODE norm inflation time.

Indeed, for scale separation, using the ODEs \eqref{ode_full}, one finds
\[
    \frac{\tilde x_{k+1}(t)}{\tilde x_k(t)}
    \sim
    \exp\!\left(
        \int_0^t
        b_k(s)
        \,ds
    \right), \quad b_k(s)= x_k(s) \Gamma_k(s)^2 \Lambda (\Gamma_k(s)).
\]
Thus, controlling the ratio $\tilde x_{k+1}(t)/\tilde x_k(t)$ reduces to estimating
the time integral of $b_k$.

Fix an arbitrary target amplitude $A>0$ and define the \emph{norm inflation time}
$T_N(A,m)$ by
\[
    \sup_{1\le k\le m} x_k(T_N(A,m)) = A.
\]
Norm inflation at the ODE level implies that $T_N(A,m)\to 0$ as $m\to\infty$ for any
fixed $A$. One can in fact obtain an algebraic decay rate.

Using the bound $x_k(t) \leq A$ on the time interval $[0,T_N(A,m)]$, one can at once estimate the time-integral of $b_k$:
    \begin{align}
    \label{small_time_int_b}
        \int_0^t b_k(s) ds \leq C A T_N(A,m)\to 0\quad \text{as }m\to\infty.
    \end{align}
Therefore, by choosing a sufficiently small distance factor $d$ and large $m$,
we ensure that \eqref{scale_sep} holds
for all $k$ and all $t\in[0,T_N(A,m)]$. (In our actual estimate, we will use a bootstrap time interval $[0,T_B]$ instead of the norm inflation time interval $[0,T_N]$. In addition, we will use a rescaled amplitude $\bar A$ instead of $A$. See \S~\ref{roadmap_boot_Euler} for an explanation.)

Regarding stability of profiles $W_k$,  bootstrap improvement is also controlled by the same quantity $AT_N(A,m)$. 
Indeed, the profile stability is governed by the time integral of discrepancy between the real solution and our approximate ansatz solution. Following the heuristic argument behind the ansatz (\S~\ref{sec:framework}), one can estimate the discrepancy by $AT_N(A,m)$.
Hence choosing large $m$ as like \eqref{small_time_int_b} will improve the bootstrap bounds.

Therefore, using scale-separation and stable profiles, we can transfer ODE-level norm inflation to the PDEs, using the smallness of $AT_N(A,m)$ for large $m$.
It finishes the main ideas of our proof for Theorem \ref{main_thm1}.

Before turning to the infinite-ring argument, we record why the bootstrap estimate used below contains an extra logarithmic factor. See \eqref{eq:desired_bd_res_vel_updated}. For smooth finite-ring solutions one could, in principle, exploit Calderón--Zygmund cancellations to obtain sharper self-interaction estimates. Since the same type of estimate must later be used in the contradiction framework for merely bounded Yudovich-type solutions, we instead use a robust \(L^\infty\)-based estimate already in the finite-ring bootstrap. This produces the factor \(\log m\), which remains harmless because the ODE analysis gives an algebraically small time scale \(T_B(A,m)\lesssim m^{-\beta_q}\).

\subsection{Instantaneous Blow-Up}
\label{subsec:instant_blow-up}

To pass from norm inflation to instantaneous blow-up, we take $m=\infty$ in \eqref{KJdata2}, that is, we superpose infinitely many dyadic rings with $\alpha > 1/q$. Since this construction is tailored to model the worst-case profile \eqref{worst_m}, the resulting data satisfy
\[
\omega_0^{(\infty)} \in L^\infty \cap L^1(\mathbb R^3),
\qquad 
\omega_0^{(\infty)}/r \in L^{3,q}(\mathbb R^3),
\]
but $\omega_0^{(\infty)}$ necessarily loses Hölder regularity. This loss of regularity is exactly what makes the instantaneous blow-up problem qualitatively different from the finite-time framework of \cite{CordobaMartinezZheng25}.

\medskip

\noindent
\textbf{Lack of Calderón--Zygmund cancellation.}
In the Hölder-based setting of \cite{CordobaMartinezZheng25}, one exploits Calderón--Zygmund cancellations in the Biot--Savart law to estimate $\nabla u^{(\infty)}$. Making use of that mechanism requires more than bounded vorticity: one must propagate a Hölder norm of $\omega^{(\infty)}$ and, in turn, control the deformation gradient of the flow map through
\[
\frac{d}{dt} DY(t) = \nabla u^{(\infty)}(Y(t),t)\, DY(t).
\]
By contrast, in the contradiction argument for instantaneous blow-up (more precisely, non-existence of Yudovich-type weak solutions), we may assume only that a solution
\[
\omega^{(\infty)} \in L^\infty\bigl(0,T; L^\infty \cap L^1(\mathbb R^3)\bigr)
\]
exists on some short interval. At this level, the associated velocity field is merely log-Lipschitz (see \cite[Lemma~1]{AzzamBedrossian15}), so the Lagrangian flow map $Y$ is still well-defined, but there is no available bootstrap bound on $DY$ or on any Hölder norm of $\omega^{(\infty)}(t)$. In particular, the cancellation mechanism available in the Hölder regime is no longer accessible.

\medskip

\noindent
\textbf{Logarithmic penalty.}
Accordingly, in the instantaneous blow-up argument we estimate the self-induced velocity using only the $L^\infty$ bound on the vorticity together with the geometry of its support, without exploiting any cancellation.

More precisely, using the crude bound
\[
|u(x)| \;\lesssim\; \int_{\mathbb R^3} \frac{|\omega(y)|}{|x-y|^2}\,dy,
\]
and the fact that $\omega_k(\cdot,t)$ is supported in a region of radial scale $R_k$ and vertical thickness $H_k$, one obtains
\[
|u_k(x)| 
\;\lesssim\;
\|\omega_k\|_{L^\infty}
\int_{\supp \omega_k(\cdot,t)} \frac{1}{|x-y|^2}\,dy
\;\lesssim\;
\|\omega_k\|_{L^\infty}\, H_k \log\!\left(\frac{R_k}{H_k}\right).
\]
Therefore, for $(r,z)\in \supp \omega_k(\cdot,t)$, where $|z| \sim H_k$, we obtain
\[
\frac{|u^{z}_{k}(r,z,t)|}{|z|}
\;\lesssim\;
\|\omega_k \|_{L^\infty}
\log\!\left(\frac{R_k}{H_k}\right)
=
\|\omega_k\|_{L^\infty}
\log\!\left(\frac{1}{\Gamma_k}\right),
\]
where $\Gamma_k = H_k/R_k \ll 1$ is the aspect ratio.

Thus, in the absence of Calderón--Zygmund cancellation, a geometric logarithmic loss of order $\log(1/\Gamma_k)$ naturally appears in the self-interaction estimate.

\medskip

\noindent
\textbf{How to overcome the logarithmic loss.}
The key point is that our construction is designed so that this logarithmic loss is dominated by the shrinking time scale.

Along characteristics, one needs to control
\[
\int_0^{T_N(A,k)} 
\left\| \frac{u^{z}_{k}}{z} \right\|_{L^\infty(\supp \omega_k(\cdot,t))}
\,dt
\;\lesssim\;
A\,
\log\!\left(\frac{1}{\Gamma_k(T_N(A,k))}\right)
\, T_N(A,k).
\]
To reach amplitude $A$, the $k$-th ring must satisfy
\[
\Gamma_k(T_N(A,k)) \sim k^{-3\alpha},
\qquad
\log\!\left(\frac{1}{\Gamma_k}\right) \sim \log k.
\]
On the other hand, the ODE analysis shows that the norm-inflation time decays \emph{algebraically}.
Hence, $\log k \cdot T_N(A,k) \to 0,$ so the logarithmic geometric loss remains perturbative.

\medskip

\noindent
\textbf{Further remarks.}
For finite-ring solutions, one can bootstrap the deformation gradient and recover Calderón--Zygmund cancellation, eliminating the logarithmic loss. The appearance of the logarithmic factor is therefore a genuinely infinite-ring phenomenon tied to the lack of Hölder regularity.

It is  worth noting that the extra log factor also appears in a heuristic estimate of $u_k^{z}/z$ based on the local induction approximation.
To see this, we approximate one vortex ring $\supp \omega_k(\cdot,\cdot,t)\cap \{z>0\}$ by a vortex filament given by the centerline circle and compute the self-induced vertical velocity  through local induction approximation. For the notion of local induction approximation, see \cite[Chapter 7]{MajdaBertozzi01}.

In addition, the infinite-ring configuration leads to a vanishing time scale. To handle this, we decompose the configuration into head and tail rings, applying the bootstrap argument only to the head and controlling the tail via outer-ring dominance (see  Proposition~\ref{prop:cone_free_key_lem}).

\subsection{Comparison with the C\'ordoba--Martinez-Zoroa--Zheng Construction}
\label{subsec:CMZ_comparison}
The recent work of C\'ordoba--Martinez-Zoroa--Zheng \,\cite{CordobaMartinezZheng25} constructs finite-time
singularities for the $3$D axisymmetric Euler equations without swirl from initial data in a
H\"older-based local well-posedness regime.
At a very high level, both \,\cite{CordobaMartinezZheng25} and the present paper (and \cite{KimJeong22}) exploit the same multiscale
``vortex-ring cascade'' principle: one superposes rings indexed by $k$ and uses that, at suitable
points, the dominant contribution to the stretching of an inner ring comes from outer rings. 
(Multi-scale constructions also appear in \cite{CordobaMartinezZoroa23ForcedEuler}, \cite{CordobaLainSanclementeMartinezZoroa25}.)

However, the comparison is subtle because there are \emph{two different criticalities} at play:
\begin{itemize}
\item \textbf{H\"older criticality} (pointwise regularity of $\omega$ or $u$ near the axis), which is the
natural scale for classical local well-posedness;
\item \textbf{Lorentz criticality} (integrability of the transported quantity $\omega/r$ in \eqref{eq:rel_vor_eq}), which is
the natural scale for Danchin-type global estimates such as \eqref{Danchin_est}.
\end{itemize}
In particular, even though $\omega/r$ is transported by \eqref{eq:rel_vor_eq}, its Lorentz criticality is 
sensitive to the \emph{geometry of the support near $r=0$}, since the weight $r^{-1}$ amplifies
concentration near the axis. This is precisely where the CMZ configuration differs from ours.

In this section~\ref{subsub:comparison}, we will see that: a) CMZ's data are supercritical in the Lorentz scale because their rings concentrate near the axis, yet this same anisotropy weakens the stretching kernel; b) Our construction instead targets the worst-case conical geometry relevant to Danchin's threshold. c) Additionally, we will also study why our problem must be organized forward in time instead of backward as opposed to CMZ's problem.

\subsubsection{The CMZ Construction: Data Given at the Singular Time, and Backward Evolution}

CMZ prove that there exist axisymmetric no-swirl solutions on a time interval $[0,T)$
whose velocity remains in a H\"older class for every fixed time $0\leq t<T$ but whose vorticity
blows up as $t\uparrow T$. More precisely, they construct solutions $u(x,t)$ to the Euler equation \eqref{eq:Euler}
on $t\in[0,T)$ such that for every $t<T$,
\[
u(\cdot, t)\in C^\infty(\mathbb{R}^3\setminus\{0\})\cap C^{1,\alpha_{\mathrm{CMZ}}}(\mathbb{R}^3)\cap L^2(\mathbb{R}^3),
\qquad
\omega(\cdot,t)=\nabla\times u(\cdot,t)\in C^{\alpha_{\mathrm{CMZ}}}(\mathbb{R}^3),
\]
while
$\max_{x\in\mathbb{R}^3}|\omega(x,t)|\approx 1/|T-t|$ as $t\uparrow T.$
In fact, they worked on a time interval $[-T,0)$ where $t=0$ corresponds to the final time and $t=-T$ the initial time, but for the sake of comparison, let us shift their time interval to $[0,T)$.

A key structural feature is that CMZ organize the cascade \emph{backward in time}.
They start by prescribing a singular ``final state'' at the blow-up time $t=T$ via a dyadic
superposition of vortex rings, and then solve the Euler dynamics backward to obtain
regular initial data at $t=0$.

To avoid confusion, no matter whether we consider a backward problem or a forward problem,
we will always consider the time interval $(0,T)$, and the initial time always means $t=0$, at which the evolution starts, whereas the final time means $t=T$ where the evolution ends.

Concretely, for each truncation level $m$ they prescribe at $t=T$ the (truncated) vorticity
\begin{equation}\label{eq:CMZ_terminal_data}
\omega^{\mathrm{CMZ}}_{(m)}(r,z,T)
=\sum_{k=0}^{m}\omega^{\mathrm{CMZ}}_{k}(r,z,T),
\qquad
\omega^{\mathrm{CMZ}}_{k}(r,z,T)
= A_{\mathrm{CMZ}}^{k}\,
\varphi
\left(
\frac{r}{d_{\mathrm{CMZ}}^{k}},
\frac{z}{d_{\mathrm{CMZ}}^{k}}
\right),
\end{equation}
where $0<A_{\mathrm{CMZ}}-1\ll1, \, 0<d_{\mathrm{CMZ}}\ll 1$, and $\varphi$ is a smooth profile supported near $r=1$ horizontally and in $[-2Z,2Z]$ for large parameter $Z>0$ vertically.
The profile $\varphi$ is odd in $z$. (see \cite{CordobaMartinezZheng25} for details). We have simplified some structure of the profile $\varphi$ that is not essential for our discussion.

This terminal configuration makes the infinite sum (as $m\to\infty$) singular at $t=T$, but
for each finite $m$, the data are smooth and hence generate a global-in-time classical solution in the axisymmetric no-swirl class.
CMZ then pass to a limit $m\to\infty$ to obtain the singular solution on $[0,T)$.

We will call $\omega_{(m)}^{\mathrm{CMZ}}$ given in \eqref{eq:CMZ_terminal_data} \emph{CMZ's terminal data} whereas we call $\omega_{(m)}^{\mathrm{CMZ}}(r,z,0)$ \emph{CMZ's backward-evolved initial data}.

Their multiscale ansatz for each $\omega_k^{\mathrm{CMZ}}$ may be written in the same form as ours but with $x_k(0)$ replaced by $x_k(T)$. Indeed, \cite[Eq. (15)]{CordobaMartinezZheng25} has the form
\[
\omega^{\mathrm{CMZ}}_{k}(r,z,t)
=
x^{\mathrm{CMZ}}_{k}(t)\,
W^{\mathrm{CMZ}}_{k}\!\left(
\frac{r}{R^{\mathrm{CMZ}}_{k}(t)},
\frac{z}{H^{\mathrm{CMZ}}_{k}(t)},
t\right),
\]
where
\begin{equation*}
R^{\mathrm{CMZ}}_{k}(t)
= d_{\mathrm{CMZ}}^{k}\frac{x^{\mathrm{CMZ}}_{k}(t)}{x_k^{\mathrm{CMZ}}(T)},
\qquad
H^{\mathrm{CMZ}}_{k}(t)
= d_{\mathrm{CMZ}}^{k}\frac{(x_k^{\mathrm{CMZ}}(T))^2}{\bigl(x^{\mathrm{CMZ}}_{k}(t)\bigr)^2}, 
\qquad
x_k^{\mathrm{CMZ}}(T)=A^{k}_{\mathrm{CMZ}}.
\end{equation*}
Here $d_{\mathrm{CMZ}}$ is chosen to be small, and $A_{\mathrm{CMZ}}>1$ is chosen to be close to $1$.
At the singular time $t=T$, one has $x^{\mathrm{CMZ}}_{k}(T)=A_{\mathrm{CMZ}}^{k}$, so
$R^{\mathrm{CMZ}}_{k}(T)= H^{\mathrm{CMZ}}_{k}(T)=d_{\mathrm{CMZ}}^{k}$ and the rings
are essentially ``conical'' at $t=T$.  In the backward direction $t<T$, $x^{\mathrm{CMZ}}_{k}(t)$
decreases and therefore the aspect ratio grows:
\begin{equation}\label{eq:CMZ_aspect_ratio}
\Gamma^{\mathrm{CMZ}}_{k}(t)
:=
\frac{H^{\mathrm{CMZ}}_{k}(t)}{R^{\mathrm{CMZ}}_{k}(t)}
=
\left(\frac{A_{\mathrm{CMZ}}^{k}}{x^{\mathrm{CMZ}}_{k}(t)}\right)^{3},
\qquad t<T.
\end{equation}
Thus the backward evolution produces initial data at $t=0$ supported in increasingly
\emph{thin} regions (large $\Gamma^{\mathrm{CMZ}}_{k}$.)

As CMZ utilized a backward problem, their backward-evolved initial data at $t=0$ is not given explicitly. Hence
for later comparison, it is convenient to summarize the CMZ \emph{initial} geometry and pointwise
size by a \emph{continuous envelope} model.  To this end, the following schematic asymptotic behavior is useful: with a fixed $t<T$,
    \begin{align}
    \label{CMZ_x_k}
        x_k^{\mathrm{CMZ}}(t) \sim A^{-k}_{\mathrm{CMZ}} , \quad 
    \end{align}
    for large $k$ when $A_{\mathrm{CMZ}}>1$ is close to $1$.
See \cite[\S~1.3.2]{CordobaMartinezZheng25} for more detailed heuristic computations on this asymptotic behavior. We have simplified the asymptote, $A^{-k}_{\mathrm{CMZ}}$, from the exact asymptote in \cite[\S~1.3.2]{CordobaMartinezZheng25}, using the fact that $A_{\mathrm{CMZ}}>1$ is close to $1$.

A minimal envelope that captures both
\begin{enumerate}
\item the pointwise H\"older-type decay of $\omega$ near the origin, and
\item the concentration of the support toward the axis,
\end{enumerate}
is
\begin{equation}\label{eq:CMZ_envelope}
\omega_{\mathrm{env}}(r,z)
:=
-|(r,z)|^{\beta}\,
\mathbf{1}_{\{\,c|z|^{\gamma}\le r\le C|z|^{\gamma}\,\}}\,
\mathbf{1}_{\{|z|\le 1\}}\,
\mathrm{sgn}(z),
\qquad
\beta>0,\ \gamma>1,
\end{equation}
where 
    \begin{align*}
        \gamma \sim \frac{\log \frac{1}{d_{\mathrm{CMZ}}}+2\log A_{\mathrm{CMZ}}}{\log \frac{1}{d_{\mathrm{CMZ}}}-4\log A_{\mathrm{CMZ}}}>1, 
    \qquad
        \beta\sim  \frac{\log A_{\mathrm{CMZ}}}{\log \frac{1}{d_{\mathrm{CMZ}}} +2\log A_{\mathrm{CMZ}}} \ll1.
    \end{align*}
Here $\gamma$ is close to $1.$ The expressions of $\gamma, \beta$ are valid for an amplitude $A_{\mathrm{CMZ}}>1$ close to $1.$
One can find the expression of $\gamma,\beta$ by using \eqref{CMZ_x_k} together with $r\sim R_k^{\mathrm{CMZ}}(t)$  and $|z|\sim H_k^{\mathrm{CMZ}}(t)$. For the expression of $\beta$, it can be also found in  \cite[\S~1.3.2]{CordobaMartinezZheng25}. 
 We will use \eqref{eq:CMZ_envelope} below to make the Lorentz/H\"older comparison
transparent.

\subsubsection{Comparison with Our data and Solution: Lorentz Spaces, Geometry, and the Impossibility of a Backward Problem} \,
\label{subsub:comparison} 

\medskip
\noindent\textbf{(1) CMZ backward-evolved initial data are more singular than ours in the Lorentz scale.}
The key quantity for Danchin-type criticality is the transported relative vorticity $\omega/r$,
cf.\ \eqref{eq:rel_vor_eq}, and Lorentz norms of $\omega/r$ are propagated along the flow.

It is therefore enough to compute $\omega/r$ from the explicit terminal configuration
\eqref{eq:CMZ_terminal_data}.  On the support of the $k$-th CMZ ring at $t=T$ one has
$r\sim d_{\mathrm{CMZ}}^{k}$, and hence
\begin{equation*}
\left|\frac{\omega^{\mathrm{CMZ}}_{k}(r,z,T)}{r}\right|
\sim
\frac{A_{\mathrm{CMZ}}^{k}}{d_{\mathrm{CMZ}}^{k}}
=\left(\frac{A_{\mathrm{CMZ}}}{d_{\mathrm{CMZ}}}\right)^{k}.
\end{equation*}
Moreover, since each ring is localized at scale $d_{\mathrm{CMZ}}^{k}$ in all directions at $t=T$,
its support has volume $|\supp \omega^{\mathrm{CMZ}}_{k}(T)|\sim d_{\mathrm{CMZ}}^{3k}$.
Consequently, for any $1\le p<\infty$,
\begin{equation}\label{eq:CMZ_Lp_layer}
\left\|\frac{\omega^{\mathrm{CMZ}}_{k}(T)}{r}\right\|_{L^{p}(\mathbb{R}^3)}^{p}
\sim
\left(\frac{A_{\mathrm{CMZ}}^{k}}{d_{\mathrm{CMZ}}^{k}}\right)^{p}
\,d_{\mathrm{CMZ}}^{3k}
=
A_{\mathrm{CMZ}}^{kp}\,d_{\mathrm{CMZ}}^{k(3-p)}.
\end{equation}
Summing over $k$ gives a sharp integrability threshold:
\begin{equation*}
\frac{\omega^{\mathrm{CMZ}}_{(\infty)}(t)}{r}\in L^{p}_{\mathrm{loc}}
\quad\Longleftrightarrow\quad
p<p_\ast,
\qquad
p_\ast:=\frac{3\log(1/d_{\mathrm{CMZ}})}{\log A_{\mathrm{CMZ}}+\log(1/d_{\mathrm{CMZ}})}
\;<\;3.
\end{equation*}
In particular, $\omega_{(\infty)}^{\mathrm{CMZ}}(t)/r\notin L^{3}_{\mathrm{loc}}$ and hence it lies  outside
Danchin's endpoint class $L^{3,1}$ (and outside $L^{3,q}$ for any finite $q$).  More precisely,
the ring counting implicit in \eqref{eq:CMZ_Lp_layer} yields the distribution estimate
$|\{|\omega^{\mathrm{CMZ}}_{(\infty)}/r|>\lambda\}|\sim \lambda^{-p_\ast}$, so
\begin{equation}\label{eq:CMZ_Lorentz_membership}
\frac{\omega^{\mathrm{CMZ}}_{(\infty)}(t)}{r}\in L^{p_\ast,\infty}_{\mathrm{loc}}
\quad\text{but}\quad
\frac{\omega^{\mathrm{CMZ}}_{(\infty)}(t)}{r}\notin L^{p_\ast,q}_{\mathrm{loc}}\ \text{for any }q<\infty.
\end{equation}

By contrast, our construction starts from  data $\omega_0^{(\infty)}$ with
$\omega_0^{(\infty)}/r\in L^{3,q}(\mathbb{R}^3)$ for some $q>1,$
and we can make this $L^{3,q}$ norm arbitrarily small.
Thus, in terms of Lorentz criticality for $\omega^{(\infty)}/r$, the CMZ backward-evolved initial data are \emph{strictly more
singular} than ours: they sit only at a weaker exponent $p_\ast<3$.

\medskip
\noindent\textbf{(2) CMZ's backward-evolved data are more regular than ours in the H\"older scale.}
Despite being more singular than our data in the Lorentz scale, 
CMZ's backward-evolved data (with infinite rings) is more regular than ours in the H\"older scale. Indeed, CMZ's backward-evolved data is H\"older continuous at the origin whereas ours $\omega_0^{(\infty)}$ is not. Hence the evolution of their data is locally classical and well-posed in that topology (local well-posedness in H\"older spaces is explained in the next section \ref{subsec:holder_spaces}) whereas our data are in an ill-posedness regime.

\medskip
\noindent\textbf{(3) Why there is no contradiction: pointwise vs.\ integral scales, and the role of geometry.}
The comparison between CMZ's backward-evolved data and ours is subtle because the two scales measure
different aspects of singularity.

On the one hand, H\"older regularity is a \emph{pointwise}, $L^\infty$-based notion.
Near the origin it asks how fast $\omega(r,z)$ itself vanishes at each point.
On the other hand, the Lorentz scale relevant here is not attached to $\omega$ itself,
but to the transported quantity $\omega/r$ in \eqref{eq:rel_vor_eq}.  This is an
\emph{integral} $L^p$-type measurement, so it depends not only on the pointwise size of
$\omega$, but also on how much of the support lies in the region where $r$ is very small.
Thus the factor $r^{-1}$ and the geometry of the support near the axis are largely
invisible in the H\"older scale, but they are decisive in the Lorentz scale.

For CMZ, the backward-evolved initial data at $t=0$ are concentrated in increasingly thin regions near
the axis; equivalently, their aspect ratios $\Gamma_k^{\mathrm{CMZ}}$ are large, see
\eqref{eq:CMZ_aspect_ratio}.  In the envelope model \eqref{eq:CMZ_envelope}, this is
encoded by
\[
r\sim |z|^\gamma,\qquad \gamma>1,
\qquad
\omega_{\mathrm{env}}(r,z)\sim -|(r,z)|^\beta\mathrm{sgn}(z) .
\]
Hence CMZ have a genuine \emph{pointwise power decay} of $\omega$ near the origin, which
is compatible with a H\"older-based local well-posedness regime.  However, because
$r\ll |z|$ on the support, dividing by $r$ creates a much larger quantity:
\[
\frac{\omega_{\mathrm{env}}(r,z)}{r}
\sim
-\frac{|z|^\beta}{|z|^\gamma}\mathrm{sgn}(z)
=
-|z|^{\beta-\gamma}\mathrm{sgn}(z).
\]
Therefore $\omega/r$ can be highly singular even though $\omega$ itself is H\"older.
Since Lorentz norms measure the distribution of large values of $\omega/r$, this thin
concentration near the axis makes CMZ's backward-evolved data more singular in the Lorentz scale.

Our data have the opposite feature.  Recall the continuous model of our data \eqref{worst_m}, that is,
\[
\omega_0(r,z)\sim -\frac{1}{(\log(1/|(r,z)|))^\alpha}\mathrm{sgn}(z)
\qquad\text{on } \{r\sim |z|\}
\]
This is worse pointwise than any power $|(r,z)|^\beta$, so our data are not H\"older at
the origin.  But now there is no extra concentration toward the axis: on the support,
the factor $r^{-1}$ is  comparable to $|(r,z)|^{-1}$.  Hence, 
\[
\frac{\omega_0(r,z)}{r}
\sim
-\frac{1}{|(r,z)|\,(\log(1/|(r,z)|))^\alpha}\mathrm{sgn}(z) \qquad \text{on } \{r\sim |z|\}
\]
which is the borderline behavior tailored to belong to $L^{3,q}$ for $q>1$ as we already observed in Section~\ref{General_L_space}.

So the difference is not merely that CMZ's backward-evolved data and our data have different norms.
Rather, they optimize different features.  CMZ are \emph{pointwise milder} but
\emph{geometrically more concentrated} near the axis; we are \emph{pointwise rougher}
but \emph{geometrically less concentrated} toward the axis.  This is exactly why CMZ's
backward-evolved data are more regular in the H\"older scale while being more singular in the Lorentz
scale.

\medskip
\noindent\textbf{(4) CMZ are not worst-case in their Lorentz class; a worst-case model.}
Even though CMZ's backward-evolved initial data is more singular than ours in terms of Lorentz criticality, their solution does not exhibit more ``singular" behavior than ours. To explain this discrepancy,
the stretching rate $u^r/r$ is controlled by an $|x|^{-2}$-type singular integral of $\omega/r$,
as reflected in the heuristic behind \eqref{Danchin_est}.  For a fixed Lorentz size of $\omega/r$, the strongest
amplification is obtained by placing the vorticity where the kernel in the Biot--Savart law is
largest.  In particular, the kernel in \eqref{ini_vs} carries a factor $r^2 z/(r^2+z^2)^{5/2}$, so the
dominant contribution comes from the conical region $r\sim |z|$ (this is precisely why the
worst-case scenario \eqref{worst_m} is conical).

CMZ's configuration at the initial time $t=0$ is instead highly anisotropic
(large $\Gamma^{\mathrm{CMZ}}_{k}$), meaning it is concentrated closer to the axis, where the
kernel in \eqref{ini_vs} is suppressed by extra powers of $r/|z|$.
This can be seen transparently on the envelope \eqref{eq:CMZ_envelope}.  Using \eqref{ini_vs} and
$r\sim |z|^\gamma$ (so $r\ll|z|$), we have $(r^2+z^2)^{5/2}\sim |z|^5$.
Therefore, the stretching functional has the scaling
\begin{align*}
\int\!\!\int \frac{r^2 z}{(r^2+z^2)^{5/2}}\,\omega_{\mathrm{env}}(r,z)\,dr\,dz
\sim
\int_{0}^{1}|z|^{\beta}\int_{c|z|^\gamma}^{C|z|^\gamma}\frac{r^2}{|z|^4}\,dr\,dz \nonumber
\sim
\int_{0}^{1} z^{3\gamma+\beta-4}\,dz,
\end{align*}
which is finite whenever $3\gamma+\beta>3$, in particular for any $\gamma>1$ and any
$\beta>0$.
Thus CMZ can be Lorentz-supercritical while still having finite initial stretching in the
sense of \eqref{ini_vs}, because their support geometry avoids the worst-case region for the kernel.

By contrast, a natural ``worst-case'' model in the Lorentz class $L^{p_\ast,\infty}$ suggested by
\eqref{eq:CMZ_Lorentz_membership} is to place $\omega/r$ conically with the borderline power law
\begin{equation*}
\frac{\omega_{\mathrm{worst}}(r,z)}{r}
\sim
-\frac{1}{|(r,z)|^{3/p_\ast}}\,
\mathbf{1}_{\{\,r\sim |z|,\ |(r,z)|\le 1\,\}}\,
\mathrm{sgn}(z),
\qquad p_\ast<3.
\end{equation*}
This saturates the weak-$L^{p_\ast}$ scaling and, when inserted into the heuristic
$u^r/r\sim |x|^{-2}*(\omega/r)$, yields a divergent stretching rate at the origin.
In this sense, the CMZ configuration is not the worst-case configuration inside its Lorentz class.

\medskip
\noindent\textbf{(5) Why CMZ can run a backward problem, while we cannot (and the link to self-slowdown).} 
CMZ's backward organization is compatible with their goal: they prescribe a singular final
state at $t=T$ and evolve backward to obtain regular data at $t=0$ in a H\"older class. The required H\"older regularity does not have the same scaling as the blow-up quantity, the $L^\infty$ norm of vorticity.

For our norm inflation, however, the required Lorentz regularity, $\omega_0/r \in L^{3,q}(\mathbb{R}^3)$, has the \emph{same scaling} as the inflation quantity, the $L^\infty$ norm of vorticity, and 
we must start from \emph{small} data in the critical Lorentz
class $L^{3,q}$ for $\omega/r$.  

Since $\omega/r$ is transported by~\eqref{eq:rel_vor_eq}, its Lorentz norm is
conserved in time.  Moreover, for our dyadic vortex-ring architecture (finite superposition
as in \eqref{KJdata2}), the $L^{3,q}$ norm is essentially an $\ell^q$ norm of the ring amplitudes.
Indeed, 
assume
at the final time $t=T$ (more precisely, norm inflation time) the $k$-th ring has the canonical form
\[
\omega_k(r,z,T)=x_k^T\,W\!\left(\frac{r}{d^k},\frac{z}{d^k}\right),
\]
a scaling computation gives
\begin{equation}\label{eq:ellq_identity_intro}
\left\|\frac{\omega^{(m)}|_{t=T}}{r}\right\|_{L^{3,q}(\mathbb{R}^3)}^{q}
\sim \sum_{k=1}^{m}|x_k^T|^{q}.
\end{equation}
Since \eqref{eq:ellq_identity_intro} is invariant under the transport \eqref{eq:rel_vor_eq},
any attempt to prescribe a \emph{final} state at $t=T$ and solve backward forces
\[
\left\|\frac{\omega^{(m)}|_{t=0}}{r}\right\|_{L^{3,q}}^{q}
=
\left\|\frac{\omega^{(m)}|_{t=T}}{r}\right\|_{L^{3,q}}^{q}
\sim \sum_{k=1}^{m}|x_k^T|^{q}
\ge \Bigl(\max_{1\le k\le m}|x_k^T|\Bigr)^q
\sim \|\omega^{(m)}|_{t=T}\|_{L^\infty(\mathbb{R}^3)}^{q}.
\]
Hence one cannot make $\|\omega^{(m)}|_{t=T}\|_{L^\infty}$ large while keeping
$\|\omega^{(m)}|_{t=0}/r\|_{L^{3,q}}$ small: the backward problem is incompatible with the
small-data norm inflation goal.

This obstruction is precisely why our analysis must be organized \emph{forward in time}.
But in forward time, as explained in Subsection~\ref{geo_ss_mc} and encoded in the weak cascade
model \eqref{original_ODEs}, vortex stretching drives the aspect ratios $\Gamma_k(t)=H_k(t)/R_k(t)$ to become
small, which suppresses the Biot--Savart coefficient and produces the \emph{geometric
self-slowdown} mechanism.  Overcoming this self-slowdown in a Lorentz-critical regime is
the main additional difficulty in our problem compared to the backward cascade strategy of
CMZ.

\subsection{H\"older Spaces}
\label{subsec:holder_spaces}
Lastly, as the study of well-posedness/ill-posedness in Lorentz spaces
$L^{3,q}(\mathbb{R}^3)$ is closely related to borderline H\"older regularity, we briefly
review some relevant background.

First of all, without imposing axisymmetry (and allowing swirl), classical local well-posedness for
$u_0\in C^{k,\alpha}(\mathbb{R}^3)$ ($k\ge 1$, $\alpha\in(0,1)$) goes back to
\cite{Lichtenstein30,Gunther27}. It was later extended to Sobolev data
$u_0\in H^s(\mathbb{R}^3)$ with $s>\frac{5}{2}$ (more generally, $s>\frac{n}{2}+1$ in
$\mathbb{R}^n$) in \cite{EbinMarsden70,Kato72}.

On the other hand,
ill-posedness in the critical Sobolev space $H^{\frac{5}{2}}(\mathbb{R}^3)$ was proved
in \cite{BourgainLi15,BourgainLi21}. A simplified proof in a closely related spirit was
later found by Kim--Jeong \cite{KimJeong22}.

\medskip
\noindent\textbf{Axisymmetric no-swirl blow-up vs.\ global regularity.}
Finite-time blow-up solutions in the axisymmetric no-swirl class have been constructed
from low H\"older regularity in \cite{Elgindi21,CordobaMartinezZheng25}.
Equivalently, these results may be viewed as singularity formation for
$C^{1,\alpha}$ velocities, or $C^\alpha$ vorticities, with small $\alpha>0$.

On the other hand, Danchin's global well-posedness theory implies global regularity
in the axisymmetric no-swirl setting for $\alpha>\frac{1}{3}$, assuming suitable decay
at infinity. Indeed, if $\omega_0\in C^\alpha(\mathbb{R}^3)$ is axisymmetric without
swirl, then $\omega_0^\theta$ vanishes on the axis and satisfies
$|\omega_0(r,z)|\lesssim r^\alpha$ near $r=0$, hence
\begin{equation*}
    \frac{|\omega_0(r,z)|}{r}
    \;\lesssim\;
    r^{\alpha-1}
    \in L^{3,1}_{\mathrm{loc}}(\mathbb{R}^3)
    \qquad \text{if and only if}\qquad
    \alpha>\frac{1}{3}.
\end{equation*}
Combining this with local well-posedness in H\"older spaces and the
Beale--Kato--Majda criterion yields global regularity for
$\omega_0\in C^\alpha$ with $\alpha>\frac{1}{3}$.

During the final preparation of this paper,  a recent preprint by Shkoller \cite{Shkoller26} announced
a construction of finite-time blow-up for every $\alpha\in(0,\frac{1}{3})$ for axisymmetric no-swirl
data in $C^{1,\alpha}(\mathbb{R}^3)\cap L^2(\mathbb{R}^3)$, which means the
threshold $\alpha=\frac{1}{3}$ is sharp up to the endpoint. Moreover, his blow-up
mechanism is structurally stable, in the sense that it persists for an open set of admissible angular profiles.

Both our work and Shkoller's reduce the axisymmetric Euler dynamics to an ODE mechanism capturing vortex stretching near the symmetry axis. The reductions, however, are fundamentally different: Shkoller works with a Lagrangian clock-and-driver system for the on-axis strain and pressure Hessian, while our argument uses a multiscale ansatz for dyadic vortex rings, leading to an outer-ring-driven ODE cascade for amplitudes and aspect ratios. Our key input is the geometric localization and front-migration mechanism for the Biot--Savart coefficient, whereas his key input is a nonperturbative strain-versus-pressure estimate.

\medskip
\noindent\textbf{The critical H\"older space $C^{1/3}$ for vorticity.}

Given Shkoller's result \cite{Shkoller26}, it is worth noting that at the borderline H\"older exponent $\alpha=1/3$, Danchin's criterion still yields global regularity provided the vanishing of the angular vorticity on the symmetry axis is improved by an integrable logarithmic factor; if
    \begin{align*}
        |\omega_0(r,z)| \sim  \frac{r^{1/3}}{(\log(e/r))^{\beta}}, \quad \beta\geq 0, \quad \text{near $r=0$,}
    \end{align*}
 and $\omega_0$ is bounded with suitable decay at infinity, then 
    \begin{align*}
        \frac{\omega_0}{r}\in L^{3,q}(\mathbb{R}^3), \quad \beta \, q>1, \quad 1\leq q<\infty.
    \end{align*}
Therefore, if $\beta>1$, then $\omega_0/r\in L^{3,1}(\mathbb{R}^3)$ so
the corresponding axisymmetric no-swirl Euler solution is global. 

On the other hand, when $\beta\in (0,1]$, the relative vorticity $\omega_0/r$ only belongs to $L^{3,q}(\mathbb{R}^3)$ for $q>1/\beta\geq 1.$ As our ill-posedness results imply these Lorentz spaces are ill-posed for any $q>1$, one cannot extend the global regularity in the H\"older scale beyond $\beta>1$ towards $\beta\in(0,1]$ by merely using well-posedness in the Lorentz scale for $q>1$.

A schematic relation between the H\"older and the Lorentz criticality is drawn in Figure~\ref{Holder_Lorentz}.

\begin{figure}[t]
\centering
\begin{tikzpicture}[
    x=1cm,y=1cm,
    >=Latex,
    line cap=round,
    line join=round,
    axis/.style={very thick},
    tick/.style={line width=.9pt},
    imp/.style={->,line width=1pt},
    overbrace/.style={decorate,decoration={brace,amplitude=4.5pt}},
    underbrace/.style={decorate,decoration={brace,mirror,amplitude=4.5pt}},
    lab/.style={font=\small},
    sml/.style={font=\footnotesize,align=center}
]

% vertical levels
\def\ytop{6.2}
\def\ymid{3.5}
\def\ybot{0.8}

% common x-locations
\def\xL{3.4}
\def\xR{10.9}
\def\xOneTop{4.8}
\def\xThird{8.55}
\def\xZeroTop{10.35}
\def\xOneMid{4.15}
\def\xKJ{9.95}
\def\xOneBot{8.95}
\def\xZeroBot{10.35}

% left labels
\node[lab,anchor=east] at (2.65,\ytop) {$\omega_0\in C^\alpha$};
\node[lab,anchor=east] at (2.65,\ymid) {$\omega_0/r\in L^{3,q}$};
\node[lab,anchor=east] at (2.65,\ybot)
{$\omega_0\sim \dfrac{r^{1/3}}{(\log(e/r))^{\beta}}$};

% ----------------------------------------------------------------
% Top axis: alpha increases to the left
% ----------------------------------------------------------------
\node[lab] at (3.0,\ytop-0.02) {$\alpha$};
\draw[axis,<-] (\xL,\ytop) -- (\xR,\ytop);
\draw[tick] (\xOneTop,\ytop-0.13) -- (\xOneTop,\ytop+0.13);
% \draw[tick] (\xThird,\ytop-0.13) -- (\xThird,\ytop+0.13);
\draw (\xThird,\ytop) circle (2pt);
\draw (\xZeroTop,\ytop) circle (2pt);

\node[lab,above] at (\xOneTop,\ytop+0.16) {$1$};
\node[lab,above] at (\xThird,\ytop+0.16) {$\frac13$};
\node[lab,above] at (\xZeroTop,\ytop+0.10) {$0$};

\draw[underbrace] (\xOneTop,\ytop-0.20) -- (\xThird,\ytop-0.20)
    node[midway,below=7pt,lab] {Glob. Reg.};

\draw[underbrace] (\xThird,\ytop-0.20) -- (10.45,\ytop-0.20)
    node[midway,below=7pt,lab] {FTB};

\node[sml] at (9.85,7.35) {Elgindi '21\\ CMZ '25};
\draw[imp] (9.95,6.95) -- (9.95,6.33);

\node[lab] at (9.55,5.15) {Shkoller '26};

% ----------------------------------------------------------------
% Middle axis: q increases to the right
% ----------------------------------------------------------------
\draw[axis,-{Latex[length=2.8mm,width=2.2mm]}] (\xL,\ymid) -- (11.15,\ymid);
\node[lab] at (11.35,\ymid-0.02) {$q$};

\draw[tick] (\xOneMid,\ymid-0.13) -- (\xOneMid,\ymid+0.13);
\node[lab,above] at (\xOneMid,\ymid+0.16) {$1$};

\node[lab] at (4.45,2.85) {Glob. Reg.};
\node[lab] at (4.30,4.30) {Danchin '07};

\draw[underbrace] (\xOneMid,\ymid-0.20) -- (10.75,\ymid-0.20)
    node[midway,below=7pt,lab] {our ill-posedness};

\node[sml] at (9.95,4.40) {Kim--Jeong '22\\
norm inflation};
\draw[imp] (\xKJ,4.00) -- (\xKJ,3.63);

% implication to the top line
\draw[imp,double,double distance=0.8pt] (4.95,4.55) -- (6.05,5.20);

% ----------------------------------------------------------------
% Bottom axis: beta increases to the left
% ----------------------------------------------------------------
\node[lab] at (3.0,\ybot-0.02) {$\beta$};
\draw[axis,<-] (\xL,\ybot) -- (\xR,\ybot);

% \draw[tick] (\xOneBot,\ybot-0.13) -- (\xOneBot,\ybot+0.13);
\draw (\xOneBot,\ybot) circle (2pt);
\draw[tick] (\xZeroBot,\ybot-0.13) -- (\xZeroBot,\ybot+0.13);

\node[lab,below] at (\xOneBot,\ybot-0.16) {$1$};
\node[lab,below] at (\xZeroBot,\ybot-0.16) {$0$};

\draw[overbrace] (3.55,\ybot+0.20) -- (\xOneBot,\ybot+0.20)
    node[midway,above=7pt,lab] {Glob. Reg.};

% implication from Danchin endpoint to the logarithmic refinement
\draw[imp,double,double distance=0.8pt] (4.43,2.60) -- (5.55,1.85);

% implication from our Lorentz ill-posedness to the logarithmic borderline
\draw[imp,double,double distance=0.8pt] (8.40,2.45) -- (8.95,1.55);
\node[lab,align=left,anchor=west] at (9.15,2.10)
{cannot extend\\ Glob. Reg. further\\ by using Lorentz};

\end{tikzpicture}
\caption{Schematic relation between the H\"older threshold and the Lorentz threshold for axisymmetric no-swirl Euler. Here, Glob. Reg. and FTB stand for global regularity and finite-time blow-up respectively.
Top: Danchin plus local H\"older theory imply global regularity for $\alpha>\frac13$, while Shkoller's result gives finite-time blow-up for every $\alpha\in(0,\frac13)$; earlier blow-up results of Elgindi and CMZ covered the case when $\alpha>0$ is small.
Middle: Danchin gives global well-posedness at $\omega_0/r\in L^{3,1}$, Kim--Jeong proved ill-posedness for sufficiently large $q$, whereas our result is ill-posedness for every $q>1$.
Bottom: at the borderline $\alpha=\frac13$, the logarithmically improved model $\omega_0\sim r^{1/3}/(\log(e/r))^\beta$ is still covered by Danchin when $\beta>1$; the right-hand arrow indicates that our Lorentz-space ill-posedness prevents one from extending this global regularity statement further by Lorentz methods alone.}
\label{Holder_Lorentz}
\end{figure}
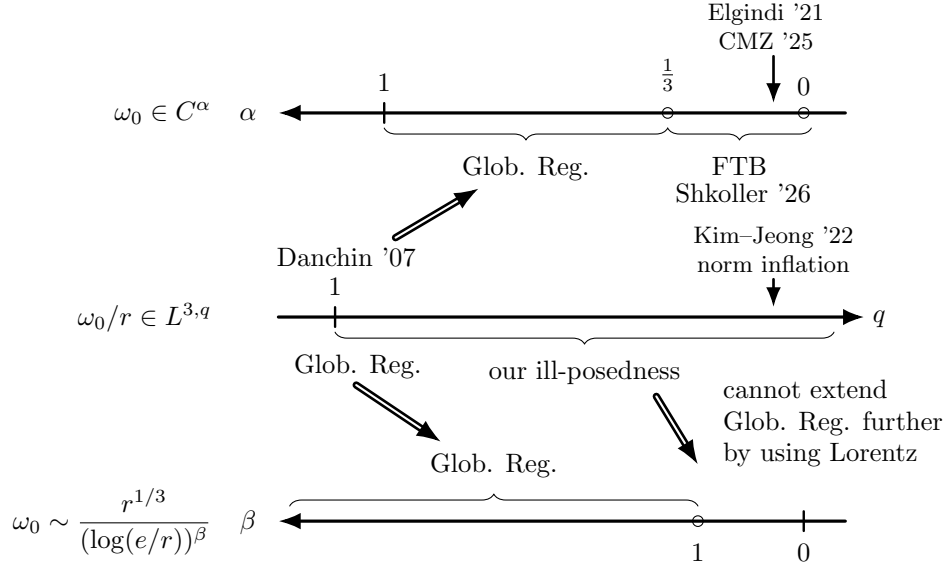

\subsection{Notation Conventions}
\label{sec:out_paper_not}

We record the notation conventions used throughout the paper. First, we distinguish three levels of variables. The physical solution is denoted by $u$ and $\omega$, and physical points are written in cylindrical coordinates $(r,\theta,z)$. When the reference point matters, we write $(r_x,\theta_x,z_x)$ for $x\in\mathbb{R}^3$. The $k$-th vortex ring is described in rescaled coordinates by the profile variables $W_k$, $V_k$, and the flow map $X_k$, together with the amplitude and geometry parameters $x_k$, $R_k$, and $H_k$. Since the $k$-th ring is placed at spatial scale $d^k$, smaller indices correspond to larger, more outer rings; accordingly, sums over $j<k$ represent the influence of outer rings on the $k$-th ring. Finally, the quantities $\Gamma_k$, $\Lambda$, $\Lambda_j$, $b_j$, $S_k$, $\zeta_k$, and $\Psi$ belong to the ODE/stretching cascade and should be read as reduced bookkeeping variables for the multiscale dynamics. A comprehensive lookup table for all these symbols is given in Appendix~\ref{sec:notation}.

In the main parts of the proofs of Theorems~\ref{main_thm1} and ~\ref{thm:instant_blow-up}, we use the following convention for constants. The fundamental parameters are $\tilde\varepsilon>0$ and $q>1$. After $q$ is fixed, the associated quantities $L_q$, $\alpha_q$, and $\phi_q$ are also regarded as fixed. After $\tilde\varepsilon$ is fixed and $\varepsilon=\varepsilon(\tilde\varepsilon,q)$ is chosen in \S~\ref{sec:ini_data_ansatz}, this $\varepsilon$ is likewise treated as fixed. Unsubscripted constants $C,c,\dots$ may depend on all of these fixed parameters including the globally fixed, geometric parameters introduced below, but unsubscripted constants are independent of the running parameters $k,m,A,t,T_N,T_B$ unless explicitly indicated.

We also regard the geometric parameters $r_0,\eta,\mu,d$ as fixed once and for all. In the present paper we take
\[
    r_0=1,\qquad \eta=\frac14,\qquad \mu=\frac1{20},\qquad d=10^{-2}.
\]
We continue to write the letters $r_0$, $\eta$, $\mu$, and $d$ in the estimates in order to keep the geometric meaning of the formulas visible. The particular numerical choice of these constants is not essential. Harmless constants may depend on these fixed parameters without further comment, but we suppress this dependence.

When we write $C(q)$ or $C_E(q)$, the displayed $q$ indicates only the dependence relevant to the discussion; such constants may still depend on the other fixed background parameters introduced above. Similarly, named constants such as $C_E$ or $C_{\mathrm{sep}}$ are simply labels for constants and obey the same convention unless a different dependence is stated locally.

When the domain of integration or of a norm is the whole space and no confusion can arise, we omit it. Unsubscripted constants $C,c,\dots$ are positive and may change from line to line. We likewise suppress nonessential dependence in symbols such as $\omega^{(m)}$, $W_k$, $x_k$, $R_k$, $H_k$, $X_k$, $T_N$, and $T_B$, unless that dependence is being used explicitly. This convention is used only in the main parts of the proofs of Theorems~\ref{main_thm1} and ~\ref{thm:instant_blow-up}. In lemmas, propositions, and appendices that are stated independently of the main theorems, any special dependence and constant conventions are stated locally.

Lastly,  meridional $(r,z)-$functions, such as $W_k,\phi_q$ and $\omega_k$, are also identified with their axisymmetric lifts to $\mathbb{R}^3$. For example, all $L^\infty$ norms are taken with respect to this identification. We will write $\|\phi_q\|_{L^\infty(\mathbb{R}^3)}$ instead of $\|\phi_q\|_{L^\infty(\mathbb{R}_+\times \mathbb{R})}$.

\subsection{Outline of the Paper}

Before we prove the main result, we will first establish a tool (cone-free outer-ring dominance, Proposition~\ref{prop:cone_free_key_lem}) in Section~\ref{sec:cone_free_key_lem}.  Section~\ref{sec:norm_inflation} treats the finite-ring problem and proves the norm inflation statement, Theorem~\ref{main_thm1}. Section~\ref{sec:inst_bu} treats the infinite-ring problem and proves instantaneous blow-up, Theorem~\ref{thm:instant_blow-up}. The first part develops the PDE-to-ODE reduction and closes a bootstrap argument; the second part revisits the same multiscale mechanism in a contradiction framework, where the infinitely many rings are handled through a finite-head bootstrap together with a separate control of the tail.

More precisely, in Section~\ref{sec:norm_inflation} we first fix the initial data, the rescaled variables, and the ansatz in \S~\ref{sec:ini_data_ansatz}, and we summarize the bootstrap strategy in \S~\ref{roadmap_boot_Euler}. In \S~\ref{sec:est_pro}, we derive the Cauchy formula for the rescaled profiles and use the bootstrap assumptions to freeze the transported profiles inside the Biot--Savart integrals. Section~\ref{sec:ODEs_Euler} is the ODE core of the paper: there we analyze the weak cascade for the ring amplitudes and aspect ratios, establish the front-migration lower bound, and obtain quantitative control of the ODE norm-inflation time. In \S~\ref{subsec:Euler_bootstrap_improvement}, we return to the PDE and prove confinement of supports, scale separation, and bounds for the inner, outer, and self-induced velocity contributions; these estimates improve the bootstrap assumptions and identify the bootstrap time with the ODE norm-inflation time. The proof of Theorem~\ref{main_thm1} is then completed in \S~\ref{sec:proof_Thm1}.

Section~\ref{sec:inst_bu} proves Theorem~\ref{thm:instant_blow-up}. In \S~\ref{sec:prel_lem}, we collect the preliminary ingredients needed for Theorem~\ref{thm:instant_blow-up}-(1). In \S~\ref{subsec:non_existence}, we carry out the finite-head bootstrap for the infinite-ring configuration and combine it with the contradiction hypothesis to prove the nonexistence result, Theorem~\ref{thm:instant_blow-up}-(1).  Section~\ref{sec:pf_rem}  shows that subsequential limits of the finite-ring approximations yields an (axisymmetric no-swirl) distributional solution that blows up instantaneously, which proves Theorem~\ref{thm:instant_blow-up}-(2).

The appendices collect auxiliary material. Appendix~\ref{app:two_lemmas} contains the proofs of the Danskin-type lemma and the Cauchy formula. Appendix~\ref{app:ODE_models} records two auxiliary ODE analyses discussed heuristically in \S~\ref{ODE_intro} but not used in the proofs of the main theorems: \S~\ref{HAodes} treats the flattened-coefficient model, while \S~\ref{frozen_pro_ODE_appdx} treats the frozen-profile model with a non-flattened Biot--Savart coefficient. Finally, Appendix~\ref{sec:notation} is a list of notation for quick reference.

\section{Cone-Free Outer-Region Dominance}
\label{sec:cone_free_key_lem}
In this section, before we prove our main results, we prove one of the tools we will use. We will prove that the outer region is dominant when estimating the velocity field $u$ induced by vorticity $\omega$ via the Biot--Savart law provided that $\omega$ is \emph{odd-symmetric in $z$}. 

See our sign convention related to the Biot--Savart law in Appendix~\ref{sec:notation}.

Let us first briefly review a lemma on this topic established by Kim--Jeong \cite[Lemma 2.1]{KimJeong22}. They proved the dominance of the outer-region, that is,
    \begin{align}
    \label{eq:KJ_key_lem}
    \begin{aligned}
        \left| \frac{u^r(x)}{r_x} - \frac{3}{8\pi} \int_{Q(r_x)} \frac{r_y z_y}{|y|^5}  \big( -\omega(y) \big) \, dy \right| 
    &\le 
    C \, \|\omega\|_{L^\infty}, 
    \\
    \left| \frac{u^z(x)}{z_x} + \frac{3}{4\pi}   \int_{Q(r_x)} \frac{r_y z_y}{|y|^5} \big( -\omega(y) \big) \, dy \right| 
    &\le C  \left( 1 + \left| \log \frac{r_x}{|z_x|} \right| \right) \|\omega\|_{L^\infty}
    \end{aligned}
\end{align}
where $Q(r_x) := \{ y \in \mathbb{R}^3 : r_y \ge 2r_x \}$,
under the odd-symmetry assumption together with $\omega\in L^\infty \cap L^2(\mathbb{R}^3)$ and a cone restriction
    \begin{align}
    \label{eq:KJ_cone}
        r_x \geq |z_x|.
    \end{align}

Due to the cone restriction \eqref{eq:KJ_cone}, it is not applicable to our situation since a large cone-slope $L_q=z_{0,q}/r_0$ plays a central role in our analysis. Hence, one needs to generalize their cone restriction to an arbitrary cone slope $L\geq |z_x|/r_x$. At the same time, it is reasonable to try to generalize it while keeping the outer-region $Q(r_x)$, that is independent of the cone slope $L$. Otherwise, a large cone slope $L_q$ will impose a more restrictive scale separation between rings. 

To this end, our strategy is to first completely remove the cone restriction from Kim--Jeong's lemma in order to isolate the role of the cone restriction. Then we apply a generalized cone restriction to the cone-free lemma when we utilize it for our main results.

With the cone restriction, it holds $|x|\sim r_x$. Hence, it is natural to  try to make estimates with $Q(r_x)$ in \eqref{eq:KJ_key_lem} replaced by $Q(|x|)$. However, we still want to keep the outer-region $Q(r_x)$. Hence, we will first derive the $Q(|x|)$-estimates and extend it to $Q(r_x)$-estimates. 

\begin{proposition}[Cone-free outer-region dominance]
\label{prop:cone_free_key_lem}
    Assume that a function $\omega=\omega(r,z) \in L^{\infty}\cap L^2(\mathbb{R}^3)$ and an axisymmetric, no-swirl velocity field $u$ are related via the axisymmetric Biot--Savart law \eqref{eq:axis_BS_law}.
Assume further that $\omega$ is odd with respect to $z$: $\omega(r,z) = -\omega(r,-z)$.
For any $x\in \mathbb{R}^3$ and any positive number $\rho> 0$, define an outer-region domain $Q(\rho)$ by
\[
Q(\rho) := \{ y \in \mathbb{R}^3 : r_y \ge 2\rho \}.
\]
and denote the velocity-errors that we want to estimate by
    \begin{align*}
        E^r(x;\rho)
        &:=
        \frac{u^r(x)}{r_x} - \frac{3}{8\pi} \int_{Q(\rho)} \frac{r_y z_y}{|y|^5} \big(-\omega(y) \big) \, dy, \qquad r_x>0
        \\
        E^z(x;\rho)
        &:=
         \frac{u^z(x)}{z_x} + \frac{3}{4\pi}   \int_{Q(\rho)} \frac{r_y z_y}{|y|^5} \big(-\omega(y) \big) \, dy, \qquad z_x\neq 0.
    \end{align*}
Then there exists an absolute constant $C > 0$ such that for any $x \in \mathbb{R}^3$, the following estimates hold whenever the displayed denominators are nonzero:
\begin{itemize}
\item \emph{\textbf{The outer-region $Q(|x|)$ with the largest cut-off scale $|x|$.}}
 \begin{align}
 \label{eq:Er_x}
     |E^r(x;|x|)| 
     &\leq C \frac{|x|}{r_x} \|\omega\|_{L^\infty (\mathbb{R}^3)},
     \\
     \nonumber
     |E^z(x;|x|)|
     &\leq C\left( 1+ \log \frac{|x|}{|z_x|}\right)\|\omega\|_{L^\infty (\mathbb{R}^3)};
 \end{align}
\item \emph{\textbf{An outer-region $Q(\rho)$ with a smaller cut-off scale $\rho \leq |x|$.}}
For any $0<\rho \leq |x|$,
\begin{align*}
     \quad\quad\quad\quad  |E^r(x;\rho)| 
     &\leq C \left( \frac{|x|}{r_x} 
     + \log \frac{|x|}{\rho} \right)
     \|\omega\|_{L^\infty (\mathbb{R}^3)}, 
     \\
     \quad \quad\quad\quad |E^z(x;\rho)|
     &\leq C\left( 1+ \log \frac{|x|}{|z_x|} + \log \frac{|x|}{\rho }\right)\|\omega\|_{L^\infty (\mathbb{R}^3)};
 \end{align*}
\item  \emph{\textbf{A special outer-region $Q(r_x)$ with $\rho=r_x$}.}
\begin{align}
    \label{eq:key_lem_ur_rx}
     \quad\quad\quad\quad\quad |E^r(x;r_x)| 
     &\leq C \frac{|x|}{r_x} 
     \|\omega\|_{L^\infty (\mathbb{R}^3)}, 
     \\
     \nonumber
     \quad\quad\quad\quad\quad |E^z(x;r_x)|
     &\leq C\left( 1+ \left| \log \frac{r_x}{|z_x|}\right|\right)\|\omega\|_{L^\infty (\mathbb{R}^3)}.
 \end{align}
\end{itemize}
\end{proposition}
The estimates for $Q(\rho), \rho \leq |x|,$ and $Q(r_x)$ follow from the $Q(|x|)$ estimates and the elementary annular enlargement bound
\[
\left|
\int_{Q(\rho)\setminus Q(|x|)}
\frac{r_yz_y}{|y|^5}\omega(y)\,dy
\right|
\le
\frac{4\pi}{3}\log\frac{|x|}{\rho}\|\omega\|_{L^\infty},
\qquad 0<\rho\le |x|.
\]
Indeed, in cylindrical coordinates, 
\begin{align}
\label{eq:ele_bd}
\int_{Q(\rho)\setminus Q(|x|)}
\left|\frac{r_yz_y}{|y|^5}\right|\,dy
=2\pi
\int_{2\rho}^{2|x|}
\int_{\mathbb R}
\frac{s^2|z|}{(s^2+z^2)^{5/2}}\,dz\,ds
=\frac{4\pi}{3}
\int_{2\rho}^{2|x|}\frac{ds}{s}
=\frac{4\pi}{3}
\log\frac{|x|}{\rho}.
\end{align}

As a result of removing the cone restriction, the right-hand side of our radial estimate~\eqref{eq:key_lem_ur_rx} has a new factor $|x|/r_x$ as opposed to Kim--Jeong's Key Lemma.
However, a generalized cone condition $L\, r_x \geq |z_x|$ will allow us to bound the new factor $\frac{|x|}{r_x} $ by $\sqrt{1+L^2}$. Thus, the new radial error factor is harmless in our application. 

The $L^2$ condition of $\omega$ is not optimal but suffices for our proof.

To prove our main result, we will only use the $Q(r_x)$-estimates, and we will do so with a generalized cone condition $Lr_x\geq |z_x|$. But we still stated $Q(|x|)$-and $Q(\rho)$-estimates since this statement indicates the structure of our proof, and this general statement might be useful for other problems as various versions of this lemma were derived and used in other literature.
 See \cite[Page~6]{KimJeong22} for a historical note of the ``Key Lemma" in the topic of ill-posedness of the Euler equations for more references.

\begin{proof}
Throughout the proof, \(C>0\) denotes an absolute constant which may change
from line to line. By rotational invariance, we may assume without loss of
generality that
\[
    x=(r_x,0,z_x).
\]
We write
\[
    y=(r_y\cos\theta_y,r_y\sin\theta_y,z_y),
\]
so that
\[
    x\cdot y=r_xr_y\cos\theta_y+z_xz_y.
\]
For every axisymmetric set \(E\subset\mathbb R^3\) and every integrable
axisymmetric function \(F=F(r_y,z_y)\), we shall use
\begin{equation}
\label{eq:angular_cancel_conefree}
    \int_E \cos\theta_y\,F(r_y,z_y)\,dy=0,
    \qquad
    \int_E \cos^2\theta_y\,F(r_y,z_y)\,dy
    =
    \frac12\int_E F(r_y,z_y)\,dy .
\end{equation}
These identities are first justified on truncated sets \(E\cap B_R(0)\) and
then by letting \(R\to\infty\). The required convergence follows from
\(\omega\in L^\infty\cap L^2\) and the fact that the far-field kernels appearing
below are square-integrable at infinity.

First note that the main integral is well-defined. Indeed, on \(Q(\rho)\) we have
\(|y|\ge r_y\ge2\rho>0\), so there is no singularity near the origin, and
\[
    \left|\frac{r_yz_y}{|y|^5}\right|
    \le \frac{1}{|y|^3}.
\]
The function \(|y|^{-3}\) is locally integrable away from the origin and belongs
to \(L^2(\{|y|\ge1\})\), while \(\omega\in L^\infty\cap L^2\).

We will mainly estimate the velocity-errors $E^r(x;|x|), E^z(x,|x|)$ with the outer-region $Q(|x|)$. The other estimates follow from elementary computations, which we will carry out at the last step.

\medskip
\noindent\textbf{Step 1: the radial estimate.}
From the radial component of the axisymmetric Biot--Savart law,
\[
    u^r(x)
    =
    \frac{1}{4\pi}
    \int_{\mathbb R^3}
    \frac{(z_x-z_y)\cos\theta_y}{|x-y|^3}\omega(y) \,dy.
\]
We write
\[
    u^r(x)=\frac{1}{4\pi}(I_{Q}+I_{Q^c}),
\]
where
\begin{align*}
    I_{Q}
    :&=
    \int_{Q(|x|)}
    \frac{(z_x-z_y)\cos\theta_y}{|x-y|^3} \omega(y) \,dy,
    \\
    I_{Q^c}
    :&=
    \int_{Q(|x|)^c}
    \frac{(z_x-z_y)\cos\theta_y}{|x-y|^3}\omega(y)\,dy
\end{align*}
where $Q(|x|)^c$ denotes the complement of $Q(|x|)$.

On \(Q(|x|)\), we have \(|y|\ge r_y\ge 2|x|\), and therefore
\[
    \left|
    \frac{1}{|x-y|^3}
    -
    \frac{1}{|y|^3}
    -
    \frac{3x\cdot y}{|y|^5}
    \right|
    \le
    C\frac{|x|^2}{|y|^5}.
\]
Since \(|z_x-z_y|\le |x|+|y|\le C|y|\) on \(Q(|x|)\), the error
satisfies
\[
    |\mathcal{E}_{Q}|
    \le C|x|^2\|\omega\|_{L^\infty}
    \int_{|y|\ge 2|x|}\frac{1}{|y|^4}\,dy
    \le
    C|x|\|\omega\|_{L^\infty}.
\]
Thus
\[
\begin{aligned}
    I_{Q}
    &=
    \int_{Q(|x|)}
    (z_x-z_y)\cos\theta_y
    \left(
    \frac1{|y|^3}
    +
    \frac{3x\cdot y}{|y|^5}
    \right)\omega(y) \,dy
    +\mathcal{E}_{Q}.
\end{aligned}
\]
The term containing \(|y|^{-3}\) vanishes by the first identity in
\eqref{eq:angular_cancel_conefree}. Since
\[
    x\cdot y=r_xr_y\cos\theta_y+z_xz_y,
\]
the part containing \(z_xz_y\cos\theta_y\) also vanishes by angular integration.
Hence
\[
    I_{Q}
    =
    3r_x
    \int_{Q(|x|)}
    \frac{(z_x-z_y)r_y\cos^2\theta_y}{|y|^5}
    \omega(y) \,dy
    +\mathcal{E}_{Q}.
\]
The contribution of \(z_x\) vanishes because \(Q(|x|)\) is symmetric
with respect to \(z_y\mapsto -z_y\), while
\(r_y|y|^{-5}\cos^2\theta_y\) is even in \(z_y\) and \(\omega\) is odd in \(z_y\).
Therefore, using the second identity in \eqref{eq:angular_cancel_conefree},
\begin{align}
\label{eq:radial_far_error_conefree}
    I_{Q}
    =
    \frac{3r_x}{2}
    \int_{Q(|x|)}
    \frac{r_yz_y}{|y|^5} \big(-\omega(y)\big) \,dy
    +\mathcal{E}_{Q},
\qquad
    |\mathcal{E}_{Q}|
    \le
    C|x|\|\omega\|_{L^\infty}.
\end{align}

Next, we estimate the complementary region \(Q(|x|)^c=\{r_y<2|x|\}\). Since
\[
    |z_x-z_y|\le |x-y|,
\]
we have
\[
    |I_{Q^c}|
    \le
    \|\omega\|_{L^\infty}
    \int_{\{r_y<2|x|\}}\frac{1}{|x-y|^2}\,dy.
\]
Denoting the projections of $x,y$ onto the horizontal plane by $x_h, y_h\in \mathbb{R}^2$,
let \(h=y_h-x_h\) and \(s=z_y-z_x\). If \(r_y<2|x|\), then
\[
    |h|\le r_y+r_x\le 3|x|.
\]
Consequently,
\[
\begin{aligned}
    \int_{\{r_y<2|x|\}}\frac{1}{|x-y|^2}\,dy
    &\le
    \int_{\{|h|\le 3|x|\}}\int_{\mathbb R}
    \frac{1}{|h|^2+s^2}\,ds\,dh
    =
    C\int_0^{3|x|}d\rho=C|x|.
\end{aligned}
\]
Thus
\[
    |I_{Q^c}|
    \le
    C|x|\|\omega\|_{L^\infty}.
\]

Combining the estimates above gives
\[
    \left|
    I_{Q}+I_{Q^c}
    -
    \frac{3r_x}{2}
    \int_{Q(|x|)}
    \frac{r_yz_y}{|y|^5}\big( -\omega(y) \big)\,dy
    \right|
    \le
    C|x|\|\omega\|_{L^\infty}.
\]
Since $u^r(x)=\frac{1}{4\pi}
    (I_{Q}+I_{Q^c}),$
dividing by \(r_x\) proves \eqref{eq:Er_x}.

\medskip
\noindent\textbf{Step 2: an elementary logarithmic integral.}
We shall use the following estimate: for all \(R>0\) and \(a>0\),
\begin{equation}
\label{eq:log_integral_conefree}
    \int_{\{|h|\le R\}}\int_0^\infty
    \frac{z}
    {\big(|h|^2+(z-a)^2\big)\big(|h|^2+(z+a)^2\big)}
    \,dz\,dh
    \le
    C\left(1+\log_+\frac{R}{a}\right),
\end{equation}
where \(\log_+s:=\max\{\log s,0\}\).

Indeed, writing \(\rho=|h|\), we first observe that
\[
\begin{aligned}
    \int_0^\infty
    \frac{z}
    {(\rho^2+(z-a)^2)(\rho^2+(z+a)^2)}
    \,dz
    &=
    \frac{1}{4a}
    \int_0^\infty
    \left[
    \frac{1}{\rho^2+(z-a)^2}
    -
    \frac{1}{\rho^2+(z+a)^2}
    \right]dz
    \\
    &=
    \frac{1}{4a}
    \int_{-a}^{a}\frac{1}{\rho^2+s^2}\,ds
    =
    \frac{1}{2a\rho}\arctan\left(\frac{a}{\rho}\right).
\end{aligned}
\]
 Therefore,
\[
\begin{aligned}
    \int_{\{|h|\le R\}}\int_0^\infty
    \frac{z}
    {\big(|h|^2+(z-a)^2\big)\big(|h|^2+(z+a)^2\big)}
    \,dz\,dh
    =
    \frac{\pi}{a}\int_0^R
    \arctan\left(\frac{a}{\rho}\right)\,d\rho .
\end{aligned}
\]
The part \(0< \rho\le a\) is bounded by an absolute constant, and for
\(\rho\ge a\) we use \(\arctan(a/\rho)\le a/\rho\). This proves
\eqref{eq:log_integral_conefree}.

\medskip
\noindent\textbf{Step 3: the vertical estimate.}
We will estimate for \(u^z\). By the reflection symmetry in \(z\), it is
enough to consider \(z_x>0\). Indeed, if \(\widetilde x=(r_x,0,-z_x)\), then
\[
    u^z(\widetilde x)=-u^z(x),
    \qquad
    \widetilde z_x=-z_x,
    \qquad
    Q(|\widetilde x|)=Q(|x|).
\]
Thus the estimate for \(\widetilde x\) is equivalent to the estimate for \(x\).

Assume now \(z_x>0\). Let
\[
    \bar y:=(y_h,-z_y),
    \qquad
    A:=|x-y|,
    \qquad
    B:=|x-\bar y|.
\]
Using the oddness of \(\omega\) in \(z_y\), we rewrite the vertical Biot--Savart law as
\[
    u^z(x)
    =
    -
    \frac{1}{4\pi}
    \int_{\{z_y>0\}}
    (r_x\cos\theta_y-r_y)
    \left(\frac1{A^3}-\frac1{B^3}\right)
\omega(y)\,dy.
\]
Since
\[
    B^2-A^2=4z_xz_y,
\]
we have the exact identity
\begin{equation}
\label{eq:diff_kernel_vertical_conefree}
    \frac1{A^3}-\frac1{B^3}
    =
    \frac{4z_xz_y(A^2+AB+B^2)}{A^3B^3(A+B)}.
\end{equation}
Write
\[
    u^z(x)=\frac{1}{4\pi}(J_{Q}+J_{Q^c}),
\]
where the two terms are the integrals over
\[
    Q(|x|)\cap\{z_y>0\},
    \qquad
    Q(|x|)^c\cap\{z_y>0\},
\]
respectively.

We first estimate \(J_{Q}\). On \(Q(|x|)\), we have
\[
    \frac12|y|\le A,B\le \frac32|y|.
\]
Define
\[
    G(a,b):=\frac{a^2+ab+b^2}{a^3b^3(a+b)}.
\]
Then \(G\) is smooth near \((1,1)\) and \(G(1,1)=3/2\). Hence
\[
    \left|
    \frac{A^2+AB+B^2}{A^3B^3(A+B)}
    -
    \frac{3}{2|y|^5}
    \right|
    \le
    C\frac{|x|}{|y|^6}
    \qquad\text{on }Q(|x|).
\]
Using \eqref{eq:diff_kernel_vertical_conefree}, we obtain
\[
\begin{aligned}
    J_{Q}
    &=
    -6z_x
    \int_{Q(|x|)\cap\{z_y>0\}}
    (r_x\cos\theta_y-r_y)
    \frac{z_y}{|y|^5}
    \omega(y) \,dy
    +\mathcal{E}_{Q}^z .
\end{aligned}
\]
The term containing \(r_x\cos\theta_y\) vanishes by angular integration, and therefore
\[
    J_{Q}
    =
    6z_x
    \int_{Q(|x|)\cap\{z_y>0\}}
    \frac{r_yz_y}{|y|^5}\omega(y) \,dy
    +\mathcal{E}_{Q}^z .
\]
The error satisfies
\[
\begin{aligned}
    |\mathcal{E}_{Q}^z|
    &\le
    C z_x |x| \|\omega\|_{L^\infty}
    \int_{|y|\ge 2|x|}
    \frac{(r_x+r_y)z_y}{|y|^6}\,dy
    \\
    &\le
    C z_x |x| \|\omega\|_{L^\infty}
    \int_{|y|\ge 2|x|}
    \frac{1}{|y|^4}\,dy
    \le
    C z_x\|\omega\|_{L^\infty}.
\end{aligned}
\]

It remains to estimate \(J_{Q^c}\), the contribution of
\(Q(|x|)^c\cap\{z_y>0\}\). As above,
\[
    \left|
    (r_x\cos\theta_y-r_y)
    \left(\frac1{A^3}-\frac1{B^3}\right)
    \right|
    \le
    C z_x
    \frac{z_y}{A^2B^2}.
\]
On \(Q(|x|)^c=\{r_y<2|x|\}\), if \(h=y_h-x_h\), then \(|h|\le 3|x|\). Thus,
by \eqref{eq:log_integral_conefree} with \(R=3|x|\) and \(a=z_x\),
\[
\begin{aligned}
    |J_{Q^c}|
    &\le
    C z_x\|\omega\|_{L^\infty}
    \int_{\{|h|\le 3|x|\}}\int_0^\infty
    \frac{z_y}
    {\big(|h|^2+(z_y-z_x)^2\big)
     \big(|h|^2+(z_y+z_x)^2\big)}
    \,dz_y\,dh
    \\
    &\le
    C z_x
    \left(1+\log\frac{|x|}{z_x}\right)
    \|\omega\|_{L^\infty}.
\end{aligned}
\]

Combining the estimates for \(J_{Q},J_{Q^c}\), we obtain
\[
    \left|
    J_{Q}+J_{Q^c}
    +
    6z_x
    \int_{Q(|x|)\cap\{z_y>0\}}
    \frac{r_yz_y}{|y|^5}\big( -\omega(y) \big)\,dy
    \right|
    \le
    C z_x
    \left(1+\log\frac{|x|}{z_x}\right)
    \|\omega\|_{L^\infty}.
\]
Finally, since \(z_y\omega(y)\) is even in \(z_y\),
\[
    \int_{Q(|x|)\cap\{z_y>0\}}
    \frac{r_yz_y}{|y|^5}\omega(y)\,dy
    =
    \frac12
    \int_{Q(|x|)}
    \frac{r_yz_y}{|y|^5}\omega(y)\,dy.
\]
Therefore,
\[
    \left|
    J_{Q}+J_{Q^c}
    +
    3z_x
    \int_{Q(|x|)}
    \frac{r_yz_y}{|y|^5} \big( -\omega(y)\big)\,dy
    \right|
    \le
    C z_x
    \left(1+\log\frac{|x|}{z_x}\right)
    \|\omega\|_{L^\infty}.
\]
Since
\[
    u^z(x)
    =
    \frac1{4\pi}
    (J_{Q}+J_{Q^c})
\]
dividing by \(z_x>0\) gives
\[
    \left|
    \frac{u^z(x)}{z_x}
    +
    \frac{3}{4\pi}
    \int_{Q(|x|)}
    \frac{r_yz_y}{|y|^5}\big( -\omega(y) \big)\,dy
    \right|
    \le
    C
    \left(1+\log\frac{|x|}{z_x}\right)
    \|\omega\|_{L^\infty}.
\]
As explained at the beginning of Step 3, the case \(z_x<0\) follows by reflection.
Replacing \(z_x\) by \(|z_x|\) inside the logarithm finishes the proof for the $Q(|x|)$-estimate.

\medskip
\noindent\textbf{Step 4: smaller cutoffs and the special choice \(\rho=r_x\).}
Let
\[
A(\rho):=
\int_{Q(\rho)}
\frac{r_yz_y}{|y|^5} \big( -\omega(y) \big)\,dy.
\]
For \(0<\rho\le |x|\), the annular estimate \eqref{eq:ele_bd} gives
\[
|A(\rho)-A(|x|)|
\le
\frac{4\pi}{3}
\log\frac{|x|}{\rho}\,
\|\omega\|_{L^\infty}.
\]
Hence
\[
|E^r(x;\rho)|
\le
|E^r(x;|x|)|
+
\frac{3}{8\pi}|A(\rho)-A(|x|)|
\le
C\left(
\frac{|x|}{r_x}
+
\log\frac{|x|}{\rho}
\right)
\|\omega\|_{L^\infty},
\]
and similarly
\[
|E^z(x;\rho)|
\le
C\left(
1+\log\frac{|x|}{|z_x|}
+
\log\frac{|x|}{\rho}
\right)
\|\omega\|_{L^\infty}.
\]

Taking \(\rho=r_x\), the radial estimate follows from
\[
\log\frac{|x|}{r_x}\le C\frac{|x|}{r_x}.
\]
For the vertical estimate,  
\[
\log\frac{|x|}{|z_x|}+\log\frac{|x|}{r_x}
=
\log\frac{|x|^2}{|z_x|r_x}
=
\log\left(\frac{r_x}{|z_x|}+\frac{|z_x|}{r_x}\right)
\le
C\left(1+\left|\log\frac{r_x}{|z_x|}\right|\right).
\]
Therefore
\[
|E^z(x;r_x)|
\le
C\left(
1+\left|\log\frac{r_x}{|z_x|}\right|
\right)
\|\omega\|_{L^\infty}.
\]

\end{proof}

\section{Norm Inflation}
\label{sec:norm_inflation}
We are now ready to establish norm inflation.
\subsection{Initial Data and Ansatz}
\label{sec:ini_data_ansatz}
Given any Lorentz exponent $q>1$, we fix
a cone-slope parameter $L_q$ by \eqref{eq:cond_L5} and the logarithmic decay exponent $\alpha_q$ by \eqref{eq:norm_inflation_range_rem}. Then we fix a profile $\phi_q(r,z)$ by \eqref{cond_phi}, which will determine our data \eqref{KJdata2}.

As we fix $r_0=1$, the fixed cone-slope $L_q=z_0/r_0$ determines $z_0$ uniquely. Let us denote this by $z_{0,q}$.

These parameters $L_q$, $\alpha_q$, and $\phi_q$ are fixed once and for all throughout the paper.  We will indicate explicitly where the assumptions on them are used.

For an integer $m\geq2$ and a parameter $\varepsilon>0$, define initial data  $\omega_0^{(m)}$ by \eqref{KJdata2} with the fixed profile $\phi_q $.

As $0<\alpha_q<1$,
\begin{align}
\label{ini_lorentz14}
        \left\| 
        \frac{\omega_0^{(m)}(x)}{r}  
        \right \|^q_{L^{3,q}(\mathbb{R}^3)}
    \leq C(q)
        \varepsilon^q \,\left\|\frac{\phi_q}{r} \right\|_{L^{3,q}(\mathbb{R}^3)}^q    \sum_{k=1}^m \frac{1}{k^{\alpha_q\,  q}}
\leq C(q)
        \varepsilon^q \,\left\|\frac{\phi_q}{r} \right\|_{L^{3,q}(\mathbb{R}^3)}^q    \sum_{k=1}^\infty \frac{1}{k^{\alpha_q\,  q}}
    <\infty.
\end{align}
In addition, as $\omega_{0,k}$'s have  disjoint support,
    \begin{align*}
        \|\omega_0^{(m)}\|_{L^1(\mathbb{R}^3)}
        =\sum_{k=1}^m \|\omega_{0,k}\|_{L^1(\mathbb{R}^3)} =
        \|\phi_q\|_{L^1(\mathbb{R}^3)}
        \sum_{k=1}^m
        \frac{\varepsilon}{k^{\alpha_q}}
        d^{3k} 
        \leq
        \|\phi_q\|_{L^1(\mathbb{R}^3)}
        \, \varepsilon
        \sum_{k=1}^\infty d^{3k}.
    \end{align*}
And $\|\omega_0^{(m)}\|_{L^\infty(\mathbb{R}^3)}\leq \varepsilon \|\phi_q\|_{L^\infty(\mathbb{R}^3)}. $
Therefore, given any $\tilde \varepsilon>0$, we can choose a small $\varepsilon\in(0,\tilde \varepsilon)$ depending only on $\tilde \varepsilon, q$  so that $\|\omega_0^{(m)}\|_{L^\infty\cap L^{1}(\mathbb{R}^3) } +\|\omega_0^{(m)}/r\|_{L^{3,q}(\mathbb{R}^3)}<\tilde \varepsilon$.

Several quantities in the proof, such as $T_N$ and $T_B$, depend on $\varepsilon$. 
However, once $\varepsilon=\varepsilon(\tilde\varepsilon,q)$ is chosen, it is regarded as fixed, and this dependence will be suppressed unless it is used explicitly.
One can refer to \S~\ref{sec:out_paper_not} for notation conventions regarding all these quantities, such as $L_q,\alpha_q,\phi_q$ and $\varepsilon(\tilde\varepsilon,q)$.

As our initial data \eqref{KJdata2} is smooth and has compact support, by Danchin's global existence and uniqueness \cite{Danchin07} together with classical local existence theory and the Beale-Kato-Majda criterion, there exists a unique global-in-time smooth solution $\omega^{(m)}$.

Regarding the solution, $\omega^{(m)},$ odd symmetry of $\omega^{(m)}$ in $z$ and non-positivity of $\omega^{(m)}$ in $\{z>0\}$ are preserved under the evolution. Indeed, the reflected vorticity $\bar \omega^{(m)}(r,z,t):= -\omega^{(m)}(r,-z,t)$ solves the equation \eqref{vor_eq} with the same initial data. Hence from uniqueness, one can obtain $\bar \omega^{(m)}\equiv \omega^{(m)}$. Then one can find corresponding symmetry of the velocity field $u$ via the Biot--Savart law, which implies that the flow map of $u$ preserves the lower and upper half-spaces. Then the transport of $\omega^{(m)}/r$ implies the non-positivity.

The relative vorticity $\omega^{(m)}/r$ is transported by the flow map associated with the velocity field $u^{(m)}$ induced by $\omega^{(m)}$. Hence, the vortex ring that evolved from $\omega_{0,k}$ is well-defined for $t\in [0,\infty)$, and let us denote it by $\omega_k(r,z,t)$. It allows us to decompose our solution into multi vortex rings:
    \begin{align}
    \label{eq:multi_vortex_ring_dec}
        \omega^{(m)}(r,z,t)
    =
        \sum_{k=1}^{m} \omega_k(r,z,t), \qquad 0\leq t <\infty.
    \end{align}

We first define amplitude factors $\{x_k(t)\}_{k=1}^m$ by the ODE systems on $t\in [0,\infty)$:
\begin{align}
    \label{EulerODE}
    \left. 
    \begin{aligned}
        x_k'(t)&= x_k(t)  \, \partial_r u_-^r (0,0,t), \qquad k\geq 2,
        \\
         x_1'(t)&=0,
        \\
        x_k(0)&= \frac{\varepsilon}{k^{\alpha_q}}, \qquad k\geq 1.
    \end{aligned}
    \right\}
    \end{align}
Recall for a fixed scale index $k$, we denote, by $u_-(r,z,t)=\sum_{j=1}^{k-1} u_j(r,z,t)$, the velocity field induced by outer rings $\omega_-(r,z,t)=\sum_{j=1}^{k-1} \omega_j(r,z,t)$.

Then we define $\tilde x_k(t), R_k(t), H_k(t)$ for $k=1,\ldots, m$ by \eqref{tilde_xk_rk_hk}. 
We define time-dependent profiles $W_k(r,z,t)$ by
    \begin{align}
    \label{eq:def_wk94}
       W_k(r,z,t):=
       \frac{1}{x_k(t)} \omega_k (R_k(t) \, r, H_k(t) \, z, t), \quad k=1,\ldots,m
    \end{align}
for $(r,z)$ such that $(R_k(t)\, r, H_k(t) \, z)\in \supp \omega_k(\cdot,\cdot,t).$
This implies the ansatz that we will use:
    \begin{align*}
    \omega_k(r,z,t)= x_k(t) W_k\left(
        \frac{r}{R_k(t)}, \frac{z}{H_k(t)}, t
    \right), \quad k=1,\ldots,m
    \end{align*}
together with 
    \begin{align*}
        W_k(r,z,0)= \phi_q(r,z), \quad k=1,\ldots,m.
    \end{align*}

Our bootstrap argument will show $W_k$ remains close to $\phi_q$ uniformly in $k$, on the relevant time interval. 
The quantity $x_k(t)$ is the amplitude factor and, under the bootstrap bounds, is comparable to the $L^\infty$-size of $\omega_k$.
The point $(R_k(t)r_0, H_k (t)z_{0,q})$ encodes a representative location in the $rz$ plane of the $k$-th vortex ring, and $2\, \eta\,  r_0 \,R_k(t)$  represents the radial width of $\mathrm{supp} \, \omega_k$ whereas $2\, \eta\, z_{0,q}\, H_k(t)$ represents the width of the support in the $z$ direction.

The profiles $W_k$ satisfy the following equation: for all $k=1,\ldots, m$
    \begin{align}
    \label{eq_W_k2}
        \partial_t W_k 
        + V^r_k  \partial_r W_k
        + V^z_k \partial_z W_k = W_k \frac{V^r_k}{r},
    \end{align}
where
    \begin{align}
    \label{def_v_k14}
    \left\{
    \begin{aligned}
        V^r_k (r,z,t)
    &=
        \frac{u^{(m),r} 
        \left( R_k(t) \, r, \, H_k(t)\,  z, \,
        t \right)
        }{ R_k(t) }
        -r \, \frac{x_k'(t)}{x_k(t)},
        \\
        V^z_k (r,z,t)
    &=
        \frac{u^{(m),z} 
        \left( R_k(t) \, r,\, H_k(t) \, z, \, t \right)
        }{
        H_k(t)
        }
        +2z \, \frac{x_k'(t)}{x_k(t)}.
    \end{aligned}
    \right.
    \end{align}

    Note that for $(r,z)\in \mathrm{supp} \, W_k(\cdot, \cdot,t)$, the quantity $V^r_k(r,z,t)/r$ represents the difference between the real and approximate values of the vortex stretching rate $u^r/r$ at the point $(R_k(t) \,r, H_k (t)\, z)$ in the $k$-th vortex ring.
    
    Note that all these quantities $\omega_k,   W_k, x_k,R_k,H_k$ in fact also depend on $m$. However, for simplicity, we do not indicate this dependence in the notation. See \S~\ref{sec:out_paper_not} for more notation conventions.

As a reminder for a notation convention, meridional $(r,z)-$functions, such as $W_k,\phi_q$ and $\omega_k$, are also identified with their axisymmetric lifts to $\mathbb{R}^3$. For example, all $L^\infty$ norms are taken with respect to this identification. We will write $\|\phi_q\|_{L^\infty(\mathbb{R}^3)}$ instead of $\|\phi_q\|_{L^\infty(\mathbb{R}_+\times \mathbb{R})}$.

For $0\leq t$ and $k=1,\ldots,m$, let $X_k(r,z,t)=(X^r_k(r,z,t), X^z_k(r,z,t))$ be the flow map for $W_k$'s given by
    \begin{align*}
        X_k(r,z,  0)&=(r,z) \, \in \mathrm{supp}\,\phi_q,
        \\
        \frac{\partial }{\partial t} 
        X_k (r,z,t)&=
        V^r_k \left( 
                X_k(r,z, t),\, t
            \right) \, e_r
        +V^z_k \left(
                X_k(r,z, t),\, t
            \right) \, e_z.
    \end{align*}

\subsection{Roadmap of a Bootstrap Argument}
\label{roadmap_boot_Euler}

For an arbitrary target amplitude $A>  \varepsilon$, define
\begin{equation}\label{def:Euler_ni_time}
\mu:= \frac{1}{20}, \quad \bar A:= \frac{A}{1-\mu}, 
\quad T_N(A,m):=
\sup \left\{t\geq 0: \max_{1\leq k \leq m} x_k(t) \leq \bar A\right\}
.
\end{equation}
We will call $T_N(A,m)$ the  \emph{norm-inflation time} at the ODE level.

In addition, it holds $T_N(A,m)>0$ because
$\bar A > \varepsilon= \max_{1\leq k\leq m} x_k(0).$ If $T_N(A,m)<\infty$, then by continuity of the finite family $(x_k)_{k=1}^m$, 
    \begin{align*}
        \max_{1\leq k \leq m}
        x_k(T_N(A,m))=\bar A.
    \end{align*}

At this moment, $T_N(A,m)$ could be infinite even though we will show later it is in fact finite.

If $A\leq\varepsilon$, then norm inflation is trivial.

\medskip
\noindent\textbf{Bootstrap bounds.}
We impose the following bootstrap bounds: for all $t\in [0,T],$ 
\begin{align}
    \label{boot:X}
    F_m(t):=
    \sup_{\substack{(r,z)\in \supp \phi_q\\ k=1,\ldots,m}}
    \max\left\{
    \left| \frac{X_k^r(r,z,t)}{r}-1\right|,
    \left|\frac{X_k^z(r,z,t)}{z}-1\right|
    \right\} \leq \mu.
\end{align}
with the bootstrap parameter $\mu=1/20.$ It is satisfied at least at $t=0$.

For each $k,$ the function $(r,z,t)\mapsto (X_k^r/r,X_k^z/z)$ is continuous on the compact set $\supp \phi_q \times [0,T]$, which gives uniform continuity. Hence thanks to finiteness of $k=1,\ldots,m$, the bootstrap deviation functional $F_m(t)$ is continuous.

We define the maximal bootstrap time $T_B(A,m)$ by
\begin{align}
\label{def_TB_Euler}
T_B(A, m)
:=\sup\left\{
\,T\in[0,T_N(A,m)]:
\ \eqref{boot:X} \text{ holds for all }t\in[0,T]
\right\} .
\end{align}
Since $F_m(0)=0$ and $F_m(t)$ is continuous for  $t\geq 0$,  it holds that $T_B(A,m)>0$.
At this stage we do not know whether $T_B(A,m)=T_N(A,m)$ or not; if not, then we cannot use the bootstrap bounds at the norm--inflation time $T_N(A,m)$.

In the roadmap below, comparison symbols such as $\ll$, $\gtrsim$, and $\lesssim$ are used only heuristically to indicate the mechanism. 
All rigorous constants and displayed dependences will be stated in the subsequent lemmas and propositions.

\medskip
\noindent\textbf{Logical structure.}
The proof of the norm inflation statement (Theorem~\ref{main_thm1}) reduces to:

\begin{enumerate}
\item \textbf{ODE control of $T_B(A, m)$.}
Using the bootstrap bound \eqref{boot:X} on $[0,T_B(A,m)]$, we show in \S~\ref{sec:est_pro} that the stretching rate
$\partial_r u^r_-(0,0,t)$ felt by the $k$th ring can be expressed in terms of the profiles $W_j$
and aspect ratios $\Gamma_j(t)$, yielding a weak ODE cascade  for
$(x_k(t))_{k=1}^m$ up to a small multiplicative error.
Section~\ref{sec:ODEs_Euler} then analyzes this ODE system; by using the conditions \eqref{cond_phi},\eqref{eq:cond_L5},\eqref{eq:norm_inflation_range_rem} of $L_q, \phi_q,\alpha_q$, we establish
quantitative upper bounds on $T_B(A,m)$ and in particular $T_B(A,m)\lesssim  m^{-\beta_q}\to0$ as $m\to\infty$.

\item \textbf{Bootstrap improvement.}
Using the flow--map estimates derived under \eqref{boot:X}, we prove in
\S~\ref{subsec:Euler_bootstrap_improvement} that on a short time interval the bootstrap inequalities improve with a strict margin,
provided the short--time smallness mechanism
\[
\bar A\,\log m \, T_B(A,m)\ll1
\]
for large $m$. 
By a standard continuity argument this forces $T_B(A,m)=T_N(A,m)$, so that
\eqref{boot:X} holds on $[0,T_N(A,m)]$, and $T_N(A,m)\to 0$ as $m\to\infty$.

\item \textbf{Completion of Norm Inflation.}
Once \eqref{boot:X} holds on $[0,T_N(A,m)]$, the Cauchy formula for $W_k$ implies a uniform
lower bound $\|W_k(\cdot,\cdot,t)\|_{L^\infty}\gtrsim \|\phi_q\|_{L^\infty}$ on this interval.
Evaluating at $t=T_N(A,m)$ then converts the amplitude bound defining $T_N(A,m)$ into
$\|\omega^{(m)}(\cdot,T_N(A,m))\|_{L^\infty}\gtrsim A$.
Finally, since $T_N(A,m)$ can be made arbitrarily small by choosing $m$ large (via the ODE step
and the choice of $\phi_{q}$), this yields norm inflation on any prescribed time scale.
\end{enumerate}

To summarize it, let us provide a dependency diagram:
\begin{equation*}
\begin{aligned}
\text{initial data}
&\;\Longrightarrow\;
\text{weak ODEs on } [0,T_B(A,m)] \\
&\;\Longrightarrow\;
\text{quantitative control of } T_B(A,m) \text{ using the bootstrap bounds}\\
&\;\Longrightarrow\;
A \log m\, T_B(A,m) \ll 1 \text{ for large }m\\
&\;\Longrightarrow\;
\text{bootstrap improvement} \\
&\;\Longrightarrow\;
T_B = T_N
\;\Longrightarrow\;
L^\infty\text{-norm inflation}.
\end{aligned}
\end{equation*}

It is worth noting that we only use a bootstrap bound on $X$ without introducing a bootstrap bound on $DY$ for a flow map $Y$ associated with the physical flow $u$ as opposed to \cite{CordobaMartinezZheng25}. For norm inflation, in fact, such a bound on $DY$ is available and we can even use it for self-interaction estimate later on, which will give us a stronger estimate without a logarithmic penalty as discussed in \S~\ref{subsec:instant_blow-up}. However, it does not carry over to prove instantaneous blow-up. Therefore, we do not use a bootstrap bound on $DY$ even for norm inflation.

\subsection{Cauchy Formula and Freezing the Profiles}
\label{sec:est_pro}
In this first section, we record the Cauchy formula for the rescaled relative vorticity $W_k/r$ and ``freeze" the profiles $W_k(r,z,t)$ by using the bootstrap bounds \eqref{boot:X}. Keep in mind that we freeze the profiles $W_k(r,z,t)$ not just in terms of time $t$ but also in terms of scale index $k$.

\begin{lemma}[Cauchy formula for $W_k/r$]\label{lem:Euler_cauchy_geometry1}
Fix $k\in\{1,\dots,m\}$. For every $t\in[0,\infty)$, the flow map $X_k(\cdot,\cdot,t)$ associated with $V_k$ satisfies
\begin{equation}\label{cauchy_for_W}
\frac{W_k\bigl(X_k(r,z,t),t\bigr)}{X_k^r(r,z,t)}=\frac{\phi_q(r,z)}{r},\qquad (r,z)\in\supp\phi_q.
\end{equation}
In particular, using the bootstrap bound \eqref{boot:X} for $t\in [0,T_B(A,m)]$,
\[
\|W_k(\cdot,\cdot,t)\|_{L^\infty(\R^3)}\le (1+\mu)\|\phi_q\|_{L^\infty(\R^3)},
\qquad
\|\omega_k(\cdot,t)\|_{L^\infty(\R^3)}\le (1+\mu)x_k(t)\|\phi_q\|_{L^\infty(\R^3)}.
\]
\end{lemma}

\begin{proof}
Starting from \eqref{eq_W_k2}, divide by $r$ and use $\partial_r(1/r)=-1/r^2$ to obtain
\[
\partial_t\left(\frac{W_k}{r}\right)+V_{k}^r\,\partial_r\left(\frac{W_k}{r}\right)+V_{k}^z\,\partial_z\left(\frac{W_k}{r}\right)=0.
\]
Therefore $W_k/r$ is transported by the characteristics of $V_k$, i.e. by $X_k$. Since $X_k(r,z,0)=(r,z)$ and $W_k(\cdot,\cdot,0)=\phi_q$, the transport identity gives \eqref{cauchy_for_W}. The $L^\infty$ bounds follow immediately from \eqref{cauchy_for_W} and the bootstrap control of $X_k^r/r$.
\end{proof}

The discussion below is heuristic until the next lemma; comparison symbols are used only to describe the expected mechanism.

Let us now freeze the profiles $W_k(r,z,t)$ using the bootstrap bounds on $[0,T_B(A,m)]$.  
We need to estimate
    \begin{align*}
        -\iint 
        \frac{r^2 z}{(r^2+z^2)^{\frac{5}{2}}}
        W_k\left(
            \frac{r}{R_k(t)}, \frac{z}{H_k(t)}, t
        \right) \, dr \, dz \geq 0.
    \end{align*}
As explained in \S~\ref{ODE_intro}, heuristically, as $W_k(\cdot,z,t)$ is concentrated near $r=r_0$, one can estimate 
    \begin{align*}
        -\iint 
        \frac{r^2 z}{(r^2+z^2)^{\frac{5}{2}}}
        W_k\left(
            \frac{r}{R_k}, \frac{z}{H_k}, t
        \right) \, dr \, dz
    &\sim
        \frac{H_k^2}{R_k^2}
        \int_{\{z\sim z_0\}}
        \frac{z}{ \left (1+ \frac{H_k^2}{R_k^2}  \, z^2 \right)^{\frac{5}{2}} } \, dz.
    \end{align*}
Here, $H_k(t)^2/R_k(t)^2$ decreases in time from $H_k^2(0)/R_k^2(0)=1$ to a small number, and it also decreases in $k$. When $H_k^2/R_k^2\ll1$, we can further simplify the estimate:
    \begin{align*}
         -\iint 
        \frac{r^2 z}{(r^2+z^2)^{\frac{5}{2}}}
        W_k\left(
            \frac{r}{R_k}, \frac{z}{H_k}, t
        \right) \, dr \, dz \sim \frac{H_k(t)^2}{R_k(t)^2}.
    \end{align*}
To be more clear, this estimate only makes sense for large $k$'s, and a fixed time $t>0$ that is not so small. 

Now, let us make this estimate rigorous.
    \begin{lemma}
    \label{lem:pro_est}
        For any $k=1,\ldots,m$ and $t\in[0,T_B(A,m)]$, it holds
            \begin{align}
            \label{heu05}
            c_\mu
                \, 
                \Gamma_k(t)^2
                \,
                \Lambda_{\mathrm{froz}} (\Gamma_k(t))
            \leq
                -\iint \frac{r^2 z}{(r^2+z^2)^{\frac{5}{2}}} W_k
                \left(
                    \frac{r}{R_k (t)},
                    \frac{z}{H_k(t)}, t
                \right)
                 dr dz
            \leq
                C_\mu 
                \, 
                \Gamma_k(t)^2
                \,
                \Lambda_{\mathrm{froz}} (\Gamma_k(t)).
            \end{align}
        where
            \begin{align*}
                \Gamma_k(t) :=\frac{H_k(t)}{R_k(t)},
            \qquad 
            \Lambda_{\mathrm{froz}}(\Gamma):=\int_{-\infty}^\infty \int_0^\infty\frac{r^2 z}{\big(r^2+\Gamma^2 z^2\big)^{5/2}}\, \big(-\phi_q(r,z)\big)\,dr\,dz,
\quad \Gamma\in (0,1]
            \end{align*}
        and
            \begin{align*}
                c_\mu:= \frac{(1-\mu)^3}{(1+\mu)^5},
            \qquad
                C_\mu:=\frac{(1+\mu)^3}{(1-\mu)^5}.
            \end{align*}
    \end{lemma}
    \begin{proof}
    Fix $k=1,\ldots,m$ and $t\in [0,T_B(A,m)]$. We will omit the argument $t$ in $R_k(t), H_k(t)$ and $ \Gamma_k(t)$ for simplicity.
    
    First, through change of variables, we rescale the variables:
    \begin{align*}
        \iint \frac{r^2\, z}{(r^2+z^2)^{\frac{5}{2}}} 
        W_k \left( \frac{r}{R_k} , \frac{z}{H_k}, t \right) 
        \, dr \, dz 
    &=
        \Gamma_k^2\iint 
        \frac{r^2\, z}{\left( r^2+\Gamma_k^2 \, z^2 \right)^{\frac{5}{2}}} W_k (r,z,t)
        \, dr \, dz.
    \end{align*}
A direct computation shows that $\mathrm{div}_{\mathbb{R}^3}V_k =0$ viewing $V_k$ as a vector fields in $\mathbb{R}^3$ written in cylindrical coordinates. Then the corresponding 3D Jacobian determinant of the flow map $X_k$ is $1$.

Using it and re-writing the variables $(r,z)$ as $(r_x,z_x)$, we  apply change of variables via $(r_x,z_x) = X_k(r_y, z_y,t)$ to obtain
    \begin{align*}
        \int_{-\infty }^\infty \int_0^\infty  \frac{r^2_x \, z_x}{(r_x^2+\Gamma_k^2z_x^2)^{\frac{5}{2}}} W_k(r_x,&z_x,t) \, dr_x dz_x
    =
        \frac{1}{2\pi}\int_{\mathbb{R}^3} \frac{r_x \, z_x}{\big(r_x^2+\Gamma_k^2z_x^2\big)^{\frac{5}{2}}} W_k(r_x,z_x,t) \, dx
    \\
    &=
        \frac{1}{2\pi}\int_{\mathbb{R}^3} \frac{X^r_k \, X^z_k}{\big((X_k^r)^2+\Gamma_k^2(X_k^z)^2\big)^{\frac{5}{2}}} W_k\big(X_k(r_y,z_y,t),t\big) \, dy
    \\
    &=
         \int_{-\infty}^\infty \int_0^{\infty} \frac{X^r_k \, X^z_k}{\big((X_k^r)^2+\Gamma_k^2(X_k^z)^2\big)^{\frac{5}{2}}} W_k\big(X_k(r_y,z_y,t),t\big) \, r_y \, dr_y \, dz_y
    \end{align*}
where $X_k^r:=X_k^r (r_y, z_y,t), X_k^z:= X_k^z(r_y,z_y,t)$. 
Now using the bootstrap bound \eqref{boot:X}, one can obtain:
    \begin{align*}
    \nonumber 
        -\int_{-\infty}^\infty \int_{0}^\infty 
        \frac{r_x^2 \, z_x }
        {(r_x^2 + \Gamma_k^2 \, z_x^2)^{\frac{5}{2}}} \, 
        & W_k (r_x,z_x,t)
        \, dr_x \, dz_x
    \\
    &=
        -\int_{-\infty}^\infty \int_{0}^\infty
        \frac{X^r_k \, X^z_k}{\big((X^r_k)^2+ \Gamma_k^2 \, (X^z_k)^2\big)^{\frac{5}{2}}} W_k \big(X_k(r_y,z_y,t),t\big) \, r_ydr_y\, dz_y
    \\\nonumber
    &=
        -\int_{-\infty}^\infty \int_{0}^\infty
        \frac{(X^r_k)^2 \, X_k^z}{\big((X^r_k)^2+ \Gamma_k^2 \, (X^z_k)^2 \big)^{\frac{5}{2}}} \,
         \phi_q (r_y,z_y) \, dr_y\,  dz_y
    \\
    &\leq
        -\frac{(1+\mu)^3}{(1-\mu)^5}
        \int_{-\infty}^\infty \int_{0}^\infty
        \frac{r_y^2\, z_y}{\big(r_y^2 + \Gamma_k^2 \, z_y^2\big)^{\frac{5}{2}}} \, \phi_q (r_y,z_y) \, dr _y\, dz_y
    \end{align*}
where we used the Cauchy formula for $W_k/r$ \eqref{cauchy_for_W}. From here, one can obtain the upper bound in \eqref{heu05} at once.
Similarly, the lower bound in \eqref{heu05} follows by replacing the upper bootstrap inequalities by the lower ones in the same argument.
\end{proof}

\subsection{Study of ODEs}
\label{sec:ODEs_Euler}

As explained in \S~\ref{roadmap_boot_Euler}, throughout this section we study the weak ODE system \eqref{original_ODEs2} below
on the bootstrap interval $[0,T_B(A,m)]$, using the bootstrap bounds \eqref{boot:X}. 
The constants introduced in this subsection may depend on the fixed background parameters from \S~\ref{sec:out_paper_not}, but they are independent of the running parameters $k,m,A,t,T_N,T_B$ unless such dependence is displayed explicitly. 
In particular, the constants $c_\mu$, $C_\mu$, and $\Theta_\mu$ are fixed once introduced. See \S~\ref{sec:out_paper_not} for more notation conventions.

We use superscript/subscript, such as orig, loc, froz, to distinguish various types of simplifications. In addition, we will denote $x_k^{\mathrm{orig}}(t):=x_k(t)$ throughout this Section~\ref{sec:ODEs_Euler}, but after this section, we will come back to the original notation $x_k(t)$.

\medskip
\noindent\textbf{Roadmap of this subsection.}
As already foreshadowed in the introduction, the purpose of this subsection is to isolate the ODE cascade that drives norm inflation. 
\begin{itemize}
\item  \S~\ref{subsec:freez_prof}:
On the bootstrap interval, we first use the flow-map control to freeze the transported profiles \(W_j\) against the fixed profile \(\phi_q\), thereby reducing the weak ODEs to outer-ring contributions depending essentially on the amplitudes \(x_j\) and aspect ratios \(\Gamma_j\).
\item  \S~\ref{subsec:boot_comp}:
We then rewrite the system in cascade form, which yields an exact identity for the evolution of the \(\Gamma_j\)'s, and then we  compare the frozen and unfrozen quantities. 
\item  \S~\ref{subsec:loc_prof}:
After introducing the cone variables \(\zeta_j=L_q\Gamma_j\), we exploit the monotonicity of the localized Biot--Savart kernel on a fixed productive range: while the front moves toward smaller scales, the cumulative outer stretching must build up at a quantitative rate. 
\item \S~\ref{subsec:front_mig}, \S~\ref{subsec:ODE_norm_inflation}:
This front-migration mechanism gives a front-induced \(t^{-1}\) lower bound for the total stretching, and from there we derive an upper bound for the ODE norm-inflation time. These ODE estimates are the main output of the subsection and will be used in the next subsection to close the bootstrap.
\end{itemize}

\subsubsection{Weak ODEs and Freezing of Profiles.}
\label{subsec:freez_prof}
Fix $k\in\{2,\dots,m\}$.  As derived in \S~\ref{der_weak_ODE}, one needs to study the weak ODEs
\begin{equation}
\label{original_ODEs2}
\frac{d}{dt} \log x_k^{\mathrm{orig}}(t)
=
\sum_{j=1}^{k-1} x_j^{\mathrm{orig}}(t)\,\Gamma_j^{\mathrm{orig}}(t)^2
\Lambda_j(t,\Gamma_j^{\mathrm{orig}}(t)),
\qquad 0\le t<T_B(A,m).
\end{equation}
Introduce the true outer-ring contributions
\begin{equation*}
b_j^{\mathrm{orig}}(t):=x_j^{\mathrm{orig}}(t)\Gamma_j^{\mathrm{orig}}(t)^2\Lambda_j\bigl(t,\Gamma_j^{\mathrm{orig}}(t)\bigr),
\quad
S_k^{\mathrm{orig}}(t):=\sum_{j=1}^k b_j^{\mathrm{orig}}(t),
\quad
B_j^{\mathrm{orig}}(t):=\int_0^t b_j^{\mathrm{orig}}(s)\,ds.
\end{equation*}
Then the cascade form is
\begin{equation}
\label{eq:orig_ODE_cascade_form}
\frac{d}{dt}\log x_k^{\mathrm{orig}}(t)=S_{k-1}^{\mathrm{orig}}(t)\qquad (2\le k\le m),
\end{equation}

On $[0,T_B(A,m)]$, Lemma~\ref{lem:pro_est} allows one to \emph{freeze} each transported profile
$W_j(\cdot,\cdot,t)$ against the initial profile $\phi_q$.
One has the uniform bounds
\begin{equation}
\label{mu_ODE}
c_\mu
\sum_{j=1}^{k-1} x_j^{\mathrm{orig}}(t)\Gamma_j^{\mathrm{orig}}(t)^2 \Lambda_{\mathrm{froz}}(\Gamma_j^{\mathrm{orig}}(t))
\ \leq\
\frac{d}{dt}\log x_k^{\mathrm{orig}} (t)
\ \leq\
C_\mu \sum_{j=1}^{k-1} x_j^{\mathrm{orig}}(t)\Gamma_j^{\mathrm{orig}}(t)^2 \Lambda_{\mathrm{froz}}(\Gamma_j^{\mathrm{orig}}(t)).
\end{equation}

As in \S~\ref{intro:front_mig}, it is convenient to rewrite the outer-ring contribution in terms of
$\Gamma_j^{\mathrm{orig}}(t)$ alone.  One has the exact identities
\begin{equation*}
\Gamma_j^{\mathrm{orig}}(t)
=\left(\frac{x_j^{\mathrm{orig}}(0)}{x_j^{\mathrm{orig}}(t)}\right)^{3},
\qquad
x_j^{\mathrm{orig}}(t)\Gamma_j^{\mathrm{orig}}(t)^2
= x_j^{\mathrm{orig}}(0)\,\Gamma_j^{\mathrm{orig}}(t)^{5/3}.
\end{equation*}
Define the frozen outer-ring contributions and partial sums
\begin{align*}
\begin{aligned}
b_{\mathrm{froz},j}^{\mathrm{orig}}(t)&:=x_j^{\mathrm{orig}}(t)\Gamma_j^{\mathrm{orig}}(t)^2\Lambda_{\mathrm{froz}}(\Gamma_j^{\mathrm{orig}}(t))
= x_j^{\mathrm{orig}}(0)\,\Gamma_j^{\mathrm{orig}}(t)^{5/3}\Lambda_{\mathrm{froz}}(\Gamma_j^{\mathrm{orig}}(t)),
\\
S_{\mathrm{froz},k}^{\mathrm{orig}}(t)&:=\sum_{j=1}^{k} b_{\mathrm{froz},j}^{\mathrm{orig}}(t).
\end{aligned}
\end{align*}
Note that the Biot--Savart coefficient $\Lambda_{\mathrm{froz}}$ is frozen. These definitions of $b_{\mathrm{froz},j}^{\mathrm{orig}}, S_{\mathrm{froz},k}^{\mathrm{orig}}$ are slightly different from \eqref{eq:orig_def_bj_froz} because we are considering both the frozen and original models  to compare them in this rigorous proof whereas we only considered one model in heuristic arguments.
Then \eqref{mu_ODE} becomes
\begin{equation*}
c_\mu\,S_{\mathrm{froz},k-1}^{\mathrm{orig}}(t)
\ \leq\
\frac{d}{dt}\log x_k^{\mathrm{orig}}(t)=S_{k-1}^{\mathrm{orig}}(t)
\ \leq\
C_\mu\,S_{\mathrm{froz},k-1}^{\mathrm{orig}}(t)
\end{equation*}
for $ 2\le k\le m$ and $ 0\le t<T_B(A,m).$

\subsubsection{Cascade Form for the Original Weak ODEs and Bootstrap Comparison.}
\label{subsec:boot_comp}
The cascade form \eqref{eq:orig_ODE_cascade_form} implies the exact cascade identity
\begin{equation}
\label{eq:Gamma_cascade_orig}
\Gamma_{j+1}^{\mathrm{orig}}(t)=\Gamma_j^{\mathrm{orig}} (t)\,e^{-3B_j^{\mathrm{orig}}(t)}.
\end{equation}

On $[0,T_B(A,m)]$, Lemma~\ref{lem:pro_est} yields the pointwise comparison
\begin{equation}
\label{eq:bj_compare_boot}
c_\mu\,b_{\mathrm{froz},j}^{\mathrm{orig}}(t)
\le 
b_j^{\mathrm{orig}}(t)
\le
C_\mu\,b_{\mathrm{froz},j}^{\mathrm{orig}}(t),
\end{equation}
and hence also
\begin{equation}
\label{eq:Sk_compare_boot}
c_\mu\,S_{\mathrm{froz},k}^{\mathrm{orig}}(t)
\le
S_k^{\mathrm{orig}}(t)
\le
C_\mu\,S_{\mathrm{froz},k}^{\mathrm{orig}}(t)
\qquad (1\le k\le m-1).
\end{equation}
We record the fixed ratio
\begin{equation}
\label{eq:def_Theta_mu}
\Theta_\mu:=\frac{C_\mu}{c_\mu}=\left( \frac{1+\mu}{1-\mu} \right)^8.
\end{equation}

Then differentiating $b_{\mathrm{froz},j}^{\mathrm{orig}}$ and using the cascade form \eqref{eq:orig_ODE_cascade_form} gives a differential equation for $b_{\mathrm{froz},j}^{\mathrm{orig}}$.

\begin{lemma}[A differential equation for $b_{\mathrm{froz},j}^{\mathrm{orig}}$]
 For each $j\ge 2$, it holds
\begin{equation}
\label{eq:bjprime_kappa}
\frac{d}{dt}b_{\mathrm{froz},j}^{\mathrm{orig}}(t)
=-\kappa_{\mathrm{froz}}(\Gamma_j^{\mathrm{orig}}(t))\,b_{\mathrm{froz},j}^{\mathrm{orig}}(t)\,
S_{j-1}^{\mathrm{orig}}(t),
\qquad
\kappa_{\mathrm{froz}}(\Gamma):=5+3\Gamma\frac{\Lambda'_{\mathrm{froz}}(\Gamma)}{\Lambda_{\mathrm{froz}}(\Gamma)}.
\end{equation}
\end{lemma}

This equation isolates the sign issue in the ODE analysis. Once $\kappa_{\mathrm{froz}}(\Gamma)$ is shown to be non-positive on the relevant cone range, the frozen quantities $b_{\mathrm{froz},j}^{\mathrm{orig}}$ become monotone, and that monotonicity is exactly what later drives the front-migration argument.

\begin{proof}
Differentiate $b_{\mathrm{froz},j}^{\mathrm{orig}}(t)=\Lambda_{\mathrm{froz}}(\Gamma_j^{\mathrm{orig}}(t))x_j^{\mathrm{orig}}(t)\Gamma_j^{\mathrm{orig}}(t)^2$ and use
\begin{align*}
\frac{d}{dt}\log x_j^{\mathrm{orig}}(t)=S_{j-1}^{\mathrm{orig}}(t)
\end{align*}
from \eqref{eq:orig_ODE_cascade_form}.
Together with $\Gamma_j^{\mathrm{orig}}(t)=(x_j^{\mathrm{orig}}(0)/x_j^{\mathrm{orig}}(t))^3$, we have
\begin{align*}
\frac{d}{dt}\Gamma_j^{\mathrm{orig}}(t)=-3\Gamma_j^{\mathrm{orig}}(t)\,S_{j-1}^{\mathrm{orig}}(t),
\end{align*}
and a direct computation gives
\eqref{eq:bjprime_kappa}.
\end{proof}

\subsubsection{Localization of $\phi_q$ and a Productive Regime.}
\label{subsec:loc_prof}
Define the Biot--Savart profile
\begin{equation*}
\Psi(\zeta):=\zeta^{5/3}(1+\zeta^2)^{-5/2}.
\end{equation*}
Direct computation yields
\begin{align*}
\Psi'(\zeta)\leq 0
\quad \text{for all} \quad \zeta\geq \zeta_*=1/\sqrt2.
\end{align*}
Recall that $\phi_q$ is odd in $z$ and nonpositive on $\{z>0\}$ (cf.\ \eqref{cond_phi}), and we
impose the fixed geometric localization
\begin{equation}
\label{phi_cond2}
\supp \phi_q\cap\{z>0\}\ \subset\
\big[(1-\eta) r_0, (1+\eta)r_0\big]\times \big[(1-\eta) z_{0,q},(1+\eta ) z_{0,q}\big],
\end{equation}
Define the fixed cone slopes
\begin{equation}
\label{eq:def_Lpm}
L_{q,-}:=\frac{1-\eta}{1+\eta}\,L_q,
\qquad
L_{q,+}:=\frac{1+\eta}{1-\eta}\,L_q.
\end{equation}

Localization reduces the profile dependence of $\Lambda_{\mathrm{froz}}(\Gamma) \Gamma^{5/3}$ to a one-parameter shape function of $\zeta=L_q \Gamma$, which is the Biot--Savart profile $\Psi$. 
More precisely, we can estimate $\kappa_{\mathrm{froz}}(\Gamma)$ in terms of $\Psi'$, which yields a monotonicity range for $t\mapsto b_{\mathrm{froz},j}^{\mathrm{orig}}(t)$. 
\begin{lemma}[Comparison to $\Psi$]
\label{lem:comp_to_Psi}
For every $\Gamma\in(0,1]$,
    \begin{align}
    \label{eq:comp_kapp_psi}
        \kappa_{\mathrm{froz}} (\Gamma) 
        \leq 
        3 \, \Gamma L_{q,-}
        \frac{\Psi'(\Gamma L_{q,-})}{\Psi (\Gamma L_{q,-})},
    \end{align}
hence, for $j\geq 2$,
    \begin{align}
    \label{eq:mon_b_froz}
        \frac{d}{dt}b_{\mathrm{froz},j}^{\mathrm{orig}} (t) \geq 0 \qquad \text{whenever} \quad
        \Gamma_j^{\mathrm{orig}}(t) L_q \geq \frac{1+\eta}{1-\eta} \zeta_*.
    \end{align}
\end{lemma}
Therefore, as long as the cone variable $\Gamma_j^{\mathrm{orig}}L_q$ is in the range \eqref{eq:mon_b_froz}, the frozen individual stretching rate $b_{\mathrm{froz},j}^{\mathrm{orig}}(t)$ is non-decreasing despite the geometric slow-down mechanism. We call the range \eqref{eq:mon_b_froz} a \emph{productive} regime.

\begin{proof}
Differentiate under the integral sign:
\[
\Lambda'_{\mathrm{froz}}(\Gamma)=-5\Gamma\int_{-\infty}^{\infty} \int _{0}^{\infty}
\frac{r^2 z^3}{(r^2+\Gamma^2 z^2)^{7/2}}\big(-\phi_q(r,z) \big)\,dr\,dz\le 0.
\]
Moreover,
\[
-\Gamma\frac{\Lambda'_{\mathrm{froz}}(\Gamma)}{\Lambda_{\mathrm{froz}}(\Gamma)}
=
5\,
\frac{\iint K_\Gamma(r,z)\,\frac{\Gamma^2 z^2}{r^2+\Gamma^2 z^2}\,\big(-\phi_q(r,z) \big)\,dr\,dz}
{\iint K_\Gamma(r,z)\,\big(-\phi_q(r,z) \big)\,dr\,dz},
\]
so $-\Gamma\Lambda'_{\mathrm{froz}}/\Lambda_{\mathrm{froz}}$ is a weighted average of $\frac{\Gamma^2 z^2}{r^2+\Gamma^2 z^2}$.
Since $(r,z)\in\supp\phi_q$ implies $z/r\in[L_{q,-},L_{q,+}]$ and $a\mapsto \frac{a^2}{1+a^2}$ is increasing
for $a\ge 0$, we obtain
\[
\frac{(\Gamma L_{q,-})^2}{1+(\Gamma L_{q,-})^2}\le
\frac{\Gamma^2 z^2}{r^2+\Gamma^2 z^2},
\]
and multiplying by $15$ yields 
\begin{equation}
\label{eq:log_derivative_bounds}
15\,\frac{(\Gamma L_{q,-})^2}{1+(\Gamma L_{q,-})^2}
\ \le\
-3\Gamma\frac{\Lambda'_{\mathrm{froz}}(\Gamma)}{\Lambda_{\mathrm{froz}}(\Gamma)}.
\end{equation}
Direct computation yields
\begin{equation}
\label{eq:log_der_Psi}
    3\zeta \frac{\Psi'(\zeta)}{\Psi(\zeta)}
    =5- \frac{15\zeta^2}{1+\zeta^2}.
\end{equation}
Using \eqref{eq:log_derivative_bounds} and \eqref{eq:log_der_Psi}, one can obtain \eqref{eq:comp_kapp_psi} at once.
And \eqref{eq:mon_b_froz} follows from \eqref{eq:bjprime_kappa}.
\end{proof}

\subsubsection{Front-Migration and a Front-Induced $t^{-1}$ Lower Bound.}
\label{subsec:front_mig}
Based on \eqref{eq:mon_b_froz}, it is reasonable to introduce the cone variables
\[
\zeta_j^{\mathrm{orig}}(t):=L_q\,\Gamma_j^{\mathrm{orig}}(t)\in(0,L_q],
\]
and define a fixed threshold $\zeta_\eta$ by 
\begin{align}
\label{eq:def_xi_eta}
 \zeta_\eta:= \frac{1+\eta}{1- \eta} \zeta_*.
\end{align}
From the exact cascade identity \eqref{eq:Gamma_cascade_orig}, $\Gamma_{j+1}^{\mathrm{orig}}(t)=\Gamma_j^{\mathrm{orig}}(t) e^{-3B_j^{\mathrm{orig}}(t)}$, it follows that  
\begin{align}
\label{eq:ordering_zeta}
\zeta_{j+1}^{\mathrm{orig}}(t)\leq \zeta_j^{\mathrm{orig}}(t).
\end{align}
For each $t>0$ define the front index at the fixed threshold $\zeta_\eta$ by
\begin{equation}
\label{eq:def_front_index_frozen}
J_{\zeta_\eta }^{\mathrm{orig}}(t):=\max\Big\{1\le j\le m-1:\ \zeta_j^{\mathrm{orig}}(t)\ge \zeta_\eta\Big\}.
\end{equation}
Then $\zeta_k^{\mathrm{orig}}(t) \geq \zeta_\eta $ for every $k\leq J_{\zeta_\eta}^{\mathrm{orig}}(t)$.

When the front has crossed the $m$-th scale, i.e. when $\zeta_m^{\mathrm{orig}}(t) <\zeta_\eta$, one can estimate  the cumulative outer stretching $S_{m-1}^{\mathrm{orig}}(t)$ using the monotonicity of $b_{\mathrm{froz},j}^{\mathrm{orig}}$ in the productive regime.

\begin{lemma}[Front-migration and a front-induced $t^{-1}$ lower bound]
\label{lem:front_tinv_varcoef}
For  $t\in(0,T_B(A,m))$, if 
    \begin{align*}
    \zeta_m^{\mathrm{orig}}(t)<\zeta_\eta ,
    \end{align*}
 then
\begin{equation}
\label{eq:Sm1_tinv_varcoef}
S_{m-1}^{\mathrm{orig}}(t)\ge \frac{1}{3\Theta_\mu \,t}\log\frac{L_q}{\zeta_\eta},
\qquad \Theta_\mu\ \text{as in \eqref{eq:def_Theta_mu}}.
\end{equation}
\end{lemma}

\begin{proof}
Let $t\in(0,T_B(A,m))$ and set $J:=J_{\zeta_\eta}^{\mathrm{orig}}(t)$ as in \eqref{eq:def_front_index_frozen}.  If $\zeta_m^{\mathrm{orig}}(t)<\zeta_\eta$,
then $\zeta_{J+1}^{\mathrm{orig}}(t)<\zeta_\eta$ by definition.  Using \eqref{eq:Gamma_cascade_orig} and $\zeta_1^{\mathrm{orig}}(t)\equiv L_q$,
\[
\zeta_{J+1}^{\mathrm{orig}}(t)
=
L_q\exp\Big(-3\sum_{k=1}^{J}B_k^{\mathrm{orig}}(t)\Big),
\]
hence
\begin{equation}
\label{eq:sum_B_lower}
\sum_{k=1}^{J}B_k^{\mathrm{orig}}(t)
=
\frac13\log\frac{L_q}{\zeta_{J+1}^{\mathrm{orig}}(t)}
\ge
\frac13\log\frac{L_q}{\zeta_\eta}.
\end{equation}

Now fix $1\le k\le J$.  Since $\zeta_k^{\mathrm{orig}}(t)\ge\zeta_\eta$ and $\Gamma_k^{\mathrm{orig}}(\cdot)$ is nonincreasing in time,
we have $\Gamma_k^{\mathrm{orig}}(s)L_q\ge\zeta_\eta$ for all $s\in[0,t]$, which implies $\Gamma_k^{\mathrm{orig}}(s) L_{q,-}\geq \zeta_*$.  By Lemma~\ref{lem:comp_to_Psi}
 the frozen
$b_{\mathrm{froz},k}^{\mathrm{orig}}(s)$ is non-decreasing on $[0,t]$ when $k\geq 2$, while $b_{\mathrm{froz},1}^{\mathrm{orig}}(s)$ is  constant, hence non-decreasing.  Therefore $\int_0^t b_{\mathrm{froz},k}^{\mathrm{orig}}(s)\,ds\le t\,b_{\mathrm{froz},k}^{\mathrm{orig}}(t)$, and using the
bootstrap comparison \eqref{eq:bj_compare_boot},
\[
B_k^{\mathrm{orig}}(t)
=
\int_0^t b_k^{\mathrm{orig}}(s)\,ds
\le
C_\mu\int_0^t b_{\mathrm{froz},k}^{\mathrm{orig}}(s)\,ds
\le
C_\mu t\,b_{\mathrm{froz},k}^{\mathrm{orig}}(t)
\le
\frac{C_\mu}{c_\mu}\,t\,b_k^{\mathrm{orig}}(t)
=
\Theta_\mu\,t\,b_k^{\mathrm{orig}}(t).
\]
Summing over $k\le J$ and combining with \eqref{eq:sum_B_lower} yields
\[
S_{m-1}^{\mathrm{orig}}(t)
\ge
\sum_{k=1}^{J}b_k^{\mathrm{orig}}(t)
\ge
\frac{1}{\Theta_\mu\,t}\sum_{k=1}^{J}B_k^{\mathrm{orig}}(t)
\ge
\frac{1}{3\Theta_\mu \,t}\log\frac{L_q}{\zeta_\eta},
\]
which is \eqref{eq:Sm1_tinv_varcoef}.
\end{proof}

\subsubsection{ODE-Level Norm Inflation and Decay of $T_B(A,m)$.}
\label{subsec:ODE_norm_inflation}
Define a threshold-hitting time (at the fixed threshold $\zeta_\eta$)
\begin{equation*}
t_{\zeta_\eta}^{\mathrm{orig}}:=\inf\{t>0:\ \zeta_m^{\mathrm{orig}}(t)=\zeta_\eta\}.
\end{equation*}
If the set is empty, $t_{\zeta_\eta}^{\mathrm{orig}}$ is understood as $+\infty$.

In the next proposition, the target amplitude $A$ is still a free parameter. 
Accordingly, any bound that depends on $A$ will display that dependence explicitly.

Combining the pre-hitting monotonicity estimate with the front-induced $1/t$ lower bound after $t_{\zeta_\eta}^{\mathrm{orig}}$, we obtain an upper bound for $T_B(A,m)$. The argument splits at the threshold-hitting time $t_{\zeta_\eta}^{\mathrm{orig}}$. For $0\le t\le t_{\zeta_\eta}^{\mathrm{orig}}$, one still has $\zeta_m^{\mathrm{orig}}(t)\ge \zeta_\eta$, so the frozen coefficients remain monotone and $S_{m-1}^{\mathrm{orig}}(t)$ is bounded from below by its initial frozen value; this already shows that $t_{\zeta_\eta}^{\mathrm{orig}}$ is short. For $t_{\zeta_\eta}^{\mathrm{orig}}<t\le T_B$, the front-migration estimate yields a $1/t$ lower bound for $S_{m-1}^{\mathrm{orig}}(t)$. Integrating
\[
\frac{d}{dt}\log x_m^{\mathrm{orig}}(t)=S_{m-1}^{\mathrm{orig}}(t)
\]
from $t_{\zeta_\eta}^{\mathrm{orig}}$ to $T_B$ gives algebraic growth of $x_m^{\mathrm{orig}}(t)$, and comparing this with the bootstrap bound $x_m^{\mathrm{orig}}(T_B)\le \bar A$ yields the desired estimate for $T_B(A,m)$.

\begin{proposition}[Norm inflation for the original weak ODEs]
\label{prop:norm_inflation_varcoef}
Assume $A>\varepsilon$.
It holds
\[
T_B(A,m)\le C(q)\,\bar A^{\gamma_q}\,
m^{\alpha_q-1+\gamma_q\alpha_q},
\qquad 
\gamma_q:=\frac{3\Theta_\mu}{\log(L_q/\zeta_\eta)}.
\]
In particular, if
\[
\beta_q:=
\frac{1}{1+\gamma_q}-\alpha_q
=
\frac{\log(L_q/\zeta_\eta)}{3\Theta_\mu+\log(L_q/\zeta_\eta)}-\alpha_q>0,
\]
then
\begin{align}
\label{eq:est_TB}
T_B(A,m)\le C(q,A)\,m^{-\beta_q}\to 0
\qquad\text{as }m\to\infty.
\end{align}
\end{proposition}
\begin{remark}
\label{rem:role_L}
    Compared to a fixed cone-slope $L$, the entire role of a large cone-slope $L_q$  is to make $\beta_q>0$. 
\end{remark}
\begin{proof}[Proof of Proposition~\ref{prop:norm_inflation_varcoef}]
Write $T_B:=T_B(A,m)$.  Since $T_B\le T_N(A,m)$, we have 
\begin{align*}
\max_{k\le m}x_k^{\mathrm{orig}}(t)\le \bar A 
\end{align*}
for all
$t\in[0,T_B]$, and in particular $x_m^{\mathrm{orig}}(t)\le \bar A$ on $[0,T_B]$.

\smallskip
\noindent{\bf Step 1: a bound up to the threshold-hitting time.}
Set $t_*:=\min\{T_B,t_{\zeta_\eta}^{\mathrm{orig}}\}$.  For $t\in[0,t_*)$ we have $\zeta_m^{\mathrm{orig}}(t)\ge\zeta_\eta$, hence, by \eqref{eq:ordering_zeta},
$\Gamma_j^{\mathrm{orig}}(t)L_q\ge\zeta_\eta$ for all $1\le j\le m$.  By Lemma~\ref{lem:comp_to_Psi}, each frozen $b_{\mathrm{froz},j}^{\mathrm{orig}}(t)$ is nondecreasing on $[0,t_*)$ for $j\geq 2$,  while $b_{\mathrm{froz},1}^{\mathrm{orig}}(t)$ is constant; hence
$S_{\mathrm{froz},m-1}^{\mathrm{orig}}(t)\ge S_{\mathrm{froz},m-1}^{\mathrm{orig}}(0)$ on $[0,t_*)$.  Using \eqref{eq:Sk_compare_boot} we obtain
\begin{align}
\label{eq:lower_bd_s_orig}
S_{m-1}^{\mathrm{orig}}(t)
\ge c_\mu 
S_{\mathrm{froz},m-1}^{\mathrm{orig}}(t)
\ge c_\mu 
S_{\mathrm{froz},m-1}^{\mathrm{orig}}(0)\qquad (0\le t< t_*).
\end{align}
Since $\zeta_m^{\mathrm{orig}}(t)=L_q\Gamma_m^{\mathrm{orig}}(t)=L_q(x_m^{\mathrm{orig}}(0)/x_m^{\mathrm{orig}}(t))^3$ and $\frac{d}{dt}\log x_m^{\mathrm{orig}}(t)=S_{m-1}^{\mathrm{orig}}(t)$,
we have the exact identity
\[
\zeta_m^{\mathrm{orig}}(t)=L_q\exp\Big(-3\int_0^t S_{m-1}^{\mathrm{orig}}(s)\,ds\Big).
\]

Therefore, for every $\tau\in[0,t_*)$,
\[
\int_0^{\tau}S_{m-1}^{\mathrm{orig}}(s)\,ds
=
\frac13\log\frac{L_q}{\zeta_m^{\mathrm{orig}}(\tau)}
\le \frac13\log\frac{L_q}{\zeta_\eta},
\]
because $\zeta_m^{\mathrm{orig}}(\tau)\ge \zeta_\eta$ for $\tau<t_*$. Combining this with
\eqref{eq:lower_bd_s_orig} gives
\[
\tau\le \frac{1}{3c_\mu S_{\mathrm{froz},m-1}^{\mathrm{orig}}(0)}\log\frac{L_q}{\zeta_\eta}
\qquad\text{for all }\tau<t_*.
\]
Hence
\[
t_*\le \frac{1}{3c_\mu S_{\mathrm{froz},m-1}^{\mathrm{orig}}(0)}\log\frac{L_q}{\zeta_\eta}.
\]
In particular, $t_* < \infty$.

At $t=0$ one has $\Gamma_j^{\mathrm{orig}}(0)=1$, hence
\[
b_{\mathrm{froz},j}^{\mathrm{orig}}(0)=x_j^{\mathrm{orig}}(0)\Lambda_{\mathrm{froz}}(1)=\varepsilon j^{-\alpha_q}\Lambda_{\mathrm{froz}}(1),
\]
and therefore
\[
S_{\mathrm{froz},m-1}^{\mathrm{orig}}(0)=\varepsilon\Lambda_{\mathrm{froz}}(1)\sum_{j=1}^{m-1} j^{-\alpha_q}
\ge c(\alpha_q)\,\varepsilon\Lambda_{\mathrm{froz}}(1)\,m^{1-\alpha_q},
\]
for some constant $c(\alpha_q)>0$.

    The frozen coefficient $\Lambda_{\mathrm{froz}}(1)$ evaluated at $\Gamma=1$ only depends on $\phi_q$.
Hence
\begin{equation}
\label{eq:tstar_bound_final_new}
t_* \le 
\frac{1}{3c_\mu S_{\mathrm{froz},m-1}^{\mathrm{orig}}(0)}\log\frac{L_q}{\zeta_\eta}
\le C(q)\,m^{\alpha_q-1}\log\frac{L_q}{\zeta_\eta},
\end{equation}
where the constant $C(q)$ absorbs the fixed dependence on $\tilde\varepsilon$, $\eta$, $\phi_q$, and the other fixed background parameters.
We know that $t_*\leq t_{\zeta_\eta}^{\mathrm{orig}}$. In particular, if $t_{\zeta_\eta}^{\mathrm{orig}}\le T_B$, then $t_{\zeta_\eta}^{\mathrm{orig}}=t_*$ and it satisfies the same bound.

\smallskip
\noindent{\bf Step 2: algebraic growth after the threshold-hitting time.}
If $T_B\le t_{\zeta_\eta}^{\mathrm{orig}}$, then $T_B=t_*$. Hence by \eqref{eq:tstar_bound_final_new},
\[
T_B \le C(q)\,m^{\alpha_q-1}.
\]
Since $A>\varepsilon$ and $\varepsilon=\varepsilon(\tilde\varepsilon,q)$ is fixed, we have
\[
\bar A=\frac{A}{1-\mu}\ge \frac{\varepsilon}{1-\mu},
\]
so $\bar A^{\gamma_q}\ge c(q)>0$. Absorbing this fixed lower bound into the constant and using
$m^{\gamma_q\alpha_q}\ge 1$, we obtain
\[
T_B \le C(q)\,\bar A^{\gamma_q}\,m^{\alpha_q-1+\gamma_q\alpha_q}.
\]

Assume now that $T_B>t_{\zeta_\eta}^{\mathrm{orig}}$. Then for every
$t\in(t_{\zeta_\eta}^{\mathrm{orig}},T_B)$,
Lemma~\ref{lem:front_tinv_varcoef} gives
\[
S_{m-1}^{\mathrm{orig}}(t)\ge \frac{1}{3\Theta_\mu\,t}\log\frac{L_q}{\zeta_\eta}
= \frac{1}{\gamma_q\,t}.
\]
Integrating $\frac{d}{dt}\log x_m^{\mathrm{orig}}(t)=S_{m-1}^{\mathrm{orig}}(t)$ from
$t_{\zeta_\eta}^{\mathrm{orig}}$ to $t$ yields
\[
\log\frac{x_m^{\mathrm{orig}}(t)}{x_m^{\mathrm{orig}}(t_{\zeta_\eta}^{\mathrm{orig}})}
\ge \frac{1}{\gamma_q}\log\frac{t}{t_{\zeta_\eta}^{\mathrm{orig}}},
\]
hence
\[
x_m^{\mathrm{orig}}(t)\ge x_m^{\mathrm{orig}}(t_{\zeta_\eta}^{\mathrm{orig}})
\left(\frac{t}{t_{\zeta_\eta}^{\mathrm{orig}}}\right)^{1/\gamma_q}.
\]
Since $t<T_B\le T_N(A,m)$, we have $x_m^{\mathrm{orig}}(t)\le \bar A$. Therefore
\[
t\le t_{\zeta_\eta}^{\mathrm{orig}}
\left(\frac{\bar A}{x_m^{\mathrm{orig}}(t_{\zeta_\eta}^{\mathrm{orig}})}\right)^{\gamma_q}
\qquad\text{for all }t\in(t_{\zeta_\eta}^{\mathrm{orig}},T_B).
\]
Hence
\[
T_B\le t_{\zeta_\eta}^{\mathrm{orig}}
\left(\frac{\bar A}{x_m^{\mathrm{orig}}(t_{\zeta_\eta}^{\mathrm{orig}})}\right)^{\gamma_q}.
\]
In particular, $T_B< \infty$.

By definition of $t_{\zeta_\eta}^{\mathrm{orig}}$, we have $\zeta_m^{\mathrm{orig}}(t_{\zeta_\eta}^{\mathrm{orig}})=\zeta_\eta$. Since
\[
\zeta_m^{\mathrm{orig}}(t)=L_q\Bigl(\frac{x_m^{\mathrm{orig}}(0)}{x_m^{\mathrm{orig}}(t)}\Bigr)^3,
\]
it follows that
\[
x_m^{\mathrm{orig}}(t_{\zeta_\eta}^{\mathrm{orig}})
=
x_m^{\mathrm{orig}}(0)\Bigl(\frac{L_q}{\zeta_\eta}\Bigr)^{1/3}
=
\varepsilon m^{-\alpha_q}\Bigl(\frac{L_q}{\zeta_\eta}\Bigr)^{1/3}.
\]
Moreover, because $T_B>t_{\zeta_\eta}^{\mathrm{orig}}$, we have $t_{\zeta_\eta}^{\mathrm{orig}}=t_*$, and thus
\eqref{eq:tstar_bound_final_new} gives
\[
t_{\zeta_\eta}^{\mathrm{orig}}\le C(q)\,m^{\alpha_q-1}.
\]
Combining the last three inequalities and absorbing the fixed factors depending only on the
background parameters into $C(q)$, we arrive at
\[
T_B(A,m)\le C(q)\,\bar A^{\gamma_q}\,m^{\alpha_q-1+\gamma_q\alpha_q}.
\]

Finally, since
\[
\alpha_q-1+\gamma_q\alpha_q
=-(1+\gamma_q)\Bigl(\frac{\log(L_q/\zeta_\eta)}{3\Theta_\mu+\log(L_q/\zeta_\eta)}-\alpha_q\Bigr)
=-(1+\gamma_q)\beta_q \le -\beta_q,
\]
the first bound implies
\[
T_B(A,m)\le C(q,A)\,m^{-\beta_q}\to 0
\qquad\text{as }m\to\infty.
\]
\end{proof}

\subsection{A bootstrap Argument}\label{subsec:Euler_bootstrap_improvement}

In this subsection, fix $A>\varepsilon(\tilde \varepsilon,q)$. 
For $m\geq2$, we work on the time interval $[0,T_B(A,m)]$ so that we can use the bootstrap bounds in \eqref{boot:X}. In addition, we use the norm inflation at the ODE level, Proposition \ref{prop:norm_inflation_varcoef}, that is, $T_B(A,m)\to 0$ as $m\to\infty$, in order to improve the bootstrap bounds in \eqref{boot:X}.

The argument has three steps:
\begin{itemize}
\item[(1)] Confinement for $W_k$ and scale separation  of  $\supp \omega_k(\cdot, t)$ (Lemmas~\ref{lem:Euler_cauchy_geometry}, \ref{lem:Euler_scale_separation}).
\item[(2)] Estimates on velocity-error $V_k$ (Lemma~\ref{lem:in_out_cont_updated}).
\item[(3)] Integration along characteristics and bootstrap improvement (Lemma~\ref{lem:Euler_close_bootstrap_X_DY}).
\end{itemize}

From this point on, ODE simplifications such as localization and freezing are no longer essential, so we return to the original notation $x_k(t)$ and $\Gamma_k(t)$ instead of $x_k^{\mathrm{orig}}(t)$ and $\Gamma_k^{\mathrm{orig}}(t)$.

\subsubsection{Confinement and Scale Separation}

\begin{lemma}[Confinement from $X_k$]\label{lem:Euler_cauchy_geometry}
Fix $k\in\{1,\dots,m\}$.
For every $t\in[0,T_B(A,m)]$,
\[
\supp W_k(\cdot,\cdot,t)\subset [r_-,r_+]\times\bigl([z_-,z_+]\cup[-z_+,-z_-]\bigr),
\]
where
\begin{align*}
r_-&:=(1-\mu)(1-\eta)r_0
&
r_+&:=(1+\mu)(1+\eta)r_0,
\\
z_-&:=(1-\mu)(1-\eta)z_0,
&
z_+&:=(1+\mu)(1+\eta)z_0.
\end{align*}
Hence
\[
\supp\omega_k(\cdot,t)\subset
[R_k(t)r_-,R_k(t)r_+]\times 
\Big( 
[H_k(t)z_-,H_k(t)z_+] \cup
[-H_k(t)z_+, -H_k(t)z_-]
\Big).
\]
\end{lemma}

\begin{proof}

For any $(r,z)\in\supp\phi_q $ and any $t\in[0,T_B]$, the bootstrap bounds in \eqref{boot:X} imply
\[
X_{k}^r(r,z,t)\in[r_-,r_+],\qquad |X_{k}^z(r,z,t)|\in[z_-,z_+].
\]
 Since the support of $W_k$ is transported by $X_k$ by \eqref{cauchy_for_W} , this yields the confinement of $\supp W_k(\cdot,\cdot,t)$. Finally, if $(r,z)\in\supp \, \omega_k (\cdot,t)$ then $(r/R_k(t),z/H_k(t))\in\supp W_k(\cdot,\cdot,t)$ by the ansatz $\omega_k(r,z,t)=x_k(t)W_k(r/R_k(t),z/H_k(t),t)$, which gives the confinement of $\supp \omega_k (\cdot,t)$.
\end{proof}

Using Lemma \ref{lem:Euler_cauchy_geometry}, the ODEs \eqref{EulerODE}, and the bound $x_k(t) \leq \bar A$ on $[0,T_B(A,m)]$, we prove scale separation and disjointness of $\supp \omega_k(\cdot,t)$.
\begin{lemma}[Scale separation]\label{lem:Euler_scale_separation}
There exists an integer $m_0(A,q)$ (depending on $A,q$ and the fixed background parameters) such that 
for all $m\geq m_0(A,q)$, $k\in \{1,\ldots,m-1\}$ and $t\in[0,T_B(A,m)]$,
\begin{equation}\label{eq:sep_R_ratio}
R_{k+1}(t)\,r_+\le \frac12\,R_k(t)\,r_-,
\end{equation}
and for any $j>k$, any $x\in \supp \omega_k(\cdot,t)$ and any $y\in\supp \omega_j(\cdot,t)$,
\begin{equation}\label{eq:sep_distance_estimate}
|x-y|\ge \frac12\,r_-\,R_k(t).
\end{equation}
In particular, the supports $\{\supp \omega_k(\cdot,t)\}_{k=1}^m$ are pairwise disjoint on $[0,T_B(A,m)]$.
\end{lemma}

\begin{proof}
\textbf{Step 1: control of $\tilde x_{k+1}/\tilde x_k$.}
Using the ODEs $x_k'/x_k=\partial_r u_-^r(0,0,t)$, one can obtain
\[
\frac{d}{dt}\log\!\Bigl(\frac{\tilde x_{k+1}(t)}{\tilde x_k(t)}\Bigr)
=\frac{d}{dt}\log\!\Bigl(\frac{x_{k+1}(t)}{x_k(t)}\Bigr)
=\frac{x'_{k+1}(t)}{x_{k+1}(t)}-\frac{x'_k(t)}{x_k(t)}
=\partial_r u^r_k(0,0,t).
\]
By the axisymmetric Biot--Savart representation at the origin,
\[
\partial_r u^r_k(0,0,t)
=-x_k(t)\,
\int_{-\infty}^\infty \int_{0}^\infty
\frac{r^2z}{(r^2+ z^2)^{5/2}}\,
W_k\left(
        \frac{r}{R_k (t)},
        \frac{z}{H_k(t)}, t
    \right)\,dr\,dz \geq 0.
\]
Define
\[
\Lambda_{\max}:=\sup_{0<\Gamma\le 1}\Gamma^2\Lambda_{\mathrm{froz}}(\Gamma)<\infty,
\]
where $\Lambda_{\mathrm{froz}}$ is defined in \eqref{lem:pro_est}.
Using the comparison estimate \eqref{heu05} (together with the definition of $\Lambda_{\max}$) yields
\[
\partial_r u^r_k(0,0,t)\le \frac{(1+\mu)^3}{(1-\mu)^5}\,x_k(t)\,\Lambda_{\max}\le C_{\rm sep}\,x_k(t)\le C_{\rm sep}\,\bar A.
\]
for the constant $C_{\rm sep}:=\frac{(1+\mu)^3}{(1-\mu)^5}\Lambda_{\max}$.
Integrating in time gives that for every $k\in \{1,\ldots,m-1\}$ and every $t\in[0,T_B(A,m)]$,
\begin{equation}\label{eq:sep_log_ratio}
\log\!\Bigl(\frac{\tilde x_{k+1}(t)}{\tilde x_k(t)}\Bigr)\le C_{\rm sep}\,\bar A\,t.
\end{equation}

\medskip
\noindent \textbf{Step 2: disjointness of supports and the distance bound.}
From Proposition~\ref{prop:norm_inflation_varcoef}, it follows that $C_{\mathrm{sep}} \, \bar A \, T_B(A,m)\to 0$ as $m\to\infty$. Hence, there exists an integer $m_0(A,q)$ (which depends on the fixed background parameters in addition to $A,q$) such that  for all $m\geq m_0(A,q)$, it holds
\begin{equation*}
d\,e^{C_{\rm sep}\bar A T_B(A,m)}\le \frac12\,\frac{r_-}{r_+},
\end{equation*}
Then by \eqref{eq:sep_log_ratio},
\[
\frac{R_{k+1}(t)}{R_k(t)}=d\,\frac{\tilde x_{k+1}(t)}{\tilde x_k(t)}
\le d\,e^{C_{\rm sep}\bar At}
\le d\,e^{C_{\rm sep} \bar A T_B}\le \frac12\,\frac{r_-}{r_+},
\]
which implies \eqref{eq:sep_R_ratio}. By Lemma~\ref{lem:Euler_cauchy_geometry},
\[
\supp \omega_k(\cdot,t)\subset\{r\in[R_k(t)r_-,R_k(t)r_+]\}.
\]
If $j>k$, then iterating \eqref{eq:sep_R_ratio} yields $R_j(t)r_+\le R_{k+1}(t)r_+\le \frac12R_k(t)r_-$. Hence for $x\in \supp \omega_k(\cdot,t)$ and $y\in\supp \omega_j(\cdot,t)$,
\[
|x-y|\ge |r_x-r_y|\ge R_k(t)r_--R_j(t)r_+\ge \frac12\,R_k(t)r_-,
\]
which is \eqref{eq:sep_distance_estimate}. The annuli in $r$ therefore do not overlap, so the supports are pairwise disjoint.
\end{proof}

\subsubsection{Velocity-Error Estimates}
\label{sec:bftrv_new}

The goal of this subsection is to estimate the rescaled velocity error
appearing in the equation for \(W_k\).  We
use Proposition~\ref{prop:cone_free_key_lem} directly.  The point
is that the leading outer-region integral in the $Q(r_x)$ estimates of Proposition~\ref{prop:cone_free_key_lem} agrees with the
stretching rate \(\partial_r u_-^r(0,0,t)\) used in the definition of the
ODE for \(x_k(t)\), up to a harmless possible contribution from the \(k\)-th
ring itself.

\begin{lemma}[Velocity-error estimate from Proposition~\ref{prop:cone_free_key_lem}]
\label{lem:in_out_cont_updated}
\label{lem:asymptotic_integral_bound_updated}
\label{lem:Euler_Vk_grad_u_updated}
Choose \(m_0(A,q)\) as in Lemma~\ref{lem:Euler_scale_separation}.
Then there exists a constant \(C_E=C_E(q)>0\), depending only on \(q\) and the fixed
background parameters, such that for every \(m\ge m_0(A,q)\),
\begin{equation}
\label{eq:desired_bd_res_vel_updated}
\sup_{(r,z),t,k}
\left(
\left|
\frac{V_k^r(r,z,t)}{r}
\right|
+
\left|
\frac{V_k^z(r,z,t)}{z}
\right|
\right)
\le
C_E\,\bar A
\left(
1+\log\left(\frac{\bar A}{\varepsilon}\right)+\log m
\right),
\end{equation}
where the supremum is taken over all
\[
1\le k\le m,\qquad
0\le t\le T_B(A,m),\qquad
(r,z)\in\supp W_k(\cdot,\cdot,t).
\]
\end{lemma}

\begin{proof}
Throughout the proof, \(C(q)>0\) denotes a constant depending only on \(q\)
and the fixed background parameters, and it may change from line to line.
Fix
\[
1\le k\le m,\qquad 0\le t\le T_B(A,m),\qquad
(r,z)\in\supp W_k(\cdot,\cdot,t),
\]
and set
\[
x=(r_x,z_x):=(R_k(t)r,\ H_k(t)z).
\]
By Lemma~\ref{lem:Euler_cauchy_geometry},
\[
r\in [r_-,r_+],\qquad |z|\in [z_-,z_+].
\]

We use the thin-annulus condition
\[
r_+<2r_-.
\]
This condition is ensured by our choice of $\mu=1/20, \eta=1/4$.

\medskip
\noindent\textbf{Step 1: preliminary bounds and the outer-region integral.}
By Lemma~\ref{lem:Euler_cauchy_geometry1} and the bootstrap bound
\(x_j(t)\le \bar A\) on \([0,T_B(A,m)]\),
\[
\|\omega_j(\cdot,t)\|_{L^\infty}
\le C(q)x_j(t)
\le C(q)\bar A,
\qquad 1\le j\le m.
\]
Since the supports are pairwise disjoint by Lemma~\ref{lem:Euler_scale_separation},
we obtain
\begin{equation}
\label{eq:omega_m_Linf_conefree_boot}
\|\omega^{(m)}(\cdot,t)\|_{L^\infty}\le C(q)\bar A.
\end{equation}
Moreover, \(\omega^{(m)}(\cdot,t)\in L^\infty\cap L^2(\mathbb R^3)\), and it is odd
with respect to \(z\). Hence Proposition~\ref{prop:cone_free_key_lem} applies to
\(\omega^{(m)}(\cdot,t)\).

Since \(\omega_-\) is
supported away from the origin, differentiating the axisymmetric Biot--Savart law
at the origin gives
\begin{equation}
\label{eq:outer_origin_moment_identity_conefree}
\partial_r u_-^r(0,0,t)
=
\frac{3}{8\pi}
\int_{\mathbb R^3}
\frac{r_yz_y}{|y|^5}\,\big(-\omega_-(y,t)\big)\,dy .
\end{equation}

\medskip
\noindent\textbf{Step 2: exact identification of the outer region.}
Let
\[
Q(r_x):=\{y\in\mathbb R^3:\ r_y\ge 2r_x\}.
\]
We claim that
\begin{equation}
\label{eq:Q_exact_outer_rings_conefree}
Q(r_x)\cap\supp\omega^{(m)}(\cdot,t)
=
\bigcup_{j=1}^{k-1}\supp\omega_j(\cdot,t),
\end{equation}
where the right-hand side is empty when \(k=1\).

Indeed, if \(j<k\), then Lemma~\ref{lem:Euler_scale_separation} gives
\[
R_k(t)r_+\le \frac12 R_{k-1}(t)r_-\le \frac12 R_j(t)r_- .
\]
Since \(r_x\le R_k(t)r_+\), for every \(y\in\supp\omega_j(\cdot,t)\) we have
\[
r_y\ge R_j(t)r_-
\ge 2R_k(t)r_+
\ge 2r_x.
\]
Thus every outer ring \(j<k\) is contained in \(Q(r_x)\).

If \(j>k\), then Lemma~\ref{lem:Euler_scale_separation} gives
\[
R_j(t)r_+\le R_{k+1}(t)r_+\le \frac12 R_k(t)r_-.
\]
Since \(r_x\ge R_k(t)r_-\), for every \(y\in\supp\omega_j(\cdot,t)\) we have
\[
r_y\le R_j(t)r_+
\le \frac12 R_k(t)r_-
\le \frac12 r_x
<2r_x.
\]
Thus no inner ring \(j>k\) intersects \(Q(r_x)\).

Finally, if \(y\in\supp\omega_k(\cdot,t)\), then the thin-annulus condition gives
\[
r_y\le R_k(t)r_+
<2R_k(t)r_-
\le 2r_x.
\]
Hence the \(k\)-th ring itself is disjoint from \(Q(r_x)\). This proves
\eqref{eq:Q_exact_outer_rings_conefree}.

Consequently, using \eqref{eq:outer_origin_moment_identity_conefree},
\begin{equation}
\label{eq:Q_moment_exact_outer_conefree}
\frac{3}{8\pi}
\int_{Q(r_x)}
\frac{r_yz_y}{|y|^5}\, \big(-\omega^{(m)}(y,t) \big)\,dy
=
\partial_r u_-^r(0,0,t).
\end{equation}

\medskip
\noindent\textbf{Step 3: radial component.}
By the definition of \(V_k\) in \eqref{def_v_k14} and by the ODEs \eqref{EulerODE}
with the convention \(x_1'(t)=0\) and \(u_-=0\) for \(k=1\), we have
\[
\frac{V_k^r(r,z,t)}{r}
=
\frac{u^{(m),r}(x,t)}{R_k(t)r}
-
\partial_r u_-^r(0,0,t)
=
\frac{u^{(m),r}(x,t)}{r_x}
-
\partial_r u_-^r(0,0,t).
\]
Applying Proposition~\ref{prop:cone_free_key_lem} to \(\omega^{(m)}(\cdot,t)\), and then
using \eqref{eq:Q_moment_exact_outer_conefree}, gives
\[
\left|
\frac{V_k^r(r,z,t)}{r}
\right|
\le
C\frac{|x|}{r_x}\,
\|\omega^{(m)}(\cdot,t)\|_{L^\infty}.
\]
On the support under consideration,
\[
\frac{|x|}{r_x}
\le
1+\frac{|z_x|}{r_x}
=
1+\Gamma_k(t)\frac{|z|}{r}
\le C(q),
\]
where we used $\Gamma_k \leq 1$, \(r\ge r_-\), and
\(|z|\le z_+\). Combining with \eqref{eq:omega_m_Linf_conefree_boot}, we obtain
\begin{equation}
\label{eq:Vk_radial_conefree_bound_no_remainder}
\left|
\frac{V_k^r(r,z,t)}{r}
\right|
\le C(q)\bar A.
\end{equation}

\medskip
\noindent\textbf{Step 4: vertical component.}
Similarly, by the definition of \(V_k\),
\[
\frac{V_k^z(r,z,t)}{z}
=
\frac{u^{(m),z}(x,t)}{H_k(t)z}
+
2\partial_r u_-^r(0,0,t)
=
\frac{u^{(m),z}(x,t)}{z_x}
+
2\partial_r u_-^r(0,0,t).
\]
The vertical one in the $Q(r_x)$-estimates of Proposition~\ref{prop:cone_free_key_lem}, together with
\eqref{eq:Q_moment_exact_outer_conefree}, gives
\[
\left|
\frac{V_k^z(r,z,t)}{z}
\right|
\le
C\left(1+\left|\log\frac{r_x}{|z_x|}\right|\right)
\|\omega^{(m)}(\cdot,t)\|_{L^\infty}.
\]
Therefore, by \eqref{eq:omega_m_Linf_conefree_boot},
\begin{equation}
\label{eq:Vk_vertical_conefree_prelog_no_remainder}
\left|
\frac{V_k^z(r,z,t)}{z}
\right|
\le
C(q)\bar A
\left(1+\left|\log\frac{r_x}{|z_x|}\right|\right).
\end{equation}

It remains to bound the logarithm. Since
\[
\frac{r_x}{|z_x|}
=
\frac{R_k(t)r}{H_k(t)|z|}
=
\frac{r}{|z|}\frac{1}{\Gamma_k(t)},
\]
and \(r/|z|\) is bounded above and below on \(\supp W_k\) by constants depending
only on the fixed geometry,
\[
1+\left|\log\frac{r_x}{|z_x|}\right|
\le
C(q)\log\left(\frac{e}{\Gamma_k(t)}\right).
\]
Since \(x_k(t)\le \bar A\) on \([0,T_B(A,m)]\) and
\(x_k(0)=\varepsilon k^{-\alpha_q}\), we get
\[
\frac1{\Gamma_k(t)}
=
\left(\frac{x_k(t)}{x_k(0)}\right)^3
\le
\left(\frac{\bar A\,k^{\alpha_q}}{\varepsilon}\right)^3
\le
\left(\frac{\bar A\,m^{\alpha_q}}{\varepsilon}\right)^3.
\]
Hence
\[
\log\left(\frac{e}{\Gamma_k(t)}\right)
\le
C(q)
\left(
1+\log\left(\frac{\bar A}{\varepsilon}\right)+\log m
\right).
\]
Substituting this into
\eqref{eq:Vk_vertical_conefree_prelog_no_remainder} yields
\begin{equation}
\label{eq:Vk_vertical_conefree_bound_no_remainder}
\left|
\frac{V_k^z(r,z,t)}{z}
\right|
\le
C(q)\bar A
\left(
1+\log\left(\frac{\bar A}{\varepsilon}\right)+\log m
\right).
\end{equation}

Combining
\eqref{eq:Vk_radial_conefree_bound_no_remainder} and
\eqref{eq:Vk_vertical_conefree_bound_no_remainder}, and enlarging the constant,
proves \eqref{eq:desired_bd_res_vel_updated}.
\end{proof}

\subsubsection{Bootstrap Improvement}
For any $A>\varepsilon(\tilde \varepsilon,q)$, thanks to Proposition~\ref{prop:norm_inflation_varcoef}, 
\[
T_B(A,m)\leq C(q,A)\, m^{-\beta_q}\to 0 \quad  \text{as }m\to\infty.
\]
As the algebraic decay $m^{-\beta_q}$ wins over the logarithmic growth $\log m$, there exists a large integer $m_1(A,q)\geq m_0(A,q)$ such that any $m\geq m_1(A,q)$ satisfies 
    \begin{align}
    \label{eq:improve_bds}
        C_E \, \bar A \left( 1 + \log\left(\frac{\bar A}{\varepsilon}\right) + \log m \right)T_B(A,m)
        \leq 
        \min\left\{ \log \left( 1+\frac{\mu}{2}\right), \log \left(\frac{1}{1-\mu/2} \right) \right\}
    \end{align}
where $C_E$ is the constant defined in Lemma~\ref{lem:in_out_cont_updated}.
\begin{proposition}[Improved bounds for $X_k$ and bootstrap improvement]\label{lem:Euler_close_bootstrap_X_DY}
Fix $m_1(A,q)\geq m_0(A,q)$ so that \eqref{eq:improve_bds} holds for all $m\geq m_1(A,q)$, and fix such an $m$.
Then for every $t\in[0,T_B(A,m)]$ the following hold:
\begin{enumerate}
\item[(i)] \textbf{Logarithmic control of $X_k^r/r$ and $X_k^z/z$.} 
\begin{equation}\label{eq:log_control_X}
\sup_{\substack{(r,z)\in \supp \phi_q \\ k=1,\ldots,m}}\left(
\left|\log\frac{X_k^r(r,z,t)}{r}\right|
+\left|\log\frac{X_k^z(r,z,t)}{z}\right|
\right)
\le C_E\,\bar A\, \left( 1 + \log\left(\frac{\bar A}{\varepsilon}\right) + \log m \right)\,t,
\end{equation}
where $C_E$ is the constant defined in Lemma~\ref{lem:in_out_cont_updated}.
\item[(ii)] \textbf{Bootstrap improvement.} 
\[
    \sup_{\substack{(r,z)\in \supp \phi_q\\ k=1,\ldots,m}}
    \max\left\{
    \left| \frac{X_k^r(r,z,t)}{r}-1\right|,
    \left|\frac{X_k^z(r,z,t)}{z}-1\right|
    \right\}
    \leq \frac{\mu}{2}
\]
In particular, the bootstrap bounds \eqref{boot:X} improve with margin on $[0,T_B]$, and therefore $T_B(A,m)=T_N(A,m)$.
\end{enumerate}
\end{proposition}

\begin{proof}
Fix $k\in\{1,\dots,m\}$ and $t\in[0,T_B]$.

\medskip
\noindent \textbf{Part (i).} Fix $(r,z)\in\supp\phi_q$ and consider the characteristic $s\mapsto X_k(r,z,s)$ solving
\[
\partial_s X_k(r,z,s)=V_k\bigl(X_k(r,z,s),s\bigr),\qquad X_k(r,z,0)=(r,z).
\]
Since the bootstrap bounds \eqref{boot:X} imply $X_k^r(r,z,s)/r>0$ and $X_k^z(r,z,s)/z>0$ for $s\in[0,T_B]$, the logarithms below are well-defined. Differentiating and using $\partial_s X_k^r=V_k^r(X_k,s)$ and $\partial_s X_k^z=V_k^z(X_k,s)$,
\[
\frac{d}{ds}\log\frac{X_k^r(r,z,s)}{r}
=\frac{V_k^r(X_k(r,z,s),s)}{X_k^r(r,z,s)},
\qquad
\frac{d}{ds}\log\frac{X_k^z(r,z,s)}{z}
=\frac{V_k^z(X_k(r,z,s),s)}{X_k^z(r,z,s)}.
\]
Since $(r,z)\in \supp \phi_q =\supp W_k(\cdot,\cdot,0)$ and $W_k/r$ is transported by $X_k$ via the Cauchy formula, the trajectory remains in the transported support, that is,  $X_k(r,z,s)\in\supp W_k(\cdot,\cdot,s)$. Hence \eqref{eq:desired_bd_res_vel_updated}  yields
\[
\Bigl|\frac{V_k^r(X_k(r,z,s),s)}{X_k^r(r,z,s)}\Bigr|
+\Bigl|\frac{V_k^z(X_k(r,z,s),s)}{X_k^z(r,z,s)}\Bigr|
\le C_E\,\bar A \left( 1 + \log\left(\frac{\bar A}{\varepsilon}\right) + \log m \right)
\]
Integrating from $0$ to $t$ gives \eqref{eq:log_control_X}.

\medskip
\noindent \textbf{Part (ii)}
Using \eqref{eq:log_control_X}, as $m\geq m_1(A,q)$ satisfies \eqref{eq:improve_bds}, it holds that 
$1-\mu/2\le X_k^r(r,z,t)/r\le 1+\mu/2$ and the same for $X_k^z/z$.
Hence $F_m(t) \leq \mu/2<\mu$ on $[0,T_B].$ If $T_B<T_N$, continuity of $F_m$ yields that there exists $\tilde{\delta}>0$ such that $F_m(t) < \mu$ on $[0,T_B+\tilde{\delta}]$, contradicting maximality of $T_B$. Therefore, $T_B=T_N$.
\end{proof}

\subsection{Norm inflation and Proof of Theorem~\ref{main_thm1}}
\label{sec:proof_Thm1}

\begin{proof}[Proof of Theorem~\ref{main_thm1}]
Fix $q>1$ and let $\widetilde\varepsilon>0$, $\delta>0$, and $A>0$ be given.

\medskip
\noindent
\textbf{Step 1: Choice of parameters and initial data.}
Fix $r_0=1, \eta=1/4, \mu=20^{-1}. $ (Recall $\eta$ denotes the localization parameter that first appears in \eqref{eq:def_Lpm} whereas $\mu$  the bootstrap parameter from \eqref{boot:X}). 
Define $\Theta_\mu$ by \eqref{eq:def_Theta_mu} and $\zeta_\eta$ by
\eqref{eq:def_xi_eta} where $\zeta_*=1/\sqrt{2}$. 
Fix a cone slope parameter $L_q$  by
\begin{align*}
    \frac{1}{q} < \frac{\log (L_q/\zeta_\eta )}{3\Theta _\mu + \log (L_q/\zeta_\eta)}, 
    \quad 
\end{align*}
Fix $\alpha_q$ by
    \begin{align}
    \label{cond_alpha66}
        \frac{1}{q}<\alpha_q < \frac{\log (L_q/\zeta_\eta)}{3\Theta _\mu + \log (L_q/\zeta_\eta)}(<1).
    \end{align}
Fix a profile $\phi=\phi_q(r,z)$ satisfying \eqref{cond_phi} with $z_0=z_{0,q}= L_q\, r_0$. 

For any $\varepsilon>0$, any integer $m\geq 2$ and $d=10^{-2}$ , we define our initial data $\omega_0^{(m)}$ by \eqref{KJdata2} with $\phi=\phi_q$. From \eqref{ini_lorentz14}, it follows that there exists $\varepsilon=\varepsilon( \tilde \varepsilon,q)\in (0,\tilde \varepsilon)$, independent of $m$, such that
\[
\|\omega_0^{(m)}\|_{L^\infty\cap L^{1}(\mathbb R^3)}
+\left\|\frac{\omega_0^{(m)}}{r}\right\|_{L^{3,q}(\mathbb R^3)}
<\widetilde\varepsilon
\]

For an integer $m\ge 2$, there exists a unique global-in-time smooth axisymmetric, no-swirl solution $\omega^{(m)}$ to the Euler equation \eqref{main_eq} with our initial data $\omega_0^{(m)}$. We decompose $\omega^{(m)}$ into multi vortex rings, $\omega_k$'s, as in \eqref{eq:multi_vortex_ring_dec}.

\medskip
\noindent \textbf{Step 2: Bootstrap Improvement.}
Define the amplitude factors $\{x_k(t)\}_{k=1}^m$ and their initial values $x_k(0)$ by \eqref{EulerODE} and the profiles $\{W_k(r,z,t)\}_{k=1}^m$ by  \eqref{eq:def_wk94}. Define the renormalized target amplitude $\bar A$ by \eqref{def:Euler_ni_time}. 
Define the (ODE) norm inflation time $T_N(A,m)$ by \eqref{def:Euler_ni_time} and the maximal bootstrap time $T_B(A,m)\geq 0$ by \eqref{def_TB_Euler}. 

If $A\leq \varepsilon$, then norm inflation already holds at $t=0$. Hence we only need to consider the case $A> \varepsilon$. Then $T_N(A,m)>0$.

By Proposition \ref{prop:norm_inflation_varcoef},  thanks to the choice of $\alpha_q$, \eqref{cond_alpha66}, it holds that $T_B(A,m)\to 0$  as $m\to\infty$.
Therefore, there exists $m_1=m_1(A,q)\geq m_0(A,q)$ such that all $m\geq m_1(A,q)$ satisfy \eqref{eq:improve_bds}.
Then by bootstrap improvement, Proposition~\ref{lem:Euler_close_bootstrap_X_DY}, it holds $T_B(A,m)=T_N(A,m)$, hence the bootstrap bound holds on $[0,T_N(A, m)]$, and also $T_N(A,m)\to 0$ as $m\to\infty$.

\medskip
\noindent
\textbf{Step 3: Norm inflation in $L^\infty$.}
Fix any $m\geq m_1(A,q)$ and $k\in\{1,\dots,m\}$.
On $[0,T_N(A,m)]$, Lemma~\ref{lem:Euler_cauchy_geometry1} gives the exact Cauchy formula for the profile:
for every $(r,z)\in\supp\phi_q$ and every $t\in[0,T_N(A, m)]$,
\[
W_k(X_k(r,z,t),t)\;=\;\frac{X^r_k(r,z,t)}{r}\,\phi_q(r,z).
\]
Using the bootstrap lower bound $X^r_k(r,z,t)/r\ge 1-\mu$ on $\supp\phi_q$, we obtain
\[
\|W_k(\cdot,\cdot,t)\|_{L^\infty(\mathbb{R}^3)}
\;\ge\;(1-\mu)\,\|\phi_q\|_{L^\infty(\mathbb{R}^3)}
\qquad\text{for all }t\in[0,T_N(A,m)].
\]
Using scale separation Lemma~\ref{lem:Euler_scale_separation}, the supports $\{\supp \omega_k(\cdot,t)\}_{k=1}^m$ are pairwise disjoint on $[0,T_B(A,m)]$. Hence, we  can estimate the $L^\infty$ norm of the solution $\omega^{(m)}$ by
\[
\|\omega^{(m)}(\cdot,t)\|_{L^\infty(\mathbb R^3)}
=\sup_{1\leq k\leq m}\|\omega_k(\cdot,t)\|_{L^\infty(\mathbb R^3)}
\;=\; \sup_{1\leq k \leq m} x_k(t)\,\|W_k(\cdot,\cdot,t)\|_{L^\infty(\mathbb{R}^3)}.
\]
Evaluating at $t=T_N(A, m)$ yields
\[
\|\omega^{(m)}(\cdot,T_N(A, m))\|_{L^\infty(\mathbb R^3)}
\;\ge\;(1-\mu)\|\phi_q\|_{L^\infty(\mathbb{R}^3)}
\sup_{1\le k\le m} x_k(T_N(A, m))
\;\geq \;(1-\mu)\|\phi_q\|_{L^\infty}\,\bar A.
\]
We conclude
\begin{align*}
\|\omega^{(m)}(\cdot,T_N(A,m))\|_{L^\infty(\mathbb R^3)}
\;\ge\;A.
\end{align*}
Finally, since $T_N(A,m)\to 0$ as $m\to\infty$, there exists an index $m_2(A,\delta)\geq m_1(A,q)$ such that $T_N(A,m)\le \delta$ for all $m\ge m_2(A,\delta)$. 
Therefore, for $m\geq  m_2(A,\delta),$ it holds 
\[
\sup_{0\le t\le \delta}\|\omega^{(m)}(\cdot,t)\|_{L^\infty(\mathbb R^3)} \;\ge\; A.
\]
This completes the proof of Theorem~\ref{main_thm1}.
\end{proof}

\section{Instantaneous blow-up}
\label{sec:inst_bu}

This section is devoted to establishing Theorem~\ref{thm:instant_blow-up}. 
This section is organized as follows:
\begin{itemize}
\item 
We first prove preliminary lemmas for Yudovich-type weak solutions \((u,\omega)\). (\S~\ref{sec:prel_lem}). 
\item Using the preliminary lemmas, we will prove the non-existence result, Theorem~\ref{thm:instant_blow-up}-(1)  (\S~\ref{subsec:non_existence}). 
\item Next, we will prove the instantaneous blow-up result, Theorem~\ref{thm:instant_blow-up}-(2) by taking the limit of the sequence of our finite-ring approximations (\S~\ref{sec:pf_rem}).
\end{itemize}

\subsection{Preliminary Lemmas for Instantaneous Blow-Up}
\label{sec:prel_lem}

This subsection is devoted to preliminary lemmas: primitive-variable formulation (Lemma~\ref{lem:weak_to_euler}), uniqueness (Lemma~\ref{lem:uniq_bd}), a flow map and the Cauchy formula (Lemma~\ref{lem:flow_cauchy_bd}) and a Danskin-type lemma (Lemma~\ref{lem:envelope_radial_extrema}).

We will use these lemmas in the proof of Theorem~\ref{thm:instant_blow-up}-(1) as follows: Fix $T_0>0$ and assume for contradiction that there exists a Yudovich-type weak solution $(u,\omega)$ on $(0,T_0)$. We first derive primitive-variable and uniqueness consequences of this assumption, then use them to propagate the symmetry and sign structure, and finally run a finite-head bootstrap on the first $M$ rings while treating the tail by Proposition~\ref{prop:cone_free_key_lem}, the flow map lemma and the Danskin-type lemma.

We prove Lemmas~\ref{lem:weak_to_euler},~\ref{lem:uniq_bd} in this subsection, but the proof of  Lemma~\ref{lem:envelope_radial_extrema} is deferred to Appendix~\ref{app:two_lemmas}, which appears after the proofs  of Theorem~\ref{thm:instant_blow-up} so that the main proofs are not interrupted.

For instantaneous blow-up, we argue by contradiction, which leads to the existence of a bounded solution in a time interval. 

As opposed to the proof for norm inflation, continuity of the bootstrap deviation functional $F_m(t)$ breaks down when $m=\infty$; in this case, the supremum over \emph{infinitely many} indices $k$ in the definition \eqref{boot:X} of $F_m(t)$ prevents the functional from being continuous. It is not obvious how to make a bootstrap argument for infinitely many rings all together.

Instead, for infinitely many rings $k=1,2,\ldots,$ we make a bootstrap argument only for finitely many rings $k=1,\ldots,M$. Hence in the proof below, $M$ denotes the number of rings for which we make a bootstrap argument.

\medskip
\noindent\textbf{Primitive-Variable Formulation.}
The next lemma isolates the passage from the weak formulation of the relative vorticity
equation to the primitive-variable Euler formulation. The Lorentz exponent $q$ plays no role
here; only the weak transport identity \eqref{eq:transport_distributional} and the uniform
$L^1\cap L^\infty$ bound are used.

\begin{lemma}[Primitive-variable formulation and strong $L^2$ trace]
\label{lem:weak_to_euler}
Let $\omega_0=\omega_0(r,z)\in L^1\cap L^\infty(\mathbb R^3)$, and assume that $\omega_0/r\in L^1_{\mathrm{loc}}(\mathbb{R}^3)$.
 For $T_0>0$, let
$(u,\omega)$ be a Yudovich-type weak solution of
\eqref{eq:IVP_rel_vor} on $(0,T_0)$ with initial data $\omega_0$ in the  sense of Definition~\ref{def:dist_sol_rel_vor}. Let $u_0$ be the velocity induced by $\omega_0$.
Then there exists a distribution
$p\in\mathcal D'((0,T_0)\times\mathbb R^3)$ such that
\[
\partial_t u+\div (u\otimes u)+\nabla p=0,
\qquad
\div u=0
\]
in $\mathcal D'((0,T_0)\times\mathbb R^3)$, and
\[
u\in L^\infty\!\bigl(0,T_0;H^1\cap L^\infty(\mathbb R^3)\bigr)
\cap W^{1,\infty}\!\bigl(0,T_0;H^{-1}(\mathbb R^3)\bigr)
\subset C\bigl([0,T_0];L^2(\mathbb R^3)\bigr),
\]
with
\[
u|_{t=0}=u_0
\qquad\text{in }L^2(\mathbb R^3).
\]
\end{lemma}

\begin{proof}
Define
\[
\boldsymbol{\omega}:=\omega e_\theta, 
\qquad A_1:=1+\|\omega\|_{L^\infty((0,T_0);L^1\cap L^\infty(\mathbb R^3))}.
\]
By interpolation,
\[
\|\omega(t)\|_{L^p(\mathbb R^3)}\le A_1,
\qquad 1\le p\le \infty.
\]
Standard Biot--Savart and Calder\'on--Zygmund estimates yield
\begin{equation}
\label{eq:uniq_est_u}
\|u(t)\|_{L^2\cap L^\infty(\mathbb R^3)}\le CA_1,
\end{equation}
and, for every $p\in[2,\infty)$,
\begin{equation}
\label{eq:uniq_est_grad_u}
\|\nabla u(t)\|_{L^p(\mathbb R^3)}
\le C p\,\|\omega(t)\|_{L^p(\mathbb R^3)}
\le C p\ A_1
\end{equation}
for some constant $C$ that is independent of $p$.
In particular,
\[
u\in L^\infty\big(0,T_0;H^1\cap L^\infty(\mathbb R^3)\big).
\]

We now recover the scalar vorticity equation from \eqref{eq:transport_distributional}.
Let $\chi\in C^\infty([0,\infty))$ satisfy
\[
0\le \chi\le 1,\qquad
\chi(\rho)=0\ \text{for }\rho\le 1,\qquad
\chi(\rho)=1\ \text{for }\rho\ge 2,
\]
and define $\chi_\varepsilon(r):=\chi(r/\varepsilon)$.
For a smooth compactly supported axisymmetric scalar test function
$\psi=\psi(r,z,t)$, set
$ \varphi_\varepsilon(x,t):=\chi_\varepsilon(r)\,r\,\psi(r,z,t).$
Since $\varphi_\varepsilon\in C_c^\infty(\mathbb R^3\times[0,T_0))$, it is admissible in
\eqref{eq:transport_distributional}. Using
\[
\frac1r\,\partial_t \varphi_\varepsilon=\chi_\varepsilon\,\partial_t\psi,
\qquad
\frac1r\,u\cdot\nabla \varphi_\varepsilon
=
\chi_\varepsilon\,(u^r\partial_r\psi+u^z\partial_z\psi)
+\frac{u^r}{r}\chi_\varepsilon\,\psi
+u^r\chi_\varepsilon'(r)\psi,
\]
we obtain
\begin{align*}
&\int_0^{T_0}\!\!\int_{\mathbb R^3}
\omega
\Bigl[
\chi_\varepsilon\Bigl(\partial_t\psi
+u^r\partial_r\psi
+u^z\partial_z\psi
+\frac{u^r}{r}\psi\Bigr)
+u^r\chi_\varepsilon'(r)\psi
\Bigr]\,dx\,dt
+\int_{\mathbb R^3}\omega_0\,\chi_\varepsilon\,\psi(\cdot,0)\,dx
=0.
\end{align*}
Since $\chi_\varepsilon\to 1$ pointwise and $0\le \chi_\varepsilon\le 1$,
the terms involving \(\omega\chi_\varepsilon\partial_t\psi\),
\(\omega\chi_\varepsilon u\cdot\nabla\psi\), and the initial datum converge by
dominated convergence using \(\omega\in L^\infty(0,T_0;L^1_{\rm loc})\) ,
\(u\in L^\infty((0,T_0)\times\mathbb R^3)\) and $\omega_0\in L^1(\mathbb{R}^3)$. The stretching term
$\omega\,\frac{u^r}{r}\,\chi_\varepsilon\psi$ 
converges by dominated convergence using
\(u\,\omega/r\in L^1_{\rm loc}([0,T_0)\times\mathbb R^3)\).
For the error term, since $\chi_\varepsilon'$ is supported in $\{\varepsilon\le r\le 2\varepsilon\}$,
\[
\left|
\int_0^{T_0}\!\!\int_{\mathbb R^3}\omega\,u^r\,\chi_\varepsilon'(r)\psi\,dx\,dt
\right|
\le C_{\psi}\|\omega\|_{L^\infty_{t,x}}\|u\|_{L^\infty_{t,x}}
\int_{\{\varepsilon\le r\le 2\varepsilon\}} \frac{r}{\varepsilon}\,dr\,d\theta\,dz\,dt
\le C_\psi \, \varepsilon\to 0.
\]
Therefore,
\begin{align*}
&\int_0^{T_0}\!\!\int_{\mathbb R^3}
\omega
\left(
\partial_t\psi
+u^r\partial_r\psi
+u^z\partial_z\psi
+\frac{u^r}{r}\psi
\right)\,dx\,dt
+\int_{\mathbb R^3}\omega_0\,\psi(\cdot,0)\,dx
=0,
\end{align*}
that is, the vorticity equation \eqref{vor_eq} holding in distributions on $(0,T_0)\times\mathbb R^3$.

Equivalently, the vector field $\boldsymbol{\omega}=\omega e_\theta$ satisfies
\[
\partial_t\boldsymbol{\omega}
+\div\bigl(u\otimes\boldsymbol{\omega}-\boldsymbol{\omega}\otimes u\bigr)=0
\qquad\text{in }\mathcal D'((0,T_0)\times\mathbb R^3).
\]
Indeed, the equivalence follows because one can average a general test field in the angular variable and reduce to smooth axisymmetric swirl test fields of the form $r\psi e_\theta$.

Since $\boldsymbol{\omega}=\mathrm{curl}\,  u$ and $\div u=0$ by construction of the Biot--Savart law,
\[
\mathrm{curl}\left(\partial_tu+\div(u\otimes u)\right)
=
\partial_t(\mathrm{curl} \ u)+\div\bigl(u\otimes \mathrm{curl}\ u-( \mathrm{curl}\ u)\otimes u\bigr)=0
\]
in distributions. Hence there exists
$p\in\mathcal D'((0,T_0)\times\mathbb R^3)$ such that
\[
\partial_t u+\div(u\otimes u)+\nabla p=0,
\qquad
\div u=0
\]
in $\mathcal D'((0,T_0)\times\mathbb R^3)$.

Next, by \eqref{eq:uniq_est_u},
\[
u\otimes u\in L^\infty\!\bigl(0,T_0;L^2(\mathbb R^3)\bigr),
\quad \text{hence} \quad
\div(u\otimes u)\in L^\infty\!\bigl(0,T_0;H^{-1}(\mathbb R^3)\bigr).
\]
Applying the Leray projector $\mathbb P$ to the Euler equation, we obtain
\[
\partial_tu=-\mathbb P\div(u\otimes u)
\in L^\infty\!\bigl(0,T_0;H^{-1}(\mathbb R^3)\bigr),
\quad \text{so} \quad 
u\in W^{1,\infty}\!\bigl(0,T_0;H^{-1}(\mathbb R^3)\bigr).
\]
Since $W^{1,\infty}\big(0,T_0;H^{-1}\big)\hookrightarrow C\big([0,T_0];H^{-1}\big)$, for any
$t,s\in[0,T_0]$,
\[
\|u(t)-u(s)\|_{L^2}^2
\le
\|u(t)-u(s)\|_{H^{-1}}
\|u(t)-u(s)\|_{H^1}
\le
2\|u\|_{L^\infty_tH^1}\|u(t)-u(s)\|_{H^{-1}},
\]
which tends to $0$ as $t\to s$. Thus
\[
u\in C\big([0,T_0];L^2(\mathbb R^3)\big).
\]

It remains to identify the initial trace. 
By the weak formulation for \(\xi=\omega/r\), we have
\[
\xi(t)\to \xi_0
\qquad\text{in }\mathcal D'(\mathbb R^3)
\quad\text{as }t\downarrow0.
\]
Since
\[
\boldsymbol\omega(t)=\omega(t)e_\theta
=\big(-x_2\xi(t),x_1\xi(t),0\big),
\]
it follows that
\[
\boldsymbol\omega(t)\to \boldsymbol\omega_0
:=\omega_0 e_\theta
\qquad\text{in }\mathcal D'(\mathbb R^3).
\]
Since $u\in C([0,T_0];L^2)$, we have $u(t)\to u(0)$ in $L^2$, hence in
$\mathcal D'(\mathbb R^3)$. Passing to the limit in
\[
\mathrm{curl}\, u(t)=\boldsymbol{\omega}(t),
\qquad
\div \, u(t)=0,
\]
we obtain
\[
\mathrm{curl}\, u(0)=\omega_0 e_\theta,
\qquad
\div\,  u(0)=0
\quad\text{in }\mathcal D'(\mathbb R^3).
\]
By construction, $u_0$ also satisfies
\[
\mathrm{curl}\, u_0=\omega_0 e_\theta,
\qquad
\div \, u_0=0
\quad\text{in }\mathcal D'(\mathbb R^3).
\]
Hence $w:=u(0)-u_0\in L^2(\mathbb R^3)$
satisfies
\[
\mathrm{curl}\, w=0,
\qquad
\div\, w=0
\quad\text{in }\mathcal D'(\mathbb R^3).
\]
Then by the Helmholtz decomposition, it follows that $w=0$, i.e.
\[
u|_{t=0}=u_0
\qquad\text{in }L^2(\mathbb R^3).
\]
\end{proof}

\medskip
\noindent\textbf{Uniqueness.}
To prove preservation of odd-symmetry and non-positivity in the upper half space, we need uniqueness   beyond Danchin's critical regime for Yudovich-type weak solutions $(u,\omega)$, which satisfy $\omega \in L^\infty((0,T_0); L^1\cap L^\infty (\mathbb{R}^3)) $. The proof adapts the classical strategy of Yudovich's uniqueness result for two-dimensional flows (see \cite[Section 3]{Yudovich63} and \cite[Theorem 8.2]{MajdaBertozzi01}); see also \cite[Theorem 1.2]{Danchin07Note}. Although the argument is a straightforward generalization of the 2D case, we provide a precise statement and full proof for the sake of completeness. In the proof, we will use a primitive-variable formulation proved in Lemma~\ref{lem:weak_to_euler}.
\begin{lemma}[Uniqueness]
\label{lem:uniq_bd}
Let $\omega_0=\omega_0(r,z)\in L^1\cap L^\infty(\mathbb R^3)$, and assume $\omega_0/r\in L^{1}_{\mathrm{loc}}(\mathbb{R}^3)$.
For $T_0>0$, let
$(u_1,\omega_1),(u_2,\omega_2)$ be Yudovich-type weak solutions of \eqref{eq:IVP_rel_vor} on
$(0,T_0)$ with the same initial data $\omega_0$ in the sense of Definition~\ref{def:dist_sol_rel_vor}. 
 Then
\[
u_1\equiv u_2,
\qquad\text{hence}\qquad
\omega_1\equiv \omega_2
\quad\text{on }[0,T_0]\times\mathbb R^3.
\]
\end{lemma}

\begin{proof}
Set
\[
A_2:=1+\sum_{i=1}^2
\|\omega_i\|_{L^\infty((0,T_0);L^1\cap L^\infty(\mathbb R^3))}.
\]
By Lemma~\ref{lem:weak_to_euler}, for $i=1,2$,
\[
u_i\in L^\infty\!\bigl(0,T_0;H^1\cap L^\infty(\mathbb R^3)\bigr)
\cap W^{1,\infty}\!\bigl(0,T_0;H^{-1}(\mathbb R^3)\bigr)
\subset C\bigl([0,T_0];L^2(\mathbb R^3)\bigr),
\]
with $u_i|_{t=0}=u_0$. Moreover, by \eqref{eq:uniq_est_u} and
\eqref{eq:uniq_est_grad_u},
\[
\|u_i(t)\|_{L^2\cap L^\infty(\mathbb R^3)}\le C A_2,
\qquad
\|\nabla u_i(t)\|_{L^p(\mathbb R^3)}\le C p\,A_2,
\qquad p\in[2,\infty)
\]
for some constant $C>0$ that is independent of $p$.

Let
\[
v:=u_1-u_2.
\]
Then
\[
v\in L^\infty\!\bigl(0,T_0;H^1(\mathbb R^3)\bigr)
\cap W^{1,\infty}\!\bigl(0,T_0;H^{-1}(\mathbb R^3)\bigr)
\subset C\bigl([0,T_0];L^2(\mathbb R^3)\bigr),
\qquad
v(0)=0,
\]
and, subtracting the two Euler equations and applying the Leray projector,
\[
\partial_t v
=
-\mathbb P\bigl((u_2\cdot\nabla)v+(v\cdot\nabla)u_1\bigr)
\quad\text{in }L^\infty\!\bigl(0,T_0;H^{-1}(\mathbb R^3)\bigr).
\]

Set
\[
E(t):=\|v(t)\|_{L^2(\mathbb R^3)}^2.
\]
Since $v\in L^\infty_tH^1_x\cap W^{1,\infty}_tH^{-1}_x$, the Lions--Magenes lemma yields
that $E$ is absolutely continuous on $[0,T_0]$ and
\[
\frac12 E'(t)=\langle \partial_t v(t),v(t)\rangle_{H^{-1},H^1}
\qquad\text{for a.e. }t\in[0,T_0].
\]
Because $\div\, v=0$, we have $\mathbb P v=v$, and therefore
\[
\langle \mathbb P f,v\rangle_{H^{-1},H^1}
=
\langle f,\mathbb P v\rangle_{H^{-1},H^1}
=
\langle f,v\rangle_{H^{-1},H^1}.
\]
Thus
\begin{equation}
\label{eq:uniq_L2_identity}
\frac12 E'(t)
=
-\big\langle (u_2\cdot\nabla)v(t),v(t)\big\rangle_{H^{-1},H^1}
-\big\langle (v\cdot\nabla)u_1(t),v(t)\big\rangle_{H^{-1},H^1}
\end{equation}
for a.e. $t$.

For a.e. fixed $t$, define
\[
B_t(a,b):=\int_{\mathbb R^3}(u_2(t)\cdot\nabla a)\cdot b\,dx,
\qquad a,b\in H^1(\mathbb R^3).
\]
Since $u_2(t)\in L^\infty(\mathbb R^3)$, $B_t$ is continuous on
$H^1(\mathbb R^3)\times H^1(\mathbb R^3)$. If
$a\in C_c^\infty(\mathbb R^3)$, then
\[
B_t(a,a)
=
\frac12\int_{\mathbb R^3}u_2(t)\cdot\nabla \left( |a|^2 \right)\,dx
=0
\]
because $\div u_2(t)=0$. By density, $B_t(a,a)=0$ for every
$a\in H^1(\mathbb R^3)$, hence
\[
\big\langle (u_2\cdot\nabla)v(t),v(t)\big\rangle_{H^{-1},H^1}=0
\qquad\text{for a.e. }t.
\]
Also,
\[
(v\cdot\nabla)u_1\in L^\infty\!\bigl(0,T_0;L^2(\mathbb R^3)\bigr),
\]
since $v\in L^\infty_{t,x}$ and $\nabla u_1\in L^\infty_tL^2_x$.
Therefore \eqref{eq:uniq_L2_identity} becomes
\[
\frac12 E'(t)
=
-\int_{\mathbb R^3}(v\cdot\nabla)u_1\cdot v\,dx
\qquad\text{for a.e. }t\in[0,T_0].
\]

Let $p\in[2,\infty)$. By H\"older's inequality, interpolation, and
\eqref{eq:uniq_est_grad_u},
\begin{align*}
\left|
\int_{\mathbb R^3}(v\cdot\nabla)u_1\cdot v\,dx
\right|
&\le
\|\nabla u_1\|_{L^p}
\|v\|_{L^{\frac{2p}{p-1}}}^2
\le
C p\,A_2\,
\|v\|_{L^2}^{2(1-\frac1p)}
\|v\|_{L^\infty}^{\frac2p}
\end{align*}
for some constant $C>0$ that is independent of $p$.
Using \eqref{eq:uniq_est_u},
\[
\|v(t)\|_{L^\infty}
\le
\|u_1(t)\|_{L^\infty}+\|u_2(t)\|_{L^\infty}
\le CA_2,
\]
and therefore
\[
E'(t)\le C p\,A_2^2\,E(t)^{1-\frac1p}
\qquad\text{for a.e. }t\in[0,T_0].
\]

Fix $\varepsilon>0$. Then, for a.e. $t\in[0,T_0]$,
\[
\frac{d}{dt}(E(t)+\varepsilon)^{1/p}
=
\frac1p(E(t)+\varepsilon)^{\frac1p-1}E'(t)
\le
C A_2^2
\left(\frac{E(t)}{E(t)+\varepsilon}\right)^{1-\frac1p}
\le C A_2^2.
\]
Since $E(0)=0$, integration gives
\[
(E(t)+\varepsilon)^{1/p}\le \varepsilon^{1/p}+C A_2^2 t.
\]
Letting $\varepsilon\downarrow0$, we obtain
\[
E(t)\le (C A_2^2 t)^p
\qquad\text{for every }p\in[2,\infty).
\]
Hence, if $t<(CA_2^2)^{-1}$, letting $p\to\infty$ yields $E(t)=0$. Therefore
\[
v\equiv 0
\qquad\text{on }\bigl[0,(CA_2^2)^{-1}\bigr].
\]

Since $v\in C([0,T_0];L^2)$, we still have
$v((CA_2^2)^{-1})=0$. Restarting the same argument finitely many times, we conclude that
\[
v\equiv 0
\qquad\text{on }[0,T_0]\times\mathbb R^3.
\]
Thus $u_1\equiv u_2$. Taking curl, we obtain
\[
\omega_1 e_\theta=\mathrm{curl}\, u_1=\mathrm{curl}\, u_2=\omega_2 e_\theta,
\]
hence $\omega_1=\omega_2$ a.e. on $(0,T_0)\times \mathbb{R}^3$.
\end{proof}

\medskip
\noindent \textbf{Flow map and Cauchy formula under bounded vorticity.}
Next, we prove the existence of a flow map and the Cauchy formula under bounded vorticity with a velocity field of finite energy. 
This lemma implies that as we consider bounded vorticity, it yields $\nabla u \in BMO$, hence $u$ is log-Lipschitz. Then the Lagrangian transport of $\xi=\omega/r$ is available.

We will use this lemma for both (1) and (2) of Theorem~\ref{thm:instant_blow-up}.
\begin{lemma}[Flow map and Cauchy formula under bounded vorticity]
\label{lem:flow_cauchy_bd}
Let $T_0>0$. Let $u$ be divergence-free and axisymmetric without swirl on $(0,T_0)\times\mathbb R^3$, and
let $\omega_0=\omega_0(r,z), \omega=\omega(r,z,t)$ be axisymmetric. Set
\[
\xi:=\frac{\omega}{r},\qquad \xi_0:=\frac{\omega_0}{r}.
\]
Assume that
\[
u\in L^\infty\bigl(0,T_0;L^2(\mathbb R^3)\bigr),
\qquad
\omega\in L^\infty\bigl((0,T_0)\times\mathbb R^3\bigr),
\]
that
\[
\operatorname{curl}u=\omega e_\theta,
\qquad
\operatorname{div}u=0
\]
in $\mathcal D'((0,T_0)\times\mathbb R^3)$, and that $\xi$ satisfies the distributional
transport identity \eqref{eq:transport_distributional} with initial data $\xi_0$,
where $\xi_0\in L^1(\mathbb R^3)$.

Then the following hold.

\smallskip
\noindent
(i) There exists a constant $C>0$ such that
\begin{equation}
\label{eq:u_log_Lip_common}
|u(x,t)-u(y,t)|
\le
C |x-y|\log\!\left(2+\frac{1}{|x-y|}\right)
\end{equation}
for all $x,y\in\mathbb R^3$ and for a.e. $t\in(0,T_0)$. In addition, $u\in L^\infty\bigl((0,T_0)\times\mathbb R^3\bigr).$

\smallskip
\noindent
(ii) After modifying $u$ on a null set in time, the ODE
\[
\partial_t Y(a,t)=u(Y(a,t),t),
\qquad
Y(a,0)=a
\]
admits a unique flow map
\[
Y:\mathbb R^3\times[0,T_0]\to\mathbb R^3.
\]
For each $a\in\mathbb R^3$, the map $t\mapsto Y(a,t)$ is absolutely continuous,
whereas $(a,t)\mapsto Y(a,t)$ is continuous. Moreover, for each
$t\in[0,T_0]$, the map $Y_t:=Y(\cdot,t)$ is a homeomorphism of $\mathbb R^3$
onto itself and is measure-preserving. Furthermore, 
there exists $\gamma\in(0,1]$ such that 
for every compact set
$K\subset\mathbb R^3$,  there exists a constant $C_{K,T_0}>0$ such
that for all $
x_1,x_2\in K$ and $t\in[0,T_0]$,
\begin{align}
\label{eq:Y_reg_x_common}
\begin{aligned}
|Y_t(x_1)-Y_t(x_2)|
+
|Y_t^{-1}(x_1)-Y_t^{-1}(x_2)|
\le
C_{K,T_0} |x_1-x_2|^\gamma.
\end{aligned}
\end{align}

\smallskip
\noindent
(iii) The relative vorticity is transported by the flow:
\begin{equation}
\label{eq:cauchy_xi_common}
\xi(Y(a,t),t)=\xi_0(a)
\qquad
\text{for a.e. }(a,t)\in\mathbb R^3\times(0,T_0).
\end{equation}
Equivalently,
\begin{equation}
\label{eq:cauchy_omega_common}
\frac{\omega(Y(a,t),t)}{Y^r(a,t)}
=
\frac{\omega_0(a)}{r_a}
\qquad
\text{for a.e. }(a,t)\in\mathbb R^3\times(0,T_0).
\end{equation}

\end{lemma}

\begin{proof}[Proof of Lemma~\ref{lem:flow_cauchy_bd}]
For a.e. fixed \(t\), set \(\Omega(t):=\omega(t)e_\theta\). Since
\[
\operatorname{curl}u(t)=\Omega(t),
\qquad
\operatorname{div}u(t)=0
\]
in \(\mathcal D'(\mathbb R^3)\), we have
\[
-\Delta u(t)=\operatorname{curl}\Omega(t).
\]
The finite-energy condition \(u(t)\in L^2(\mathbb R^3)\) removes the possible
harmonic ambiguity in this div--curl inversion. Hence, for every \(i,j\),
\[
\partial_j u_i(t)
=
\sum_{\ell,m=1}^3
\varepsilon_{i\ell m} R_jR_\ell \Omega_m(t)
\qquad\text{in }\mathcal D'(\mathbb R^3),
\]
where \(R_j\) denotes the \(j\)-th Riesz transform. Since the double Riesz
transforms are Calder\'on--Zygmund operators and are bounded from
\(L^\infty(\mathbb R^3)\) to \(BMO(\mathbb R^3)\), it follows that
\begin{align}
\label{eq:BMO_grad_u_common}
[\nabla u(t)]_{BMO(\mathbb R^3)}
\le C\|\Omega(t)\|_{L^\infty(\mathbb R^3)}
=
C\|\omega(t)\|_{L^\infty(\mathbb R^3)}.
\end{align}

On the other hand, standard local div-curl estimates imply that for every
$p\in[1,\infty)$,
\begin{equation}
\label{eq:local_W1p_u_common}
\sup_{x_0\in\mathbb R^3}
\|u(t)\|_{W^{1,p}(B_1(x_0))}
\le
C(p)\Bigl(
\|u(t)\|_{L^2(\mathbb R^3)}
+\|\omega(t)\|_{L^\infty(\mathbb R^3)}
\Bigr)
\end{equation}
for a.e. $t\in(0,T_0)$. Choosing $p>3$ and using Sobolev embedding on unit balls, we obtain
\[
\|u(t)\|_{L^\infty(\mathbb R^3)}
\le
C(p)\Bigl(
\|u(t)\|_{L^2(\mathbb R^3)}
+\|\omega(t)\|_{L^\infty(\mathbb R^3)}
\Bigr)
\]
for a.e. $t\in(0,T_0)$, and therefore
\[
u\in L^\infty\bigl((0,T_0)\times\mathbb R^3\bigr).
\]
Also, \eqref{eq:local_W1p_u_common} with $p=1$ gives
\[
\sup_{x_0\in\mathbb R^3}\int_{B_1(x_0)}|\nabla u(t,x)|\,dx
\le
C\Bigl(
\|u(t)\|_{L^2(\mathbb R^3)}
+\|\omega(t)\|_{L^\infty(\mathbb R^3)}
\Bigr).
\]

Combining this bound with \eqref{eq:BMO_grad_u_common}, Lemma~1 in
\cite{AzzamBedrossian15} yields
\[
|u(x,t)-u(y,t)|
\le
C |x-y|\,|\log |x-y||
\qquad
\text{for }|x-y|\le \frac12,
\]
uniformly for a.e. $t\in(0,T_0)$. Enlarging the constant and using the
$L^\infty$ bound for $u$, we obtain \eqref{eq:u_log_Lip_common} for all
$x,y\in\mathbb R^3$ and a.e. $t\in(0,T_0)$. In particular,
\[
u\in L^1\bigl(0,T_0;LL(\mathbb R^3)\bigr)
\]
in the sense of \cite[Définition~5.2.1]{Chemin95}.

Hence, after modifying $u$ on a null set in time, \cite[Th\'eor\`eme~5.2.1]{Chemin95}
implies that $u$ admits a unique flow map $Y$.
For each $a\in\mathbb R^3$, the map $t\mapsto Y(a,t)$ is absolutely continuous,
whereas $(a,t)\mapsto Y(a,t)$ is continuous. Moreover, for each
$t\in[0,T_0]$, the map $Y_t:=Y(\cdot,t)$ is a homeomorphism of $\mathbb R^3$
onto itself. Since $\operatorname{div}u=0$, the flow is measure-preserving.
Finally,  the local H\"older estimate
\eqref{eq:Y_reg_x_common} holds on compact sets for $Y_t$. 

Then using the H\"older estimate for $Y_t$, one can also prove the H\"older estimate for the inverse $Y_t^{-1}$ by considering the reversed vector field and applying Chemin's theorem again.
See also \cite[Lemma~8.2]{MajdaBertozzi01} for another related reference for this matter.

Applying Proposition~\ref{prop:transport_by_flow} (which will be proved in Appendix~\ref{app:two_lemmas} below) to the distributional solution
$\xi$ with the flow map $Y$ from part~(ii), we obtain
\[
\xi(x,t)=\xi_0(Y_t^{-1}(x))
\quad\text{for a.e. }(x,t)\in\mathbb R^3\times(0,T_0),
\]
hence
\[
\xi(Y(a,t),t)=\xi_0(a)
\quad\text{for a.e. }(a,t)\in\mathbb R^3\times(0,T_0).
\]
Since $\omega=r\xi$, this gives
\[
\frac{\omega(Y(a,t),t)}{Y^r(a,t)}=\frac{\omega_0(a)}{r_a}
\quad\text{for a.e. }(a,t)\in\mathbb R^3\times(0,T_0).
\]
This completes the proof.
\end{proof}

\medskip
\noindent\textbf{Danskin-Type Lemma.}
Next, we record a Danskin-type lemma, which will be used to handle rings closer to the origin. This lemma is  likely known in related forms, but as we could not find a statement that exactly fits our situation, we include it here for completeness.
This is an absolutely-continuous-in-time envelope lemma tailored to our Lagrangian radial extrema, rather than a direct invocation  of the classical smooth Danskin theorem (See \cite{Danskin67}, \cite{BernhardRapaport95}).

\begin{lemma}[Danskin-type lemma]\label{lem:envelope_radial_extrema}
Fix $T_0>0$, and let $K\subset \mathbb{R}^N$ be compact, and let
\[
Y^r:K\times [0,T_0]\to \mathbb{R}
\]
be continuous. Assume that there exists a bounded function
\[
g:K\times (0,T_0)\to \mathbb{R}
\] measurable in $t$ for each fixed $x\in K$, such that, for every
$x\in K$ and every $t\in [0,T_0]$,
\[
Y^r(x,t)=Y^r(x,0)+\int_0^t g(x,s)\,ds.
\]
Assume furthermore that there exists a full-measure set $E_g\subset(0,T_0)$ and  a modulus of continuity $\rho:[0,\infty)\to [0,\infty)$  with $\lim_{s\to0+} \rho(s)=0$, such that
\begin{align}
\label{eq:mod_cont}
|g(x,t)-g(y,t)|\le \rho(|x-y|)
\qquad
\text{for all }x,y\in K,\ t\in E_g.
\end{align}

Define
\[
R_K(t):=\max_{x\in K} Y^r(x,t),
\qquad
r_K(t):=\min_{x\in K} Y^r(x,t),
\]
and
\[
\mathcal M_K(t):=\operatorname*{argmax}_{x\in K} Y^r(x,t),
\qquad
\mathcal A_K(t):=\operatorname*{argmin}_{x\in K} Y^r(x,t).
\]
The definition of $\operatorname*{argmax}_{x\in K} Y^r(x,t)$ is as follows:
    \begin{align*}
        \operatorname*{argmax}_{x\in K} Y^r(x,t)
    =
        \{y\in K: Y^r(y,t)= \max_{x\in K}Y^r(x,t)\},
    \end{align*}
and we similarly define $\operatorname{argmin}_{x\in K} Y^r(x,t)$.
Then $R_K$ and $r_K$ are Lipschitz on $[0,T_0]$. Moreover, there exists a full-measure set $E\subset E_g$ such that for every $t\in E$
\[
R_K'(t)=\max_{x\in \mathcal M_K(t)} g(x,t),
\qquad
r_K'(t)=\min_{x\in \mathcal A_K(t)} g(x,t).
\]
In particular, these identities hold for a.e. $t\in (0,T_0)$.
\end{lemma}
A proof of this Lemma~\ref{lem:envelope_radial_extrema} is written in Appendix~\ref{app:two_lemmas}.

\subsection{Proof of Theorem~\ref{thm:instant_blow-up}-(1)}
\label{subsec:non_existence}
We are now ready to prove Theorem~\ref{thm:instant_blow-up}-(1). 
Fix $q>1$ and let $\tilde\varepsilon>0$ be given. We fix once and for all the parameters
$L_q,\alpha_q,\phi_q$ and the geometric constants $r_0,\mu,\eta,d$ exactly as in the proof of
Theorem~\ref{main_thm1} (see Section~\ref{sec:ini_data_ansatz}). For a parameter $\varepsilon>0$, let $\omega_0^{(\infty)}$ be the initial data obtained from \eqref{KJdata2}
with infinitely many vortex rings, i.e. $m=\infty$.

The smallness in Theorem~\ref{thm:instant_blow-up}-(1) follows from \eqref{ini_lorentz14}: since $\alpha_q q>1$, the series
$\sum_{k\ge1}k^{-\alpha_q q}$ converges, hence there exists $\varepsilon=\varepsilon(\tilde\varepsilon,q)$ such that
\[
\left\|\omega_0^{(\infty)}\right\|_{L^\infty\cap L^{1}(\R^3)}+\left\|\frac{\omega_0^{(\infty)}}{r}\right\|_{L^{3,q}(\R^3)}\le \tilde\varepsilon.
\]

Assume for contradiction there exist $T_0>0$ and a Yudovich-type weak solution $(u^{(\infty)},\omega^{(\infty)})$
to \eqref{eq:IVP_rel_vor} on $[0,T_0]$ with initial data $\omega_0^{(\infty)}$ in the sense of Definition~\ref{def:dist_sol_rel_vor}.
Recall a Yudovich-type weak solution satisfies, by definition,
\[
\omega^{(\infty)}\in L^\infty([0,T_0];L^\infty\cap L^1(\R^3)).
\]
Set $A_\infty:=\|\omega^{(\infty)}\|_{L^\infty([0,T_0]\times\R^3)}<\infty$. Let
\[
A:=2\max\{A_\infty,\varepsilon(\tilde\varepsilon,q)\},\qquad \bar A:=\frac{A}{1-\mu}.
\]

Since $\omega^{(\infty)} \in L^\infty(0,T_0; L^\infty\cap L^1(\mathbb{R}^3))$, it holds $u^{(\infty)}\in L^\infty (0,T_0; L^2(\mathbb{R}^3))$ by Lemma~\ref{lem:weak_to_euler}. Hence we can apply Lemma~\ref{lem:flow_cauchy_bd} to obtain a unique flow map $Y$ and the Cauchy formula~\eqref{eq:cauchy_omega_common}.

We claim that 
\begin{align}
\label{eq:odd_sign}
    \begin{cases}
        \text{$\omega^{(\infty)}(r,z,t)=-\omega^{(\infty)}(r,-z,t)$ almost everywhere,  and }
        \\
        \text{$\omega^{(\infty)}(r,z,t)\leq 0$ almost everywhere in $\{z>0\}$. }
    \end{cases}
\end{align}
\begin{proof}[Proof of \eqref{eq:odd_sign}]
Let
\[
\bar \omega^{(\infty)}(r,z,t):=-\omega^{(\infty)}(r,-z,t),
\]
and let $\bar u^{(\infty)}$ be the velocity induced by $\bar \omega^{(\infty)}$ through the
Biot--Savart law. By symmetry of \eqref{eq:rel_vor_eq} and \eqref{eq:axis_BS_law} under
$z\mapsto -z$, the pair $(\bar u^{(\infty)},\bar \omega^{(\infty)})$ is also an
(axisymmetric no-swirl) Yudovich-type weak solution with the same initial data
$\omega_0^{(\infty)}$. Hence Lemma~\ref{lem:uniq_bd} gives
\[
\bar \omega^{(\infty)}\equiv \omega^{(\infty)},
\qquad
\bar u^{(\infty)}\equiv u^{(\infty)}.
\]
Thus $\omega^{(\infty)}(r,z,t)=-\omega^{(\infty)}(r,-z,t)$ almost everywhere, and
\[
u^{(\infty),r}(r,z,t)=u^{(\infty),r}(r,-z,t),\qquad
u^{(\infty),z}(r,z,t)=-u^{(\infty),z}(r,-z,t).
\]
Hence the flow map associated with $u^{(\infty)}$ preserves the upper and lower half-spaces.
Since $\omega^{(\infty)}/r$ is transported by the flow and
$\omega_0^{(\infty)}\le 0$ on $\{z>0\}$, it follows that
$\omega^{(\infty)}(r,z,t)\le 0$ almost everywhere in $\{z>0\}$.
\end{proof}

We decompose $\omega_0^{(\infty)}=\sum_{k=1}^\infty \omega_{0,k}$ as in \eqref{KJdata2}. 
For each $k\geq 1$, \emph{define} $\omega_k(x,t)$ by 
    \begin{align*}
        \omega_k(x,t):=
        r_x
        \left(
        \frac{\omega_{0,k}}{r}
        \right)\big( Y^{-1}(x,t)\big) , \quad x=(r_x,z_x).
    \end{align*}
Hence the Cauchy formula for $\omega_k/r$ holds \emph{by definition}.
Then 
$\omega^{(\infty)}(r,z,t)=\sum_{k=1}^\infty\omega_k(r,z,t)$ for a.e. $(r,z)$ and $t$. We define $u_k(r,z,t)$ by using $\omega_k$ through the Biot--Savart law, and for $k \geq 2$, we define $u_-:=\sum_{j=1}^{k-1} u_j,\,  u_+:=\sum_{j=k+1}^\infty u_j.$ Then $u^{(\infty)}=u_-+u_k +u_+$ for every $(r,z)$ and $t$.

Next, we will use the existence of the flow map $Y$ to prove the following claim:
    \begin{align}
    \label{claim_reg_outer}
    \begin{cases}
        \text{ $\exists \rho_k>0$ s.t. for all $t\in[0,T_0],$ the velocity $u_-(\cdot,t)$ is $C^1$ on $B_{\rho_k} (0) \subset \mathbb{R}^3$.}
    \\
       \text{the map: $t\mapsto \nabla u_-(0,0,t)$ is continuous on $[0,T_0]$.}  
    \\
        \text{In particular, the map: $t\mapsto \partial_r u^r_-(0,0,t)$ is continuous on $[0,T_0]$.} 
    \end{cases}
    \end{align}

\begin{proof}[Proof of \eqref{claim_reg_outer}]

For each fixed $t\in[0,T_0]$, the flow map $Y_t:=Y(\cdot,t)$ is a homeomorphism of
$\mathbb R^3$ with inverse $Y_t^{-1}=Y^{-1}(\cdot,t)$. Since $u^{(\infty)}(0,0,t)=0$, the origin is a
fixed trajectory, so $Y_t(0)=0$ and hence $Y_t^{-1}(0)=0$. Therefore, for each $1\le j\le k-1$, because
$\supp \omega_{0,j}$ is compact and does not contain the origin, its image
$Y_t(\supp \omega_{0,j})$ also does not contain the origin for any $t\in[0,T_0]$.
Since $(a,t)\mapsto Y(a,t)$ is continuous on the compact set
$\supp \omega_{0,j}\times[0,T_0]$, the set
\[
K_j:=\bigcup_{t\in[0,T_0]}Y(\supp \omega_{0,j},t)
\]
is a compact subset of $\mathbb R^3\setminus\{0\}$.
 Hence,
\[
\rho_k:=\frac12\,\operatorname{dist}\!\left(0,\bigcup_{j=1}^{k-1}K_j\right)>0.
\]

Fix $t\in[0,T_0]$. For $x\in B_{\rho_k}(0)$ and $1\le j\le k-1$, the distance from $x$ to
$\supp\omega_j(\cdot,t)$ is at least $\rho_k$. Hence the Biot--Savart kernel and its first
$x$-derivatives are smooth and bounded on
\[
B_{\rho_k}(0)\times \left(\bigcup_{j=1}^{k-1}K_j\right).
\]
Differentiating under the integral sign in the Biot--Savart formula therefore gives that each
$u_j(\cdot,t)$ is $C^1$ on $B_{\rho_k}(0)$, and hence so is $u_-(\cdot,t)$.

To prove continuity in time at the origin, let $\mathcal{K}(y)$ denote the matrix-valued kernel
appearing in the formula for $\nabla u(0,t)$ after differentiating the Biot--Savart law in the
evaluation variable. Since $\mathcal{K}$ is smooth on $\mathbb{R}^3\setminus\{0\}$, for each
$1\le j\le k-1$ we may write
\[
\nabla u_j(0,t)=\int_{\mathbb{R}^3}\mathcal{K}(y)\,\omega_j(y,t)\,dy.
\]
Using volume preservation of $Y$ and the definition of $\omega_j$, this becomes
\[
\nabla u_j(0,t)
=
\int_{\supp\omega_{0,j}}
\mathcal{K}(Y(a,t))\,
\frac{Y^r(a,t)}{r_a}\,
\omega_{0,j}(a)\,da,
\]
where $r_a$ denotes the radial coordinate of $a$.

For $a\in\supp\omega_{0,j}$ and $t\in[0,T_0]$, the point $Y(a,t)$ stays in the compact set
$K_j\subset \mathbb{R}^3\setminus\{0\}$. Thus the integrand is continuous in $t$ and uniformly
bounded on $\supp\omega_{0,j}\times[0,T_0]$. By dominated convergence,
$t\mapsto \nabla u_j(0,t)$ is continuous. Summing over $j=1,\dots,k-1$ yields continuity of
$t\mapsto \nabla u_-(0,t)$. By axisymmetry, this implies continuity of
$t\mapsto \partial_r u_-^r(0,0,t)$. It finishes the proof of \eqref{claim_reg_outer}.
\end{proof}

For every $k\ge 2$, define $x_k(t)$ by
\[
x_k'(t)=x_k(t)\,\partial_r u_-^r(0,0,t),
\qquad
x_k(0)=\frac{\varepsilon}{k^{\alpha_q}},
\]
and set
\[
x_1(t)\equiv \varepsilon.
\]
Define $\tilde x_k$, $R_k$, and $H_k$ from $x_k$ by \eqref{tilde_xk_rk_hk}, and define
$W_k$ by \eqref{eq:def_wk94}. Also define $V_k$ by \eqref{def_v_k14}.

In the present contradiction setting, we define the rescaled flow map $X_k$ directly from
the global flow map $Y$. For $(r,z)\in \supp \phi_q$, set
\[
X_k(r,z,t):=
\left(
\frac{Y^r(d^k r,d^k z,t)}{R_k(t)},
\frac{Y^z(d^k r,d^k z,t)}{H_k(t)}
\right).
\]
Since $\tilde x_k(0)=1$, we have $R_k(0)=H_k(0)=d^k$, hence
\[
X_k(r,z,0)=(r,z).
\]
Moreover, since $Y$ is continuous and $R_k,H_k$ are continuous and strictly positive on
$[0,T_0]$, the map $(r,z,t)\mapsto X_k(r,z,t)$ is continuous on
$\supp \phi_q\times [0,T_0]$.

Just by using the definition of $X_k$ and $W_k$, we claim that the Cauchy formula \eqref{cauchy_for_W} holds pointwise in the present contradiction setting as well:
\begin{align}
\label{eq:Cauchy_contradiction}
\frac{W_k(X_k(r,z,t),t)}{X_k^r(r,z,t)}
=
\frac{\phi_q(r,z)}{r},
\qquad
(r,z)\in \supp \phi_q,\ \ t\in[0,T_0].
\end{align}

\begin{proof}[Proof of the Cauchy formula~\eqref{eq:Cauchy_contradiction}]
By \eqref{tilde_xk_rk_hk},
\[
\frac{R_k'(t)}{R_k(t)}=\frac{x_k'(t)}{x_k(t)},
\qquad
\frac{H_k'(t)}{H_k(t)}=-2\frac{x_k'(t)}{x_k(t)}.
\]
Therefore, differentiating the definition of $X_k$ in time (for a.e. $t$), we obtain
\[
\partial_t X_k^r(r,z,t)
=
\frac{u^r(R_k(t)X_k^r(r,z,t),\,H_k(t)X_k^z(r,z,t),\,t)}{R_k(t)}
-
X_k^r(r,z,t)\frac{x_k'(t)}{x_k(t)}
=
V_k^r(X_k(r,z,t),t),
\]
and similarly
\[
\partial_t X_k^z(r,z,t)
=
\frac{u^z(R_k(t)X_k^r(r,z,t),\,H_k(t)X_k^z(r,z,t),\,t)}{H_k(t)}
+
2X_k^z(r,z,t)\frac{x_k'(t)}{x_k(t)}
=
V_k^z(X_k(r,z,t),t).
\]
Hence $X_k$ is precisely the flow map associated with $V_k$.

Finally, fix $(r,z)\in \supp \phi_q$ and set
\[
a:=(d^k r,d^k z).
\]
By the definition of $\omega_k$ through the global flow $Y$,
\[
\omega_k(Y(a,t),t)
=
Y^r(a,t)\frac{\omega_{0,k}(a)}{r_a}.
\]
Since
\[
\omega_{0,k}(a)=x_k(0)\phi_q(r,z),
\qquad
r_a=d^k r,
\qquad
Y^r(a,t)=R_k(t)X_k^r(r,z,t)
=
d^k\frac{x_k(t)}{x_k(0)}X_k^r(r,z,t),
\]
we obtain
\[
\omega_k(Y(a,t),t)
=
x_k(t)\,X_k^r(r,z,t)\,\frac{\phi_q(r,z)}{r}.
\]
On the other hand, by the definition of $W_k$,
\[
\omega_k(Y(a,t),t)
=
\omega_k(R_k(t)X_k^r(r,z,t),\,H_k(t)X_k^z(r,z,t),\,t)
=
x_k(t)\,W_k(X_k(r,z,t),t).
\]
Therefore, one can obtain the pointwise Cauchy formula~\eqref{eq:Cauchy_contradiction}, and it finishes the proof of~\eqref{eq:Cauchy_contradiction}.
\end{proof}

\medskip
\noindent\textbf{Hierarchy of index thresholds.}
Throughout the proof, we introduce several integer thresholds depending on $(A,q)$:
\[
m_0(A,q), \qquad m_{\mathrm{cone}}(A,q), \qquad m_{\mathrm{tail}}(A,q), \qquad m_1(A,q), \qquad m_2(A,q).
\]
For clarity, we summarize their roles and dependencies.

\begin{itemize}
\item[(i)] $m_0(A,q)$ ensures scale separation for the first $M$ rings (Lemma~\ref{lem:Euler_scale_separation}). In particular, for all $M \ge m_0(A,q)$, the supports $\{\supp \omega_k\}_{k=1}^M$ are pairwise disjoint and satisfy the geometric separation bounds \eqref{eq:sep_R_ratio}--\eqref{eq:sep_distance_estimate}. See \eqref{eq:def_m0_blow_up} and the paragraph following Proposition~\ref{prop:finite_head_reduction}.

\item[(ii)] $m_{\mathrm{cone}}(A,q) \ge m_0(A,q)$ guarantees that all trajectories starting from $\supp \omega_0^{(\infty)}$ remain in the cone region $L r \ge |z|$ on $[0,T_B(A,M)]$ for $M \ge m_{\mathrm{cone}}(A,q)$. This allows us to use Proposition~\ref{prop:cone_free_key_lem} uniformly in time with the factor $|x|/r_x$ bounded by $\sqrt{1+L^2}$. See \eqref{eq:def_m_cone}.

\item[(iii)] $m_{\mathrm{tail}}(A,q) \ge m_{\mathrm{cone}}(A,q)$ ensures strict separation between the head ($k \le M$) and the tail ($k > M$), namely
\[
\frac{r_{M,-}(t)}{R_{\mathrm{tail}}(t)} \ge 4
\quad \text{for all } t \in [0,T_B(A,M)],
\]
which yields the uniform distance bound \eqref{eq:tail_distance_bound} and allows us to control the tail contribution. See \eqref{eq:m_tail_def}.

\item[(iv)] $m_1(A,q) \ge m_{\mathrm{tail}}(A,q)$ is chosen large enough so that
\[
C_E \bar A \bigl(1+\log(\bar A/\varepsilon)+\log M\bigr) T_B(A,M) \ll 1,
\]
which enables the bootstrap improvement and yields $T_B(A,M)=T_N(A,M)$. See \eqref{eq:m_1_def2}.

\item[(v)] $m_2(A,q) \ge m_1(A,q)$ is chosen so that the norm-inflation time satisfies
\[
T_N(A,M) < T_0 
\]
for all $M \ge m_2(A,q)$. This ensures that the ODE norm inflation occurs within the prescribed time interval and allows us to evaluate the solution at $t = T_N(A,M)$.

\end{itemize}

In particular, for all $M \ge m_2(A,q)$, all the above properties hold simultaneously.

\medskip
\noindent{\bf Step 1: Finite--ring bootstrap  (tail handled by the contradiction hypothesis).}

For any integer $M\geq 2,$ define
$T_N(A,M)$ and $T_B(A,M)$ by 
    \begin{align*}
        T_N(A,M)
        &:= \sup \left\{t\in [0,T_0]: \max_{1\leq k \leq M} x_k(t) \leq \bar A\right\}
        \\
        T_B(A,M)
        &:= \sup \left\{T\in [0,T_N(A,M)]: F_M(t) \leq \mu  \text{ all t} \in [0,T] \right\}.
    \end{align*}
where the bootstrap functional
$F_M(t)$ is defined in  \eqref{boot:X}.
Since $M$ is finite and the function $(r,z,t)\mapsto X_k(r,z,t)$ is continuous, $F_M(t)$ is continuous and $F_M(0)=0$, hence $T_B(A,M)>0$. 

By the definition, $T_B(A,M)\leq T_N(A,M)$, hence $x_k(t) \leq \bar A$ on $[0,T_B(A,M)]$.

Before we state the main goal of Step 1, 
by using the bootstrap bounds on $[0,T_B(A,M)]$ and the fact that $x_k(t) \leq \bar A $ on $[0,T_B(A,M)]$, we prove confinement and scale-separation for the head rings in the same way as in the proof for norm inflation.
\begin{proposition}[Finite-head reduction for the first $M$ rings]
\label{prop:finite_head_reduction}
 The following statements hold on
$[0,T_B(A,M)]$ for the first $M$ rings.

\begin{enumerate}
\item[(i)] For every $1\le k\le M$, the Cauchy formula \eqref{eq:Cauchy_contradiction} holds and
\[
(1-\mu)\|\phi_q\|_{L^\infty(\mathbb{R}^3)}
\le
\|W_k(\cdot,\cdot,t)\|_{L^\infty(\mathbb{R}^3)}
\le
(1+\mu)\|\phi_q\|_{L^\infty(\mathbb{R}^3)},
\]
\[
\|\omega_k(\cdot,t)\|_{L^\infty(\mathbb{R}^3)}
\le
(1+\mu)x_k(t)\|\phi_q\|_{L^\infty(\mathbb{R}^3)}
\le
(1+\mu)\bar A\|\phi_q\|_{L^\infty(\mathbb{R}^3)}.
\]

\item[(ii)] For every $2\le k\le M$, the amplitudes satisfy the same weak ODE inequalities as in
\eqref{mu_ODE}. Consequently, repeating the proof of
Proposition~\ref{prop:norm_inflation_varcoef} with $T_B:=T_B(A,M)$ gives
\[
T_B(A,M)\le \frac{C(q,A)}{M^{\beta_q}}.
\]

\item[(iii)] The confinement and scale-separation conclusions of
Lemmas~\ref{lem:Euler_cauchy_geometry} and \ref{lem:Euler_scale_separation} hold for
$1\le k\le M$ on $[0,T_B(A,M)]$. In particular, if the following inequality holds:
\begin{equation}
\label{eq:def_m0_blow_up}
d\,e^{C_{\rm sep}\bar A T_B(A,M)}\le \frac{r_-}{2r_+}.
\end{equation}
 then
\eqref{eq:sep_R_ratio} and \eqref{eq:sep_distance_estimate} are valid for the first $M$
rings on $[0,T_B(A,M)]$.
\end{enumerate}
\end{proposition}

Since $T_B(A,M)\to0$ as $M\to\infty$, there exists a large integer $m_0(A,q)$
such that for all $M\ge m_0(A,q)$, the inequality \eqref{eq:def_m0_blow_up} holds.
Hence by Proposition~\ref{prop:finite_head_reduction}, \eqref{eq:sep_R_ratio}, \eqref{eq:sep_distance_estimate} hold for all $M\geq m_0(A,q)$.

\begin{proof}[Proof of Proposition~\ref{prop:finite_head_reduction}.]
The pointwise Cauchy formula \eqref{eq:Cauchy_contradiction} was established above directly from the
global flow map $Y$ and the definition of $\omega_k$; in particular, it remains valid in the
contradiction setting and does not rely on Lemma~\ref{lem:Euler_cauchy_geometry1}.
Combining \eqref{eq:Cauchy_contradiction} with the bootstrap bound $F_M(t)\le \mu$ gives the
$L^\infty$ bounds in part (i).

Next fix $2\le k\le M$.  The ODE
\[
x_k'(t)=x_k(t)\,\partial_r u_-^r(0,0,t)
\]
depends only on the rings with indices $j<k\le M$. Therefore the proofs of
Lemma~\ref{lem:pro_est} and the weak ODE bounds \eqref{mu_ODE} in
Section~\ref{sec:ODEs_Euler} apply without change to the family $\{\omega_j\}_{j=1}^M$ on
$[0,T_B(A,M)]$. Since $T_B(A,M)\leq T_N(A,M)$, we have $x_k(t)\le \bar A$ for all $1\le k\le M$
and $0\le t\le T_B(A,M)$. Thus, because it only uses the first $k-1$ rings and the finite-head bootstrap $F_M\leq \mu$, the proof of Proposition~\ref{prop:norm_inflation_varcoef}
yields the bound in (ii).

Finally, Lemma~\ref{lem:Euler_cauchy_geometry} depends only on the bootstrap control of
$X_k$, and Lemma~\ref{lem:Euler_scale_separation} uses
Lemma~\ref{lem:Euler_cauchy_geometry}, the ODE \eqref{EulerODE}, and the bound
$x_k(t)\le \bar A$ on $[0,T_B(A,M)]$. Hence the same proofs give (iii) for the first $M$ rings. It finishes the proof of Proposition~\ref{prop:finite_head_reduction}.
\end{proof}

The goal of Step 1 is to prove, using a bootstrap argument, that for large $M,$
\begin{equation}\label{eq:TB_ge_min_T0_TN}
T_B(A,M)=  T_N(A,M).
\end{equation}
Assume for contradiction that $T_B(A,M)< T_N(A,M).$

For a bootstrap argument, for large integers $M$, which are at least bigger than $m_0(A,q),$ we control the rescaled velocities $V_k=(V_k^r,V_k^z)$ (defined in \eqref{def_v_k14}) on $\supp W_k(\cdot,\cdot,t)$ for $1\leq k \leq M$ and $0\leq t\leq T_B(A,M)$.
Decompose the vorticity $\omega^{(\infty)}$ into \emph{head} ($k\leq M$) and \emph{tail} ($k>M$):
\[
\omega^{(\infty)}=\omega^{(\leq M)}+\omega^{(>M)},\qquad
\omega^{(\leq M)}:=\sum_{j=1}^M\omega_j,\quad \omega^{(>M)}:=\sum_{j=M+1}^\infty\omega_j,
\]
and decompose $u^{(\infty)}=u^{(\leq M)}+u^{(>M)}$ accordingly.
Let  $V_{k,(\leq M)}=(V_{k,(\leq M)}^r,V_{k,(\leq M)}^z)$ denote the rescaled velocity built from $u^{(\leq M)}$ by the same formula as in
\eqref{def_v_k14}. Then for $1\leq k \leq M$
    \begin{align*}
        V_k^r(r,z,t)
    =
        V_{k,(\leq M)}^r (r,z,t) +\frac{u^{(>M),r}(R_k(t) r, H_k(t) z, t)}{R_k(t)}
    \end{align*}
and we make a corresponding decomposition for $V_k^z$. 
We divide estimates of $V_k$ into three parts: (a) head contribution $V_{k,(\leq M)}$; (b) tail contribution $u^{(>M)}$; (c) Conclusion from (a) and (b).

\medskip
\noindent \emph{Step 1-(a): Estimate the head contribution $V_{k,(\leq M)}^r, V_{k,(\leq M)}^z$}

For the head vorticity $\omega^{(\leq M)}$, using Proposition~\ref{prop:finite_head_reduction}(i), (iii) together with Proposition~\ref{prop:cone_free_key_lem}, one can 
argue in the same way as Lemma~\ref{lem:in_out_cont_updated}  to obtain that for all $M\geq m_0(A,q),$
\begin{align}\label{eq:Vk_est_first_m}
\sup_{(r,z),t,k}
\left(
\left|\frac{V_{k,(\leq M)}^r(r,z,t)}{r}\right|+\left|\frac{V_{k,(\leq M)}^z(r,z,t)}{z}\right| \right) 
\ \le\ C_E(q)\bar A\left( 1 + \log\left(\frac{\bar A}{\varepsilon}\right) + \log M \right),
\end{align}
where the supremum is taken over all $x,t,k$ such that
    \begin{align*}
        1\leq k \leq M,\quad
        0\leq t \leq T_B(A,M),\quad
        (r,z)\in \supp W_k(\cdot,t).
    \end{align*}

\medskip
\noindent \emph{Step 1-(b): Estimate the tail contribution $u^{(>M),r}, u^{(>M),z}$}

In this part, we use Proposition~\ref{prop:cone_free_key_lem}. Under the flow, the odd-symmetry of $\omega^{(\infty)}$ in $z$ is preserved.   Therefore, the odd-symmetry assumption of Proposition~\ref{prop:cone_free_key_lem} is satisfied for $\omega^{(\infty)}(\cdot,t)$ for all $t>0$. In addition, as we are considering a Yudovich-type weak solution $\omega^{(\infty)}\in L^\infty_t (L^1\cap L^\infty)$, all the assumptions of Proposition~\ref{prop:cone_free_key_lem} are satisfied.

By the construction of the initial data \eqref{KJdata2} and the profile support condition \eqref{phi_cond2}, we know that $\supp \omega_0 ^{(\infty)}\cap \{z>0\}$ is strictly contained in the cone $z \le L_{q,+} r$, where $L_{q,+} = \frac{1+\eta}{1-\eta} L_q$.

Now, we want to prove that for all $t\in [0,T_B(A,M)]$ 
    \begin{align}
    \label{eq:cone_condition}
        2L_{q,+}r\geq |z|>0 
        \quad
        \text{for  all $(r,z)\in \cup_{j=1}^\infty \supp \omega_j(\cdot,\cdot,t)$,}
    \end{align}
for a sufficiently large $M\geq m_{\mathrm{cone}}(A,q)$ (defined in \eqref{eq:def_m_cone} below), which will allow us to 
estimate the factor $|x|/r_x$
in the $Q(r_x)$-estimates of Proposition~\ref{prop:cone_free_key_lem}
  by $\sqrt{1+4L_{q,+}^2}$   
for all time $t\in[0,T_B(A,M)]$.
\begin{proof}[\textbf{Proof of the cone condition \eqref{eq:cone_condition}}] 
This proof is motivated by the proof of Lemma~3.1 in \cite{KimJeong22}.
At time $t=0$, we have $Y^z(x,0)/Y^r(x,0) \le L_{q,+}$ for all $x \in \supp \omega_0^{(\infty)} \cap \{z>0\}$.

We analyze the evolution of the ratio $Y^z(x,t)/Y^r(x,t)$ for a Lagrangian trajectory $Y(x,t)$ starting at $x \in \supp \omega_0 ^{(\infty)}\cap \{z>0\}$. 
Consider the intermediate region $\mathcal{R} := \{ (r,z) : L_{q,+} \le z/r \le 2L_{q,+} \}$.
While the trajectory $Y(x,t)$ is in the region $\mathcal{R}$, we can bound the factor $|x|/r_x$ by a constant, $\sqrt{1+4L_{q,+}^2}$, when applying the $Q(r_x)$-estimate of Proposition~\ref{prop:cone_free_key_lem} since $2L_{q,+}r \ge z > 0$.
The contradiction hypothesis yields $\|\omega^{(\infty)}(\cdot,t)\|_{L^\infty(\mathbb{R}^3)} \le A_\infty \le \bar A$. 
Because $-z_y \omega^{(\infty)}(y,t) \ge 0$ everywhere (due to odd symmetry and non-positivity on the upper half-plane), the integral term in Proposition~\ref{prop:cone_free_key_lem} is non-negative.
Therefore, 
Proposition~\ref{prop:cone_free_key_lem} implies that whenever $Y(x,t)$ is in $\mathcal{R}$,
\begin{align*}
    \frac{u^{(\infty),r}(Y,t)}{Y^r} \ge -C \sqrt{1+4L_{q,+}^2} \,\|\omega^{(\infty)}(\cdot,t)\|_{L^\infty} \ge -C\sqrt{1+4L_{q,+}^2} \bar A.
\end{align*}
Similarly, by an analogous application of Proposition~\ref{prop:cone_free_key_lem} for the vertical component $u^{(\infty),z}$, the rescaled velocity  satisfies an upper bound:
\begin{align*}
    \frac{u^{(\infty),z}(Y,t)}{Y^z} \le C \left(1 + \left| \log \frac{Y^r}{Y^z}  \right| \right) \|\omega^{(\infty)}(\cdot,t)\|_{L^\infty} \le C (1+|\log (2L_{q,+})|) \bar A,
\end{align*}
for an absolute constant $C>0$, bounding the logarithmic term since $Y^r/Y^z \in [1/(2L_{q,+}), 1/L_{q,+}]$  in the region $\mathcal{R}$.

Using these inequalities, we monitor the aspect ratio of the trajectory inside $\mathcal{R}$:
\begin{align*}
    \frac{d}{dt} \log \frac{Y^z(x,t)}{Y^r(x,t)} 
    &= \frac{u^{(\infty),z}(Y(x,t),t)}{Y^z(x,t)} - \frac{u^{(\infty),r}(Y(x,t),t)}{Y^r(x,t)}
    \\
    &\le \left(1+|\log (4L_{q,+})|\, + \sqrt{1+4L_{q,+}^2}\right)\bar A 
    = C(q)\, \bar A
\end{align*}
where the constant $C$ ultimately depends on $q$ and fixed background parameters, but independent of $M,t,k$.

Suppose that a trajectory $Y(x,t)$ enters the region $\mathcal{R}$ at some time $t=T_{\mathcal{R}} \le T_B(A,M)$. 
Until it escapes the region $\mathcal{R}$, the trajectory complies with the above differential inequality. 
Integrating from $T_{\mathcal{R}}$ to $t \le T_B(A,M)$, we obtain:
\begin{align*}
    \log \frac{Y^z(x,t)}{Y^r(x,t)} 
    &\le \log \frac{Y^z(T_{\mathcal{R}},x)}{Y^r(T_{\mathcal{R}},x)} + C(q) \bar A (t - T_{\mathcal{R}}) \\
    &\le \log L_{q,+} + C(q) \bar A T_B(A,M).
\end{align*}

Since $T_B(A,M) \to 0$ as $M \to \infty$ by Proposition~\ref{prop:finite_head_reduction}, there exists a large integer $m_{\mathrm{cone}}(A,q)\geq m_0(A,q)$ such that all $M\geq m_{\mathrm{cone}}(A,q)$ satisfy
\begin{align}
\label{eq:def_m_cone}
C(q) \bar A T_B(A,M) \le \log \frac32.
\end{align}
Under this condition, it follows that
\begin{align*}
    \log \frac{Y^z(x,t)}{Y^r(x,t)} \le \log \left(L_{q,+}\right) + \log\left( \frac32 \right) = \log \left( \frac{3}{2}L_{q,+} \right).
\end{align*}
From this, we conclude that $Y^z(x,t) / Y^r(x,t) \le \frac{3}{2}L_{q,+}$, meaning the trajectory $Y(x,t)$ can never reach the boundary $z = 2L_{q,+} r$ for all $t \in [0, T_B(A,M)]$. 
Since $Y^z(x,t) > 0$ is also strictly preserved by the flow (as the $z=0$ plane is invariant due to odd symmetry), the condition $2L_{q,+}r \ge |z| > 0$ holds for all $x \in \cup_{j=1}^\infty \supp \omega_j(\cdot,t)$. This completes the proof of the cone condition \eqref{eq:cone_condition}.
\end{proof}

Thanks to \eqref{eq:cone_condition}, we can apply the $Q(r_x)$-estimates of Proposition~\ref{prop:cone_free_key_lem} to $\omega^{(\infty)}(\cdot,\cdot,t)$ with the factor $|x|/r_x$  bounded by $\sqrt{1+4L_{q,+}^2}$, for all $t\in[0,T_B(A,M)]$ for all  $M\geq m_{\mathrm{cone}}(A,q)$.
Set
\[
K_{\mathrm{tail}}
:=
\overline{\bigcup_{j=M+1}^{\infty}\operatorname{supp}\omega_{0,j}}
=
\Bigl(\bigcup_{j=M+1}^{\infty}\operatorname{supp}\omega_{0,j}\Bigr)\cup \{0\},
\]
and define the outer radial boundary of the tail $R_{\mathrm{tail}}(t)$ and  the inner radial boundary of the head $r_{M,-}(t)$ by
\[
r_{M,-}(t):=
\min_{x\in \operatorname{supp}\omega_{0,M}} Y^r(x,t),
\qquad
R_{\mathrm{tail}}(t):=
\max_{x\in K_{\mathrm{tail}}} Y^r(x,t).
\]
We claim that the head and the tail rings are separated. More precisely, there exists a large $m_{\mathrm{tail}}(A,q)\geq m_{\mathrm{cone}}(A,q)$ such that all $M\geq m_{\mathrm{tail}}(A,q)$ satisfy
    \begin{align}
    \label{eq:tail_head_separation}
        \frac{r_{M,-}(t)}{R_{\mathrm{tail}}(t)} \geq 4 \quad \text{for all }t\in [0,T_B(A,M)].
    \end{align}

\begin{proof}[\textbf{Proof of the tail scale separation bound \eqref{eq:tail_head_separation}}]
We use a bootstrap argument. Define 
\begin{align*}
    T^*_B := \sup \left\{ T \in [0, T_B(A,M)] : \frac{r_{M,-}(t)}{R_{\mathrm{tail}}(t)} \ge 4 \quad \text{for all } t \in [0, T] \right\}.
\end{align*}

At $t=0$, the separation between the $M$-th ring and the tail is exactly:
\begin{align*}
    \frac{r_{M,-}(0)}{R_{\mathrm{tail}}(0)} \geq  \frac{(1-\eta)d^{M}}{(1+\eta)d^{M+1}} = \frac{1-\eta}{1+\eta} \, \frac{1}{d}=
    \frac{3/4}{5/4} (100) = 60.
\end{align*}

Under the contradiction hypothesis, the velocity field $u^{(\infty)}$ is also globally bounded, that is, $u^{(\infty)}\in L^\infty(0,T_0; L^\infty(\mathbb{R}^3)),$ hence the flow map $Y^r(x,t)$ is \emph{uniformly} Lipschitz continuous in time, that is,
$|Y^r(x,t)-Y^r(x,s)| \leq C |t-s|$ for all $t,s\in [0,T_0]$ and $x\in \mathbb{R}^3$ where the constant $C$ is independent of $t,s,x$. Consequently, 
both $t\mapsto r_{M,-}(t)$ and $t\mapsto R_{\mathrm{tail}}(t)$ are continuous as extrema over compact sets. Hence $T^*_B>0$.

Fix  $M\geq \max\{m_0(A,q),m_{\mathrm{cone}}(A,q)\}$. 
Since $u$ is uniformly bounded and uniformly log-Lipschitz in space on $[0,T_0]$ (see \eqref{eq:u_log_Lip_common}), and
since $Y$ is continuous on compact sets, the map
\[
g(x,t):=u^{(\infty),r}(Y(x,t),t)
\]
is bounded. Using log-Lipschitzness of $u^{(\infty)}$ and H\"older estimate \eqref{eq:Y_reg_x_common} of $Y$, it follows that $g$ admits a modulus of continuity in $x$, uniform in $t$, on both compact
sets
\[
K_M:=\operatorname{supp}\omega_{0,M},
\qquad
K_{\mathrm{tail}}
\]
using the uniform H\"older continuity of $Y$ \eqref{eq:Y_reg_x_common} on the relevant compact sets.
More precisely, $g$ admits the following modulus of continuity:
    \begin{align*}
        |g(x,t)- g(y,t)|
    \leq
        C |x-y|^\gamma \log (2+C|x-y|^{-\gamma})=: \rho (|x-y|).
    \end{align*}
Hence Lemma~\ref{lem:envelope_radial_extrema} applies with $K=K_M$ and $K=K_{\mathrm{tail}}$ and shows that $r_{M,-}$ and
$R_{\mathrm{tail}}$ are Lipschitz on $[0,T_0]$, and for a.e. $t\in (0,T_0)$,
\[
r_{M,-}'(t)
=
\min_{x\in \mathcal A_M(t)} u^{(\infty),r}(Y(x,t),t),
\qquad
R_{\mathrm{tail}}'(t)
=
\max_{x\in \mathcal M_{\mathrm{tail}}(t)} u^{(\infty),r}(Y(x,t),t),
\]
where
\[
\mathcal A_M(t)
:=
\operatorname*{argmin}_{x\in K_M} Y^r(x,t),
\qquad
\mathcal M_{\mathrm{tail}}(t)
:=
\operatorname*{argmax}_{x\in K_{\mathrm{tail}}} Y^r(x,t).
\]
In particular, for a.e. such $t$, there exist
$x_t\in \mathcal A_M(t)$ and $x_t'\in \mathcal M_{\mathrm{tail}}(t)$ such that
\[
r_{M,-}'(t)=u^r(Y(x_t,t),t),
\qquad
R_{\mathrm{tail}}'(t)=u^r(Y(x_t',t),t).
\]
Therefore,
\[
\frac{d}{dt}\log\frac{r_{M,-}(t)}{R_{\mathrm{tail}}(t)}
=
\frac{u^{(\infty),r}(Y(x_t,t),t)}{r_{M,-}(t)}
-
\frac{u^{(\infty),r}(Y(x_t',t),t)}{R_{\mathrm{tail}}(t)}
\qquad
\text{for a.e. }t\in (0,T_0).
\]
From this point on, all differential inequalities are understood for a.e. $t$, and are
then integrated in time.

Setting $Y=Y(x_t,t), Y'=Y(x_t',t)$, (hence $Y'^{r}=R_{\mathrm{tail}}(t)$)
we evaluate the contraction via the $Q(r_x)$-estimates of  Proposition~\ref{prop:cone_free_key_lem}:
\begin{align*}
    \frac{d}{dt} \log R_{\mathrm{tail}}(t) - \frac{d}{dt} \log r_{M,-}(t) \le \frac{3}{8\pi} \int_{Q(Y'^r) \setminus Q(Y^r)} \frac{r_y z_y}{|y|^5} \big(-\omega^{(\infty)}(t,y)\big) \, dy + 2C_L \bar A, \quad \text{a.e. }\, t\in [0,T_B^*].
\end{align*}
The domain of integration is $\{y \in \mathbb{R}^3 : 2 Y'^r \le r_y < 2 Y^r \}$. We check its boundaries:
\begin{itemize}
    \item \textbf{Upper boundary:} By Lemma~\ref{lem:Euler_cauchy_geometry} and \eqref{eq:sep_R_ratio}, it holds 
    \begin{align*}
    2 Y^r \le 2 R_M(t)  \, r_+\le R_{M-1}(t) \, r_-\leq r_{M-1,-}(t).
    \end{align*}
    Thus, the domain strictly avoids rings $1$ through $M-1$.
    \item \textbf{Lower boundary:} Because $Y'^r$ represents the outer boundary of the tail, $2 Y'^r> R_{\mathrm{tail}}(t)$.  Thus, the domain strictly overshoots the entire tail mass.
\end{itemize}

Consequently, the integration domain intersects no ring except possibly the $M$-th ring. Therefore, its contribution is bounded by the integral over $\supp \omega_M(\cdot,t)$.
  We bound the integral using exact volume conservation $|\supp \omega_M(\cdot,t)|=|\supp\omega_{0,M}|= C(q) \,d^{3M}$:
\begin{align*}
    \frac{d}{dt} \log \frac{r_{M,-}(t)}{R_{\mathrm{tail}}(t)} 
    &\ge - C (q) \, \bar A \int_{\supp \omega_M(\cdot,t)} \frac{1}{|y|^3} \, dy - 2C_L \bar A.
\end{align*}
 The singular integral is therefore uniformly bounded:
\begin{align*}
    \int_{\supp \omega_M(\cdot,t)} \frac{1}{|y|^3} \, dy 
    \le \frac{|\supp \, \omega_{M}(\cdot,t)|}{r_{M,-}(t)^3} 
    \le C(q)\frac{d^{3M}}{R_{M}(t)^3 }
    \le C(q).
\end{align*}
Therefore,
\begin{align*}
    \frac{d}{dt} \log \frac{r_{M,-}(t)}{R_{\mathrm{tail}}(t)} \ge -C(q)\, \bar A.
\end{align*}
Integrating this inequality from $t=0$ yields 
\begin{align*}
\frac{r_{M,-}(t)}{R_{\mathrm{tail}}(t)} 
\geq \frac{r_{M,-}(0)}{R_{\mathrm{tail}}(0)} e^{-C(q) \bar A t}
\ge \frac{1-\eta}{1+\eta} \,\frac{1}{d}\,  e^{-C(q) \bar A t}.
\end{align*}
Since $T_B(A,M)\to 0$ as $M\to\infty$,  there exists a large integer $m_{\mathrm{tail}}(A,q)\geq m_{\mathrm{cone}}(A,q)$ such that all $M\geq m_{\mathrm{tail}}(A,q)$ satisfy
    \begin{align}
    \label{eq:m_tail_def}
    \frac{1-\eta}{1+\eta} \, \frac{1}{d} \, e^{-C(q) \bar A T_B(A,M)} = 60 e^{-C(q) \bar A T_B(A,M)} >4.
    \end{align}
Then for $M\geq m_{\mathrm{tail}}(A,q)$,  the bootstrap condition strictly improves on $[0, T^*_B]$, hence it forces $T^*_B = T_B(A,M)$. It finishes the proof of the tail scale separation bound~\eqref{eq:tail_head_separation}.
\end{proof}

\noindent \emph{Completion of Step 1-(b).}
Using the strict scale separation \eqref{eq:tail_head_separation}, we can establish a uniform distance bound between the $k$-th ring and any tail ring $j > M$, allowing us to rigorously bound the tail velocity contribution.

Fix $k \in \{1,\ldots,M\}$ and $t \in [0,T_B(A,M)]$. Let $x \in \supp \omega_k(\cdot,t)$ and $y \in \supp \omega_j(\cdot,t)$ for some $j > M$. The 3D Euclidean distance is bounded below by the radial distance:
\begin{align*}
    |x-y| \ge |r_x - r_y|.
\end{align*}
For $M\geq \max\{m_{\mathrm{tail}}(A,q), m_0(A,q)\} $, by using \eqref{eq:tail_head_separation}, since $j > M \ge k$, the $j$-th ring is strictly localized closer to the $z$-axis than the $k$-th ring:
\begin{align*}
    r_y 
    \le 
        \sup_{\xi \in \supp \omega_{0,j}} Y^r(\xi, t) 
    \le R_{\mathrm{tail}}(t) 
    \le \frac{1}{2} r_{M,-}(t)
    \le 
        \frac{1}{2} \inf_{\xi \in \supp \omega_{0,k}} Y^r(\xi, t) 
    \le \frac{1}{2} r_x.
\end{align*}
This guarantees a uniform lower bound on the distance:
\begin{align}
\label{eq:tail_distance_bound}
    |x-y| \ge r_x - \frac{1}{2} r_x = \frac{1}{2} r_x\geq \frac{1}{2} R_k(t) r_-.
\end{align}
where we used the finite-ring bootstrap Lemma~\ref{lem:Euler_cauchy_geometry}, which ensures $r_x = R_k(t) r$ with $r \ge r_-$. 

Next, we estimate the velocity $u_j(x,t)$ induced by the $j$-th ring. Because the velocity field is divergence-free, the physical volume of the $j$-th ring is exactly conserved over time. From the initial data \eqref{KJdata2}, its volume is strictly determined by the initial geometric scale $d^j$:
\begin{align*}
    |\supp \omega_j(\cdot, t)|
    = |\supp \omega_{0,j}|
= C(q) (d^j)^3 .
\end{align*}
Using the Biot--Savart law, the contradiction hypothesis $\|\omega_j(\cdot,t)\|_{L^\infty} \le A_\infty\le\bar A$, and our uniform distance bound \eqref{eq:tail_distance_bound}, we obtain:
\begin{align}
\label{eq:tail_uj_bound}
    |u_j(x,t)| 
    \le C_{\mathrm{BS}} \int_{\supp \omega_j(\cdot,t)} \frac{|\omega_j(y,t)|}{|x-y|^2} \, dy 
    \le \frac{C_{\mathrm{BS}}}{\left(\frac{1}{2} r_- R_k(t)\right)^2} \|\omega_j\|_{L^\infty} |\supp \omega_j|
    \le C(q) \bar A \frac{d^{3j}}{R_k(t)^2}.
\end{align}

We now bound the rescaled velocity contributions on the support of the $k$-th ring. For the radial component, we divide \eqref{eq:tail_uj_bound} by $R_k(t) r$. Using $r \ge r_-$ and the fact that amplitudes do not decrease ($R_k(t) = d^k \tilde{x}_k(t) \ge d^k$), we get:
\begin{align*}
    \frac{|u_j^r(x,t)|}{R_k(t) r} 
    \le \frac{C(q) \bar A}{R_k(t)^3 r_-} d^{3j} 
    \le C(q) \bar A \frac{d^{3j}}{(d^k)^3} 
    = C(q) \bar A (d^3)^{j-k}.
\end{align*}
For the vertical component, we divide \eqref{eq:tail_uj_bound} by $H_k(t) z$. By definition \eqref{tilde_xk_rk_hk}, we have the exact geometric relation $R_k(t)^2 H_k(t) = (d^k \tilde{x}_k(t))^2 (d^k / \tilde{x}_k(t)^2) = d^{3k}$. Since $|z| \ge z_-$, this yields:
\begin{align*}
    \left| \frac{u_j^z(x,t)}{H_k(t) z} \right|
    \le \frac{C(q) \bar A}{R_k(t)^2 H_k(t) z_-} d^{3j} 
    = C(q) \bar A \frac{d^{3j}}{d^{3k}} 
    = C(q) \bar A (d^3)^{j-k}.
\end{align*}
Because the base of the exponent $d^3 = 10^{-6} < 1$, summing these contributions over all tail rings $j > M$ results in a rapidly converging geometric series:
\begin{align*}
    \frac{|u^{(>M),r}(x,t)|}{R_k(t) r} + \frac{|u^{(>M),z}(x,t)|}{H_k(t) |z|}
    \le C(q) \bar A \sum_{j=M+1}^\infty (d^3)^{j-k} 
    \le C(q) \bar A \frac{d^3}{1-d^3}.
\end{align*}
We conclude that the tail contribution is uniformly bounded by the target amplitude: for all  $M\geq m_{\mathrm{tail}}(A,q),$
\begin{align}
\label{eq:tail_est_37}
    \sup_{(r,z),t,k} \left( \left| \frac{u^{(>M),r}(R_k(t) r, H_k(t) z, t)}{R_k(t) r} \right| + \left| \frac{u^{(>M),z}(R_k(t) r, H_k(t) z, t)}{H_k(t) z} \right| \right) \le C(q) \bar A,
\end{align}
where the supremum is taken over all $x,t,k$ from the same set as \eqref{eq:Vk_est_first_m}.

\medskip
 \noindent \emph{Step 1-(c): Conclusion from (a) and (b).}
Combining \eqref{eq:Vk_est_first_m} and \eqref{eq:tail_est_37}, we can conclude that
for all $M\geq \max\{{m_{\mathrm{tail}}(A,q),m_0(A,q)}\}$,
\begin{align*}
\sup_{(r,z),t,k} 
\left(
\left|\frac{V_k^r(r,z,t)}{r}\right|
+\left|\frac{V_k^z(r,z,t)}{z}\right|
\right) \leq C_E(q) \bar A 
\left(
1+\log \left( \frac{\bar A}{\varepsilon} \right) 
+ \log M
\right)
\end{align*}
where the supremum is taken from the same set as in \eqref{eq:Vk_est_first_m}.

\medskip
To complete bootstrap improvement, we argue in the same way as in Proposition~\ref{lem:Euler_close_bootstrap_X_DY}.
Recall $T_B(A,M)\lesssim M^{-\beta_q}\to 0$ as $M\to\infty$. As the algebraic decay $M^{-\beta_q}$ wins over the logarithmic growth $\log M$, there exists a large integer $m_1(A,q)\geq m_{\mathrm{tail}}(A,q)$ such that for any $M\geq m_1(A,q)$,
    \begin{align}
    \label{eq:m_1_def2}
        C_E \bar A 
        \left(
        1+ \log \left( \frac{\bar A}{\varepsilon} \right) 
        + \log M
        \right) T_B(A,M)
        \leq
        \min \left\{
        \log \left( 1+\frac{\mu}{2}\right),
        \log  \left( \frac{1}{1-\mu/2}\right)
        \right\}.
    \end{align}

Therefore, we can obtain the improved bounds
$1-\mu/2\le X_k^r/r\le 1+\mu/2$ and $1-\mu/2\le X_k^z/z\le 1+\mu/2$ on $\supp\phi_q$ for all
$0\le t\le T_B(A,M)$ and all $1\le k\le M$ when $M\geq m_1(A,q)$. This contradicts maximality of $T_B(A,M)$, using the continuity
of $F_M$ and the fact that $T_B(A,M)>0$. Hence \eqref{eq:TB_ge_min_T0_TN} holds, that is, $T_B(A,M)=T_N(A,M)$, and it completes Step 1.

\medskip
\noindent{\bf Step 2: ODE norm inflation yields a contradiction.}
 Recall on
$[0,T_B(A,M)]$, we can obtain 
\begin{equation}\label{eq:Wk_lower_infty_layers}
\|W_k(\cdot,\cdot,t)\|_{L^\infty(\R^3)}\ge (1-\mu)\|\phi_q\|_{L^\infty(\R^3)}
\qquad (1\le k\le M,\ 0\le t\le T_B(A,M)).
\end{equation}
In addition, the amplitudes $\{x_k(t)\}_{k=1}^M$
satisfy the same weak ODE inequalities \eqref{mu_ODE} as in Sections~\ref{sec:est_pro}--\ref{sec:ODEs_Euler}.
By Proposition~\ref{prop:finite_head_reduction} and the choice of $\alpha_q$ in \eqref{cond_alpha66}, we have $T_B(A,M)\to0$ as
$M\to\infty$. For $M\geq m_1(A,q)$, Step~1 yields
$T_B(A,M)=T_N(A,M)$.
There exists  $m_2(A,q) \geq m_1(A,q)$ such that 
\[
\frac{C(q,A)}{M^{\beta_q}}<T_0 \qquad \text{for all } M\geq m_2(A,q)
\]
where $C(q,A)$ is the constant in \eqref{eq:est_TB}.
Since Proposition~\ref{prop:finite_head_reduction} gives
$T_B(A,M)\le C(q,A) M^{-\beta_q}$, it follows that $T_N(A,M)=T_B(A,M)<T_0$, which implies
\begin{align*}
    \sup_{1\leq k \leq M  } x_k (T_N(A,M)) = \bar A.
\end{align*}

Since each \(\omega_k(\cdot,t)\le0\) on \(\{z>0\}\), the functions
\(-\omega_k(\cdot,t)\) are nonnegative there. Hence
\[
-\omega^{(\infty)}(\cdot,t)
=
\sum_j (-\omega_j(\cdot,t))
\ge
-\omega_k(\cdot,t)
\qquad\text{on }\{z>0\}.
\]
Therefore
\[
\|\omega^{(\infty)}(\cdot,t)\|_{L^\infty}
\ge
\|\omega_k(\cdot,t)\|_{L^\infty}.
\]
Evaluating at $t=T_N(A,M)$ and using \eqref{eq:Wk_lower_infty_layers}, we obtain
\begin{align*}
\Big\|\omega^{(\infty)}\big(T_N(A,M)\big)\Big\|_{L^\infty(\mathbb{R}^3)}
&\ge \sup_{1\le k\le M}\Big\|\omega_k\big(T_N(A,M)\big)\Big\|_{L^\infty(\mathbb{R}^3)}
\\
&= \sup_{1\le k\le M}x_k\big(T_N(A,M)\big)\Big\|W_k\big(T_N(A,M)\big)\Big\|_{L^\infty(\mathbb{R}^3)}
\\
&\ge (1-\mu)\|\phi_q\|_{L^\infty}\sup_{1\le k\le M}x_k\big(T_N(A,M)\big)
\\
&= (1-\mu)\|\phi_q\|_{L^\infty}\bar A
=\|\phi_q\|_{L^\infty}A\ge A>A_\infty.
\end{align*}
In fact, 
\[
\|\omega^{(\infty)}(t)\|_{L^\infty(\mathbb R^3)}>A_\infty,
\]
for a.e. $t\in[T_N(A,M)-\tau,T_N(A,M)]$, for some sufficiently small $\tau>0$,
contradicting the definition of $A_\infty$.
In other words, this contradicts the assumed existence of an essentially bounded vorticity solution on $[0,T_0]$. Therefore, no such $T_0>0$ exists, proving Theorem~\ref{thm:instant_blow-up}-(1).
\qed

\subsection{Proof of Theorem~\ref{thm:instant_blow-up}-(2)}
\label{sec:pf_rem}
This subsection proves Theorem~\ref{thm:instant_blow-up}-(2). In fact, the argument below yields a single global-in-time distributional solution; Theorem~\ref{thm:instant_blow-up}-(2) then follows by restricting that solution to an arbitrary finite time interval.

We first show that a subsequential limit of the smooth finite-ring solutions from Section~\ref{sec:norm_inflation} gives rise to a distributional solution on an arbitrary finite time interval.

\begin{proposition}[Finite-ring approximation of a distributional solution]
\label{prop:finite_ring_limit_distributional}
Let
\[
\omega_0^{(\infty)}=\sum_{k=1}^\infty \omega_{0,k}
\]
be the infinite-ring initial data from \eqref{KJdata2} with $\alpha=\alpha_q$ and $\phi=\phi_q$, and for each \(m\ge 1\) let
\[
\omega_0^{(m)}:=\sum_{k=1}^m \omega_{0,k}.
\]
Let \((u^{(m)},\omega^{(m)})\) be the unique global smooth axisymmetric no-swirl
solution of \eqref{main_eq} with the initial condition \(\omega^{(m)}|_{t=0}=\omega_0^{(m)}\), and set
\[
\xi^{(m)}:=\frac{\omega^{(m)}}{r},
\qquad
\xi_0^{(m)}:=\frac{\omega_0^{(m)}}{r}.
\]
Then, for every \(T>0\), there exist a subsequence \(m_j\to\infty\) and
axisymmetric functions
\[
u^{(\infty)}\in L^\infty\bigl(0,T;L^2(\mathbb{R}^3)\bigr)
\cap L^\infty\bigl(0,T;H^1_{\mathrm{loc}}(\mathbb{R}^3)\bigr),
\qquad
\xi^{(\infty)}\in L^\infty\bigl(0,T;L^2_{\mathrm{loc}}(\mathbb{R}^3)\bigr),
\]
such that, writing \(\omega^{(\infty)}:=r\xi^{(\infty)}\),
\begin{align*}
u^{(m_j)} &\overset{*}{\rightharpoonup} u^{(\infty)}
\quad\text{in }L^\infty\bigl(0,T;L^2(\mathbb{R}^3)\bigr),
\\
u^{(m_j)} &\to u^{(\infty)}
\quad\text{strongly in }L^2\bigl((0,T)\times B_R\bigr)
\quad\text{for every }R>0,
\\
\xi^{(m_j)} &\overset{*}{\rightharpoonup} \xi^{(\infty)}
\quad\text{in }L^\infty\bigl(0,T;L^2(B_R)\bigr)
\quad\text{for every }R>0.
\end{align*}
Moreover, \((u^{(\infty)},\omega^{(\infty)})\) is an (axisymmetric no-swirl) distributional solution of
\eqref{eq:IVP_rel_vor} on \((0,T)\) with the initial data \(\omega_0^{(\infty)}\) in the sense of
Definition~\ref{def:dist_sol_rel_vor}.
\end{proposition}

\begin{remark}
Proposition~\ref{prop:finite_ring_limit_distributional} pertains to a sequential limit of our finite-ring approximations. If one only cares about the existence of a distributional solution with our initial data in the sense of Definition~\ref{def:dist_sol_rel_vor}, without keeping track of the finite-ring approximations, then this can also be obtained from \cite[Theorems~1--2]{NobiliSeis22}.
\end{remark}

\begin{proof}[Proof of Proposition~\ref{prop:finite_ring_limit_distributional}]
Set
\[
\xi_{0,k}:=\frac{\omega_{0,k}}{r},
\qquad
\xi_0^{(m)}:=\sum_{k=1}^m \xi_{0,k},
\qquad
\xi_0^{(\infty)}:=\sum_{k=1}^\infty \xi_{0,k}.
\]
Since the ring supports are pairwise disjoint, the scaling computations from
Section~\ref{sec:ini_data_ansatz} give, for \(p=1,2\),
\[
\|\omega_{0,k}\|_{L^p(\mathbb R^3)}^p
=
\frac{\varepsilon^p}{k^{\alpha p}}
d^{3k}\|\phi_q\|_{L^p(\mathbb R^3)}^p,
\]
and
\[
\|\xi_{0,k}\|_{L^p(\mathbb R^3)}^p
=
\frac{\varepsilon^p}{k^{\alpha p}}
d^{(3-p)k}
\left\|\frac{\phi_q}{r}\right\|_{L^p(\mathbb R^3)}^p.
\]
Here \(\phi_q/r\in L^p(\mathbb R^3)\) because \(\phi_q\) is supported away from the symmetry axis.
Consequently,
\[
\sup_{m\ge1}
\left\|\omega_0^{(m)}\right\|_{L^1\cap L^2(\mathbb R^3)}
<\infty,
\qquad
\xi_0^{(\infty)}\in L^1\cap L^2(\mathbb R^3).
\]
Moreover,
\[
\left\|\xi_0^{(\infty)}-\xi_0^{(m)}\right\|_{L^1(\mathbb R^3)}
\le
C\varepsilon
\sum_{k=m+1}^\infty
\frac{d^{2k}}{k^\alpha}
\to 0 \quad \text{as}\quad m\to\infty.
\]
Finally, by the Lorentz estimate already proved in Section~\ref{sec:ini_data_ansatz},
\[
\sup_{m\ge1}
\left\|\xi_0^{(m)}\right\|_{L^{3,q}(\mathbb R^3)}
<\infty.
\]
Thus
\begin{equation}
\label{eq:uniform_est_finite_ring_vor_rel}
\sup_{m\ge 1}\left\|\omega_0^{(m)}\right\|_{L^1\cap L^2(\mathbb{R}^3)}<\infty,
\qquad
\sup_{m\ge 1}\left\|\xi_0^{(m)}\right\|_{L^{3,q}(\mathbb{R}^3)}<\infty.
\end{equation}

From the more general form of the Biot--Savart law~\eqref{eq:gen_Biot-Savart},
we obtain
\[
\left|u_0^{(m)}(x)\right|
\lesssim
\int_{\mathbb{R}^3}\frac{\left|\omega_0^{(m)}(y)\right|}{|x-y|^2}\,dy.
\]
Write
\[
K_1(x):=|x|^{-2}\mathbf 1_{\big\{|x|\le 1\big\}},
\qquad
K_2(x):=|x|^{-2}\mathbf 1_{\big\{|x|>1\big\}}.
\]
Then \(K_1\in L^1(\mathbb{R}^3)\) and \(K_2\in L^\infty(\mathbb{R}^3)\cap L^2(\mathbb{R}^3)\).
Young's inequality yields
\[
\left\|u_0^{(m)}\right\|_{L^2}
\lesssim
\|K_1\|_{L^1} \left\|\omega_0^{(m)} \right\|_{L^2}
+\|K_2\|_{L^2}\left\|\omega_0^{(m)}\right\|_{L^1}.
\]
Hence, using \eqref{eq:uniform_est_finite_ring_vor_rel},
\[
\sup_{m\ge 1}\left\|u_0^{(m)}\right\|_{L^2(\mathbb{R}^3)}<\infty.
\]
Since \((u^{(m)},\omega^{(m)})\) is a smooth Euler solution, energy is conserved, so
\begin{equation}
\sup_{m\ge 1}\left\|u^{(m)}\right\|_{L^\infty(0,T;L^2(\mathbb{R}^3))}<\infty.
\label{eq:finite_ring_energy_bd}
\end{equation}

Moreover, \(\xi^{(m)}=\omega^{(m)}/r\) solves
\[
\partial_t\xi^{(m)}+u^{(m)}\cdot\nabla\xi^{(m)}=0
\]
and is transported by the smooth volume-preserving flow associated with \(u^{(m)}\).
Hence its Lorentz norm is conserved:
\[
\left\|\xi^{(m)}(t)\right\|_{L^{3,q}(\mathbb{R}^3)}
=
\left\|\xi_0^{(m)}\right\|_{L^{3,q}(\mathbb{R}^3)}
\le C
\qquad (0\le t\le T),
\]
with \(C\) independent of \(m\).

Fix \(R>0\). Since \(B_{2R}\) has finite measure and \(L^{3,q}(B_{2R})\hookrightarrow
L^2(B_{2R})\), it follows that
\begin{equation}
\label{eq:finite_ring_loc_rel_vor}
\sup_{m\ge 1}\sup_{0\le t\le T}
\left\|\xi^{(m)}(t)\right\|_{L^2(B_{2R})}
\le C(R).
\end{equation}
Because \(\omega^{(m)}=r\xi^{(m)}\), we also have
\begin{equation}
\sup_{m\ge 1}\sup_{0\le t\le T}
\left\|\omega^{(m)}(t)\right\|_{L^2(B_{2R})}
\le C(R).
\label{eq:finite_ring_local_vor_bd}
\end{equation}

Choose \(\chi_R\in C_c^\infty(B_{2R})\) such that \(\chi_R\equiv 1\) on \(B_R\).
The standard div--curl estimate on \(\mathbb{R}^3\) gives
\begin{multline*}
\left\|\chi_R u^{(m)}(t)\right\|_{H^1(\mathbb{R}^3)}
\\
\leq C\left( 
\left\|\chi_R u^{(m)}(t)\right\|_{L^2(\mathbb{R}^3)}
+\left\|\operatorname{div}\left(\chi_R u^{(m)}(t)\right)\right\|_{L^2(\mathbb{R}^3)}
+\left\|\operatorname{curl}\left(\chi_R u^{(m)}(t)\right)\right\|_{L^2(\mathbb{R}^3)}
\right).
\end{multline*}
Since
\[
\operatorname{div}\left(\chi_R u^{(m)}\right)=\nabla\chi_R\cdot u^{(m)},
\qquad
\operatorname{curl}\left(\chi_R u^{(m)}\right)
=\chi_R\,\omega^{(m)}e_\theta+\nabla\chi_R\times u^{(m)},
\]
the bounds \eqref{eq:finite_ring_energy_bd} and
\eqref{eq:finite_ring_local_vor_bd} imply
\begin{equation}
\sup_{m\ge 1}\left\|u^{(m)}\right\|_{L^\infty\left(0,T;H^1(B_R)\right)}\le C(R).
\label{eq:finite_ring_local_H1_bd}
\end{equation}

Next, since \(u^{(m)}\) is a smooth Euler solution, applying the Leray projector \(\mathbb P\)
gives
\[
\partial_t u^{(m)}=-\mathbb P\operatorname{div}(u^{(m)}\otimes u^{(m)}).
\]
For every \(\Phi\in H^3(\mathbb{R}^3)\),
\begin{align*}
\left|\left\langle \operatorname{div}\left(u^{(m)}\otimes u^{(m)}\right),\Phi\right\rangle\right|
=
\left|\int_{\mathbb{R}^3}\left(u^{(m)}\otimes u^{(m)}\right):\nabla\Phi\,dx\right|
&\le
\left\|u^{(m)}\right\|_{L^2}^2\|\nabla\Phi\|_{L^\infty}
\\
&\leq C
\left\|u^{(m)}\right\|_{L^2}^2\|\Phi\|_{H^3}.
\end{align*}
Hence
\[
\sup_{m\ge 1}
\left\|\partial_t u^{(m)}\right\|_{L^\infty(0,T;H^{-3}(\mathbb{R}^3))}\le C.
\]
Since multiplication by \(\chi_R\) is bounded on \(H^{-3}(\mathbb{R}^3)\),
\[
\sup_{m\ge 1}
\left\|\partial_t\left(\chi_R u^{(m)}\right)\right\|_{L^\infty(0,T;H^{-3}(\mathbb{R}^3))}\le C(R).
\]

Hence, by the Aubin--Lions lemma,
\(\{\chi_R u^{(m)}\}_{m\ge 1}\) is relatively compact in
\(L^2((0,T);L^2(\mathbb{R}^3))\).
Combining this with
\eqref{eq:finite_ring_loc_rel_vor},
\eqref{eq:finite_ring_energy_bd}, \eqref{eq:finite_ring_local_H1_bd},
and Banach--Alaoglu, and then using a diagonal argument over \(R\in\mathbb{N}\), we
obtain a subsequence, still denoted by \(m\), and functions
\[
u^{(\infty)}\in L^\infty\big(0,T;L^2(\mathbb{R}^3)\big)
\cap L^\infty\big(0,T;H^1_{\mathrm{loc}}(\mathbb{R}^3)\big),
\qquad
\xi^{(\infty)}\in L^\infty\big(0,T;L^2_{\mathrm{loc}}(\mathbb{R}^3)\big),
\]
such that
\[
u^{(m)}\overset{*}{\rightharpoonup} u^{(\infty)}
\quad\text{in }L^\infty\big(0,T;L^2(\mathbb{R}^3)\big),
\]
\[
u^{(m)}\to u^{(\infty)}
\quad\text{strongly in }L^2\big((0,T)\times B_R\big)
\quad\text{for every }R>0,
\]
and
\[
\xi^{(m)}\overset{*}{\rightharpoonup} \xi^{(\infty)}
\quad\text{in }L^\infty\big(0,T;L^2(B_R)\big)
\quad\text{for every }R>0.
\]
Since the class of axisymmetric no-swirl vector fields is closed under local
\(L^2\)-convergence, \(u^{(\infty)}\) is still axisymmetric without swirl. 
Similarly, since each \(\xi^{(m)}\) is axisymmetric and the class of axisymmetric
scalar functions is closed under weak convergence in \(L^2_{\mathrm{loc}}\),
the limit \(\xi^{(\infty)}\) is axisymmetric.
Set \(\omega^{(\infty)}:=r\xi^{(\infty)}\).
Then, for every \(R>0\),
\[
\omega^{(m)}=r\xi^{(m)}\overset{*}{\rightharpoonup}\omega^{(\infty)}
\quad\text{in }L^\infty\big(0,T;L^2(B_R)\big).
\]

We now pass to the limit in the transport identity. For each \(m\), smoothness of
\((u^{(m)},\omega^{(m)})\) gives
\[
\int_0^T\int_{\mathbb{R}^3}\xi^{(m)}(x,t)
\bigl(\partial_t\psi(x,t)+u^{(m)}(x,t)\cdot\nabla\psi(x,t)\bigr)\,dx\,dt
+\int_{\mathbb{R}^3}\xi_0^{(m)}(x)\psi(x,0)\,dx=0
\]
for every \(\psi\in C_c^\infty\big(\mathbb{R}^3\times[0,T)\big)\). Fix such a \(\psi\), and choose
\(R>0\) so that \(\operatorname{supp}\psi\subset B_R\times[0,T)\). The terms involving
\(\partial_t\psi\) and the initial data pass to the limit by the weak-\(*\) convergence
of \(\xi^{(m)}\) in \(L^\infty\big(0,T;L^2(B_R)\big)\) and the strong convergence
\(\xi_0^{(m)}\to\xi_0^{(\infty)}\) in \(L^1(\mathbb{R}^3)\). For the nonlinear term, write
\[
\int_0^T\int_{B_R}\xi^{(m)}u^{(m)}\cdot\nabla\psi
-
\int_0^T\int_{B_R}\xi^{(\infty)} u^{(\infty)}\cdot\nabla\psi
=
I_m+J_m,
\]
where
\[
I_m:=\int_0^T\int_{B_R}\xi^{(m)}(u^{(m)}-u^{(\infty)})\cdot\nabla\psi,
\qquad
J_m:=\int_0^T\int_{B_R}(\xi^{(m)}-\xi^{(\infty)})u^{(\infty)}\cdot\nabla\psi.
\]
Since \(\xi^{(m)}\) is uniformly bounded in \(L^2\big((0,T)\times B_R\big)\) and
\(u^{(m)}\to u\) strongly in \(L^2\big((0,T)\times B_R\big)\), we have \(I_m\to 0\). Also
\(u\cdot\nabla\psi\in L^1\big(0,T;L^2(B_R)\big)\), so the weak-\(*\) convergence of
\(\xi^{(m)}\) implies \(J_m\to 0\). Hence
\[
\int_0^T\int_{\mathbb{R}^3}\xi^{(\infty)}(x,t)
\bigl(\partial_t\psi(x,t)+u^{(\infty)}(x,t)\cdot\nabla\psi(x,t)\bigr)\,dx\,dt
+\int_{\mathbb{R}^3}\xi_0^{(\infty)}(x)\psi(x,0)\,dx=0.
\]
Equivalently,
\[
\int_0^T\int_{\mathbb{R}^3}\frac{\omega^{(\infty)}(x,t)}{r}
\bigl(\partial_t\psi(x,t)+u^{(\infty)}(x,t)\cdot\nabla\psi(x,t)\bigr)\,dx\,dt
+\int_{\mathbb{R}^3}\frac{\omega_0^{(\infty)}(x)}{r}\psi(x,0)\,dx=0.
\]

It remains to verify the curl and divergence conditions. For each \(m\),
\[
\operatorname{div}u^{(m)}=0,
\qquad
\operatorname{curl}u^{(m)}=\omega^{(m)}e_\theta
=
\left(-x_2\xi^{(m)},x_1\xi^{(m)},0\right)
\]
in distributions. Passing to the limit using the strong local \(L^2\) convergence of
\(u^{(m)}\) and the weak-\(*\) convergence of \(\xi^{(m)}\) on every ball gives
\[
\operatorname{div}u^{(\infty)}=0,
\qquad
\operatorname{curl}u^{(\infty)}=\omega^{(\infty)} e_\theta
\]
in \(\mathcal D'\big((0,T)\times\mathbb{R}^3\big)\).

Finally, for every \(T'<T\) and \(R>0\), the bounds
\[
u^{(\infty)}\in L^\infty\big(0,T;L^2(\mathbb R^3)\big),
\qquad
\xi^{(\infty)}\in L^\infty\big(0,T;L^2(B_R)\big)
\]
imply
\[
\xi^{(\infty)}\in L^1\big((0,T')\times B_R\big),
\qquad
u^{(\infty)}\xi^{(\infty)}\in L^1\big((0,T')\times B_R\big).
\]
Since the time slice \(\{0\}\times\mathbb R^3\) has measure zero, this gives
\[
\xi^{(\infty)}\in L^1_{\mathrm{loc}}\big([0,T)\times\mathbb R^3\big),
\qquad
u^{(\infty)}\xi^{(\infty)}\in L^1_{\mathrm{loc}}\big([0,T)\times\mathbb R^3\big).
\]
Thus all conditions in Definition~\ref{def:dist_sol_rel_vor} are satisfied, and
\((u^{(\infty)},\omega^{(\infty)})\) is a distributional solution of \eqref{eq:IVP_rel_vor} on \((0,T)\)
with initial data \(\omega_0^{(\infty)}\).
\end{proof}

\begin{corollary}[Diagonal extraction: a single global distributional solution]
\label{cor:global_distributional_limit}
There exist a subsequence \(m_j\to\infty\) and axisymmetric functions
\[
u^{(\infty)}\in L^\infty_{\mathrm{loc}}\bigl([0,\infty);L^2(\mathbb{R}^3)\bigr)
\cap L^\infty_{\mathrm{loc}}\bigl([0,\infty);H^1_{\mathrm{loc}}(\mathbb{R}^3)\bigr),
\qquad
\xi^{(\infty)}\in L^\infty_{\mathrm{loc}}\bigl([0,\infty);L^2_{\mathrm{loc}}(\mathbb{R}^3)\bigr),
\]
such that, writing \(\omega^{(\infty)}:=r\xi^{(\infty)}\), the following holds: for every \(T>0\), as $j\to\infty,$
\begin{align*}
u^{(m_j)} &\overset{*}{\rightharpoonup} u^{(\infty)}
\quad\text{in }L^\infty\bigl((0,T);L^2(\mathbb{R}^3)\bigr),
\\
u^{(m_j)} &\to u^{(\infty)}
\quad\text{strongly in }L^2\bigl((0,T)\times B_R\bigr)
\quad\text{for every }R>0,
\\
\xi^{(m_j)} &\overset{*}{\rightharpoonup} \xi^{(\infty)}
\quad\text{in }L^\infty\bigl((0,T);L^2(B_R)\bigr)
\quad\text{for every }R>0,
\end{align*}
and \((u^{(\infty)},\omega^{(\infty)})\) is an (axisymmetric no-swirl) distributional solution of
\eqref{eq:IVP_rel_vor} on \((0,\infty)\) with initial data \(\omega_0^{(\infty)}\) in the sense of
Definition~\ref{def:dist_sol_rel_vor}.
\end{corollary}

\begin{proof}
For \(N=1\), apply Proposition~\ref{prop:finite_ring_limit_distributional} with \(T=1\).
This gives a subsequence \(\{m_j^{(1)}\}_{j\ge1}\) and a limit
\((u^{[1]},\xi^{[1]})\) on \((0,1)\).

Assume by induction that, for some \(N\ge2\), we have already chosen a subsequence
\(\{m_j^{(N-1)}\}_{j\ge1}\) of positive integers such that the corresponding finite-ring
solutions converge on \((0,N-1)\). Relabel this subsequence as a sequence and apply
Proposition~\ref{prop:finite_ring_limit_distributional} with \(T=N\). We obtain a further
subsequence \(\{m_j^{(N)}\}_{j\ge1}\subset\{m_j^{(N-1)}\}_{j\ge1}\) and a limit
\((u^{[N]},\xi^{[N]})\) on \((0,N)\).

Thus we have constructed nested subsequences
\[
\{m_j^{(1)}\}\supset \{m_j^{(2)}\}\supset \cdots .
\]
Define the diagonal subsequence
\[
m_j:=m_j^{(j)}.
\]
Fix \(N\in\mathbb N\). Then for every \(j\ge N\), the index \(m_j\) belongs to
\(\{m_\ell^{(N)}\}_{\ell\ge1}\), because the subsequences are nested. Therefore
\(\{m_j\}_{j\ge N}\) is a subsequence of \(\{m_\ell^{(N)}\}_{\ell\ge1}\), and so the diagonal
subsequence converges on \((0,N)\) to \((u^{[N]},\xi^{[N]})\) with the same properties as in
Proposition~\ref{prop:finite_ring_limit_distributional}.

We now check compatibility on overlaps. Since
\(\{m_j^{(N+1)}\}_{j\ge1}\subset\{m_j^{(N)}\}_{j\ge1}\), the sequence
\(\{u^{(m_j^{(N+1)})}\}_{j\ge1}\) converges on \((0,N)\) both to \(u^{[N]}\) and to
\(u^{[N+1]}\!\restriction_{(0,N)}\) in the strong topology
\(L^2((0,N)\times B_R)\) for every \(R>0\). By uniqueness of strong limits,
\[
u^{[N+1]}=u^{[N]}
\qquad\text{a.e. on }(0,N)\times B_R
\]
for every \(R>0\), hence a.e. on \((0,N)\times\mathbb R^3\). Likewise,
\(\xi^{[N+1]}=\xi^{[N]}\) a.e. on \((0,N)\times B_R\) for every \(R>0\), by uniqueness
of the weak-\(*\) limit in \(L^\infty((0,N);L^2(B_R))\).

Therefore the local limits patch together. Define
\[
u^{(\infty)}:=u^{[N]},\qquad \xi^{(\infty)}:=\xi^{[N]}
\qquad\text{on }(0,N)\times\mathbb R^3.
\]
This is well-defined by the compatibility just proved, and the stated local-in-time bounds
follow from the bounds on each finite interval \((0,N)\). Finally, for any \(T>0\), choose
\(N\in\mathbb N\) with \(N>T\). Since \((u^{[N]},r\xi^{[N]})\) is a distributional solution on
\((0,N)\), its restriction to \((0,T)\) gives the desired statement.
\end{proof}

\begin{proposition}[Bounded infinite-ring limits are Yudovich-type]
\label{prop:bounded_limit_upgrade}
Let \((u^{(\infty)},\omega^{(\infty)})\) be the global distributional solution furnished by Corollary~\ref{cor:global_distributional_limit}.
 Set
$
\xi^{(\infty)}=\omega^{(\infty)}/r.$
Assume in addition that for some \(T\in(0,\infty)\),
\[
\omega^{(\infty)}\in L^\infty\bigl((0,T)\times \mathbb R^3\bigr).
\]
Then
\[
\omega^{(\infty)}\in L^\infty\bigl(0,T;L^1\cap L^\infty(\mathbb R^3)\bigr).
\]
In particular, \((u^{(\infty)},\omega^{(\infty)})\) is a Yudovich-type weak solution on \((0,T)\).
\end{proposition}

\begin{proof}[Proof of Proposition~\ref{prop:bounded_limit_upgrade}.]
By Corollary~\ref{cor:global_distributional_limit}, we already know that
\[
u^{(\infty)}\in L^\infty\bigl(0,T;L^2(\mathbb R^3)\bigr)
\cap L^\infty\bigl(0,T;H^1_{\mathrm{loc}}(\mathbb R^3)\bigr),
\qquad
\xi^{(\infty)}\in L^\infty\bigl(0,T;L^2_{\mathrm{loc}}(\mathbb R^3)\bigr),
\]
and that \((u^{(\infty)},\omega^{(\infty)})\) satisfies the distributional transport identity for \(\xi^{(\infty)}=\omega^{(\infty)}/r\).

Since \(\boldsymbol{\omega}^{(\infty)}:=\omega^{(\infty)} e_\theta\in L^\infty((0,T)\times\mathbb R^3)\) and
\[
\operatorname{curl}u^{(\infty)}=\boldsymbol{\omega}^{(\infty)},
\qquad
\operatorname{div}u^{(\infty)}=0
\]
in distributions, we can apply Lemma~\ref{lem:flow_cauchy_bd} to obtain a unique flow
map \(Y\), together with the Cauchy formula~\eqref{eq:cauchy_omega_common} and the
transported representative~\eqref{eq:cauchy_xi_common}.

Since \(Y_t\) is measure-preserving,
\[
\left\|\xi^{(\infty)}(t)\right\|_{L^1\cap L^2(\mathbb R^3)}
=
\left\|\xi_0^{(\infty)}\right\|_{L^1\cap L^2(\mathbb R^3)}
\qquad\text{for every }t\in[0,T].
\]
Also, by Lemma~\ref{lem:flow_cauchy_bd}, \(u^{(\infty)}\in L^\infty((0,T)\times\mathbb R^3)\), so
trajectories move at uniformly bounded speed. Since \(\xi_0^{(\infty)}\) has compact support, there
exists \(R_*>0\) such that
\[
Y_t\left(\operatorname{supp}\xi_0^{(\infty)}\right)\subset B_{R_*}
\qquad\text{for all }t\in[0,T].
\]
Therefore,
\[
\left\|\omega^{(\infty)}(t)\right\|_{L^1(\mathbb R^3)}
=
\left\|r\,\xi^{(\infty)}(t)\right\|_{L^1(\mathbb R^3)}
\le
R_*\,\left\|\xi^{(\infty)}(t)\right\|_{L^1(\mathbb R^3)}
=
R_*\,\left\|\xi_0^{(\infty)}\right\|_{L^1(\mathbb R^3)}
\]
for every \(t\in[0,T]\). Hence
\[
\omega^{(\infty)}\in L^\infty\bigl(0,T;L^1\cap L^\infty(\mathbb R^3)\bigr).
\]
Since \(\operatorname{div}u^{(\infty)}=0\), \(\operatorname{curl}u^{(\infty)}=\omega^{(\infty)} e_\theta\) in distributions, and
\(u^{(\infty)}\in L^\infty(0,T;L^2(\mathbb R^3))\), the standard Biot--Savart representation identifies
\(u^{(\infty)}\) with the velocity induced by \(\omega^{(\infty)}\). Thus \((u^{(\infty)},\omega^{(\infty)})\) is a Yudovich-type weak
solution on \((0,T)\).
\end{proof}

\begin{remark}
\label{rem:diff_bd_dis_sol}
In Proposition~\ref{prop:bounded_limit_upgrade}, we upgrade the distributional solution furnished by Corollary~\ref{cor:global_distributional_limit} to a Yudovich-type weak solution provided the distributional solution is uniformly bounded. One may not upgrade any uniformly bounded distributional solution to a Yudovich-type weak solution. In that case, the corresponding velocity field $u$ may not be uniformly bounded. Then the associated flow map might escape to infinity in finite time. 
\end{remark}

\begin{proof}[Proof of Theorem~\ref{thm:instant_blow-up}-(2)]
Let $\omega_0^{(\infty) }$ be the same initial data as in the first paragraph of the proof for non-existence (\S~\ref{subsec:non_existence}).
Let \((u^{(\infty)},\omega^{(\infty)})\) be the global distributional solution furnished by
Corollary~\ref{cor:global_distributional_limit}. We claim that
\[
\omega^{(\infty)}\notin L^\infty\bigl((0,T)\times\mathbb R^3\bigr)
\qquad\text{for every }T>0.
\]
Indeed, if this were false, then for some \(T>0\) we would have
\[
\omega^{(\infty)}\in L^\infty\bigl((0,T)\times\mathbb R^3\bigr).
\]
Applying Proposition~\ref{prop:bounded_limit_upgrade} to the restriction of
\((u^{(\infty)},\omega^{(\infty)})\) to \([0,T]\), we would obtain
\[
\omega^{(\infty)}\in L^\infty\bigl(0,T;L^1\cap L^\infty(\mathbb R^3)\bigr),
\]
so that \((u^{(\infty)},\omega^{(\infty)})\) is a Yudovich-type weak solution on \((0,T)\). This contradicts
Theorem~\ref{thm:instant_blow-up}-(1). Therefore the claim is proved, and Theorem~\ref{thm:instant_blow-up}-(2) follows.
\end{proof}

\appendix

\section{Proof of the Danskin-type lemma and the Cauchy formula}
\label{app:two_lemmas}
This appendix is devoted to prove the Danskin-type Lemma~\ref{lem:envelope_radial_extrema} and the Cauchy formula~\eqref{eq:cauchy_xi_common}. 

For the Cauchy formula, we first state a proposition before proving it. 

For both the lemma and the proposition, our proofs are standard, but we could not find a reference that fits our situation. Hence, we prove them here in this appendix for completeness.

\begin{proof}[Proof of Lemma~\ref{lem:envelope_radial_extrema}]
Let
\[
M:=\sup_{x\in K, \ 0<t<T_0}|g(x,t)|.
\]
For every $x\in K$ and every $0\le s\le t\le T_0$,
\[
|Y^r(x,t)-Y^r(x,s)|
\le \int_s^t |g(x,\tau)|\,d\tau
\le M|t-s|.
\]
Hence
\[
|R_K(t)-R_K(s)|
\le \sup_{x\in K}|Y^r(x,t)-Y^r(x,s)|
\le M|t-s|,
\]
and similarly
\[
|r_K(t)-r_K(s)|
\le \sup_{x\in K}|Y^r(x,t)-Y^r(x,s)|
\le M|t-s|.
\]
Thus $R_K$ and $r_K$ are Lipschitz.

Choose a countable dense set $\{x_n\}_{n\ge 1}\subset K$. Let $E_0\subset (0,T_0)$ be a
full-measure set on which both $R_K$ and $r_K$ are differentiable. For each $n\ge 1$,
let $E_n\subset (0,T_0)$ be a full-measure set of right Lebesgue points of the function
$t\mapsto g(x_n,t)$, namely
\[
\lim_{h\downarrow 0}\frac{1}{h}\int_t^{t+h} g(x_n,s)\,ds = g(x_n,t)
\qquad
\text{for every }t\in E_n.
\]
Set
\[
E:=E_0\cap E_g \cap \bigcap_{n=1}^\infty E_n.
\]
Then $E$ still has full measure.

We claim that for every $t\in E$ and every $x\in K$,
\[
\lim_{h\downarrow 0}\frac{1}{h}\int_t^{t+h} g(x,s)\,ds = g(x,t).
\]
Indeed, fix $t\in E$ and $x\in K$, and choose $x_{n_j}\to x$. Then
\[
\begin{aligned}
\left|
\frac{1}{h}\int_t^{t+h} g(x,s)\,ds - g(x,t)
\right|
&\le
\frac{1}{h}\int_t^{t+h}|g(x,s)-g(x_{n_j},s)|\,ds \\
&\quad +
\left|
\frac{1}{h}\int_t^{t+h} g(x_{n_j},s)\,ds - g(x_{n_j},t)
\right| \\
&\quad + |g(x_{n_j},t)-g(x,t)| \\
&\le
2\rho(|x-x_{n_j}|) +
\left|
\frac{1}{h}\int_t^{t+h} g(x_{n_j},s)\,ds - g(x_{n_j},t)
\right|.
\end{aligned}
\]
First let $h\downarrow 0$, using that $t\in E_{n_j}$, and then let $j\to\infty$. This proves
the claim.

Now fix $t\in E$. Since $Y^r(\cdot,t)$ is continuous on the compact set $K$, the sets
$\mathcal M_K(t)$ and $\mathcal A_K(t)$ are nonempty compact. Since $g(\cdot,t)$ is
continuous on $K$, the extrema of $g(\cdot,t)$ over these sets are attained.

We first prove the formula for $R_K'(t)$. Let $x\in \mathcal M_K(t)$. Then for every $h>0$
small enough,
\[
R_K(t+h)\ge Y^r(x,t+h),
\]
hence
\[
\frac{R_K(t+h)-R_K(t)}{h}
\ge
\frac{Y^r(x,t+h)-Y^r(x,t)}{h}
=
\frac{1}{h}\int_t^{t+h} g(x,s)\,ds.
\]
Letting $h\downarrow 0$ and using the claim above, we obtain
\[
\liminf_{h\downarrow 0}\frac{R_K(t+h)-R_K(t)}{h}\ge g(x,t).
\]
Since this holds for every $x\in \mathcal M_K(t)$,
\[
\liminf_{h\downarrow 0}\frac{R_K(t+h)-R_K(t)}{h}
\ge
\max_{x\in \mathcal M_K(t)} g(x,t).
\]

For the reverse inequality, for each $h>0$ choose $x_h\in \mathcal M_K(t+h)$, so that
\[
R_K(t+h)=Y^r(x_h,t+h).
\]
By compactness of $K$, along any sequence $h_j\downarrow 0$ we may pass to a subsequence,
still denoted $h_j$, such that $x_{h_j}\to x_*\in K$. Since
\[
Y^r(x_{h_j},t+h_j)=R_K(t+h_j),
\]
continuity of $Y^r$ and of $R_K$ yields
\[
Y^r(x_*,t)=R_K(t),
\]
so $x_*\in \mathcal M_K(t)$. Moreover,
\begin{align*}
\frac{R_K(t+h_j)-R_K(t)}{h_j}
&=
\frac{Y^r(x_{h_j},t+h_j)-R_K(t)}{h_j}
\\
&\le
\frac{Y^r(x_{h_j},t+h_j)-Y^r(x_{h_j},t)}{h_j}
=
\frac{1}{h_j}\int_t^{t+h_j} g(x_{h_j},s)\,ds.
\end{align*}
Therefore,
\[
\begin{aligned}
\left|
\frac{1}{h_j}\int_t^{t+h_j} g(x_{h_j},s)\,ds - g(x_*,t)
\right|
&\le
\frac{1}{h_j}\int_t^{t+h_j}|g(x_{h_j},s)-g(x_*,s)|\,ds \\
&\quad +
\left|
\frac{1}{h_j}\int_t^{t+h_j} g(x_*,s)\,ds - g(x_*,t)
\right| \\
&\le
\rho(|x_{h_j}-x_*|) +
\left|
\frac{1}{h_j}\int_t^{t+h_j} g(x_*,s)\,ds - g(x_*,t)
\right|.
\end{aligned}
\]
Letting $j\to\infty$ and using the claim again, we obtain
\[
\limsup_{h\downarrow 0}\frac{R_K(t+h)-R_K(t)}{h}
\le g(x_*,t)
\le \max_{x\in \mathcal M_K(t)} g(x,t).
\]
Combining the lower and upper bounds, and using that $R_K$ is differentiable at $t$,
we conclude
\[
R_K'(t)=\max_{x\in \mathcal M_K(t)} g(x,t)
\]
since the right derivative of $R_K$ equals to the derivative.

Finally, apply the already proved formula to the function $-Y^r$. Since
\[
-r_K(t)=\max_{x\in K}(-Y^r(x,t)),
\]
and the corresponding derivative is $-g$, we get
\[
-r_K'(t)=\max_{x\in \mathcal A_K(t)}(-g(x,t))
= - \min_{x\in \mathcal A_K(t)} g(x,t),
\]
that is,
\[
r_K'(t)=\min_{x\in \mathcal A_K(t)} g(x,t).
\]
This completes the proof of Lemma~\ref{lem:envelope_radial_extrema}.
\end{proof}

Now we state the Cauchy formula under assumptions that fit our situation. 
\begin{proposition}[Cauchy formula of the transport equation along a given flow]
\label{prop:transport_by_flow}
Let $T_0>0$. Let
\[
u\in L^\infty\bigl((0,T_0)\times\mathbb R^3\bigr)
\]
be divergence-free, and assume that there exists a flow map
\[
Y:\mathbb R^3\times[0,T_0]\to\mathbb R^3
\]
such that:

\smallskip
\noindent
(i) for every $a\in\mathbb R^3$, the map $t\mapsto Y(a,t)$ is absolutely continuous and
\[
\partial_tY(a,t)=u\big(Y(a,t),t\big)
\qquad\text{for a.e. }t\in(0,T_0),
\qquad
Y(a,0)=a;
\]

\smallskip
\noindent
(ii) for every $t\in[0,T_0]$, the map $Y_t:=Y(\cdot,t)$ is a measure-preserving
homeomorphism of $\mathbb R^3$ onto itself;

\smallskip
\noindent
(iii) there exists $\gamma\in(0,1]$ such that for every compact set $K\subset\mathbb R^3$,
\[
\sup_{t\in[0,T_0]}
\Bigl(
[Y_t]_{C^\gamma(K)}+[Y_t^{-1}]_{C^\gamma(K)}
\Bigr)
<\infty .
\]

Let $\xi_0\in L^1(\mathbb R^3)$, and let
\[
\xi\in L^1_{\mathrm{loc}}\bigl([0,T_0)\times\mathbb R^3\bigr),
\qquad
u\xi\in L^1_{\mathrm{loc}}\bigl([0,T_0)\times\mathbb R^3\bigr)
\]
satisfy
\begin{equation}
\label{eq:transport_distributional_appendix}
\int_0^{T_0}\int_{\mathbb R^3}
\xi(x,t)\bigl(\partial_t\psi(x,t)+u(x,t)\cdot\nabla\psi(x,t)\bigr)\,dx\,dt
+\int_{\mathbb R^3}\xi_0(x)\psi(x,0)\,dx
=0
\end{equation}
for every $\psi\in C_c^\infty(\mathbb R^3\times[0,T_0))$.

Assume in addition that there exists a constant $C_u>0$ such that
\begin{equation}
\label{eq:u_loglip_appendix}
|u(x,t)-u(y,t)|
\le
C_u |x-y|\log\!\left(2+\frac{1}{|x-y|}\right)
\end{equation}
for all $x,y\in\mathbb R^3$ and for a.e. $t\in(0,T_0)$.

Then
\[
\xi(x,t)=\xi_0\bigl(Y_t^{-1}(x)\bigr)
\qquad
\text{for a.e. }(x,t)\in\mathbb R^3\times(0,T_0).
\]
\end{proposition}

\begin{proof}
Define
\[
\widetilde{\xi}(x,t):=\xi_0\bigl(Y_t^{-1}(x)\bigr).
\]
Since each $Y_t$ is measure-preserving,
\[
\|\widetilde{\xi}(t)\|_{L^1(\mathbb R^3)}=\|\xi_0\|_{L^1(\mathbb R^3)}
\qquad\text{for every }t\in[0,T_0],
\]
hence
\[
\widetilde{\xi}\in L^\infty\bigl(0,T_0;L^1(\mathbb R^3)\bigr)
\subset L^1_{\mathrm{loc}}\bigl((0,T_0)\times\mathbb R^3\bigr).
\]

Fix $\eta\in C_c^\infty(\mathbb R^3)$ and $\chi\in C_c^\infty(0,T_0)$. Set
\[
K:=\operatorname{supp}\eta,
\qquad
M:=\|u\|_{L^\infty((0,T_0)\times\mathbb R^3)}.
\]
Define
\[
K_*:=\{x\in\mathbb R^3:\operatorname{dist}(x,K)\le MT_0\},
\qquad
K_{**}:=\{x\in\mathbb R^3:\operatorname{dist}(x,K)\le 2MT_0\}.
\]
For $(x,t)\in\mathbb R^3\times[0,T_0]$, define
\[
\Phi(x,t):=
\int_t^{T_0}\chi(\tau)\,\eta\Bigl(Y_\tau \big(Y_t^{-1}(x)\big)\Bigr)\,d\tau.
\]

We first record some elementary properties of $\Phi$.

\smallskip
\noindent\emph{Step 1: support and H\"older regularity of $\Phi$.}
If $\Phi(x,t)\neq 0$, then there exists $\tau\in \operatorname{supp}\chi$ such that
$Y_\tau(Y_t^{-1}(x))\in K$. Writing $a:=Y_t^{-1}(x)$ and using the flow equation,
\[
|x-Y_\tau(a)|=|Y_t(a)-Y_\tau(a)|
\le M|t-\tau|
\le MT_0,
\]
hence $x\in K_*$. Therefore
\[
\operatorname{supp}\Phi(\cdot,t)\subset K_*
\qquad\text{for every }t\in[0,T_0].
\]

Next, if $x\in K_*$ and $a=Y_t^{-1}(x)$, then
\[
\operatorname{dist}(a,K)\le \operatorname{dist}(a,x)+\operatorname{dist}(x,K)\le 2MT_0,
\]
so
\[
Y_t^{-1}(K_*)\subset K_{**}
\qquad\text{for every }t\in[0,T_0].
\]
By assumption (iii), there exists $\beta:=\gamma^2\in(0,1]$ and a constant $C>0$
depending only on $K_*,K_{**}$ such that for all $t,\tau\in[0,T_0]$,
\[
\left|Y_\tau\big(Y_t^{-1}(x_1)\big)-Y_\tau\big(Y_t^{-1}(x_2)\big)
\right|
\le C |x_1-x_2|^\beta,
\qquad x_1,x_2\in K_*.
\]
Hence
\[
\bigl[\eta\circ Y_\tau\circ Y_t^{-1}\bigr]_{C^\beta(K_*)}
\le C\|\nabla\eta\|_{L^\infty},
\]
uniformly in $t,\tau\in[0,T_0]$. It follows that
\begin{equation}
\label{eq:Phi_holder_appendix}
\sup_{t\in[0,T_0]}[\Phi(\cdot,t)]_{C^\beta(K_*)}
\le C\|\chi\|_{L^1(0,T_0)}\|\nabla\eta\|_{L^\infty}.
\end{equation}

Since $\supp \Phi(\cdot,t)\subset K_*$ and $K_*$ is closed, we have
$\Phi(\cdot,t)=0$ on $\mathbb R^3\setminus K_*$. By continuity, it also follows that
$\Phi(\cdot,t)=0$ on $\partial K_*$. Consequently,
\[
[\Phi(\cdot,t)]_{C^\beta(\mathbb R^3)}
\le
[\Phi(\cdot,t)]_{C^\beta(K_*)}.
\]
Indeed, if $x,y\in K_*$ or $x,y\notin K_*$, the claim is immediate. If $x\in K_*$ and
$y\notin K_*$, choose $z\in [x,y]\cap\partial K_*$
 where $[x,y]$ denotes the straight line segment connecting $x$ to $y$.
Then $\Phi(z,t)=\Phi(y,t)=0$, so
\[
|\Phi(x,t)-\Phi(y,t)|
=
|\Phi(x,t)-\Phi(z,t)|
\le
[\Phi(\cdot,t)]_{C^\beta(K_*)}\,|x-z|^\beta
\le
[\Phi(\cdot,t)]_{C^\beta(K_*)}\,|x-y|^\beta.
\]
The case $x\notin K_*$ and $y\in K_*$ is symmetric. Combining this with
\eqref{eq:Phi_holder_appendix}, we obtain
\begin{equation}
\label{eq:Phi_holder_global_appendix}
\sup_{t\in[0,T_0]}[\Phi(\cdot,t)]_{C^\beta(\mathbb R^3)}
\le
C\|\chi\|_{L^1(0,T_0)}\|\nabla\eta\|_{L^\infty}.
\end{equation}

\smallskip
\noindent\emph{Step 2: $\Phi$ solves a backward transport equation.}
Let $\varphi\in C_c^\infty(\mathbb R^3\times(0,T_0))$. By the change of variables
$x=Y_t(a)$ and the measure-preserving property of $Y_t$,
\begin{multline*}
\int_0^{T_0}\int_{\mathbb R^3}
\Phi(x,t)\bigl(\partial_t\varphi(x,t)+u(x,t)\cdot\nabla\varphi(x,t)\bigr)\,dx\,dt 
\\
=
\int_{\mathbb R^3}\int_0^{T_0}
\left(\int_t^{T_0}\chi(\tau)\eta(Y_\tau(a))\,d\tau\right)
\frac{d}{dt}\varphi(Y_t(a),t)\,dt\,da .
\end{multline*}
Set
\[
G_a(t):=\int_t^{T_0}\chi(\tau)\eta(Y_\tau(a))\,d\tau.
\]
Since $G_a$ is absolutely continuous and
\[
G_a'(t)=-\chi(t)\eta(Y_t(a))
\qquad\text{for a.e. }t,
\]
integration by parts in $t$ gives
\begin{align*}
\int_0^{T_0}\int_{\mathbb R^3}
\Phi(x,t)\bigl(\partial_t\varphi(x,t)+u(x,t)\cdot\nabla\varphi(x,t)\bigr)\,dx\,dt 
&=
-\int_{\mathbb R^3}\int_0^{T_0}G_a'(t)\varphi(Y_t(a),t)\,dt\,da \\
&=
\int_{\mathbb R^3}\int_0^{T_0}\chi(t)\eta(Y_t(a))\varphi\big(Y_t(a),t\big)\,dt\,da \\
&=
\int_0^{T_0}\int_{\mathbb R^3}\chi(t)\eta(x)\varphi(x,t)\,dx\,dt .
\end{align*}
Therefore, in distributions on $(0,T_0)\times\mathbb R^3$,
\begin{equation}
\label{eq:Phi_backward_appendix}
\partial_t\Phi+u\cdot\nabla\Phi=-\chi(t)\eta(x).
\end{equation}

\smallskip
\noindent\emph{Step 3: spatial mollification and commutator.}
Let $\rho\in C_c^\infty(\mathbb R^3)$ be a standard nonnegative mollifier supported in
$B_1(0)$ with $\int\rho=1$, and set
\[
\rho_\varepsilon(x):=\varepsilon^{-3}\rho(x/\varepsilon),
\qquad
\Phi_\varepsilon:=\rho_\varepsilon * \Phi,
\qquad
\eta_\varepsilon:=\rho_\varepsilon * \eta.
\]
Then $\Phi_\varepsilon$ is smooth in $x$, compactly supported in $K_*+B_\varepsilon$, and
from \eqref{eq:Phi_backward_appendix},
\[
\partial_t\Phi_\varepsilon+\rho_\varepsilon*(u\cdot\nabla\Phi)
=
-\chi(t)\eta_\varepsilon .
\]
Hence
\begin{equation}
\label{eq:Phi_eps_eq_appendix}
\partial_t\Phi_\varepsilon+u\cdot\nabla\Phi_\varepsilon
=
-\chi(t)\eta_\varepsilon+R_\varepsilon,
\end{equation}
where
\[
R_\varepsilon
:=
u\cdot\nabla\Phi_\varepsilon-\rho_\varepsilon*(u\cdot\nabla\Phi).
\]
Using $\operatorname{div}u=0$, we may rewrite $R_\varepsilon$ as
\[
R_\varepsilon(x,t)
=
\int_{\mathbb R^3}
\nabla\rho_\varepsilon(y)\cdot\bigl(u(x,t)-u(x-y,t)\bigr)
\bigl(\Phi(x-y,t)-\Phi(x,t)\bigr)\,dy .
\]
Indeed, the term containing $\Phi(x,t)$ alone vanishes because
\[
\int_{\mathbb R^3}\nabla\rho_\varepsilon(y)\cdot\bigl(u(x,t)-u(x-y,t)\bigr)\,dy
=
-\int_{\mathbb R^3}\rho_\varepsilon(y)\,\operatorname{div}u(x-y,t)\,dy
=0.
\]

If $x\notin K_*+\overline B_\varepsilon$, then $\operatorname{dist}(x,K_*)>\varepsilon$. Since
$\supp \nabla\rho_\varepsilon\subset B_\varepsilon(0)$, for every
$y\in\supp \nabla\rho_\varepsilon$ we have
\[
\operatorname{dist}(x-y,K_*)\ge \operatorname{dist}(x,K_*)-|y|>0.
\]
Hence $\Phi(x,t)=\Phi(x-y,t)=0$, and therefore $R_\varepsilon(x,t)=0$. Thus
\begin{equation}
\label{eq:R_eps_support_appendix}
\supp R_\varepsilon(\cdot,t)\subset K_*+\overline B_\varepsilon
\qquad\text{for every }t\in[0,T_0].
\end{equation}

Using \eqref{eq:u_loglip_appendix} and \eqref{eq:Phi_holder_global_appendix}, we obtain
for all $(x,t)\in\mathbb R^3\times(0,T_0)$
\begin{align*}
|R_\varepsilon(x,t)|
&\le
C[\Phi(\cdot,t)]_{C^\beta(\mathbb R^3)}
\int_{\mathbb R^3}
|\nabla\rho_\varepsilon(y)|\,|y|^\beta\,|u(x,t)-u(x-y,t)|\,dy \\
&\le
C\int_{|y|<\varepsilon}
|\nabla\rho_\varepsilon(y)|\,|y|^{1+\beta}
\log\!\left(2+\frac1{|y|}\right)\,dy \\
&\le
C\,\varepsilon^\beta\log\!\left(2+\frac1\varepsilon\right).
\end{align*}
Therefore
\begin{equation}
\label{eq:R_eps_to_zero_appendix}
\|R_\varepsilon\|_{L^\infty((0,T_0)\times\mathbb R^3)}
\longrightarrow 0
\qquad\text{as }\varepsilon\to 0.
\end{equation}

\smallskip
\noindent\emph{Step 4: testing the weak formulation with $\Phi_\varepsilon$.}
Since $\chi\in C_c^\infty(0,T_0)$, there exists $\delta_0\in(0,T_0)$ such that
\[
\chi(t)=0
\qquad\text{for }t\in[T_0-\delta_0,T_0].
\]
Hence, by the definition of $\Phi$,
\[
\Phi(x,t)=0
\qquad\text{for all }(x,t)\in\mathbb R^3\times[T_0-\delta_0,T_0],
\]
and therefore the same holds for $\Phi_\varepsilon$.

Also, \eqref{eq:Phi_eps_eq_appendix} yields
\[
\partial_t\Phi_\varepsilon
=
-u\cdot\nabla\Phi_\varepsilon-\chi(t)\eta_\varepsilon+R_\varepsilon
\in L^\infty\bigl((0,T_0)\times\mathbb R^3\bigr),
\]
so $\Phi_\varepsilon$ is compactly supported and belongs to
$W^{1,\infty}(\mathbb R^3\times[0,T_0])$.

Choose $\theta\in C_c^\infty((-1,1))$ such that
\[
\theta\equiv 1
\qquad\text{on }[-\tfrac12,\tfrac12],
\]
and define $\widetilde{\Phi}_\varepsilon$ on $\mathbb R^3\times\mathbb R$ by
\[
\widetilde{\Phi}_\varepsilon(x,t):=
\begin{cases}
\theta(t)\Phi_\varepsilon(x,0), & t<0,\\[1mm]
\Phi_\varepsilon(x,t), & 0\le t\le T_0,\\[1mm]
0, & t>T_0.
\end{cases}
\]
Then $\widetilde{\Phi}_\varepsilon\in W^{1,\infty}(\mathbb R^3\times\mathbb R)$ and has
compact support.

Let $\vartheta\in C_c^\infty(\mathbb R^4)$ be a standard mollifier, and set
\[
\vartheta_\sigma(z):=\sigma^{-4}\vartheta(z/\sigma),
\qquad
\Psi_{\varepsilon,\sigma}:=\vartheta_\sigma*\widetilde{\Phi}_\varepsilon.
\]
For $0<\sigma<\delta_0/4$, the restriction of $\Psi_{\varepsilon,\sigma}$ to
$\mathbb R^3\times[0,T_0)$ belongs to
$C_c^\infty(\mathbb R^3\times[0,T_0))$.

Moreover, as $\sigma\to0$,
\[
\Psi_{\varepsilon,\sigma}\to \Phi_\varepsilon
\qquad\text{uniformly on }\mathbb R^3\times[0,T_0],
\]
\[
\partial_t\Psi_{\varepsilon,\sigma}\to \partial_t\Phi_\varepsilon,
\qquad
\nabla\Psi_{\varepsilon,\sigma}\to \nabla\Phi_\varepsilon
\quad\text{a.e. on }\mathbb R^3\times(0,T_0),
\]
and
\[
\Psi_{\varepsilon,\sigma}(\cdot,0)\to \Phi_\varepsilon(\cdot,0)
\qquad\text{uniformly on }\mathbb R^3.
\]
These functions are uniformly bounded, and their supports are contained in one fixed compact
subset of $\mathbb R^3\times[0,T_0)$; likewise,
$\Psi_{\varepsilon,\sigma}(\cdot,0)$ are supported in one fixed compact subset of
$\mathbb R^3$.

Testing \eqref{eq:transport_distributional_appendix} with
$\Psi_{\varepsilon,\sigma}$ and passing to the limit $\sigma\to0$ by dominated convergence,
using $\xi,u\xi\in L^1_{\mathrm{loc}}$ and $\xi_0\in L^1$, we obtain
\[
0=
\int_0^{T_0}\int_{\mathbb R^3}
\xi\bigl(\partial_t\Phi_\varepsilon+u\cdot\nabla\Phi_\varepsilon\bigr)\,dx\,dt
+\int_{\mathbb R^3}\xi_0(x)\Phi_\varepsilon(x,0)\,dx .
\]
Using \eqref{eq:Phi_eps_eq_appendix},
\begin{equation}
\label{eq:key_identity_appendix}
\int_0^{T_0}\int_{\mathbb R^3}\xi(x,t)\chi(t)\eta_\varepsilon(x)\,dx\,dt
=
\int_0^{T_0}\int_{\mathbb R^3}\xi(x,t)R_\varepsilon(x,t)\,dx\,dt
+\int_{\mathbb R^3}\xi_0(x)\Phi_\varepsilon(x,0)\,dx .
\end{equation}

\smallskip
\noindent\emph{Step 5: passage to the limit.}
Because $\eta_\varepsilon\to\eta$ uniformly and $\operatorname{supp}\eta_\varepsilon$ stays
inside a fixed compact set, while $\xi\in L^1_{\mathrm{loc}}$, the left-hand side of
\eqref{eq:key_identity_appendix} converges to
\[
\int_0^{T_0}\int_{\mathbb R^3}\xi(x,t)\chi(t)\eta(x)\,dx\,dt.
\]
Also, by \eqref{eq:R_eps_support_appendix}, for $0<\varepsilon<1$ we have
\[
\supp R_\varepsilon(\cdot,t)\subset K_*+\overline B_1
\qquad\text{for every }t\in[0,T_0].
\]
Hence, using \eqref{eq:R_eps_to_zero_appendix} and $\xi\in L^1_{\mathrm{loc}}$, we obtain
\[
\left|
\int_0^{T_0}\int_{\mathbb R^3}\xi(x,t)R_\varepsilon(x,t)\,dx\,dt
\right|
\le
\|R_\varepsilon\|_{L^\infty((0,T_0)\times\mathbb R^3)}
\int_0^{T_0}\int_{K_*+\overline B_1}|\xi(x,t)|\,dx\,dt
\longrightarrow 0.
\]

Finally, since $\Phi(\cdot,0)$ is compactly supported and H\"older continuous,
$\Phi_\varepsilon(\cdot,0)\to\Phi(\cdot,0)$ uniformly, hence
\begin{align*}
\int_{\mathbb R^3}\xi_0(x)\Phi_\varepsilon(x,0)\,dx
&\longrightarrow
\int_{\mathbb R^3}\xi_0(a)\Phi(a,0)\,da \\
&=
\int_{\mathbb R^3}\xi_0(a)\int_0^{T_0}\chi(\tau)\eta\big(Y_\tau(a)\big)\,d\tau\,da \\
&=
\int_0^{T_0}\chi(\tau)\int_{\mathbb R^3}\xi_0(a)\eta\big(Y_\tau(a)\big)\,da\,d\tau 
=
\int_0^{T_0}\chi(\tau)\int_{\mathbb R^3}\widetilde{\xi}(x,\tau)\eta(x)\,dx\,d\tau,
\end{align*}
where in the last step we used the change of variables $x=Y_\tau(a)$ and the
measure-preserving property of $Y_\tau$.

Passing to the limit $\varepsilon\to0$ in \eqref{eq:key_identity_appendix}, we obtain
\[
\int_0^{T_0}\int_{\mathbb R^3}\xi(x,t)\chi(t)\eta(x)\,dx\,dt
=
\int_0^{T_0}\int_{\mathbb R^3}\widetilde{\xi}(x,t)\chi(t)\eta(x)\,dx\,dt.
\]
Since $\chi\in C_c^\infty(0,T_0)$ and $\eta\in C_c^\infty(\mathbb R^3)$ are arbitrary,
it follows that
\[
\xi=\widetilde{\xi}
\qquad\text{in }\mathcal D'((0,T_0)\times\mathbb R^3),
\]
hence a.e. on $(0,T_0)\times\mathbb R^3$ because both functions are locally integrable.
This proves
\[
\xi(x,t)=\xi_0\bigl(Y_t^{-1}(x)\bigr)
\qquad\text{for a.e. }(x,t)\in\mathbb R^3\times(0,T_0),
\]
and completes the proof.
\end{proof}

\section{ODE Models}
\label{app:ODE_models}
This appendix records the two auxiliary ODE analyses announced in
\S~\ref{ODE_intro}.  We retain the notation introduced there and only recall the
definitions needed below.  
These propositions are not used in the proof of the
main theorems; they serve only as to explain motivations behind the localized model.
This appendix provides rigorous justifications of the heuristic discussion in
\S~\ref{flattened_kernel_ODE_intro} and \S~\ref{variable_coeff_intro}.

\subsection{Sharp Dichotomy for the Flattened-Coefficient ODEs}
\label{HAodes}

We consider the flattened model \eqref{ode_simplified} from \S~\ref{flattened_kernel_ODE_intro}:
for $k\ge2$ with $x_k^{\mathrm{flat}}(0)=:x_k^0=\frac{\varepsilon}{k^\alpha}$.

\begin{proposition}\label{prop:flat_dichotomy}
Fix $0<\alpha<1$ and $0<\varepsilon\le1$. Then the unique positive solution of
\eqref{ode_simplified} with $x_k^{\mathrm{flat}}(0)=:x_k^0=\frac{\varepsilon}{k^\alpha}$ satisfies:
\begin{enumerate}
\item If $\alpha<\frac27$, then for every fixed $t>0$,
it holds $ x_k^{\mathrm{flat}}(t)\to\infty$ as $k\to\infty.$
\item If $\alpha>\frac27$, then for every fixed $t>0$,it holds $ x_k^{\mathrm{flat}}(t)\to 0$ as $k\to\infty.$
\end{enumerate}
At the borderline $\alpha=\frac27$, the family $\{x_k^{\mathrm{flat}}(t)\}_{k\ge1}$ stays
uniformly bounded for each fixed $t>0$.
\end{proposition}

\begin{proof}
Since the system is triangular and the right-hand side in \eqref{ode_simplified}
depends only on the lower modes, one constructs the solution inductively in $k$.
Moreover, $x_k^{\mathrm{flat}}(t)$ is positive and nondecreasing for every $k$.

Using the definition of $\Gamma_j^{\mathrm{flat}}$ together with $\frac{d}{dt}\log x_k^{\mathrm{flat}}(t)=S_{k-1}^{\mathrm{flat}}(t)$, it holds
\begin{equation}
\label{eq:app_flat_Gamma_prime}
\bigl(\Gamma_j^{\mathrm{flat}}\bigr)'(t)
=
-3\,\Gamma_j^{\mathrm{flat}}(t)\,S_{j-1}^{\mathrm{flat}}(t).
\end{equation}

\medskip
\noindent\textbf{Step 1: depletion identities.}
Differentiating $b_j^{\mathrm{flat}}$ 
and using $\frac{d}{dt}\log x_k^{\mathrm{flat}}(t)=S_{k-1}^{\mathrm{flat}}(t),$ and \eqref{eq:app_flat_Gamma_prime}, we obtain
\begin{equation}
\label{eq:app_flat_b_prime}
\bigl(b_j^{\mathrm{flat}}\bigr)'(t)
=
-5\,b_j^{\mathrm{flat}}(t)\,S_{j-1}^{\mathrm{flat}}(t)
\qquad (j\ge1).
\end{equation}
Here, $S_0^{\mathrm{flat}}\equiv 0.$
Summing over $1\le j\le k$ gives the equation \eqref{ode_bj} of $S_k^{\mathrm{flat}}$.
We now use the elementary identity
$S_j^{\mathrm{flat}}=S_{j-1}^{\mathrm{flat}}+b_j^{\mathrm{flat}},$
which yields
\[
\bigl(S_j^{\mathrm{flat}}\bigr)^2-\bigl(S_{j-1}^{\mathrm{flat}}\bigr)^2
=
2\,b_j^{\mathrm{flat}}\,S_{j-1}^{\mathrm{flat}}
+\bigl(b_j^{\mathrm{flat}}\bigr)^2.
\]
Summing this relation over $1\le j\le k$, we obtain
\[
2\sum_{j=1}^k b_j^{\mathrm{flat}}(t)\,S_{j-1}^{\mathrm{flat}}(t)
=
\bigl(S_k^{\mathrm{flat}}(t)\bigr)^2-\sum_{j=1}^k \bigl(b_j^{\mathrm{flat}}(t)\bigr)^2.
\]
Inserting this into the equation \eqref{ode_bj} of $S_{k}^{\mathrm{flat}}$, we arrive at \eqref{ode_Sk_quad}.

\medskip
\noindent\textbf{Step 2: norm inflation when $\alpha<\frac27$.}
Since $\sum_{j=1}^k (b_j^{\mathrm{flat}})^2\ge0$, \eqref{ode_Sk_quad} implies
\[
\bigl(S_k^{\mathrm{flat}}\bigr)'(t)\ge -\frac52\bigl(S_k^{\mathrm{flat}}(t)\bigr)^2.
\]
By comparison with the scalar ODE $U'=-\frac52U^2$, we obtain
\begin{equation*}
S_k^{\mathrm{flat}}(t)
\ge
\frac{1}{\frac52\,t+\frac1{S_k^{\mathrm{flat}}(0)}}
\qquad\text{for all }t\ge0.
\end{equation*}
Integrating in time yields
\begin{equation}
\label{eq:app_flat_intS_lower}
\int_0^t S_k^{\mathrm{flat}}(s)\,ds
\ge
\frac25\log\!\Bigl(1+\frac52\,t\,S_k^{\mathrm{flat}}(0)\Bigr).
\end{equation}

At time $t=0$ we have $\Gamma_j^{\mathrm{flat}}(0)=1$, hence
$
b_j^{\mathrm{flat}}(0)=x_j^0=\frac{\varepsilon}{j^\alpha}$ and $
S_k^{\mathrm{flat}}(0)
=
\varepsilon\sum_{j=1}^k \frac{1}{j^\alpha}.$
By the integral test, for $k\ge2$,
\begin{equation}
\label{eq:app_flat_S0_bounds}
\frac{\varepsilon}{1-\alpha}\bigl(k^{1-\alpha}-1\bigr)
\le
S_k^{\mathrm{flat}}(0)
\le
\varepsilon+\frac{\varepsilon}{1-\alpha}\bigl(k^{1-\alpha}-1\bigr)
\leq C(\alpha) \, \varepsilon\,k^{1-\alpha}.
\end{equation}
Using \eqref{ode_simplified}, \eqref{eq:app_flat_intS_lower}, and the lower bound in
\eqref{eq:app_flat_S0_bounds}, we get
\begin{align*}
x_k^{\mathrm{flat}}(t)
&=
x_k^0\exp\!\Bigl(\int_0^t S_{k-1}^{\mathrm{flat}}(s)\,ds\Bigr)
\nonumber\\
&\ge
\frac{\varepsilon}{k^\alpha}
\Bigl(1+\frac52\,t\,S_{k-1}^{\mathrm{flat}}(0)\Bigr)^{2/5}
\geq C(\varepsilon,t,\alpha)
k^{-\alpha}(k-1)^{\frac25(1-\alpha)}.
\end{align*}
Therefore, if $\frac25(1-\alpha)>\alpha,$
equivalently $\alpha<\frac27$, then $x_k^{\mathrm{flat}}(t)\to\infty$ as $k\to\infty$ for
every fixed $t>0$.

\medskip
\noindent\textbf{Step 3: decay when $\alpha>\frac27$.}
From \eqref{eq:app_flat_b_prime} and $S_{j-1}^{\mathrm{flat}}\ge0$, we have $\bigl(b_j^{\mathrm{flat}}\bigr)'(t)\le0.$
Hence $0<b_j^{\mathrm{flat}}(t)\le b_j^{\mathrm{flat}}(0)=\frac{\varepsilon}{j^\alpha}\le \varepsilon\le1.$
In particular,
\[
\sum_{j=1}^k \bigl(b_j^{\mathrm{flat}}(t)\bigr)^2
\le
\sum_{j=1}^k b_j^{\mathrm{flat}}(t)
=
S_k^{\mathrm{flat}}(t).
\]
Using this in \eqref{ode_Sk_quad}, we obtain
\begin{equation*}
\bigl(S_k^{\mathrm{flat}}\bigr)'(t)
\le
-\frac52\Bigl(\bigl(S_k^{\mathrm{flat}}(t)\bigr)^2-S_k^{\mathrm{flat}}(t)\Bigr).
\end{equation*}
Let $U_k$ solve
$U_k'(t)=-\frac52\bigl(U_k(t)^2-U_k(t)\bigr)$ with $
U_k(0)=S_k^{\mathrm{flat}}(0). $
By scalar comparison,
\[
S_k^{\mathrm{flat}}(t)\le U_k(t)\qquad\text{for all }t\ge0.
\]
The explicit solution is
\[
U_k(t)=\frac{1}{1-\Bigl(1-\frac{1}{S_k^{\mathrm{flat}}(0)}\Bigr)e^{-5t/2}},
\]
and therefore
\[
\int_0^t S_k^{\mathrm{flat}}(s)\,ds
\le
\int_0^t U_k(s)\,ds
=
\frac25\log\!\Bigl(1+\bigl(e^{5t/2}-1\bigr)S_k^{\mathrm{flat}}(0)\Bigr).
\]
Using again \eqref{ode_simplified} and \eqref{eq:app_flat_S0_bounds}, we conclude
\begin{align}
\label{eq:xk_upper_flat}
x_k^{\mathrm{flat}}(t)
&=
x_k^0\exp\!\Bigl(\int_0^t S_{k-1}^{\mathrm{flat}}(s)\,ds\Bigr)
\nonumber\\
&\le
\frac{\varepsilon}{k^\alpha}
\Bigl(1+\bigl(e^{5t/2}-1\bigr)S_{k-1}^{\mathrm{flat}}(0)\Bigr)^{2/5}
\leq
C(\varepsilon,t,\alpha)
k^{-\alpha}(k-1)^{\frac25(1-\alpha)}.
\end{align}
Thus, if $\alpha>\frac27$, then $\frac25(1-\alpha)-\alpha<0,$
and therefore $x_k^{\mathrm{flat}}(t)\to0$ as $k\to\infty$ for every fixed $t>0$.

At the borderline $\alpha=\frac27$, the bound
\eqref{eq:xk_upper_flat} show that $x_k^{\mathrm{flat}}(t)$ stays uniformly bounded in $k$
for each fixed $t>0$.
\end{proof}

\subsection{Frozen-Profile ODEs}
\label{frozen_pro_ODE_appdx}

We now treat the frozen-profile model \eqref{ode_full}
from \S~\ref{variable_coeff_intro}. Recall that
$\Lambda_{\mathrm{froz}}:(0,1]\to(0,\infty)$ is the frozen Biot--Savart coefficient introduced
there; in particular, $\Lambda_{\mathrm{froz}}$ is $C^1$, positive, and non-increasing.  We consider the frozen model \eqref{ode_full} with $x_k^{\mathrm{froz}}(0)=:x_k^0=\frac{\varepsilon}{k^\alpha}, k\geq 2$ and $x_1^{\mathrm{froz}}(t)\equiv x_1^0=\varepsilon$.

Recall the logarithmic-derivative correction $Q_{\mathrm{froz}}(\Gamma)$ defined in \eqref{eq:log_der_cor}, and we denote its supremum by $Q_*:=\sup_{\Gamma\in(0,1]}Q_{\mathrm{froz}}(\Gamma)\in[0,\infty].$

\begin{proposition}
Fix $\varepsilon>0$ and $0<\alpha<1$. Then the unique positive solution of
\eqref{ode_full} with $x_k^{\mathrm{froz}}(0)=\frac{\varepsilon}{k^\alpha}, k\geq 2$ and $x_1^{\mathrm{froz}}(t)\equiv\varepsilon$ satisfies:
\begin{enumerate}
\item If $\alpha<\frac27$, then for every fixed $t>0$, $x_k^{\mathrm{froz}}(t)\to\infty$ as $k\to\infty.$
\item Assume in addition that $Q_*<5$. Then,
for every fixed $t>0$,
\begin{equation}
\label{eq:xk_upper_frozen}
x_k^{\mathrm{froz}}(t)
\le
\frac{C(\varepsilon,t,\alpha,Q_*,\Lambda_{\mathrm{froz}}(1))}{k^\alpha}\,
(k-1)^{\frac{2}{5-Q_*}(1-\alpha)}.
\end{equation}
In particular, if $\alpha>\frac{2}{7-Q_*},$
then $x_k^{\mathrm{froz}}(t)\to0$ as $k\to\infty$.
\end{enumerate}
\end{proposition}

\begin{remark}[How the profile geometry enters the decay threshold]
The quantity $Q_*$ depends only on the frozen coefficient
$\Lambda_{\mathrm{froz}}$, hence only on the frozen profile geometry from
\S~\ref{variable_coeff_intro}. In the localized regime discussed there, the
logarithmic-derivative bounds for $\Lambda_{\mathrm{froz}}$ show that $Q_*$ can be made
arbitrarily small by 
choosing the frozen profile support inside a cone of sufficiently small slope.
Indeed, arguing in a similar way as in the proof of Lemma~\ref{lem:comp_to_Psi}, one can obtain
    \begin{align*}
        Q_{\mathrm{froz}}(\Gamma) \leq 15 \frac{(\Gamma L_+)^2}{1+(\Gamma L_+)^2}
        \leq 15 \frac{L^2_+}{1+L^2_+}, 
        \quad L_+:=\sup _{(r,z)\in \supp \phi  } \frac{|z|}{r}.
    \end{align*}
Thus, for any fixed
$\alpha>\frac27$, one can choose the frozen profile geometry so that
$\alpha>\frac{2}{7-Q_*}$ and therefore the decay conclusion in part~(2) applies.
\end{remark}

\begin{proof}
As in the flattened case, the system is triangular and therefore has a unique global
positive solution.

Define $S_0^{\mathrm{froz}}(t):=0.$
Using the definition of $S_k^{\mathrm{froz}}$, the frozen model equation \eqref{ode_full} becomes
$\frac{d}{dt}\log x_k^{\mathrm{froz}}(t)=S_{k-1}^{\mathrm{froz}}(t)$ for $k\geq 2$.
Also,$\bigl(\Gamma_j^{\mathrm{froz}}\bigr)'(t)
=
-3\,\Gamma_j^{\mathrm{froz}}(t)\,S_{j-1}^{\mathrm{froz}}(t).$
Differentiating $b_j^{\mathrm{froz}}$ yields \eqref{eq:log_der_cor} for $j\geq 1$.
(The case $j=1$ is harmless because $S_0^{\mathrm{froz}}\equiv0$.)  As before,
\begin{equation}
\label{eq:app_frozen_sum_identity}
\sum_{j=1}^k b_j^{\mathrm{froz}}(t)\,S_{j-1}^{\mathrm{froz}}(t)
=
\frac12\Bigl(\bigl(S_k^{\mathrm{froz}}(t)\bigr)^2
-\sum_{j=1}^k \bigl(b_j^{\mathrm{froz}}(t)\bigr)^2\Bigr).
\end{equation}

\medskip
\noindent\textbf{Step 1: norm inflation when $\alpha<\frac27$.}
Since $Q_{\mathrm{froz}}(\Gamma)\ge0$, from \eqref{eq:log_der_cor} we get
\[
\bigl(b_j^{\mathrm{froz}}\bigr)'(t)\ge
-5\,b_j^{\mathrm{froz}}(t)\,S_{j-1}^{\mathrm{froz}}(t).
\]
Summing over $1\le j\le k$ and using \eqref{eq:app_frozen_sum_identity}, we obtain
\[
\bigl(S_k^{\mathrm{froz}}\bigr)'(t)
\ge
-\frac52\Bigl(\bigl(S_k^{\mathrm{froz}}(t)\bigr)^2
-\sum_{j=1}^k \bigl(b_j^{\mathrm{froz}}(t)\bigr)^2\Bigr)
\ge
-\frac52\bigl(S_k^{\mathrm{froz}}(t)\bigr)^2.
\]
The rest of Step 1 proceeds in the same way as in Step 2 of Proposition~\ref{prop:flat_dichotomy} for the flat model.
Therefore, exactly as in the proof of Proposition~\ref{prop:flat_dichotomy},
\begin{equation}
\label{eq:app_frozen_intS_lower}
\int_0^t S_k^{\mathrm{froz}}(s)\,ds
\ge
\frac25\log\!\Bigl(1+\frac52\,t\,S_k^{\mathrm{froz}}(0)\Bigr).
\end{equation}
Since $\Gamma_j^{\mathrm{froz}}(0)=1$,
it holds $
b_j^{\mathrm{froz}}(0)
=
\Lambda_{\mathrm{froz}}(1)\,x_j^0
=
\Lambda_{\mathrm{froz}}(1)\frac{\varepsilon}{j^\alpha}, $
and hence
\[
S_k^{\mathrm{froz}}(0)
=
\Lambda_{\mathrm{froz}}(1)\varepsilon\sum_{j=1}^k \frac{1}{j^\alpha}
\geq 
C(\Lambda_{\mathrm{froz}}(1), \varepsilon,\alpha)
k^{1-\alpha}.
\]
Using $\frac{d}{dt}\log x_k^{\mathrm{froz}}(t)=S_{k-1}^{\mathrm{froz}}(t)$ and \eqref{eq:app_frozen_intS_lower}, we obtain
\begin{equation*}
x_k^{\mathrm{froz}}(t)
\geq
C(\Lambda_{\mathrm{froz}}(1),\varepsilon,t,\alpha)
\, 
k^{-\alpha}
\,
(k-1)^{\frac25(1-\alpha)}.
\end{equation*}
Thus, if $\alpha<\frac27$, then $x_k^{\mathrm{froz}}(t)\to\infty$ as $k\to\infty$ for every
fixed $t>0$.

\medskip
\noindent\textbf{Step 2: decay when $\alpha>\frac{2}{7-Q_*}$.}
Assume now that $Q_*<5$. Then
\eqref{eq:log_der_cor} gives
\begin{align}
\label{eq:eq_bj_froz7}
\bigl(b_j^{\mathrm{froz}}\bigr)'(t)
\le
-(5-Q_*)\,b_j^{\mathrm{froz}}(t)\,S_{j-1}^{\mathrm{froz}}(t)
\le 0.
\end{align}
Therefore,
\[
0\le b_j^{\mathrm{froz}}(t)\le b_j^{\mathrm{froz}}(0)
=
\Lambda_{\mathrm{froz}}(1)\frac{\varepsilon}{j^\alpha}
\le \Lambda_{\mathrm{froz}}(1)\varepsilon=: \Lambda^*
\]
In particular,
\[
\sum_{j=1}^k \bigl(b_j^{\mathrm{froz}}(t)\bigr)^2
\le
\Lambda^*\sum_{j=1}^k b_j^{\mathrm{froz}}(t)
=
\Lambda^*\, S_k^{\mathrm{froz}}(t).
\]
Using \eqref{eq:eq_bj_froz7} and \eqref{eq:app_frozen_sum_identity}, we get
\[
\bigl(S_k^{\mathrm{froz}}\bigr)'(t)
\le
-\frac{5-Q_*}{2}
\Bigl(
\bigl(S_k^{\mathrm{froz}}(t)\bigr)^2-\Lambda^*\, S_k^{\mathrm{froz}}(t)
\Bigr).
\]
Let $U_k$ solve
\[
U_k'(t)
=
-\frac{5-Q_*}{2}\bigl(U_k(t)^2-\Lambda^*\, U_k(t)\bigr),
\qquad
U_k(0)=S_k^{\mathrm{froz}}(0).
\]
By scalar comparison, $S_k^{\mathrm{froz}}(t)\le U_k(t)$ for all $t\ge0.$
The explicit formula for $U_k$ is
\[
U_k(t)
=
\frac{\Lambda^*}{1-\Bigl(1-\frac{\Lambda^*}{S_k^{\mathrm{froz}}(0)}\Bigr)e^{-(5-Q_*)\Lambda^*t/2}},
\]
and therefore
\[
\int_0^t S_k^{\mathrm{froz}}(s)\,ds
\le
\frac{2}{5-Q_*}
\log\!\left(1+
\frac{e^{(5-Q_*)\Lambda^*t/2}-1}{\Lambda^*}
S_k^{\mathrm{froz}}(0)\right).
\]
Using
\[
S_{k-1}^{\mathrm{froz}}(0)
=
\Lambda_{\mathrm{froz}}(1)\varepsilon\sum_{j=1}^{k-1} j^{-\alpha}
\leq C(\Lambda_{\mathrm{froz}}(1),\varepsilon,\alpha)
\, 
(k-1)^{1-\alpha},
\]
together with $\frac{d}{dt}\log x_k^{\mathrm{froz}}(t)=S_{k-1}^{\mathrm{froz}}(t)$, we obtain
\[
x_k^{\mathrm{froz}}(t)
=
x_k^0\exp\!\Bigl(\int_0^t S_{k-1}^{\mathrm{froz}}(s)\,ds\Bigr)
\le
\frac{\varepsilon}{k^\alpha}
\Bigl(1+ \frac{e^{(5-Q_*)\Lambda^*t/2}-1}{\Lambda^*}
S_{k-1}^{\mathrm{froz}}(0)\Bigr)^{\frac{2}{5-Q_*}},
\]
which yields \eqref{eq:xk_upper_frozen}. Finally,
if $
\frac{2}{5-Q_*}(1-\alpha)<\alpha
$, or equivalently $
\alpha>\frac{2}{7-Q_*},$
then the right-hand side of \eqref{eq:xk_upper_frozen} tends to zero as
$k\to\infty$.
\end{proof}

\section{List of notation}
\label{sec:notation}

For the convenience of the reader, we collect here the main notation that recurs throughout the paper. This appendix is intended as a global lookup table for symbols used repeatedly in the statements and main proofs. Proof-local symbols and locally introduced constants are usually not listed here.

\begin{itemize}
\item \textbf{Biot--Savart law and sign convention.}
We define vorticity by $\operatorname{curl}u=\omega$ in the general case and $\operatorname{curl}u= \omega e_\theta$ in the axisymmetric, no-swirl case. This determines the signs in the Biot--Savart law as indicated in \eqref{eq:gen_Biot-Savart} and \eqref{eq:axis_BS_law}, and we will use these definitions and the Biot--Savart laws with these signs. For the correct sign of the general Biot--Savart law, see \cite[the first page]{EncisoGarciaFerreroPeraltaSalas18} and \cite[Equation (1.44)]{Tsai18}. 

On the other hand, if one defines vorticity with the minus sign, that is, $\operatorname{curl}u=-\omega$ in the general case and $\operatorname{curl}u= -\omega e_\theta$ in the axisymmetric, no-swirl case. Then it will lead to an extra minus sign in the Biot--Savart law. This convention is used for example in \cite[Equations (1.3), (1.5), (1.6)]{KimJeong22}.

The displayed Biot--Savart formulas in
\cite[Equation (3)]{CordobaMartinezZheng25},
\cite[Equations (2.94), (2.95)]{MajdaBertozzi01},
\cite[Equation (1.53)]{BedrossianVicol22}, and
\cite[p.~477]{Danchin07}
use the cross-product order \((x-y)\times\omega(y)\). This is the negative of
the convention used in the present paper.
Consequently, those displayed Biot--Savart formulas correspond to the replacement
\(\omega_{\rm ref}=-\omega\) when compared with the present convention.

\item \textbf{Outer-region notation for  Proposition~\ref{prop:cone_free_key_lem}}
(Section~\ref{sec:cone_free_key_lem}).
\begin{itemize}
    \item $Q(\rho):=\{y\in\mathbb R^3:\ r_y\ge2\rho\}$:
the outer region with radial cutoff $\rho$ in Proposition~\ref{prop:cone_free_key_lem}.
\end{itemize}

\item \textbf{Physical variables and coordinates.}
\begin{itemize}
    \item $u=u^r e_r+u^z e_z$: the axisymmetric no-swirl velocity field.
   \item $\operatorname{curl}u=\omega(r,z,t)e_\theta$: the vorticity vector where $\omega$ denotes the scalar angular vorticity.
    \item $\xi:=\omega/r$: the relative vorticity.
    \item $(r,\theta,z)$: cylindrical coordinates in $\mathbb R^3$.
    \item $(r_x,\theta_x,z_x)$: the cylindrical coordinates of a point $x\in\mathbb R^3$.
\end{itemize}

\item \textbf{Parameter hierarchy and fixed background data.}
\begin{itemize}
    \item $\tilde\varepsilon>0$, $q>1$: the fundamental parameters.
    \item $L_q$, $\alpha_q$, $\phi_q$, $z_{0,q}$: quantities fixed once $q$ is fixed.
    \begin{itemize}
        \item $L_q=z_{0,q}/r_0$ is the cone-slope parameter.
        \item $\alpha_q$ is the logarithmic decay exponent.
        \item $\phi_q$ is the fixed reference profile.
        \item $z_{0,q}$ is the corresponding reference height.
    \end{itemize}
    \item $r_0,\eta,\mu,d$: globally fixed geometric parameters. In the paper they are taken to be
    \[
        r_0=1,\qquad \eta=\frac14,\qquad \mu=\frac1{20},\qquad d=10^{-2}.
    \]
    \item $\varepsilon=\varepsilon(\tilde\varepsilon,q)$: the amplitude parameter chosen after fixing $(\tilde\varepsilon,q)$ and thereafter treated as fixed.
    \item $k,m,A,t,T_N,T_B$: running parameters.
    \item $k$: the ring index; smaller $k$ corresponds to a larger, more outer ring.
    \item  $m$: the total number of rings in the finite superposition. 
\end{itemize}

\item \textbf{Initial data and multiring decomposition}
(mainly Section~\ref{sec:ini_data_ansatz} and Section~\ref{sec:inst_bu}).
\begin{itemize}
    \item $\omega_0^{(m)}=\sum_{k=1}^m \omega_{0,k}$: the finite-ring initial data; in the instantaneous blow-up argument one also uses the infinite sum  $\omega_0^{(\infty)}=\sum_{k=1}^\infty \omega_{0,k}$.
    \item $\omega^{(m)}$: the smooth finite-ring solution generated by $\omega_0^{(m)}$.
    \item $\omega_k$: the evolved $k$-th vortex ring.
    \item $u^{(m)}$: the velocity induced by $\omega^{(m)}$.
    \item $u_k$: the velocity induced by $\omega_k$.
    \item For a fixed $k$, $\omega_-:=\sum_{j=1}^{k-1}\omega_j$ and $u_-:=\sum_{j=1}^{k-1}u_j$
    denote the outer-ring contributions, while $\omega_+:=\sum_{j=k+1}^m\omega_j$ and $u_+:=\sum_{j=k+1}^mu_j$
    denote the inner-ring contributions.
    \item $\xi^{(m)}:=\omega^{(m)}/r$, $\xi_0^{(m)}:=\omega_0^{(m)}/r$: relative-vorticity variables.
\end{itemize}

\item \textbf{Rescaled ring variables and local flow maps}
(Section~\ref{sec:ini_data_ansatz}).
\begin{itemize}
    \item $x_k(t)$: the amplitude factor of the $k$-th ring.
    \item $\tilde x_k(t):=x_k(t)/x_k(0)$: the normalized amplitude.
    \item $R_k(t)=d^k\tilde x_k(t)$: the radial scale of the $k$-th ring.
    \item $H_k(t)=d^k\tilde x_k(t)^{-2}$: the vertical scale of the $k$-th ring.
    \item $W_k(r,z,t)$: the rescaled profile of the $k$-th ring, defined through
    \[
        \omega_k(r,z,t)=x_k(t)\,
        W_k\!\left(\frac{r}{R_k(t)},\frac{z}{H_k(t)},t\right).
    \]
    \item $V_k=(V_k^r,V_k^z)$: the rescaled velocity field driving $W_k$.
    \item $X_k=(X_k^r,X_k^z)$: the flow map associated with $V_k$ on the support of $\phi_q$.
\end{itemize}

\item \textbf{Bootstrap quantities}
(Section~\ref{roadmap_boot_Euler} and \S~\ref{subsec:Euler_bootstrap_improvement}).
\begin{itemize}
    \item $A>0$: the target amplitude for norm inflation.
    \item $\bar A:=A/(1-\mu)$: the renormalized target amplitude.
    \item $T_N(A,m)$: the ODE norm-inflation time, i.e. the time at which  some $x_k$ reaches $\bar A$.
    \item $F_m(t)$: the bootstrap deviation functional measuring the relative distortion of $X_k^r/r$ and $X_k^z/z$ on $\operatorname{supp}\phi_q$.
    \item $T_B(A,m)$: the maximal time on which the bootstrap bounds $F_m(t)\le \mu$ hold.
\end{itemize}
\item \textbf{ODE/stretching cascade notation}
(Section~\ref{sec:ODEs_Euler}; see also \S~\ref{ODE_intro} and Appendix~\ref{app:ODE_models}).
\begin{itemize}
    \item \emph{Superscript/subscript convention.}
    \begin{itemize}
        \item Superscript $\mathrm{orig}$ denotes quantities belonging to the \emph{original weak ODEs}, i.e.\ the PDE-driven cascade with the true coefficients $\Lambda_j(t,\Gamma)$.
        \item Subscript $\mathrm{froz}$ means that the transported profile has been \emph{frozen} against the fixed profile $\phi_q$, so $\Lambda_j(t,\Gamma)$ is replaced by $\Lambda_{\mathrm{froz}}(\Gamma)$ while the amplitude and aspect-ratio variables remain those of the original cascade.
        \item Superscript $\mathrm{loc}$ denotes the \emph{localized  model}, where $\Lambda_{\mathrm{froz}}(\Gamma)$ is replaced by the explicit localized coefficient $\Lambda_{\mathrm{loc}}(\Gamma)$.
        \item In the heuristic discussion and Appendix~\ref{app:ODE_models}, superscripts such as $\mathrm{froz}$, $\mathrm{flat}$, and $\mathrm{str}$ denote the stand-alone frozen-profile, flattened-coefficient, and strong no-self-slowdown model ODEs, respectively.
    \end{itemize}

    \item In Section~\ref{sec:ODEs_Euler} we temporarily write $x_k^{\mathrm{orig}}(t):=x_k(t)$
    (and similarly for the derived ODE quantities) in order to distinguish the original cascade from its simplified models. After that section, the paper returns to the shorter notation $x_k(t)$.

    \item
    $
        \Gamma_k^{\mathrm{orig}}(t)
        :=\frac{H_k(t)}{R_k(t)}
        =\left(\frac{x_k^{\mathrm{orig}}(0)}{x_k^{\mathrm{orig}}(t)}\right)^3
    $
    : the aspect ratio of the $k$-th ring in the original weak ODEs.

    \item
    $
        K_\Gamma(r,z):=\frac{r^2 z}{(r^2+\Gamma^2 z^2)^{5/2}}
    $
    : the rescaled Biot--Savart kernel appearing in the ODE coefficients.

    \item
    $
        \Lambda_j(t,\Gamma):=\iint K_\Gamma(r,z)\
        \big(-W_j(r,z,t)\big)\,dr\,dz
    $
    : the \emph{unfrozen} Biot--Savart coefficient coming from the true transported profile $W_j$.

    \item
    $
        \Lambda_{\mathrm{froz}}(\Gamma):=\iint K_\Gamma(r,z)\,\big(-\phi_q(r,z)\big)\,dr\,dz
    $
    : the \emph{frozen} Biot--Savart coefficient obtained by replacing $W_j(\cdot,\cdot,t)$ with the fixed profile $\phi_q$.

    \item
    $
        \Lambda_{\mathrm{loc}}(\Gamma)
        :=\frac{r_0^2 z_{0,q}}{(r_0^2+\Gamma^2 z_{0,q}^2)^{5/2}}
    $
    : the explicit localized coefficient obtained from the cone localization of $\phi_q$.

    \item
    \[
        b_j^{\mathrm{orig}}(t)
        :=x_j^{\mathrm{orig}}(t)\,\Gamma_j^{\mathrm{orig}}(t)^2
        \Lambda_j\!\bigl(t,\Gamma_j^{\mathrm{orig}}(t)\bigr),
    \]
    \[
        S_k^{\mathrm{orig}}(t):=\sum_{j=1}^k b_j^{\mathrm{orig}}(t),
        \qquad
        B_j^{\mathrm{orig}}(t):=\int_0^t b_j^{\mathrm{orig}}(s)\,ds
    \]
    : the true outer-ring contribution, its cumulative stretching rate, and its time integral.

    \item
    \[
        b_{\mathrm{froz},j}^{\mathrm{orig}}(t)
        :=x_j^{\mathrm{orig}}(t)\,\Gamma_j^{\mathrm{orig}}(t)^2
        \Lambda_{\mathrm{froz}}(\Gamma_j^{\mathrm{orig}}(t))
        =x_j^{\mathrm{orig}}(0)\,\Gamma_j^{\mathrm{orig}}(t)^{5/3}
        \Lambda_{\mathrm{froz}}(\Gamma_j^{\mathrm{orig}}(t)),
    \]
    \[
        S_{\mathrm{froz},k}^{\mathrm{orig}}(t):=\sum_{j=1}^k b_{\mathrm{froz},j}^{\mathrm{orig}}(t)
    \]
    : the frozen-coefficient comparison quantities associated with the original cascade.

    \item
    \[
        b_j^{\mathrm{loc}}(t)
        :=x_j^{\mathrm{loc}}(t)\,\Gamma_j^{\mathrm{loc}}(t)^2
        \Lambda_{\mathrm{loc}}(\Gamma_j^{\mathrm{loc}}(t)),
    \]
    \[
        S_k^{\mathrm{loc}}(t):=\sum_{j=1}^k b_j^{\mathrm{loc}}(t),
        \qquad
        B_j^{\mathrm{loc}}(t):=\int_0^t b_j^{\mathrm{loc}}(s)\,ds
    \]
    : the localized-model outer contribution, cumulative stretching rate, and time integral.

    \item$
        \kappa_{\mathrm{froz}}(\Gamma):=5+3\Gamma\frac{\Lambda_{\mathrm{froz}}'(\Gamma)}{\Lambda_{\mathrm{froz}}(\Gamma)}
    $
    : the coefficient governing the differential equation for $b_{\mathrm{froz},j}^{\mathrm{orig}}$.

    \item
    $
        c_\mu:=\frac{(1-\mu)^3}{(1+\mu)^5},
        \, 
        C_\mu:=\frac{(1+\mu)^3}{(1-\mu)^5}
    $
    : the bootstrap comparison constants between the true and frozen outer contributions.

    \item
    $
        \Theta_\mu:=\frac{C_\mu}{c_\mu}
        =\left(\frac{1+\mu}{1-\mu}\right)^8
    $
    : the comparison ratio appearing in the frozen/unfrozen bootstrap inequalities.

    \item
    $
        \Psi(\zeta):=\zeta^{5/3}(1+\zeta^2)^{-5/2}
    $
    : the one-parameter Biot--Savart profile arising after localization.

    \item
    $
        L_{q,-}:=\frac{1-\eta}{1+\eta}L_q,$ and $
        L_{q,+}:=\frac{1+\eta}{1-\eta}L_q
    $
    : lower and upper cone-slope bounds coming from the support of $\phi_q$.

    \item
    $
        \zeta_k^{\mathrm{orig}}(t):=L_q\,\Gamma_k^{\mathrm{orig}}(t),$ and $
        \zeta_k^{\mathrm{loc}}(t):=L_q\,\Gamma_k^{\mathrm{loc}}(t)
    $
    : the cone variables in the original and localized cascades.

    \item
    $
        \zeta_*:=\frac{1}{\sqrt2}
    $
    : the threshold where $\Psi$ changes monotonicity.

    \item
    $
        \zeta_\eta:=\frac{1+\eta}{1-\eta}\,\zeta_*$
    : the fixed productive threshold used in the front-migration argument.

    \item
    $
        J_{\zeta_\eta}^{\mathrm{orig}}(t)
        :=\max\Big\{1\le j\le m-1:\ \zeta_j^{\mathrm{orig}}(t)\ge \zeta_\eta\Big\}$
    : the front index in the original cascade.

    \item
    $
        t_{\zeta_\eta}^{\mathrm{orig}}
        :=\inf\Big\{t>0:\ \zeta_m^{\mathrm{orig}}(t)=\zeta_\eta\Big\}$
    : the threshold-hitting time of the last ring.
    \item
\[
\gamma_q:=\frac{3\Theta_\mu}{\log(L_q/\zeta_\eta)},
\qquad
\beta_q:=\frac{1}{1+\gamma_q}-\alpha_q
=
\frac{\log(L_q/\zeta_\eta)}
{3\Theta_\mu+\log(L_q/\zeta_\eta)}
-\alpha_q.
\]
Here $\beta_q>0$ is the exponent giving the algebraic decay
$T_B(A,m)\lesssim m^{-\beta_q}$.
\end{itemize}

\item \textbf{Global flow and radial-extremum notation}
(Section~\ref{sec:prel_lem} and Appendix~\ref{app:two_lemmas}).
\begin{itemize}
    \item $Y(a,t)$: the global Lagrangian flow map of the bounded log-Lipschitz velocity field $u$ obtained under the contradiction hypothesis.
    \item $Y_t^{-1}=Y^{-1}(\cdot,t)$: the inverse map of $Y_t:=Y(\cdot,t)$.
    \item $R_K(t):=\max_{x\in K}Y^r(x,t)$ and $r_K(t):=\min_{x\in K}Y^r(x,t)$
    are the outer and inner radial envelopes of $K$ under the flow for a compact set $K\subset \mathbb R^3$.
    \item $\mathcal M_K(t):=\operatorname*{argmax}_{x\in K}Y^r(x,t)$
    and $\mathcal A_K(t):=\operatorname*{argmin}_{x\in K}Y^r(x,t)$
    : the maximizing and minimizing sets used in the Danskin-type lemma.
\end{itemize}

\item \textbf{Finite-head contradiction notation}
(Section~\ref{subsec:non_existence}).
\begin{itemize}
    \item $\omega_0^{(\infty)}=\sum_{k=1}^\infty \omega_{0,k}$: the infinite-ring initial data.
    \item
    $\omega^{(\infty)}, u^{(\infty)}$: the infinite-ring solution and the velocity field induced by $\omega^{(\infty)}.$
    \item
    An integer $M\geq 2$ is the finite head cutoff used in the contradiction/bootstrap argument.
    \item $
        \omega^{(\le M)}:=\sum_{j=1}^M \omega_j$ and $\omega^{(>M)}:=\sum_{j=M+1}^\infty \omega_j$
    : the head and tail parts of the vorticity. The same notation is used for the corresponding velocity decomposition $u^{(\infty)}=u^{(\le M)}+u^{(>M)}.$
    \item $V_{k,(\le M)}$: the rescaled velocity built only from the head part $u^{(\le M)}$.
    \item $K_M:=\operatorname{supp}\omega_{0,M}$
    : the initial support of the $M$-th ring.
    \item
    \[
        K_{\mathrm{tail}}:=\overline{\bigcup_{j=M+1}^\infty \operatorname{supp}\omega_{0,j}}
    \]
    : the initial tail set.
    \item $
        r_{M,-}(t):=\min_{x\in K_M}Y^r(x,t)$
    : the inner radial boundary of the $M$-th head ring.
    \item$
        R_{\mathrm{tail}}(t):=\max_{x\in K_{\mathrm{tail}}}Y^r(x,t)$
    : the outer radial boundary of the tail.
    \item $\mathcal A_M(t):=\operatorname*{argmin}_{x\in K_M}Y^r(x,t)$ and $
        \mathcal M_{\mathrm{tail}}(t):=\operatorname*{argmax}_{x\in K_{\mathrm{tail}}}Y^r(x,t)$
    : the extremizing sets used in the head-tail separation argument.
\end{itemize}
\end{itemize}

This appendix complements the notation conventions from \S~\ref{sec:out_paper_not}; generic constants and suppressed dependences are governed there.

\section*{Acknowledgements}
The authors are grateful to In-Jee Jeong for introducing us to the ill-posedness problem beyond Danchin's regime, and to both In-Jee Jeong and Xiaoyutao Luo for stimulating discussions.

\bibliographystyle{alpha}
	
 \bibliography{mybib}
\end{document}